\documentclass[11pt,twoside,reqno]{amsart}
\usepackage{academic}
\theoremstyle{definition}
\newtheorem*{declaration}{Declaration}

\DeclareMathOperator{\dimconf}{dim_{conf}}

\newcommand{\uuu}{\mathtt{u}}
\newcommand{\vvv}{\mathtt{v}}
\newcommand{\cmax}{\lambda_{\max}}
\newcommand{\cmin}{\lambda_{\min}}
\newcommand{\ol}{\overline}
\newcommand{\hi}{\iii'}
\newcommand{\hj}{\jjj'}

\begin{document}

\title{Exponential separation of self-conformal systems}

\author{Antti K\"aenm\"aki}
\address[Antti K\"aenm\"aki]
        {Department of Physics and Mathematics \\
         University of Eastern Finland \\
         P.O.\ Box 111 \\
         FI-80101 Joensuu \\
         Finland}
\email{antti@kaenmaki.net}

% 28A80 - Fractals
% 37C45 - Dimension theory of smooth dynamical systems
% 30C35 - General theory of conformal mappings
% 37C20 - Generic properties, structural stability of dynamical systems
% 37F35 - Conformal densities and Hausdorff dimension for holomorphic dynamical systems
\date{\today}

\begin{abstract}
  On the real line, B\'ar\'any, Kolossv\'ary, and Troscheit \cite{BaranyKolossvaryTroscheit} made exponential separation of analytic self-conformal systems checkable through a dual iterated function system on a space of analytic functions. The quantity their condition separates is the pre-Schwarzian derivative, and this observation lets us carry the construction to conformal iterated function systems in every dimension, where the pre-Schwarzian is $T_f=\nabla\log\|Df\|$. A pointwise separation of the pre-Schwarzian cocycle implies the strong exponential separation condition, and in dimensions one and two separation of the Schwarzian cocycle implies the same condition modulo M\"obius maps. Each hypothesis is equivalent to a uniform gap condition, and these are $\mathcal C^2$-open. On the line they are also dense, so both separation conditions are $\mathcal C^2$-generic there; this sharpens the genericity theorem of B\'ar\'any, Kolossv\'ary, and Troscheit, whose open and dense set carries only the plain condition. In the plane they are dense on Jordan domains with simply connected extension domains, and both separation conditions are $\mathcal C^2$-generic there too. In dimensions at least three every conformal map is M\"obius, so no system satisfies the condition modulo M\"obius maps, and the pre-Schwarzian hypothesis becomes a pole-separation condition, dense when the generators contract strongly enough and stay away from the similarities, but not dense in general. On Jordan domains, the planar genericity and a theorem of Feng and Rapaport \cite{FengRapaport} settle the dimension drop conjecture for $\mathcal C^2$-generic planar systems with injective generators.
\end{abstract}

\maketitle

\tableofcontents

\section{Introduction}\label{sec:intro}

\subsection{Dimension drop and exponential separation} \label{sec:intro-drop}

An \emph{iterated function system} on $\R^d$ is a finite tuple $\Phi=(\fii_i)_{i\in\II}$, indexed by a finite alphabet $\II=\{1,\ldots,N\}$, $N \ge 2$, of strictly contracting maps of a domain into itself. By a theorem of Hutchinson \cite{Hutchinson} it has a unique nonempty compact set $X$ satisfying $X=\bigcup_{i\in\II}\fii_i(X)$, and, for a probability vector $\mathbf p=(p_i)_{i\in\II}$, a unique Borel probability measure $\mu_{\mathbf p}$ with $\mu_{\mathbf p}=\sum_{i\in\II}p_i\mu_{\mathbf p}\circ\fii_i^{-1}$, supported on $X$. The maps are \emph{conformal} when the derivative $D\fii_i$ is at every point a positive multiple of an orthogonal matrix; the set $X$ is then a \emph{self-conformal set} and the measure $\mu_{\mathbf p}$ is a \emph{self-conformal measure}. For such systems both the Hausdorff dimension $\dimh(X)$ of $X$ and the dimension $\dim(\mu_{\mathbf p})$ of $\mu_{\mathbf p}$ have natural upper bounds, the \emph{conformality dimension} $\dimconf(\Phi)$ for the set (see e.g.~\cite{MauldinUrbanski1996}) and the ratio $H(\mathbf p)/\chi(\mathbf p)$ of entropy to Lyapunov exponent for the measure (see e.g.~\cite{FengHu}), each capped at the ambient dimension $d$. These bounds can be strict when the pieces $\fii_i(X)$ overlap, and whether overlaps other than exact ones can force such a drop is the \emph{dimension drop conjecture}.

Exponential separation is a natural hypothesis under which no dimension drop occurs. For self-similar systems on the line, Hochman \cite{Hochman2014} proved that the exponential separation condition, which asks that distinct length-$n$ compositions stay at least $c^n$ apart for some fixed $c>0$, is a mild and verifiable hypothesis that rules out a dimension drop. It is satisfied for example if the parameters defining the self-similar set are algebraic and there are no exact overlaps; see \cite{FengFeng, Rapaport_ExactOverlaps, RapaportVarju, Varju} for generalizations in the self-similar setting. The condition is nevertheless strictly stronger than the absence of exact overlaps: Baker \cite{Baker} and B\'ar\'any and K\"aenm\"aki \cite{BaranyKaenmaki} exhibited self-similar systems on the line with no exact overlaps that are not exponentially separated. Rapaport \cite{Rapaport_SelfConfESC} extended Hochman's result to analytic self-conformal systems on the real line, showing that under exponential separation, both dimensions attain their upper bounds. What his theorem leaves unresolved is the verification of exponential separation for a concrete system: for self-similar systems the algebraic criterion above renders this routine, reducing it to the absence of exact overlaps, whereas for non-linear analytic systems it is genuinely difficult.

\subsection{The dual system and the Schwarzian hierarchy} \label{sec:intro-dual}

On the line, B\'ar\'any, Kolossv\'ary, and Troscheit \cite{BaranyKolossvaryTroscheit} introduced a verifiable sufficient condition for exponential separation through a remarkable device: a \emph{dual} iterated function system acting on a space of analytic functions. Its dual canonical projection attaches to a word $i_1\cdots i_n$ over $\II$ the explicit sum
\begin{equation*}
  \sum_{k=1}^{n}\biggl(\frac{\fii_{i_k}''}{\fii_{i_k}'}\circ\fii_{i_{k-1}}\circ\cdots\circ\fii_{i_1}\biggr)(\fii_{i_{k-1}}\circ\cdots\circ\fii_{i_1})'
\end{equation*}
along the orbit of the partial compositions of the word. They identify this sum with the derivative ratio $f''/f'$ of the full composition $f=\fii_{i_n}\circ\cdots\circ\fii_{i_1}$. That ratio is the \emph{pre-Schwarzian derivative}, on the line the scalar $T_f=f''/f'$, and recognizing their projection as such is where the present paper starts. The pre-Schwarzian is the first-order, affine counterpart of the classical Schwarzian derivative, on the line $S_f=(f''/f')'-\tfrac12(f''/f')^2$: the Schwarzian is the fundamental M\"obius invariant, vanishing exactly on the M\"obius maps, whereas the pre-Schwarzian is the fundamental affine invariant, vanishing exactly on the affine maps. Naming the projection in this way opens both of the extensions pursued here: to every dimension, where the pre-Schwarzian is $T_f=\nabla\log\|Df\|$, and one group level higher, to the Schwarzian $S_f = DT_f-T_fT_f^{\top}+\tfrac12|T_f|^2I$. Read in these terms, their theorems say that a pointwise separation of the pre-Schwarzian cocycle implies the strong exponential separation condition and that the systems satisfying it contain a $\mathcal{C}^2$-open and dense subset.

The absence of a dimension drop under exponential separation has since been established beyond the self-similar systems of the line: Hochman \cite{Hochman} proved it for self-similar systems throughout $\R^d$, Hochman and Solomyak \cite{HochmanSolomyak} for random products of M\"obius maps of the projective line, and Feng and Rapaport \cite{FengRapaport} for planar self-conformal sets and measures. Each takes exponential separation as a standing hypothesis whose verification, for a concrete non-linear system, it leaves open. This paper carries the dual construction of B\'ar\'any, Kolossv\'ary, and Troscheit to conformal iterated function systems on a bounded domain $\Omega\subset\R^d$ in every dimension, and one group level higher, from the affine invariant to the M\"obius invariant; in passing it simplifies the proof of their line theorem. Conformality is precisely what permits the extension: since $D\fii_i$ is a scalar multiple of an orthogonal matrix, the pre-Schwarzian $T_f=\nabla\log\|Df\|$ is an abelian cocycle. Two regimes arise. In the plane, that is for holomorphic maps of a domain in $\C$, the derivative is a nonzero scalar and the proofs become complex-analytic, the line theorem being recovered from the planar one by complexification. In dimension $d\ge3$, Liouville's theorem \cite{Blair} forces every conformal map to be the restriction of a M\"obius transformation, so the systems lie in a finite-dimensional family, the pre-Schwarzian reduces to a pole field, and separation becomes a separation of poles. The genuinely non-abelian case, of maps whose derivative has distinct singular values, falls outside this framework and remains open; \cref{rem:anisotropic} records that it is the composition law, not the definition of the operators, that confines the theory to the conformal class.

Our three main results build on one another, from separation through genericity to dimension. First, \cref{thm:main} gives verifiable sufficient conditions for the strong exponential separation condition, by the pre-Schwarzian in every dimension and by the Schwarzian, one weight higher, on the line and in the plane. Second, \cref{thm:genericity} establishes the genericity of these conditions, unconditional on the line and on planar \emph{Jordan domains}, bounded domains whose boundaries are Jordan curves, with simply connected extension domains, and, necessarily, conditional in dimensions at least three; on the line, the condition it makes generic is strictly stronger than the one carried by the open and dense set of B\'ar\'any, Kolossv\'ary, and Troscheit. Third, \cref{thm:main-dim} draws the dimensional consequence on the same planar domains: combined with a theorem of Feng and Rapaport \cite{FengRapaport}, the genericity above makes the absence of a dimension drop generic among the systems with injective generators.

\subsection{Sufficient conditions for exponential separation} \label{sec:intro-separation}

We work throughout with a \emph{conformal iterated function system} $\Phi=(\fii_i)_{i\in\II}$, whose generators $\fii_i$ send $\ol\Omega$ into $\Omega$, have derivative of norm below one on $\ol\Omega$, and extend conformally, and real-analytically when $d=1$, to a neighborhood of $\ol\Omega$; the precise definition is given in \cref{sec:cifs}. We write $\cmax=\max_{i \in \II}\sup_{x\in\ol{\Omega}}\|D\fii_i(x)\|$ for the largest conformal factor of a generator on $\ol{\Omega}$, so that the second requirement reads $\cmax<1$. The dual construction assigns to $\Phi$, for every finite or infinite word $\iii\in\II^\N\cup\II^*$, the \emph{pre-Schwarzian cocycle} $T_\iii$, built from $T_{f_\iii}=\nabla\log\|Df_\iii\|$, where $f_\iii$ is the composition of generators $\fii_i$ read backwards along $\iii$: on the line it is the scalar $f_\iii''/f_\iii'$, and in the plane the complex conjugate of the same holomorphic quantity. Its vanishing detects similarities, so a coincidence $T_\iii=T_\jjj$ forces $f_\iii$ and $f_\jjj$ to agree up to a similarity, and exact overlaps are a special case. One group level higher the same construction assigns the \emph{Schwarzian cocycle} $S_\iii$, which in the plane is the classical complex Schwarzian of $f_\iii$. Its vanishing detects M\"obius maps, a wider family than the similarities, so a coincidence $S_\iii=S_\jjj$ forces $f_\iii$ and $f_\jjj$ to agree up to a M\"obius map. 

Our first main result states that ruling out the pre-Schwarzian coincidences between equal-length words forces the \emph{strong exponential separation condition}, under which distinct words of length $n$ give maps at least $c^n$ apart in the supremum norm, for a fixed $c>0$ and every $n$, and that ruling out the Schwarzian coincidences forces the \emph{strong exponential separation condition modulo M\"obius maps}, which no post-composition by a M\"obius transformation can defeat.

\begin{theorem} \label{thm:main}
  Let $\Phi=(\fii_i)_{i\in\II}$ be a conformal iterated function system on $\Omega \subset \R^d$.
  \begin{enumerate}
    \item\label{it:main-pre} If $d \in \N$ and 
    \begin{equation} \label{eq:mainhyp}
      \sup_{x\in\ol{\Omega}}|T_\iii(x)-T_\jjj(x)|>0
    \end{equation}
    for all distinct $\iii,\jjj\in\II^\N\cup\II^*$ with $|\iii|=|\jjj|$, then $\Phi$ satisfies the strong exponential separation condition.
    \item\label{it:main-schw} If $d\in\{1,2\}$ and
    \begin{equation} \label{eq:mainhyp2}
      \sup_{x\in\ol{\Omega}}|S_\iii(x)-S_\jjj(x)|>0
    \end{equation}
    for all distinct $\iii,\jjj\in\II^\N\cup\II^*$ with $|\iii|=|\jjj|$, then $\Phi$ satisfies the strong exponential separation condition modulo M\"obius maps.
  \end{enumerate}
\end{theorem}

\Cref{thm:main}\cref{it:main-pre} is proved in \cref{sec:C2-pre} for the plane, in \cref{sec:C1-pre} for the line, and in \cref{sec:C3-pre} for $d\ge3$, and \cref{thm:main}\cref{it:main-schw} in \cref{sec:C2-mob} for the plane and in \cref{sec:C1-mob} for the line. Restricted to the line it contains the sufficient condition of B\'ar\'any, Kolossv\'ary, and Troscheit \cite[Theorem~1.5]{BaranyKolossvaryTroscheit}, as recorded in \cref{rem:BKT-comparison}, by a proof that passes through the plane. Where their argument transfers closeness of maps to closeness of every derivative of the pre-Schwarzian through a Taylor estimate backed by a Fa\`a di Bruno apparatus and finishes with a two-case endgame, a single Cauchy estimate performs all of these differentiations at once, and one application of Rouch\'e's theorem makes the endgame free of cases, as \cref{rem:cleaner} explains; the line is recovered afterward by lifting the analytic system to a complex neighborhood and appealing to the planar theorem. In dimension $d\ge3$ the function theory disappears entirely: contraction pushes every pole a definite distance from $\ol\Omega$, and this pole gap, with the rigidity of the M\"obius class, stands in for every analytic estimate, as \cref{rem:rigid} describes.

The Schwarzian hypothesis is the stronger of the two, forbidding the larger family of coincidences and earning the stronger conclusion; in the plane the weaker hypothesis already yields separation modulo similarities, as \cref{prop:similarities} shows. It is confined to dimensions one and two because the Schwarzian cannot replace the pre-Schwarzian in dimension $d\ge3$: there every conformal map is M\"obius and has vanishing Schwarzian, while the pre-Schwarzian still records its pole. Neither hypothesis is checkable as it stands, both being quantified over all equal-length pairs of words, but each follows from a gap between the projections of the generators at a single point, large against the geometric tail of the cocycle. This sufficient condition on the generators alone is the one-point gap criterion of \cref{prop:gencrit}, available at weight one in every dimension and at weight two for dimensions one and two, and shown by \cref{rem:gencrit-necessity} not to be necessary. On the line, its weight-one half is due to B\'ar\'any, Kolossv\'ary, and Troscheit \cite[Proposition~1.8]{BaranyKolossvaryTroscheit}. In \cref{ex:C2} we run its weight-two half on an explicit planar system of three polynomial maps, obtaining a concrete system that satisfies the strong exponential separation condition, plain and modulo M\"obius maps.

\subsection{Genericity of the conditions} \label{sec:intro-genericity}

For a dimension $d\ge1$, a bounded domain $\Omega\subset\R^d$, an extension domain $\Omega'\supset\ol{\Omega}$, and the alphabet $\II$, write $\mathcal S$ for the space of conformal iterated function systems on $\Omega$ with extension domain $\Omega'$, metrized by the $\mathcal C^2$-distance between generators over $\ol{\Omega}$. The pointwise hypotheses \cref{eq:mainhyp,eq:mainhyp2} require only that each distance $\sup_{x\in\ol{\Omega}}|T_\iii(x)-T_\jjj(x)|$, respectively $\sup_{x\in\ol{\Omega}}|S_\iii(x)-S_\jjj(x)|$, be positive, and these distances have infimum zero, since two words with a long common prefix have exponentially close projections; over the pairs with distinct first letters, however, the infimum can be positive. \Cref{sec:gaps} attaches to $\mathcal S$ the two resulting uniform gap conditions, at weight one in every dimension and at weight two in dimensions one and two; the weight-two condition implies the weight-one one, both are $\mathcal C^2$-open, and each is equivalent to the corresponding hypothesis of \cref{thm:main}, so making a gap condition dense makes that hypothesis generic, and with it the separation it forces. On the line the weight-two condition is dense with no further hypothesis, so the strong exponential separation condition modulo M\"obius maps is $\mathcal C^2$-generic there; this strictly sharpens the theorem of B\'ar\'any, Kolossv\'ary, and Troscheit \cite[Theorem~1.4]{BaranyKolossvaryTroscheit}, whose open and dense set carries the plain condition. In the plane density costs a hypothesis on the domain, the Cauchy estimates forbidding the interior localization that the line allows and pushing the perturbations to the boundary. In dimensions at least three the result is conditional, and necessarily so. These three cases are the content of our second main theorem.

\begin{theorem} \label{thm:genericity}
  Let $\Omega\subset\R^d$ be a bounded domain and $\Omega'$ be its extension domain.
  \begin{enumerate}
    \item\label{it:gen-line} If $d=1$, then the systems in $\mathcal S$ satisfying the strong exponential separation condition modulo M\"obius maps contain a $\mathcal C^2$-open and dense subset, and so do those satisfying the strong exponential separation condition.
    \item\label{it:gen-plane} If $d=2$, then the same two statements hold whenever $\Omega$ is a Jordan domain and $\Omega'$ is simply connected.
    \item\label{it:gen-higher} If $d \ge 3$, then no system satisfies the strong exponential separation condition modulo M\"obius maps, while the systems in $\mathcal S$ satisfying the strong exponential separation condition contain a $\mathcal C^2$-open set whose closure contains every system with $\cmax<N^{-2/d}$ whose generators are bounded on $\Omega'$ and satisfy
    \begin{equation} \label{eq:genhigher}
      \sup_{x\in\ol{\Omega}}|T_{\fii_i}(x)| > \frac{4}{\min_{j \in \II}\dist(\fii_j(\ol{\Omega}),\partial\Omega)}\sqrt{\frac{\cmax}{1-\cmax}}
    \end{equation}
    for every $i\in\II$.
  \end{enumerate}
\end{theorem}

\Cref{thm:genericity} is proved in \cref{sec:gen-line-dense} for the line, in \cref{sec:gen-plane-dense} for the plane, and in \cref{sec:gen-higher-dense} for dimensions at least three, in each case right after the density theorem of its section. Each proof combines two ingredients established separately: the openness of the gap classes, proved in \cref{sec:gen-line-open,sec:gen-plane-open,sec:gen-higher-open} from a Lipschitz estimate for the dual projections that is uniform over the words, and their density, the main result of each of the three sections; the counting lemma and the small-ball lemma behind the two transversality arguments are isolated in \cref{sec:gaps}. On the line, \cref{rem:BKT} records how the plain half of \cref{it:gen-line} and the open-dense theorem of B\'ar\'any, Kolossv\'ary, and Troscheit \cite[Theorem~1.4]{BaranyKolossvaryTroscheit} differ, in the space, in the conclusion, and in the method.

The hypotheses on the domain in \cref{it:gen-plane} are those of the method: \cref{ex:nonconvex} shows that on an annulus the parametrization through which the boundary perturbations act is unavailable, the argument principle locking the periods of the pre-Schwarzians, and \cref{q:nonconvex} asks whether the conclusion nevertheless holds on multiply connected domains. The first half of \cref{it:gen-higher}, by contrast, is no limitation of the method but a fact about the systems: for $d\ge3$ every composition $\fii_\iii$ is the restriction of a M\"obius transformation, so for distinct equal-length words the M\"obius map $M=\fii_\iii\circ\fii_\jjj^{-1}$ satisfies $M\circ\fii_\jjj=\fii_\iii$ on $\ol{\Omega}$, and the separation modulo M\"obius maps fails outright; this is \cref{ex:esc-not-mobius}(1). In the second half, \cref{eq:genhigher} keeps the generators quantitatively away from the similarities, on which the pre-Schwarzian vanishes, and the bound $\log N/\log(1/\cmax)<d/2$ implies $\dimm(X)<d/2$ for the Minkowski dimension of the self-conformal set $X$, which exists by Falconer \cite[Theorem~4]{Falconer1989}. It is the bound on $\cmax$ itself, and not the dimension bound it implies, that the transversality argument uses, the count there running over pairs of words rather than over the points they build; it would be interesting to know whether $\dimm(X)<d/2$ suffices. The statement is \cref{cor:C3dense}, a special case of the slightly more general \cref{thm:C3dense}, whose contraction threshold the pole geometry of the system sets and which \cref{rem:C3gen} shows to be sufficient but not necessary. Some hypothesis of this kind is needed, by \cref{prop:nondense}: when $\Omega$ is convex, the gap condition that produces the open set is not dense in $\mathcal S$, failing on an open set of near-homothetic systems with $\cmax>N^{-1/d}$. What this rules out is the density of the gap condition, not that of the separation itself, which we leave undecided, and \cref{q:C3dense} asks whether the gap condition is nevertheless dense in the open set of systems with $\cmax<N^{-1/d}$; the finite-word half of it is dense in the open set of systems whose generators have pairwise distinct poles, with no hypothesis on the domains or on $\cmax$, by \cref{prop:finite}.

\subsection{Consequences for dimension} \label{sec:intro-dimension}

These separation results feed directly into dimension. In the plane a theorem of Feng and Rapaport \cite{FengRapaport} turns exponential separation into the sharp dimensions of the self-conformal set and of every self-conformal measure of a system whose generators are injective on $\ol{\Omega}$, provided the system preserves neither a point nor a real-analytic curve and is not holomorphically conjugate to a homothetic system. Fed through our $\mathcal{C}^2$-genericity, their theorem shows that, on Jordan domains with simply connected extension domains, a planar system with injective generators generically suffers no dimension drop. Write $\mathcal S'$ for the subspace of $\mathcal S$ consisting of the conformal iterated function systems whose generators are injective on $\ol{\Omega}$.

\begin{theorem} \label{thm:main-dim}
  Let $\Omega$ be a Jordan domain and let $\Omega'$ be simply connected. Then the conformal iterated function systems in $\mathcal S'$ satisfying 
  \begin{equation*}
    \dimh(X)=\min\{2,\dimconf(\Phi)\}\qquad\text{and}\qquad\dim(\mu_{\mathbf p})=\min\biggl\{2,\frac{H(\mathbf p)}{\chi(\mathbf p)}\biggr\}
  \end{equation*}
  for every positive probability vector $\mathbf p$ contain a $\mathcal{C}^2$-open and dense subset.
\end{theorem}

\Cref{thm:main-dim} is the final assertion of \cref{thm:dim}, which we prove in \cref{sec:dimension} by verifying the hypotheses of the theorem of Feng and Rapaport on an open and dense subset of $\mathcal S'$. This subset can be chosen so that its members also satisfy the strong exponential separation condition modulo M\"obius maps. The density is relative to $\mathcal S'$, which need not be dense in $\mathcal S$: by Rouch\'e's theorem, the systems with a generator that is not injective on $\Omega$ form an open set, and \cref{rem:injectivity-restriction} shows that it can be nonempty.

Feng and Rapaport \cite{FengRapaport} pose the problem of extending the criterion of B\'ar\'any, Kolossv\'ary, and Troscheit to holomorphic systems, an extension which would produce explicit non-self-similar planar systems whose sets and measures have the expected dimensions. This paper solves it: \cref{thm:main} with \cref{prop:gencrit} is that extension and \cref{ex:C2dim} works out an explicit example. \Cref{thm:main-dim} goes further and makes the dimension formulas generic. Feng and Rapaport themselves reach a genericity statement of a different kind in \cite[Corollary~1.8]{FengRapaport}, along a non-degenerate one-parameter family and outside a set of parameters of Hausdorff dimension zero; \cref{rem:FR-genericity} compares the two, and neither implies the other.

\subsection{Organization of the paper} \label{sec:intro-organization}

The preliminaries consist of \cref{sec:schwarzian} on the two operators, \cref{sec:cifs-esc} on conformal iterated function systems and the separation conditions, and \cref{sec:dual-ifs} on the dual construction; the last of these ends with the reduction of \cref{thm:main} to a coincidence of limit projections and with the one-point gap criterion. \Cref{thm:main} is proved in \cref{sec:C2,sec:C1,sec:C3}, the plane before the line because the line is reached from it by complexification, and \cref{thm:genericity} in \cref{sec:gen-line,sec:gen-plane,sec:gen-higher}, after the common gap framework and transversality argument of \cref{sec:gaps}. The planar dimension theorem, \cref{thm:main-dim}, is derived in \cref{sec:dimension}.

\section{Pre-Schwarzian and Schwarzian derivatives} \label{sec:schwarzian}

Both operators of this paper derive from one scalar quantity, the logarithm of the operator norm of the derivative. We define them in their natural generality, for maps that need not be conformal, and then isolate the conformal maps as the class on which the operators compose exactly, with the similarities and the M\"obius maps as their respective kernels. Throughout the paper, a \emph{domain} is a nonempty connected open set, and throughout this section $\Omega\subseteq\R^d$ is one. In \cref{sec:schwarzian-defs} we define the two operators for an arbitrary $\mathcal{C}^2$-map, and in \cref{sec:conformal-maps} we record the classification of the conformal maps by the ambient dimension and compute both operators in closed form in the two regimes where conformality is a rigidity. The remaining two subsections isolate the properties on which the rest of the paper rests, and they are used in different places: in \cref{sec:composition} we prove the composition laws, which make the operators affine cocycles of weight one and weight two and which \cref{sec:dual-ifs} dualizes, and in \cref{sec:kernels} we compute the kernels, which are what turns a coincidence of two projections into a statement about the maps themselves and which \cref{sec:C2,sec:C1,sec:C3} consume. The Schwarzian, and the weight-two theory built upon it, bears on the separation results only in the plane and on the line, through the finer separation modulo M\"obius maps of \cref{thm:main}\cref{it:main-schw}; in dimension $d\ge3$ it vanishes identically, every conformal map being M\"obius.

\subsection{The operators in the general \texorpdfstring{$\mathcal{C}^2$}{C2}-setting} \label{sec:schwarzian-defs}

Let $f\colon\Omega\to\R^d$ be a $\mathcal{C}^2$-map. The \emph{pre-Schwarzian derivative} of $f$ is the vector field
\begin{equation} \label{eq:Tdef}
  T_f = \nabla\log\|Df\|,
\end{equation}
defined at every point at which the operator norm of the Jacobian $Df$, given by $\|Df(x)\|=\sup_{|v|=1}|Df(x)v|$, is positive, $\|Df(x)\|>0$, and the function $x\mapsto\log\|Df(x)\|$ is differentiable. For $d=1$ the norm $|f'|$ is the only singular value and \cref{eq:Tdef} is the classical pre-Schwarzian $T_f=(\log|f'|)'=f''/f'$, defined wherever $f'\ne0$. In higher dimensions the positivity alone does not suffice. Since $\|Df(x)\|$ is the largest singular value of $Df(x)$ and singular values are $1$-Lipschitz functions of the matrix, the function $\log\|Df\|$ is locally Lipschitz on the set $\{x\in\Omega : \|Df(x)\|>0\}$, so $T_f$ exists at least almost everywhere there.

Wherever $T_f$ is in turn differentiable, the \emph{Schwarzian derivative} of $f$ is the matrix field
\begin{equation} \label{eq:Sdef}
  S_f = DT_f-T_fT_f^{\top}+\tfrac12|T_f|^2I,
\end{equation}
built from the pre-Schwarzian and one derivative of it alone, so that $S_f$ is of second order in $\log\|Df\|$ and hence of third order in $f$. Implicit in this is that $T_f$ is defined on a neighborhood of any point at which it is differentiated, so that $S_f$ lives on the interior of the set on which $T_f$ exists and not merely on that set. At such a point the partial derivatives of $\log\|Df\|$ exist in a neighborhood and are differentiable at it, so the mixed second partial derivatives coincide by the theorem of Young \cite{Young1909}; see \cite[Theorem~12.12]{Apostol1974} for a proof and \cite{Minguzzi2015} for the hierarchy of criteria for the equality of mixed partial derivatives. Hence $DT_f$ is the Hessian of $\log\|Df\|$ and is symmetric, and the remaining two terms of \cref{eq:Sdef} are symmetric by inspection, so that $S_f$ takes values in the symmetric matrices. For $d=1$ the last two terms combine and \cref{eq:Sdef} is the classical Schwarzian $S_f=T_f'-\tfrac12T_f^2=(f''/f')'-\tfrac12(f''/f')^2$.

We single out three linear-algebraic notions attached to the symmetric matrix $S_f$. At each point we read $S_f$ as a \emph{tensor}, the symmetric bilinear form $(v,w)\mapsto v^{\top}S_fw$ or equivalently the quadratic form $v\mapsto v^{\top}S_fv$ on vectors $v,w\in\R^d$; this quadratic reading is intrinsic, the matrix transforming as $S_f\mapsto O^{\top}S_fO$ when the coordinates are rotated by $O\in O(d)$, and the \emph{trace} $\operatorname{tr}S_f$ is unchanged by such a rotation. The splitting $S_f=\tfrac1d(\operatorname{tr}S_f)I+(S_f-\tfrac1d(\operatorname{tr}S_f)I)$ into a multiple of the identity and a \emph{trace-free} remainder separates the average stretching common to every direction from the anisotropy, how unequally $S_f$ stretches different directions. When $d=1$ there is a single direction and hence no anisotropy to record, so the trace-free remainder vanishes and only the scalar $\operatorname{tr}S_f$ remains. The trace-free part of \cref{eq:Sdef} is the Schwarzian tensor that Osgood and Stowe \cite[(1.4)]{OsgoodStowe} associate to the function $\log\|Df\|$; we keep the trace: for $d=1$ nothing else survives, and in the proof of \cref{lem:kernel} below it is the trace that singles out the conformal factors of the M\"obius maps from the wider class of conformal factors on which the trace-free tensor alone vanishes, those of the M\"obius metrics of \cite{OsgoodStowe}.

The construction so far asks only that $\log\|Df\|$ be twice differentiable, and one might attempt it away from the conformal class, on the open set where the largest singular value of $Df$ is simple; that set is fragile, however, and \cref{rem:anisotropic} records why we do not pursue it. Conformality lies at the \emph{isotropic} extreme, where all singular values coincide identically, and no exceptional set need be removed from $\Omega$; it is the class on which the two operators are globally defined, as the classification below records, and the composition laws \cref{eq:Tchain} and \cref{eq:Schain} below show that they compose exactly on it. This is the assumption we now adopt.

\subsection{Conformal maps and their classification} \label{sec:conformal-maps}

A $\mathcal{C}^1$-map $f\colon\Omega\to\R^d$ is \emph{conformal} if at every point the derivative is a positive multiple of an orthogonal matrix, $Df=\lambda_fO_f$ with the \emph{conformal factor} $\lambda_f>0$ and $O_f\in O(d)$ pointwise; equivalently $(Df)^{\top}Df=\lambda_f^2I$, so $f$ preserves angles. All singular values of $Df$ then equal $\lambda_f$, so the norm $\|Df\|=\lambda_f=|\det Df|^{1/d}$ is exactly as regular as $Df$, the singular values coinciding identically instead of crossing transversally. For $d=1$ we require in addition that $f$ be real-analytic. This is a genuine restriction special to the line, where conformality asks only that $f'$ not vanish and carries no regularity beyond the assumed $\mathcal{C}^1$; in higher dimensions the classification below makes a conformal map real-analytic of its own accord. We impose it here mainly for convenience, so that a conformal map has, in every dimension, the smoothness the results of this section need, in particular the derivatives the composition laws of \cref{lem:chain} below require. Real-analyticity is needed in earnest from \cref{lem:dual} onwards, where on the line it supplies the complex neighborhood to which the generators, and with them the dual projections, extend; the identity principle that closes the proof of \cref{prop:gencrit} rests on this analyticity, and so does the reduction of the line to the plane in \cref{sec:C2}. For $d=2$ we require in addition that $f$ be orientation-preserving, $O_f\in SO(2)$; this makes $Df$ complex-linear, so by the Cauchy--Riemann equations a planar conformal map is a holomorphic map with $f'\neq0$. The added condition rules out the anti-holomorphic maps, which reverse orientation while preserving angles; for such a map $\partial_zf=0$, so the scalar pre-Schwarzian $f''/f'$ used below is unavailable. An anti-holomorphic map $f$ has the same conformal factor $\lambda_f$, hence the same $T_f$ and $S_f$, as the holomorphic map $z\mapsto\ol{f(z)}$, but this does not identify their iterated dynamics, and orientation-reversing systems remain outside the scope of this paper.

The conformal maps are classified by dimension. For $d=1$ a conformal map is, by the added assumption, a real-analytic map of nonvanishing derivative, with $\lambda_f=|f'|$; conformality on the line is otherwise no rigidity at all. For $d=2$, identifying $\R^2$ with $\C$, the conformal maps are, as just observed, the holomorphic maps of nonvanishing derivative, with $\lambda_f=|f'|$; for these, \cref{lem:scalar} below computes both operators from the classical complex quantities, identifying $T_f$ with the scalar pre-Schwarzian $f''/f'$ up to complex conjugation and $S_f$ with the classical complex Schwarzian 
\begin{equation*}
  S_f^{\C} = \biggl(\frac{f''}{f'}\biggr)'-\frac12\biggl(\frac{f''}{f'}\biggr)^2.
\end{equation*}
For $d\ge3$ rigidity takes over: by Liouville's theorem \cite[Theorem~5.5]{Blair} every conformal map of a domain in $\R^d$ is the restriction of a M\"obius transformation of $\R^d\cup\{\infty\}$, that is a finite composition of similarities $x\mapsto\lambda Ox+b$ and inversions. In every dimension a conformal map is thus real-analytic, on the line by assumption and for $d\ge2$ by the classification, so $T_f=\nabla\log\lambda_f$ and, one derivative up, $S_f$ are defined on all of $\Omega$.

\begin{lemma} \label{lem:scalar}
  Let $\Omega\subseteq\C$ be a domain and $f\colon\Omega\to\C$ conformal. Then
  \begin{equation} \label{eq:scalarpreschwarzian}
    T_f = \ol{\biggl(\frac{f''}{f'}\biggr)} \qquad\text{and}\qquad v^{\top}S_fv = \re(S_f^{\C}v^2),
  \end{equation}
  where the second identity holds for all $v\in\C$. The matrix $S_f$ is symmetric and trace-free, hence determined by the scalar $S_f^{\C}$.
\end{lemma}

\begin{proof}
  For holomorphic $f$ the derivative $Df$ is multiplication by the complex number $f'$, hence $|f'|$ times a rotation, so $\|Df\|=|f'|$ and $T_f=\nabla\log|f'|$. Fix $z_0\in\Omega$. Since $f'$ is holomorphic and nonvanishing, it has a holomorphic logarithm on a neighborhood of $z_0$. Write this branch as $\log f'=\log|f'|+i\arg f'$. For holomorphic functions the complex derivative equals the $x$-derivative, so
  \begin{equation*}
    \frac{f''}{f'} = \frac{\mathrm{d}}{\mathrm{d}z}\log f' = \partial_x\log f' = \partial_x\log|f'|+i\partial_x\arg f',
  \end{equation*}
  and the Cauchy--Riemann equations for $\log f'$ give $\partial_x\arg f'=-\partial_y\log|f'|$, so that $f''/f'=\partial_x\log|f'|-i\partial_y\log|f'|$. The gradient $T_f=\nabla\log|f'|=(\partial_x\log|f'|,\partial_y\log|f'|)$ is the complex number $\partial_x\log|f'|+i\partial_y\log|f'|$, the complex conjugate of the previous quantity. Hence $T_f=\ol{f''/f'}$.

  The second identity is \cite[(1.5)]{OsgoodStowe}, where the computation is only indicated; we include it for completeness. Set $u=\log|f'|$; then $T_f=\nabla u$, and $u$, the real part of the holomorphic $\log f'$, is harmonic, so $u_{yy}=-u_{xx}$. Thus $DT_f=D^2u$ is the symmetric Hessian, and \cref{eq:Sdef} gives $S_f=D^2u-T_fT_f^{\top}+\tfrac12|T_f|^2I$; writing subscripts for the partial derivatives of $u$, its entries are
  \begin{equation*}
    (S_f)_{11}=u_{xx}-\tfrac12(u_x^2-u_y^2)=-(S_f)_{22} \qquad\text{and}\qquad (S_f)_{12}=u_{xy}-u_xu_y,
  \end{equation*}
  so $S_f$ is trace-free, and $v^{\top}S_fv=\re(sv^2)$ with $s=(S_f)_{11}-i(S_f)_{12}$. By the first part $f''/f'=u_x-iu_y$ is holomorphic, so $\frac{\mathrm{d}}{\mathrm{d}z}(f''/f')=u_{xx}-iu_{xy}$ and
  \begin{align*}
    S_f^{\C}=(f''/f')'-\tfrac12(f''/f')^2 &= (u_{xx}-iu_{xy})-\tfrac12(u_x-iu_y)^2 \\
    &= u_{xx}-\tfrac12(u_x^2-u_y^2)-i(u_{xy}-u_xu_y)=s.
  \end{align*}
  Hence $v^{\top}S_fv=\re(S_f^{\C}v^2)$, which is the second identity of \cref{eq:scalarpreschwarzian}.
\end{proof}

The companion of \cref{lem:scalar} for $d\ge3$ is a rigidity statement. There is no counterpart of the scalar $S_f^{\C}$ to compute: by Liouville's theorem the conformal maps are the restrictions of M\"obius transformations, and on the M\"obius maps, in every dimension, both operators collapse to closed forms, the pre-Schwarzian retaining only the pole, and the Schwarzian retaining nothing.

\begin{lemma} \label{lem:mobius}
  Let $\Omega\subseteq\R^d$ be a domain and $f\colon\Omega\to\R^d$ the restriction of a M\"obius transformation of $\R^d\cup\{\infty\}$, with which we identify it. If $f$ is a similarity, then $\lambda_f$ is constant and $T_f\equiv0$. Otherwise $f$ has a unique pole $p=f^{-1}(\infty)\notin\Omega$ and conformal factor $\lambda_f(x)=\rho_f/|x-p|^2$ for a constant $\rho_f>0$, and
  \begin{equation} \label{eq:pole}
    T_f(x) = -\frac{2(x-p)}{|x-p|^2}\qquad\text{and}\qquad DT_f(x) = -\frac{2}{|x-p|^2}I + T_f(x)T_f(x)^{\top}.
  \end{equation}
  In both cases $S_f\equiv0$. In particular, for $d\ge3$ the Schwarzian vanishes identically on the entire conformal class, while the pre-Schwarzian retains the pole through \cref{eq:pole}.
\end{lemma}

\begin{proof}
  A M\"obius transformation is a bijection of $\R^d\cup\{\infty\}$, and a finite composition of similarities and inversions. An inversion has a unique finite pole, while a composition of similarities is a similarity, so either $f$ fixes $\infty$ and is a similarity, with constant $\lambda_f$ and $T_f=\nabla\log\lambda_f\equiv0$, or the pole $p=f^{-1}(\infty)$ is a unique finite point, not in $\Omega$ since $f$ maps $\Omega$ into $\R^d$. In the latter case the composition $g=f\circ\iota_p$ with the inversion $\iota_p(x)=p+(x-p)/|x-p|^2$ fixes $\infty$ and is therefore a similarity $g(x)=\rho_fOx+b$; since $\iota_p$ is an involution, $f=g\circ\iota_p$, and since $D\iota_p(x)$ is $|x-p|^{-2}$ times the reflection $I-2\widehat u\widehat u^{\top}$, where $u=x-p$ and $\widehat u=u/|u|$, the conformal factor is $\lambda_f(x)=\rho_f/|x-p|^2$. Taking the gradient of $\log\lambda_f=\log\rho_f-2\log|x-p|$ gives the first identity of \cref{eq:pole}. Differentiating gives
  \begin{equation*}
    DT_f(x) = -2D(u|u|^{-2}) = -2|u|^{-2}I + 4|u|^{-4}uu^{\top} = -2|u|^{-2}I + T_f(x)T_f(x)^{\top},
  \end{equation*}
  which is the second identity. Since $|T_f|^2=4/|x-p|^2$, substituting \cref{eq:pole} into \cref{eq:Sdef} cancels all three terms, so $S_f\equiv0$; for a similarity every term of \cref{eq:Sdef} vanishes with $T_f$. For the trace-free part of $S_f$, both cases are also contained in \cite{OsgoodStowe}: the homotheties have vanishing Schwarzian tensor by \cite[Theorem~2.2]{OsgoodStowe}, and $1/\lambda_f=|x-p|^2/\rho_f$ is of the form in \cite[Lemma~2.5]{OsgoodStowe}. The last claim follows since for $d\ge3$, by Liouville's theorem, every conformal map of a domain in $\R^d$ is the restriction of a M\"obius transformation; see also \cite[Theorem~6.4]{OsgoodStowe}, where this is proved through the Schwarzian tensor.
\end{proof}

\subsection{Composition laws} \label{sec:composition}

Both operators are cocycles on the conformal class, satisfying exact composition laws. Since a conformal map is real-analytic in every dimension, as recorded after the classification, the composition laws below need no additional $\mathcal{C}^3$-regularity hypothesis.

\begin{lemma} \label{lem:chain}
  Let $\Omega,\Omega'\subseteq\R^d$ be domains and let $f\colon\Omega\to\R^d$ and $g\colon\Omega'\to\R^d$ be conformal maps with $f(\Omega)\subseteq\Omega'$. Then $g\circ f$ is conformal and satisfies
  \begin{align}
    \lambda_{g\circ f} &= (\lambda_g\circ f)\lambda_f, \label{eq:lambdachain} \\
    T_{g\circ f} &= (Df)^{\top}(T_g\circ f)+T_f, \label{eq:Tchain} \\
    S_{g\circ f} &= (Df)^{\top}(S_g\circ f)(Df)+S_f. \label{eq:Schain}
  \end{align}
\end{lemma}

\begin{proof}
  The chain rule factorizes $D(g\circ f)=(Dg\circ f)Df=(\lambda_g\circ f)\lambda_f(O_g\circ f)O_f$, and a product of orthogonal matrices is orthogonal, of rotations again a rotation, so $g\circ f$ is conformal with the multiplicative conformal factor $\lambda_{g\circ f}=(\lambda_g\circ f)\lambda_f$ giving \cref{eq:lambdachain}. Taking first the logarithm of this identity and then the gradient, using the chain rule $\nabla(h\circ f)=(Df)^{\top}(\nabla h\circ f)$ for the scalar $h=\log\lambda_g$, gives \cref{eq:Tchain}, the pre-Schwarzian of a conformal map being the gradient of the logarithm of its conformal factor.

  For the trace-free part, \cref{eq:Schain} is \cite[Lemma~2.1 and (2.2)]{OsgoodStowe}; we give a self-contained Euclidean proof that also covers the trace. Differentiating \cref{eq:Tchain} calls for the second derivative of $f$, and conformality pins it down. Write $B_{ikl}=\langle\partial_i\partial_kf,\partial_lf\rangle$, so that $B_{ikl}=B_{kil}$ by the symmetry of second derivatives. Differentiating the entries $\langle\partial_if,\partial_kf\rangle=\lambda_f^2\delta_{ik}$ of the conformality identity $(Df)^{\top}Df=\lambda_f^2I$ with respect to $x_l$, and using $\partial_l\lambda_f^2=2\lambda_f^2(T_f)_l$ from $T_f=\nabla\log\lambda_f$, gives $B_{ilk}+B_{kli}=2\lambda_f^2(T_f)_l\delta_{ik}$, the two terms on the left being $\langle\partial_l\partial_if,\partial_kf\rangle=B_{lik}=B_{ilk}$ and $\langle\partial_if,\partial_l\partial_kf\rangle=B_{lki}=B_{kli}$. Replacing $(i,k,l)$ by its cyclic permutations $(k,l,i)$ and $(l,i,k)$ gives the two further identities
  \begin{align*}
    B_{kil}+B_{lik} &= 2\lambda_f^2(T_f)_i\delta_{kl}, \\
    B_{lki}+B_{ikl} &= 2\lambda_f^2(T_f)_k\delta_{il}.
  \end{align*}
  Adding these two and subtracting the identity itself,
  \begin{equation*}
    (B_{kil}+B_{lik})+(B_{lki}+B_{ikl})-(B_{ilk}+B_{kli}) = 2B_{ikl},
  \end{equation*}
  the symmetry in the first two indices turning $B_{lik}$ and $B_{lki}$ into $B_{ilk}$ and $B_{kli}$, which cancel the subtracted pair, while $B_{kil}=B_{ikl}$ contributes a second copy of the term already present. Combining the right-hand sides with the same signs and halving gives $B_{ikl}=\lambda_f^2((T_f)_i\delta_{kl}+(T_f)_k\delta_{il}-(T_f)_l\delta_{ik})$. Since the columns $\lambda_f^{-1}\partial_lf$ of $\lambda_f^{-1}Df$ form an orthonormal basis of $\R^d$, expanding the second derivative in this basis gives the \emph{conformality relation}
  \begin{equation} \label{eq:confrel}
    \partial_i\partial_kf = \lambda_f^{-2}\sum_{l}B_{ikl}\partial_lf = (T_f)_k\partial_if+(T_f)_i\partial_kf-\delta_{ik}(Df)T_f,
  \end{equation}
  the last term collecting $\sum_l(T_f)_l\partial_lf=(Df)T_f$. The relation \cref{eq:confrel} is the rule relating the Riemannian connections of two conformal metrics, see \cite[(2.0)]{OsgoodStowe}, applied to the Euclidean metric and its multiple by $\lambda_f^2$, which is the pullback of the Euclidean metric under $f$.

  Now write $w=(Df)^{\top}(T_g\circ f)$, so that $T_{g\circ f}=w+T_f$ by \cref{eq:Tchain} and $\langle(Df)T_f,T_g\circ f\rangle=\langle T_f,w\rangle$. Differentiating the components $w_i=\langle\partial_if,T_g\circ f\rangle$ by the product rule, the chain rule $\partial_k(T_g\circ f)=(DT_g\circ f)\partial_kf$, and \cref{eq:confrel} gives $\partial_kw_i=(T_f)_kw_i+(T_f)_iw_k-\delta_{ik}\langle T_f,w\rangle+((Df)^{\top}(DT_g\circ f)Df)_{ik}$, in matrix form
  \begin{equation*}
    Dw = (Df)^{\top}(DT_g\circ f)Df+wT_f^{\top}+T_fw^{\top}-\langle T_f,w\rangle I.
  \end{equation*}
  Substituting $T_{g\circ f}=w+T_f$ and $DT_{g\circ f}=Dw+DT_f$ into \cref{eq:Sdef} and expanding,
  \begin{align*}
    S_{g\circ f} &= Dw+DT_f-(w+T_f)(w+T_f)^{\top}+\tfrac12|w+T_f|^2I \\
    &= (Df)^{\top}(DT_g\circ f)Df-ww^{\top}+\tfrac12|w|^2I+(DT_f-T_fT_f^{\top}+\tfrac12|T_f|^2I),
  \end{align*}
  the mixed terms $wT_f^{\top}+T_fw^{\top}$ of $Dw$ canceling against those of $-(w+T_f)(w+T_f)^{\top}$, and the term $-\langle T_f,w\rangle I$ of $Dw$ against the cross term $\langle T_f,w\rangle I$ of $\tfrac12|w+T_f|^2I$. The second cancellation is where the coefficient $\tfrac12$ of \cref{eq:Sdef} is spent: with $c|T_f|^2I$ in its place the same computation leaves the defect $(2c-1)\langle T_f,w\rangle I$, so the composition law forces the normalization of the Schwarzian. The parenthesized sum is $S_f$. Using $Df(Df)^{\top} = (Df)^{\top}Df = \lambda_f^2I$, we have $ww^{\top}=(Df)^{\top}(T_g\circ f)(T_g\circ f)^{\top}Df$ and $|w|^2=\lambda_f^2|T_g\circ f|^2$, so that $\tfrac12|w|^2I=(Df)^{\top}(\tfrac12|T_g\circ f|^2I)Df$; the three therefore collect into $(Df)^{\top}(S_g\circ f)Df$, and \cref{eq:Schain} follows.
\end{proof}

We call \cref{eq:Tchain} the \emph{pre-Schwarzian composition law} and \cref{eq:Schain} the \emph{Schwarzian composition law}. The pre-Schwarzian law is affine of weight one and the Schwarzian law of weight two, the weight counting the derivative factors through which the operator of the outer map enters, while the operator of the inner map enters as a translation. For $d=1$ these are the classical composition laws $T_{g\circ f}=(T_g\circ f)f'+T_f$ and $S_{g\circ f}=(S_g\circ f)(f')^2+S_f$. For $d=2$ they persist in the classical scalar quantities of \cref{lem:scalar}, for which we write $T^{\C}_f=f''/f'=\ol{T_f}$ and, as there, $S^{\C}_f$: conjugating \cref{eq:Tchain} and reading \cref{eq:Schain} through the quadratic form of \cref{lem:scalar}, the derivative factors feeding a vector $v$ into the form as $f'v$, give $T^{\C}_{g\circ f}=(T^{\C}_g\circ f)f'+T^{\C}_f$ and $S^{\C}_{g\circ f}=(S^{\C}_g\circ f)(f')^2+S^{\C}_f$.

\subsection{Kernels} \label{sec:kernels}

The kernels are the groups whose deviation the operators measure: the similarities for the pre-Schwarzian and the M\"obius maps for the Schwarzian. On the conformal class each operator determines the map up to post-composition by a member of its kernel. This is what the separation results of \cref{sec:C2,sec:C1,sec:C3} read backwards, a coincidence of two dual projections being converted there into a relation between the two compositions, and it is also what fixes the two regimes of \cref{thm:main}, the wider kernel of the Schwarzian earning the finer conclusion.

\begin{lemma} \label{lem:kernel}
  Let $\Omega\subseteq\R^d$ be a domain and let $f,g\colon\Omega\to\R^d$ be conformal maps.
  \begin{enumerate}
    \item\label{it:tkernel} $T_f\equiv0$ if and only if $f$ is a similarity. More generally, $T_f\equiv T_g$ if and only if $f=S\circ g$ for a similarity $S$.
    \item\label{it:skernel} $S_f\equiv0$ if and only if $f$ is the restriction of a M\"obius transformation of $\R^d\cup\{\infty\}$. More generally, $S_f\equiv S_g$ if and only if $f=M\circ g$ for a M\"obius transformation $M$.
  \end{enumerate}
\end{lemma}

\begin{proof}
  For \cref{it:tkernel}, the pre-Schwarzian is the gradient $T_f=\nabla\log\lambda_f$ on the connected $\Omega$, so $T_f\equiv0$ holds if and only if $\lambda_f$ is constant; a similarity has constant factor, and a conformal map with constant factor is a similarity, by Liouville's theorem and \cref{lem:mobius} for $d\ge3$, for $d=2$ because a holomorphic derivative of constant modulus is constant, and for $d=1$ because $f'$ has constant sign on the connected $\Omega$, so the constancy of $|f'|$ makes $f'$ constant. For the level sets, fix $x_0\in\Omega$. Since $Dg(x_0)$ is invertible, there is a connected neighborhood $U\subseteq\Omega$ of $x_0$ such that $g|_U$ is a conformal diffeomorphism onto the domain $g(U)$. Thus $h=f\circ(g|_U)^{-1}$ is conformal on $g(U)$, and the composition law \cref{eq:Tchain} applied on $U$ to $f=h\circ g$ gives $T_f-T_g=(Dg)^{\top}(T_h\circ g)$. If $T_f\equiv T_g$, then the invertibility of $Dg$ gives $T_h\equiv0$, so the first claim gives a similarity $S$ such that $h=S|_{g(U)}$, and hence $f=S\circ g$ on $U$. Both sides are real-analytic on the connected domain $\Omega$, so the identity principle gives $f=S\circ g$ throughout $\Omega$. Conversely, if $f=S\circ g$, then \cref{eq:Tchain} and $T_S\equiv0$ give $T_f\equiv T_g$.

  For \cref{it:skernel}, the vanishing of the Schwarzian on the M\"obius maps is \cref{lem:mobius}. For the converse, let $d\ge2$ and let $f$ be conformal with $S_f\equiv0$. Following \cite[(1.7) and (2.6)]{OsgoodStowe}, we substitute $v=1/\lambda_f$, which linearizes the trace-free part of the equation. This gives $T_f=-\nabla v/v$ and hence $DT_f=-D^2v/v+(\nabla v)(\nabla v)^{\top}/v^2$, and turns \cref{eq:Sdef} into
  \begin{equation} \label{eq:vsub}
    S_f = -\frac1v\biggl(D^2v-\frac{|\nabla v|^2}{2v}I\biggr),
  \end{equation}
  the rank-one term of $DT_f$ canceling against the subtracted $T_fT_f^{\top}$. Put $a=|\nabla v|^2/(4v)$, a nonnegative real-analytic function on $\Omega$, so that the vanishing of $S_f$ is precisely the Hessian identity $D^2v=2aI$. For each $j \in \{1,\ldots,d\}$, choose $i \in \{1,\ldots,d\}\setminus\{j\}$, which $d\ge2$ permits. Since $\partial_i\partial_iv=2a$ and $\partial_i\partial_jv=0$, equality of mixed third derivatives, justified by the real-analyticity of $v$, gives $2\partial_ja=\partial_j\partial_i\partial_iv=\partial_i\partial_i\partial_jv=0$, and hence $a$ is constant. Integrating $D^2v=2aI$ on the connected domain gives $v(x)=a|x|^2+b\cdot x+e$, the form found in \cite[Lemma~2.5]{OsgoodStowe} for the trace-free equation, and the defining identity $|\nabla v|^2=4av$ reduces to $|b|^2=4ae$. By \cite[Lemma~2.6]{OsgoodStowe} and its proof, the Euclidean metric multiplied by $v^{-2}$ has constant curvature $4ae-|b|^2$, so for a positive function $v$ of this form the trace-free equation alone admits metrics of every constant curvature, among them the spherical and hyperbolic metrics of \cite[(2.12) and (2.13)]{OsgoodStowe}, while the trace selects the flat case. Thus either $a=0$, and then $b=0$ and $v$ is constant, so $f$ is a similarity by \cref{it:tkernel}, or $a>0$ and completing the square gives $v=a|x-p|^2$ with $p=-b/(2a)\notin\Omega$ by the positivity of $v$.

  In the latter case $\lambda_f$ is the conformal factor of an inversion-type M\"obius map, and $f$ itself is M\"obius: for $d\ge3$ by Liouville's theorem, and for $d=2$ because $f'(z)(z-p)^2$ is holomorphic of constant modulus, hence constant, so $f'(z)=c/(z-p)^2$ and $f(z)=q-c/(z-p)$. For $d=1$, put $u=T_f=f''/f'$. The equation $S_f=0$ is $u'=\tfrac12u^2$. If $u$ vanishes at one point, uniqueness gives $u\equiv0$, in which case $f$ is affine. Otherwise $u$ has no zeros and $(1/u)'=-1/2$, so $u(x)=-2/(x-p)$ for some $p\notin\Omega$, where $p\notin\Omega$ since $u$ is defined throughout $\Omega$; integration then gives $f(x)=q+c/(x-p)$. Thus in every dimension $S_f\equiv0$ if and only if $f$ is the restriction of a M\"obius transformation.

  For the level sets, fix $x_0\in\Omega$ and choose a connected neighborhood $U\subseteq\Omega$ on which $g$ is a conformal diffeomorphism, as above. With $h=f\circ(g|_U)^{-1}$ on $g(U)$, the composition law \cref{eq:Schain} applied on $U$ to $f=h\circ g$ gives $S_f-S_g=(Dg)^{\top}(S_h\circ g)(Dg)$. If $S_f\equiv S_g$, then the invertibility of $Dg$ gives $S_h\equiv0$, so the first claim gives a M\"obius transformation $M$ such that $h=M|_{g(U)}$, and hence $f=M\circ g$ on $U$. Viewed as maps into $\R^d\cup\{\infty\}$, both sides are real-analytic on the connected domain $\Omega$, so the identity principle gives $f=M\circ g$ throughout $\Omega$; in particular, the pole of $M$ does not belong to $g(\Omega)$. Conversely, if $f=M\circ g$, then \cref{eq:Schain} and $S_M\equiv0$ give $S_f\equiv S_g$.
\end{proof}

\begin{remark} \label{rem:anisotropic}
  The two operators were defined in \cref{eq:Tdef,eq:Sdef} for maps that need not be conformal, and one may ask whether the theory can be developed away from the conformal class, on the open set where the largest singular value of $Df$ is simple and positive. The set of matrices with this property is open and dense in matrix space, but for a fixed map its inverse image under $Df$ need not be dense in $\Omega$, and for a conformal map with $d\ge2$ it is empty. On this set the smooth dependence of a simple singular value on the matrix entries makes $T_f$ continuous when $f$ is a $\mathcal{C}^2$-map and defines $S_f$ when $f$ is a $\mathcal{C}^3$-map, and such a map is \emph{anisotropic}, distinguishing a single direction of greatest stretch. The set is fragile, however: it is left the instant the two largest singular values cross, where the maximum of two smooth branches need not be smooth and $T_f$ itself ceases to be defined. For example, the real-analytic diffeomorphism $f(x,y)=(x+x^2/2,y-y^2/2)$ near the origin has $Df=\diag(1+x,1-y)$ and $\log\|Df\|=\max\{\log(1+x),\log(1-y)\}$, which is not differentiable on the line $y=-x$. 
  
  The planar harmonic maps exhibit the same crossing in a classical family. If $f=h+\ol{g}$ is a sense-preserving harmonic map, with $h$ and $g$ holomorphic and $|g'|<|h'|$, then the singular values of $Df$ are $|h'|+|g'|$ and $|h'|-|g'|$, which coincide exactly at the zeros of the \emph{dilatation} $g'/h'$, and $\log\|Df\|=\log|h'|+\log(1+|g'/h'|)$ is real-analytic near such a zero precisely when its order is even; for $f(z)=z+\ol{z}^2/2$ on the unit disc, $\log\|Df\|=\log(1+|z|)$ is not differentiable at the origin. On a simply connected domain the zeros of the dilatation all have even order exactly when the dilatation is the square of a holomorphic function, and this is the hypothesis under which Chuaqui, Duren, and Osgood \cite{ChuaquiDurenOsgood} define the Schwarzian derivative of a harmonic map; read as a complex scalar as in \cref{lem:scalar}, their Schwarzian is the trace-free part of $S_f$. Hern\'andez and Mart\'in \cite{HernandezMartin} dispense with the hypothesis by building both operators from $\tfrac12\log\det Df=\tfrac12\log(|h'|^2-|g'|^2)$ instead, which is real-analytic for every sense-preserving harmonic map. Both repairs act on the definitions within a single class of maps, whereas for iterated function systems the obstruction lies in the composition laws.
\end{remark}

The composition laws, rather than the definitions, are what confine the theory to the conformal class. For matrices one has only submultiplicativity, $\|AB\|\le\|A\|\|B\|$, with equality precisely when $B$ attains its norm at some unit vector whose image direction is one at which $A$ attains its norm; the additivity of $\log\|D\cdot\|$ along a composition, behind \cref{eq:Tchain}, therefore requires the norm-attaining directions of the derivatives to stay aligned along the entire orbit. Non-conformal families can be arranged to satisfy this, diagonal maps with one dominant coordinate being an example, but the alignment is then a rigid additional constraint, and $\|Df\|$ reads a single singular value, so the pre-Schwarzian retains a one-dimensional shadow of the derivative. For self-affine systems the alignment is irrelevant: every composition is affine and hence has vanishing pre-Schwarzian, so such a system, dominated or not, carries no pre-Schwarzian information, just as a self-similar system carries none on the conformal class. Conformality is the pointwise condition making the alignment automatic and lossless: every direction attains the norm, so \cref{eq:Tchain} holds for all compositions of conformal maps, and $\|Df\|$ determines $Df$ up to a rotation, so the pre-Schwarzian retains the full conformal content of the derivative.

Only the inner map of a composition needs to be conformal. If $f$ is conformal, then $\|Dg(f(x))Df(x)\|=\lambda_f(x)\|Dg(f(x))\|$ for every map $g$, and with this identity in place of \cref{eq:lambdachain} the proof of \cref{lem:chain} uses the conformality of $f$ alone; hence \cref{eq:Tchain} holds for every outer map $g$ wherever $T_g$ is defined, \cref{eq:Schain} holds wherever $T_g$ is moreover differentiable, and both remain true when the norm is replaced by $|\det\,\cdot\,|^{1/d}$ throughout. Read in complex notation, and at weight two on trace-free parts, these one-sided laws for a harmonic outer map and a holomorphic inner map are the chain rules of \cite{ChuaquiDurenOsgood} and \cite{HernandezMartin}. They do not help along an iterated function system: the dual system of \cref{sec:dual-ifs} peels the generators off a composition from the inside, so every generator occurs as an inner map, and the one-sided laws cover every word only when every generator is conformal.

The determinant, in contrast, composes on both sides at weight one, since the field $\tfrac1d\nabla\log|\det Df|$ satisfies \cref{eq:Tchain} for arbitrary $\mathcal{C}^2$-diffeomorphisms, the determinant being multiplicative, and agrees with $T_f$ on conformal maps; for planar harmonic maps it is, up to the complex conjugation of \cref{lem:scalar}, the pre-Schwarzian of Hern\'andez and Mart\'in. One weight up it still needs a conformal inner map, since \cref{eq:Schain} can fail otherwise. At weight one it pays with faithfulness, depending on $f$ only through $|\det Df|$ up to a constant factor: for $d\ge2$ it cannot distinguish post-compositions by maps of constant Jacobian determinant, which is an infinite-dimensional family. Among planar harmonic maps it vanishes only on the affine maps (see \cite[Section~4]{HernandezMartin}), but harmonic maps are not closed under composition, so this rigidity is lost along the words of an iterated function system. On the conformal class the pre-Schwarzian is thus simultaneously everywhere defined, exactly cocyclic, and faithful up to similarities; this is the class on which the dual iterated function system of \cref{sec:dual-ifs} runs.

\section{Conformal iterated function systems and exponential separation} \label{sec:cifs-esc}

Fix the ambient dimension $d\ge1$. Throughout, $\II=\{1,\ldots,N\}$ is a finite alphabet with $N \ge 2$, $\II^\N$ is the set of infinite words, $\II^n$ is the set of words of length $n$, written $\iii=i_1\cdots i_n$, and $\II^*=\bigcup_{n\ge0}\II^n$ is the set of finite words with $\II^0=\{\varnothing\}$; $|\iii|$ is the length of $\iii$, with $|\iii|=\infty$ for $\iii\in\II^\N$, $\iii|_n=i_1\cdots i_n$ is the length-$n$ prefix of $\iii$ for $n\le|\iii|$, and $\iii\land\jjj$ is the longest common prefix of $\iii$ and $\jjj$. We write $\sigma$ for the left shift and $\overleftarrow{\iii}=i_n\cdots i_1$ for the reversal of $\iii$, and, for any $\beta\in(0,1)$, we regard $\II^\N\cup\II^*$ as a compact metric space under the ultrametric equal to $\beta^{|\iii\land\jjj|}$ when $\iii\ne\jjj$ and to $0$ otherwise; the induced topology does not depend on $\beta$, only that topology is used, and every subsequential limit of words below is taken in it.

This section fixes the objects the rest of the paper works with. In \cref{sec:cifs} we define a conformal iterated function system, in a form that imposes the contraction on the derivative rather than on distances, and record in \cref{lem:tube} the long-word contraction and the invariant neighborhood of $\ol{\Omega}$ that this form still delivers. In \cref{sec:cifs-measure} we construct the self-conformal set and the self-conformal measures, the objects whose dimensions \cref{sec:dimension} computes. In \cref{sec:esc} we define the four separation conditions, plain and modulo M\"obius maps, and show in \cref{lem:osc} that the classical separation hypotheses give the plain ones, while \cref{ex:esc-not-mobius} shows that no hypothesis on the pieces can give the conditions modulo M\"obius maps. In \cref{sec:condensation} we negate the strong conditions into super-exponential condensation, which is the form in which \cref{sec:C2,sec:C1,sec:C3} use them.

\subsection{Conformal iterated function systems} \label{sec:cifs}

Let $\Omega\subset\R^d$ be a bounded domain, an open interval when $d=1$, and let $\Phi=(\fii_i)_{i\in\II}$ be a tuple of maps that send $\ol{\Omega}$ into $\Omega$ and extend, conformally in the sense of \cref{sec:conformal-maps}, to a domain $\Omega'$ containing $\ol{\Omega}$, with $\|D\fii_i\|<1$ on $\ol{\Omega}$ for every $i \in \II$. We call such a $\Phi$ a \emph{conformal iterated function system}, the maps $\fii_i$ its \emph{generators}, and the domain $\Omega'$ its \emph{extension domain}. By conformality and compactness $0<\cmin\le\cmax<1$, where $\cmin=\min_{i \in \II}\inf_{x\in\ol{\Omega}}\|D\fii_i(x)\|$ and $\cmax=\max_{i \in \II}\sup_{x\in\ol{\Omega}}\|D\fii_i(x)\|$. The contraction is thus imposed infinitesimally. When $\Omega$ is convex, the mean value theorem along segments makes $\cmax$ a shared Lipschitz constant of the generators on $\ol{\Omega}$, but for a non-convex $\Omega$ a generator need not contract distances at all, the same estimate along an in-domain path bounding $|\fii_i(x)-\fii_i(y)|$ only by $\cmax$ times the length of that path; \cref{lem:tube} below shows that all sufficiently long compositions nevertheless are contractions of $\ol{\Omega}$, which is what the construction of the self-conformal set requires. For $d=1$ conformality by itself asks only that $\fii_i'$ not vanish, and the definition in \cref{sec:conformal-maps} therefore imposes the extra assumption that a conformal map on the line be real-analytic. For $d\ge2$ no regularity is imposed beyond conformality: by the classification of \cref{sec:conformal-maps}, a conformal extension is for $d=2$ a holomorphic map with nonvanishing derivative and for $d\ge3$ the restriction of a M\"obius transformation, so the extensions are real-analytic on $\Omega'$ without any further assumption. In every case the conformal factor $\|D\fii_i\|$, the pre-Schwarzian $T_{\fii_i}=\nabla\log\|D\fii_i\|$ of \cref{eq:Tdef}, and the Schwarzian $S_{\fii_i}$ of \cref{eq:Sdef} are defined on all of $\Omega'$.

We use two compositions along a finite word $\iii=i_1\cdots i_n$, the empty composition read as the identity of $\ol{\Omega}$. The \emph{forward} composition $\fii_{\iii}=\fii_{i_1}\circ\cdots\circ\fii_{i_n}$ carries the attractor and the separation; the \emph{reversed} composition
\begin{equation} \label{eq:frev}
  f_{\iii} = \fii_{\overleftarrow{\iii}} = \fii_{i_n}\circ\cdots\circ\fii_{i_1}
\end{equation}
carries the pre-Schwarzian cocycle, its outermost factor being the last applied; reversal is a bijection of each $\II^n$ interchanging the two families. Both compositions map $\ol{\Omega}$ into itself, their orbits under partial compositions stay in $\ol{\Omega}$, and the chain rule gives $\|D\fii_\iii\|\le\cmax^{|\iii|}$ and $\|Df_\iii\|\le\cmax^{|\iii|}$ on $\ol{\Omega}$. The next lemma turns this derivative bound into a contraction of distances, at a fixed loss determined by $\Phi$ and the geometry of $\Omega$.

\begin{lemma} \label{lem:tube}
  For every $\lambda\in(\cmax,1)$ there are $r_\Phi>0$ and $C_\Phi\ge1$ such that every finite word $\iii$ satisfies $|\fii_\iii(x)-\fii_\iii(y)|\le \lambda^{|\iii|}|x-y|$ for all $x,y\in\ol{\Omega}$ with $|x-y|<r_\Phi$, and
  \begin{equation} \label{eq:tube}
    |\fii_\iii(x)-\fii_\iii(y)| \le C_\Phi \lambda^{|\iii|}|x-y|
  \end{equation}
  for all $x,y\in\ol{\Omega}$. In particular, $\fii_\iii$ is a strict contraction of $\ol{\Omega}$ for all $\iii\in\II^n$ when $n \ge 1$ satisfies $C_\Phi \lambda^n < 1$.
\end{lemma}

\begin{proof}
  The generators are conformal on $\Omega'$, so the conformal factors $\|D\fii_i\|$ are continuous there, and on the compact set $\ol{\Omega}\subseteq\Omega'$ they are at most $\cmax<1$. By compactness there is $r_\Phi>0$ such that the neighborhood $V=\{x\in\R^d : \dist(x,\ol{\Omega})<r_\Phi\}$ lies in $\Omega'$ and $\|D\fii_i\|\le \lambda$ on $V$ for every $i \in \II$. If $z\in V$ and $x\in\ol{\Omega}$ is a nearest point, then the segment from $x$ to $z$ lies in $V$, each of its points $w$ satisfying $\dist(w,\ol{\Omega})\le|w-x|\le|z-x|<r_\Phi$, and $\fii_i(x)\in\ol{\Omega}$. Integrating $D\fii_i$ along the segment, that is, applying the fundamental theorem of calculus to the continuously differentiable curve $t\mapsto\fii_i(x+t(z-x))$ on $[0,1]$, therefore gives
  \begin{align*}
    |\fii_i(z)-\fii_i(x)| &= \biggl|\int_0^1 D\fii_i(x+t(z-x))(z-x)\dd t\biggr| \\
    &\le \int_0^1 \|D\fii_i(x+t(z-x))\||z-x|\dd t \le \lambda|z-x|,
  \end{align*}
  the points $x+t(z-x)$ of the segment staying in $V$, where $\|D\fii_i\|\le \lambda$, and hence $\dist(\fii_i(z),\ol{\Omega})\le|\fii_i(z)-\fii_i(x)|\le \lambda|z-x|<r_\Phi$. Thus every generator maps $V$ into itself, all partial orbits of points of $V$ stay in $V$, and the chain rule gives $\|D\fii_\iii\|\le \lambda^{|\iii|}$ on $V$ for every $\iii\in\II^*$. If $x,y\in\ol{\Omega}$ satisfy $|x-y|<r_\Phi$, then the segment from $x$ to $y$ lies in $V$ by the same distance estimate, and integrating $D\fii_\iii$ along it, exactly as in the display above but with the bound $\|D\fii_\iii\|\le \lambda^{|\iii|}$ on $V$ in place of $\|D\fii_i\|\le \lambda$, gives $|\fii_\iii(x)-\fii_\iii(y)|\le \lambda^{|\iii|}|x-y|$. For \cref{eq:tube}, note that $V$ is open and connected, being the union of the connected set $\Omega$ with the balls $B^o(x,r_\Phi)$, $x\in\ol{\Omega}$, each of which meets $\Omega$. Any two points of $V$ are therefore joined by a polygonal path in $V$, and the infimum $\ell_V(x,y)$ of the lengths of such paths is a metric on $V$ that agrees with $|x-y|$ whenever the segment from $x$ to $y$ lies in $V$. For each $(x,y)\in V\times V$, choose $\delta>0$ such that $B^o(x,\delta),B^o(y,\delta)\subseteq V$. If $x'\in B^o(x,\delta)$ and $y'\in B^o(y,\delta)$, then the segments from $x$ to $x'$ and from $y$ to $y'$ lie in $V$, and the triangle inequality gives $|\ell_V(x',y')-\ell_V(x,y)|\le|x'-x|+|y'-y|$. Thus $\ell_V$ is continuous on $V\times V$, hence bounded on the compact set $\ol{\Omega}\times\ol{\Omega}$, say by $D$. Since all partial orbits of points of $V$ stay in $V$, integrating $D\fii_\iii$ along the segments of a polygonal path in $V$ from $x$ to $y$ gives $|\fii_\iii(x)-\fii_\iii(y)|\le \lambda^{|\iii|}\ell_V(x,y)\le \lambda^{|\iii|}D$ for all $x,y\in\ol{\Omega}$. Together with the first bound this gives \cref{eq:tube} with $C_\Phi=\max\{1,D/r_\Phi\}$, since $\lambda^{|\iii|}D\le \lambda^{|\iii|}(D/r_\Phi)|x-y|$ when $|x-y|\ge r_\Phi$.
\end{proof}

\subsection{The self-conformal set and measure} \label{sec:cifs-measure}

The compositions of \cref{sec:cifs} contract along every infinite word, and the limit is the object the dimension theory of \cref{sec:dimension} is about. We construct it here, together with the self-conformal measures, and record the uniqueness that \cref{sec:dimension} uses.

Let $\lambda\in(\cmax,1)$ be arbitrary, and let $r_\Phi$ and $C_\Phi$ be the constants \cref{lem:tube} provides for it. By \cref{eq:tube}, for every $x_0\in\ol{\Omega}$ and $\iii\in\II^\N$ the sequence $(\fii_{\iii|_n}(x_0))_{n \ge 1}$ is Cauchy, consecutive terms being at most $C_\Phi \lambda^{n}\diam(\ol{\Omega})$ apart, a summable bound, since $\fii_{\iii|_{n+1}}(x_0)=\fii_{\iii|_n}(\fii_{i_{n+1}}(x_0))$ with $\fii_{i_{n+1}}(x_0)\in\ol{\Omega}$, so the limit $\pi(\iii)=\lim_{n \to \infty}\fii_{\iii|_n}(x_0)$ exists in $\ol{\Omega}$, uniformly over $\iii\in\II^\N$, and does not depend on $x_0$, the terms for two starting points being at most $C_\Phi \lambda^{n}\diam(\ol{\Omega})$ apart at stage $n$; the map $\pi$, continuous as a uniform limit of locally constant maps, is the \emph{canonical projection}, and its image $X=\pi(\II^\N)$, a nonempty compact set, is the \emph{self-conformal set} associated with $\Phi$. Since $\pi(i\iii)=\fii_i(\pi(\iii))$ for all $i \in \II$ and $\iii\in\II^\N$ by the continuity of $\fii_i$, it satisfies
\begin{equation*}
  X=\bigcup_{i \in \II}\fii_i(X),
\end{equation*}
and by Hutchinson \cite{Hutchinson} it is the unique nonempty compact subset of $\ol{\Omega}$ with this property: for $n \ge 1$ with $C_\Phi \lambda^n<1$, the compositions $\fii_\iii$, $\iii\in\II^n$, are strict contractions of $\ol{\Omega}$ by \cref{lem:tube}, so his theorem provides exactly one nonempty compact set invariant under them, and every nonempty compact set invariant under the generators, $X$ among them, is invariant under their length-$n$ compositions and hence equals that one set. The invariance also gives $X\subseteq\bigcup_{i \in \II}\fii_i(\ol{\Omega})\subseteq\Omega$; the left-hand side being compact, this gives $\dist(X,\partial\Omega)>0$. Finally, for every nonempty finite word $\iii$ some iterate of $\fii_\iii$ is a strict contraction of $\ol{\Omega}$ by \cref{lem:tube}, so $\fii_\iii$ has exactly one fixed point in $\ol{\Omega}$: every fixed point of $\fii_\iii$ is also fixed by the iterate, and the unique fixed point of the iterate is fixed also by $\fii_\iii$, which permutes the fixed points of the iterate. This fixed point is $\pi(\iii^\infty)\in X$, and $\|D\fii_\iii\|\le\cmax^{|\iii|}<1$ there, so it is attracting.

For a positive probability vector $\mathbf p=(p_i)_{i\in\II}$, the projection also carries the \emph{self-conformal measure} $\mu_{\mathbf p}=\mathbf p^{\N}\circ\pi^{-1}$, the push-forward under $\pi$ of the Bernoulli measure $\mathbf p^{\N}$ on $\II^\N$; it is a Borel probability measure supported on $X$. Writing $\tau_i(\iii)=i\iii$, the identity $\pi\circ\tau_i=\fii_i\circ\pi$ recorded above and the decomposition $\mathbf p^{\N}=\sum_{i \in \II}p_i\mathbf p^{\N}\circ\tau_i^{-1}$ by the first letter give
\begin{equation*}
  \mu_{\mathbf p} = \sum_{i \in \II}p_i\mu_{\mathbf p}\circ\fii_i^{-1},
\end{equation*}
and $\mu_{\mathbf p}$ is the unique Borel probability measure with this property, again by passing to length-$n$ compositions: for $n$ with $C_\Phi \lambda^n<1$ the maps $\fii_\iii$, $\iii\in\II^n$, with the weights $p_\iii=p_{i_1}\cdots p_{i_n}$ are strict contractions of $\ol{\Omega}$ carrying a positive probability vector, so the theorem of Hutchinson \cite{Hutchinson} provides exactly one invariant Borel probability measure for them, and substituting the invariance into itself shows that every Borel probability measure invariant under the generators with the weights $p_i$ is invariant under the length-$n$ compositions with the weights $p_\iii$ and hence equals that one measure.

\subsection{Separation conditions} \label{sec:esc}

The separation quantities use the supremum norm $\|h\|=\sup_{x\in\ol{\Omega}}|h(x)|$. The system $\Phi=(\fii_i)_{i\in\II}$ satisfies the \emph{exponential separation condition modulo M\"obius maps} if there is $c>0$ such that, for infinitely many $n\in\N$,
\begin{equation} \label{eq:sescmob}
  \|\fii_\iii - M \circ \fii_\jjj\| \ge c^n
\end{equation}
for all distinct $\iii,\jjj\in\II^n$ and all M\"obius transformations $M$ of $\RS^d=\R^d\cup\{\infty\}$, orientation-reversing ones included, the norm read as $\infty$ when the pole of $M$ meets $\fii_\jjj(\ol{\Omega})$; it satisfies the \emph{strong exponential separation condition modulo M\"obius maps} if this holds for every $n\in\N$. Throughout, the logarithm of $0$ is read as $-\infty$ and the logarithm of $\infty$ as $\infty$, so that quantities such as $\frac1n\log\|\fii_\iii-M\circ\fii_\jjj\|$ take values in $[-\infty,\infty]$.

Requiring $M$ to be the identity in \cref{eq:sescmob} instead defines the \emph{exponential separation condition} and the \emph{strong exponential separation condition}. Restricting $M$ to the similarities of $\R^d$, orientation-reversing ones included, defines in the same way the \emph{exponential separation condition modulo similarities} and the \emph{strong exponential separation condition modulo similarities}. A similarity is a M\"obius transformation fixing $\infty$ and the identity is a similarity, so the conditions modulo M\"obius maps imply those modulo similarities, and these in turn imply their plain counterparts. The plain strong condition follows from the classical separation hypotheses on the pieces, with a proviso special to the plane. The system $\Phi$ satisfies the \emph{open set condition} if there is a nonempty open set $U\subseteq\Omega$ such that $\fii_i(U)\subseteq U$ for every $i\in\II$ and $\fii_i(U)\cap\fii_j(U)=\emptyset$ for all distinct $i,j\in\II$, and the \emph{strong separation condition} if the sets $\fii_i(X)$, $i\in\II$, are pairwise disjoint; both conditions are classical \cite{Hutchinson,MauldinUrbanski1996}. The proviso is forced by the generators themselves: for $d\ne2$ every generator is injective on all of $\Omega$, being strictly monotone for $d=1$, the nonvanishing derivative keeping a constant sign on the interval, and being the restriction of a M\"obius bijection of $\RS^d$ for $d\ge3$ by the classification of \cref{sec:conformal-maps}, whereas for $d=2$ conformality carries only local injectivity, so either injectivity on the open set or the cylinder form of the open set condition recorded below enters as a hypothesis. The following lemma verifies that the exponential separation condition holds for conformal systems satisfying the open set condition whenever the generators are injective.

\begin{lemma} \label{lem:osc}
  Let $\Phi=(\fii_i)_{i\in\II}$ be a conformal iterated function system satisfying the open set condition with open set $U$. Suppose that either $d\ne2$, or $d=2$ and at least one of the following holds:
  \begin{enumerate}
    \item\label{it:osc-inj} the generators $\fii_i$ are injective on $U$,
    \item\label{it:osc-cyl} $\fii_\iii(U)\cap\fii_\jjj(U)=\emptyset$ for all $n\in\N$ and all distinct $\iii,\jjj\in\II^n$.
  \end{enumerate}
  Then $\Phi$ satisfies the strong exponential separation condition. Likewise, if $\Phi$ satisfies the strong separation condition and either $d\ne2$, or $d=2$ and the generators are injective on $X$, then $\Phi$ satisfies the strong exponential separation condition.
\end{lemma}

\begin{proof}
  We first obtain the cylinder disjointness $\fii_\iii(U)\cap\fii_\jjj(U)=\emptyset$ for all distinct $\iii,\jjj\in\II^n$. If $d=2$ and \cref{it:osc-cyl} holds, this is the hypothesis. Otherwise either $d\ne2$, in which case every generator is injective on $\Omega$ and hence on $U$, or $d=2$ and \cref{it:osc-inj} holds; in both remaining cases the generators are injective on $U$. Since the generators map $U$ into $U$, every $\fii_\iii$ sends $U$ into $U$ and is injective on $U$, being a composition of injections along partial images that stay in $U$. The cylinder claim then follows by induction on $n$. The case $n=1$ is the open set condition. Let $n\ge2$ and let $\iii,\jjj\in\II^n$ be distinct. If $i_1\ne j_1$, then iterating the invariance of $U$ gives $\fii_\iii(U)\subseteq\fii_{i_1}(U)$ and $\fii_\jjj(U)\subseteq\fii_{j_1}(U)$, and the latter two sets are disjoint. If $i_1=j_1$, then $\sigma\iii$ and $\sigma\jjj$ are distinct words in $\II^{n-1}$, so $\fii_{\sigma\iii}(U)$ and $\fii_{\sigma\jjj}(U)$ are disjoint subsets of $U$ by the induction hypothesis, and the injectivity of $\fii_{i_1}$ on $U$ keeps the images $\fii_\iii(U)=\fii_{i_1}(\fii_{\sigma\iii}(U))$ and $\fii_\jjj(U)=\fii_{i_1}(\fii_{\sigma\jjj}(U))$ disjoint.

  Fix $z_0\in U$ and $r>0$ with $B(z_0,r)\subseteq U$. By \cref{eq:lambdachain}, applied along the partial images, which stay in $\ol{\Omega}$, the conformal factor of $\fii_\iii$ is the product of the conformal factors of its letters along the orbit, so $\cmin^{|\iii|}\le\lambda_{\fii_\iii}\le\cmax^{|\iii|}$ on $\ol{\Omega}$ for every $\iii\in\II^*$. We claim that
  \begin{equation} \label{eq:innerball}
    \fii_\iii(B^o(z_0,r)) \supseteq B^o(\fii_\iii(z_0),r\cmin^{|\iii|}/4)
  \end{equation}
  for every $\iii\in\II^*$, the empty word being trivial.

  For $d=1$ the derivative $\fii_\iii'$ is, by the chain rule, a product along the orbit of factors each of constant sign on the interval, so it has constant sign and $\fii_\iii$ is strictly monotone on $\Omega$; moreover $|\fii_\iii'|=\lambda_{\fii_\iii}\ge\cmin^{|\iii|}$ on $\ol{\Omega}$, so the mean value theorem along the segments from $z_0$ to $z_0\pm r$, which lie in $\Omega$, gives $|\fii_\iii(z_0\pm r)-\fii_\iii(z_0)|\ge\cmin^{|\iii|}r$; the image $\fii_\iii(B^o(z_0,r))$ is the open interval with endpoints $\fii_\iii(z_0-r)$ and $\fii_\iii(z_0+r)$, which lie on opposite sides of $\fii_\iii(z_0)$ at distance at least $\cmin^{|\iii|}r$, so \cref{eq:innerball} holds with the radius $\cmin^{|\iii|}r$. For $d=2$ we first arrange that every composition $\fii_\iii$ is univalent on the disc $B^o(z_0,r)$. If \cref{it:osc-inj} holds, this is immediate from the injectivity on $U$ recorded above. If instead \cref{it:osc-cyl} is the standing planar hypothesis, local injectivity of the finitely many generators on a neighborhood of $\ol{\Omega}$ supplies, by compactness, $\delta>0$ such that each generator is injective on every ball of radius $\delta$ centered at a point of $\ol{\Omega}$; shrink $r$ so that $2r<\min\{\delta,r_\Phi\}$, with $r_\Phi$ from \cref{lem:tube}, and $B(z_0,r)\subseteq U$. Every composition $\fii_\iii$ is then injective on $B^o(z_0,r)$ by induction on length: the empty word is the identity, and if $\fii_{\sigma\iii}$ is injective on the disc then $\fii_{\sigma\iii}(B^o(z_0,r))$ has diameter at most $2r<\delta$ by \cref{lem:tube}, two points of the disc being less than $r_\Phi$ apart, and lies in $U\subseteq\Omega$, hence sits in a ball of radius $\delta$ centered at one of its points, on which $\fii_{i_1}$ is injective, so the composite $\fii_\iii=\fii_{i_1}\circ\fii_{\sigma\iii}$ is injective on the disc. With univalence and $|\fii_\iii'(z_0)|=\lambda_{\fii_\iii}(z_0)\ge\cmin^{|\iii|}$, the Koebe one-quarter theorem \cite[Theorem 14.14(b)]{Rudin} gives $\fii_\iii(B^o(z_0,r))\supseteq B^o(\fii_\iii(z_0),r|\fii_\iii'(z_0)|/4)$, which contains the ball of \cref{eq:innerball}.

  For $d\ge3$ the composition $\fii_\iii$ is the restriction of a M\"obius transformation of $\RS^d$, with which we identify it; the M\"obius transformations, finite compositions of similarities and inversions, are homeomorphisms of $\RS^d$. Every M\"obius transformation $f$ obeys the \emph{two-point identity}
  \begin{equation} \label{eq:twopoint}
    |f(u)-f(v)| = \sqrt{\lambda_f(u)\lambda_f(v)}|u-v|,
  \end{equation}
  away from the pole: both sides are multiplicative under composition, the identity is trivial for similarities, and for the unit inversion $\iota_0(y)=y/|y|^2$ it is the computation $|\iota_0(u)-\iota_0(v)|^2=(|v|^2-2\langle u,v\rangle+|u|^2)/(|u|^2|v|^2)=|u-v|^2\lambda_{\iota_0}(u)\lambda_{\iota_0}(v)$, an inversion in an arbitrary sphere being a similarity conjugate of $\iota_0$; the factorwise induction establishes the identity off the finitely many intermediate poles, and both sides being continuous off the pole of $f$ itself, it extends there. For $u\in\ol{\Omega}$ with $|u-z_0|=r$, the identity and the factor bound give $|\fii_\iii(u)-\fii_\iii(z_0)|\ge\cmin^{|\iii|}r$. Write $W=\fii_\iii(B^o(z_0,r))$. The set $W$ is open in $\RS^d$ and does not contain $\infty$, the preimage of $\infty$ under $\fii_\iii$ being either the pole, a point outside $\Omega$ by \cref{lem:mobius}, or, for a similarity, the point $\infty$ itself, so $W$ is an open subset of $\R^d$; moreover $\ol{W}\subseteq\fii_\iii(B(z_0,r))$, the image of the closed ball being compact, so the injectivity of $\fii_\iii$ gives $\partial W\subseteq\fii_\iii(B(z_0,r))\setminus\fii_\iii(B^o(z_0,r))=\fii_\iii(\{u\in\R^d : |u-z_0|=r\})$, a set at distance at least $\cmin^{|\iii|}r$ from $\fii_\iii(z_0)$ by the previous estimate. If a point $w$ with $|w-\fii_\iii(z_0)|<\cmin^{|\iii|}r$ were outside $W$, the segment from $\fii_\iii(z_0)$ to $w$, a connected set meeting $W$ and its complement, would meet $\partial W$ at distance less than $\cmin^{|\iii|}r$ from $\fii_\iii(z_0)$, which is impossible. Hence \cref{eq:innerball} holds with the radius $\cmin^{|\iii|}r$, and the claim is proved in every dimension.

  Let now $n\ge1$ and let $\iii,\jjj\in\II^n$ be distinct. By \cref{eq:innerball} and $B^o(z_0,r)\subseteq U$, the balls $B^o(\fii_\iii(z_0),r\cmin^{n}/4)$ and $B^o(\fii_\jjj(z_0),r\cmin^{n}/4)$ lie in $\fii_\iii(U)$ and $\fii_\jjj(U)$, respectively, so they are disjoint by the cylinder disjointness of the first paragraph; and two disjoint open balls of a common radius have centers at distance at least twice that radius, the midpoint of the centers lying in both balls otherwise. Since $z_0\in\ol{\Omega}$, this gives
  \begin{equation*}
    \|\fii_\iii-\fii_\jjj\| \ge |\fii_\iii(z_0)-\fii_\jjj(z_0)| \ge r\cmin^{n}/2,
  \end{equation*}
  and the constant $c=\min\{r/2,1\}\cmin$ satisfies $\|\fii_\iii-\fii_\jjj\|\ge c^{n}$ for every $n\in\N$ and all distinct $\iii,\jjj\in\II^n$: this is the strong exponential separation condition.

  Finally, suppose that $\Phi$ satisfies the strong separation condition and, when $d=2$, that the generators are injective on $X$. The compact sets $\fii_i(X)$ are pairwise disjoint, so $\delta=\min\{\dist(\fii_i(X),\fii_j(X)) : i,j\in\II \text{ so that } i\ne j\}>0$. For $t>0$ write $U_t=\{x\in\R^d : \dist(x,X)<t\}$. In the planar case, compactness also gives $\eta>0$ such that every generator is injective on $U_\eta$. Indeed, otherwise finiteness of the alphabet gives a generator $\fii_i$ and distinct pairs $x_k,y_k$ tending to points $x,y\in X$ such that $\fii_i(x_k)=\fii_i(y_k)$; injectivity on $X$ forces $x=y$, while $\fii_i'(x)\ne0$ makes $\fii_i$ injective on a ball about $x$ that contains both $x_k$ and $y_k$ for large $k$, a contradiction.

  Let $\lambda$ and $r_\Phi$ be as above and choose $0<\eps<\min\{\dist(X,\partial\Omega),r_\Phi,\delta/(2\lambda)\}$, with also $\eps<\eta$ when $d=2$. Then $U_\eps$ is a nonempty open subset of $\Omega$. For $x\in U_\eps$, choose a nearest point $y\in X$. By \cref{lem:tube},
  \begin{equation*}
    \dist(\fii_i(x),\fii_i(X)) \le |\fii_i(x)-\fii_i(y)| \le \lambda|x-y| < \lambda\eps.
  \end{equation*}
  Since $\fii_i(X)\subseteq X$, this gives $\fii_i(U_\eps)\subseteq U_\eps$, while $2\lambda\eps<\delta$ makes the $\lambda\eps$-neighborhoods of the sets $\fii_i(X)$ pairwise disjoint and hence $\fii_i(U_\eps)\cap\fii_j(U_\eps)=\emptyset$ for distinct $i$ and $j$. Thus $U_\eps$ is an open set for the open set condition. When $d=2$, the additional restriction $\eps<\eta$ ensures that every generator is injective on $U_\eps$, so \cref{it:osc-inj} holds. The first part of the proof therefore shows that $\Phi$ satisfies the strong exponential separation condition.
\end{proof}

The conclusion of \cref{lem:osc} is the plain condition, and necessarily so: no separation hypothesis on the pieces reaches the conditions modulo M\"obius maps. The next example records the obstruction in every dimension. For $d\ge3$ every conformal map is itself M\"obius, so a conjugating M\"obius transformation collapses every equal-length pair. In the plane, and on the line by restriction, the middle-third Cantor set satisfies the strong separation condition, hence the plain strong exponential separation condition by \cref{lem:osc}, while a translation collapses every equal-length pair.

\begin{example} \label{ex:esc-not-mobius}
  (1) Let $d\ge3$ and let $\Phi=(\fii_i)_{i\in\II}$ be any conformal iterated function system. By the classification of \cref{sec:conformal-maps}, every generator is the restriction of a M\"obius transformation of $\RS^d$, and hence so is every composition on $\ol{\Omega}$. Let $n\in\N$ and let $\iii,\jjj\in\II^n$ be distinct. The composition $M=\fii_\iii\circ \fii_\jjj^{-1}$ is a M\"obius transformation of $\RS^d$ and satisfies $M\circ\fii_\jjj=\fii_\iii$ throughout $\ol{\Omega}$. The pole of $M$ lies off $\fii_\jjj(\ol{\Omega})$, since $\fii_\iii$ takes values in $\R^d$, so
  \begin{equation*}
    \|\fii_\iii-M\circ\fii_\jjj\| = 0.
  \end{equation*}
  Hence \cref{eq:sescmob} fails for every $c>0$, every $n\in\N$, and every pair of distinct words in $\II^n$: no conformal iterated function system in dimension $d\ge3$ satisfies the exponential separation condition modulo M\"obius maps.

  (2) Let $\Omega=B^o(0,2)$, let $\Omega'=\C$, and consider the two maps
  \begin{equation*}
    \fii_1(z) = \frac{z}{3} \qquad\text{and}\qquad \fii_2(z) = \frac{z}{3}+\frac{2}{3}.
  \end{equation*}
  Both maps are entire conformal similarities, their common Lipschitz constant is $\tfrac{1}{3}$, and $|\fii_1(z)|\le\tfrac{2}{3}$ and $|\fii_2(z)|\le\tfrac{4}{3}$ on $\ol{\Omega}$. Hence they map $\ol{\Omega}$ into $\Omega$ and define a conformal iterated function system whose attractor is the middle-third Cantor set $X$ in $[0,1]\subset\Omega$. Since the maps are injective similarities and the sets $\fii_1(X)\subseteq[0,\tfrac13]$ and $\fii_2(X)\subseteq[\tfrac23,1]$ are disjoint, the system satisfies the strong separation condition, so \cref{lem:osc} already implies strong exponential separation; the following calculation gives the explicit constant.

  For $\iii=i_1\cdots i_n\in\II^n$, let $\eps_k(\iii)=0$ if $i_k=1$ and $\eps_k(\iii)=1$ if $i_k=2$. Direct composition gives
  \begin{equation*}
    \fii_\iii(z) = 3^{-n}z+t_\iii,\qquad \text{where} \qquad t_\iii = 2\sum_{k=1}^{n}\eps_k(\iii)3^{-k}.
  \end{equation*}
  The integer
  \begin{equation*}
    \frac{3^n t_\iii}{2} = \sum_{k=1}^{n}\eps_k(\iii)3^{n-k}
  \end{equation*}
  has base-three digits $\eps_1(\iii),\ldots,\eps_n(\iii)$. Hence distinct words $\iii,\jjj\in\II^n$ give distinct integers, so $3^n(t_\iii-t_\jjj)$ is a nonzero even integer. Consequently
  \begin{equation*}
    \|\fii_\iii-\fii_\jjj\| = |t_\iii-t_\jjj| \ge 2\cdot3^{-n} \ge \biggl(\frac{1}{3}\biggr)^n.
  \end{equation*}
  Thus $\Phi=(\fii_1,\fii_2)$ satisfies the strong exponential separation condition with $c=\tfrac{1}{3}$.

  On the other hand, for distinct $\iii,\jjj\in\II^n$, define $M_{\iii,\jjj}(w)=w+t_\iii-t_\jjj$. This translation is an orientation-preserving M\"obius transformation of $\RS^2$ with pole $\infty$, so its pole misses $\fii_\jjj(\ol{\Omega})$, and it satisfies $M_{\iii,\jjj}\circ\fii_\jjj=\fii_\iii$ on $\C$. Therefore
  \begin{equation*}
    \|\fii_\iii-M_{\iii,\jjj}\circ\fii_\jjj\| = 0.
  \end{equation*}
  Hence \cref{eq:sescmob} fails for every $c>0$; in fact, it fails for every $n\in\N$ and every pair of distinct words in $\II^n$. The same two maps, restricted to a bounded open interval containing $[0,1]$ and with extension domain $\R$, form a conformal iterated function system with $d=1$ to which the same conclusions apply, each map contracting toward its fixed point, $0$ for $\fii_1$ and $1$ for $\fii_2$, and the collapsing translations being M\"obius transformations of $\RS^1$.
\end{example}

\subsection{Super-exponential condensation} \label{sec:condensation}

The proofs of \cref{thm:main} given in \cref{sec:C2,sec:C1,sec:C3} proceed by contraposition. We call the negation of the strong exponential separation condition modulo M\"obius maps \emph{weak super-exponential condensation modulo M\"obius maps}, after B\'ar\'any and K\"aenm\"aki \cite{BaranyKaenmaki}, whose self-similar systems on the line, like those of Baker \cite{Baker}, condense super-exponentially without generating exact overlaps. The next lemma characterizes this condition by requiring $\|\fii_\iii-M\circ\fii_\jjj\|$ to decay super-exponentially along a sequence of lengths for suitable distinct equal-length words and M\"obius maps $M$.

\begin{lemma} \label{lem:condensation-mobius}
  Let $\Phi=(\fii_i)_{i\in\II}$ be a conformal iterated function system. Then $\Phi$ satisfies weak super-exponential condensation modulo M\"obius maps if and only if there are a strictly increasing sequence $(n_\ell)_{\ell \ge 1}$ and, for each $\ell$, distinct $\iii_\ell,\jjj_\ell\in\II^{n_\ell}$ and a M\"obius transformation $M_\ell$ such that
  \begin{equation*}
    \frac{1}{n_\ell}\log\|\fii_{\iii_\ell}-M_\ell\circ\fii_{\jjj_\ell}\| \to -\infty
  \end{equation*}
  as $\ell\to\infty$.
\end{lemma}

\begin{proof}
  Choose violations of \cref{eq:sescmob} with $c=e^{-k}$: since the strong condition fails, for each $k\in\N$ there are a length $n$, distinct $\iii,\jjj\in\II^{n}$, and a M\"obius transformation $M$ with $\|\fii_\iii-M\circ\fii_\jjj\|<e^{-kn}$. If their lengths are unbounded, thin them so that $k\ge\ell$ and $n_\ell$ is strictly increasing, the $\ell$-th violation then having error below $e^{-\ell n_\ell}$. If they are bounded, then some pair $\iii,\jjj$ of a fixed length $m$ recurs for arbitrarily large $k$, so its infimum over the M\"obius maps is zero; choose any strictly increasing $n_\ell\ge m$ and append a common tail $\kkk$ of length $n_\ell-m$. Since $\fii_{\iii\kkk}=\fii_\iii\circ\fii_\kkk$ with $\fii_\kkk(\ol{\Omega})\subseteq\ol{\Omega}$, appending the tail cannot increase the supremum norm, so choosing the M\"obius map $M_\ell$ before appending it, with $\|\fii_\iii-M_\ell\circ\fii_\jjj\|\le e^{-\ell n_\ell}$, keeps the error at most $e^{-\ell n_\ell}$. Thus in both cases there are, for each $\ell$, distinct $\iii_\ell,\jjj_\ell\in\II^{n_\ell}$ and a M\"obius transformation $M_\ell$ with error at most $e^{-\ell n_\ell}$, so $\frac{1}{n_\ell}\log\|\fii_{\iii_\ell}-M_\ell\circ\fii_{\jjj_\ell}\|\le-\ell\to-\infty$. 
  
  Conversely, if such witnesses exist and the strong condition held with some $c>0$, then $c^{n_\ell}\le\|\fii_{\iii_\ell}-M_\ell\circ\fii_{\jjj_\ell}\|$ for every $\ell$, so $\log c\le\frac{1}{n_\ell}\log\|\fii_{\iii_\ell}-M_\ell\circ\fii_{\jjj_\ell}\|\to-\infty$, which is absurd.
\end{proof}

With $M$ and every $M_\ell$ read as the identity, the negation of the plain strong condition is likewise \emph{weak super-exponential condensation}, and \cref{lem:condensation-mobius} holds verbatim for it, as it does with $M$ and every $M_\ell$ restricted to the similarities, the proof using nothing about these maps beyond their being fixed maps composed on the left. Since $\fii_\iii=f_{\overleftarrow{\iii}}$ and reversal is a bijection of each $\II^n$, all the preceding separation and condensation properties are reversal invariant.

\section{The dual iterated function system} \label{sec:dual-ifs}

This section builds the dual iterated function system, a system of affine contractions whose orbit of the zero field runs through the pre-Schwarzian cocycle $T_\iii$ and, one weight up, the Schwarzian cocycle $S_\iii$. In \cref{sec:dual-cocycles} we unfold the two cocycles along a finite word, in \cref{sec:dual-maps} we assemble the dual systems and their canonical projections, and in \cref{sec:dual-properties} we extend those projections to infinite words and record the peeling identity, in \cref{lem:dual,lem:dual-peel,lem:dual2}. The last two subsections are what \cref{sec:C2,sec:C1,sec:C3} carry out and apply: \cref{sec:reduction} reduces \cref{thm:main} to a coincidence of two limit projections, and \cref{sec:gencrit} makes its hypotheses checkable on a concrete system through the one-point gap criterion of \cref{prop:gencrit}. The dual objects live on the space $\mathcal{C}(\ol{\Omega};\R^d)$ of $\R^d$-valued continuous maps on $\ol{\Omega}$, normed by $\|\cdot\|$. Since $\Omega$ is bounded, $\ol{\Omega}$ is compact, so $\mathcal{C}(\ol{\Omega};\R^d)$ is a Banach space under the supremum norm. Analyticity is deliberately left out of the function space: real-analyticity does not pass to uniform limits, the polynomials being dense in $\mathcal{C}(\ol{\Omega};\R^d)$ by the Stone--Weierstrass theorem, so the real-analytic maps form no closed subspace; instead, \cref{lem:dual} below proves that all pre-Schwarzian projections extend real-analytically to one common neighborhood of $\ol{\Omega}$, while \cref{lem:dual2} treats the Schwarzian projections. The plane is the exception at the level of the ambient space: there holomorphy does pass to uniform limits, and the dual system of \cref{eq:dualdef2} below is housed in a space of holomorphic maps outright. 

\subsection{Cocycles along words} \label{sec:dual-cocycles}

The \emph{Ruelle transfer operator} $\LL^{(1)}$ of $\Phi=(\fii_i)_{i\in\II}$ at weight one is defined by
\begin{equation*}
  \LL^{(1)} h = \sum_{i\in\II}(D\fii_i)^{\top}(h\circ\fii_i)
\end{equation*}
for all $h \in \mathcal{C}(\ol{\Omega};\R^d)$. The transfer operator $\LL^{(1)}$ acts by composition with the generators and multiplication by a power of the derivative, the weight, and here the multiplier $(D\fii_i)^{\top}=\|D\fii_i\|O_{\fii_i}^{\top}$ carries the first power of the conformal factor together with an isometric orthogonal part. The dual construction iterates the generators one at a time rather than their sum: the \emph{weight-one composition operator} $\LL^{(1)}_i$ of the generator $\fii_i$ is the $i$-th summand
\begin{equation*}
  \LL^{(1)}_i h = (D\fii_i)^{\top}(h\circ\fii_i),
\end{equation*}
so that $\LL^{(1)}=\sum_{i\in\II}\LL^{(1)}_i$. It maps $\mathcal{C}(\ol{\Omega};\R^d)$ into itself, the image $\fii_i(\ol{\Omega})$ lying in $\ol{\Omega}$, and satisfies $\|\LL^{(1)}_i h\|\le\cmax\|h\|$, since the orthogonal part preserves norms and $\|D\fii_i\|\le\cmax$ on $\ol{\Omega}$; thus each $\LL^{(1)}_i$ is a linear $\cmax$-contraction of $\mathcal{C}(\ol{\Omega};\R^d)$. Since $\fii_i$ is real-analytic and maps $\Omega$ into $\Omega$, the operator also preserves real-analyticity on $\Omega$. The weight is forced: applied to a composition $g\circ\fii_i$, with $g$ conformal on a domain containing $\ol{\Omega}$, the pre-Schwarzian composition law \cref{eq:Tchain} reads 
\begin{equation} \label{eq:Tcocycle}
  T_{g\circ\fii_i} = (D\fii_i)^{\top}(T_g\circ \fii_i)+T_{\fii_i} = \LL^{(1)}_i T_g+T_{\fii_i},
\end{equation}
so the pre-Schwarzian is an affine cocycle over the weight-one action, the translation part being the pre-Schwarzian of the generator.

The construction repeats one weight up, the superscript recording the weight; the ambient space is now the Banach space of continuous symmetric matrix fields on $\ol{\Omega}$, normed by $\|H\|=\sup_{x\in\ol{\Omega}}\|H(x)\|$, the inner norm being the matrix operator norm. The \emph{Ruelle transfer operator at weight two} is
\begin{equation*}
  \LL^{(2)}H = \sum_{i\in\II}(D\fii_i)^{\top}(H\circ\fii_i)(D\fii_i),
\end{equation*}
its $i$-th summand being the \emph{weight-two composition operator}
\begin{equation*}
  \LL^{(2)}_i H = (D\fii_i)^{\top}(H\circ\fii_i)(D\fii_i),
\end{equation*}
so that $\LL^{(2)}=\sum_{i\in\II}\LL^{(2)}_i$; the derivative now enters on both sides as $(D\fii_i)^{\top}H(D\fii_i)=\|D\fii_i\|^2 O_{\fii_i}^{\top}HO_{\fii_i}$, the orthogonal conjugation preserving norms, so $\|\LL^{(2)}_i H\|\le\cmax^2\|H\|$ and each $\LL^{(2)}_i$ is a $\cmax^2$-contraction. Applied to a composition $g\circ\fii_i$, the Schwarzian composition law \cref{eq:Schain} reads
\begin{equation} \label{eq:Scocycle}
  S_{g\circ\fii_i} = (D\fii_i)^{\top}(S_g\circ\fii_i)(D\fii_i)+S_{\fii_i} = \LL^{(2)}_i S_g+S_{\fii_i},
\end{equation}
so the Schwarzian is an affine cocycle over the weight-two action, its translation part being the Schwarzian of the generator, exactly as the pre-Schwarzian is an affine cocycle over $\LL^{(1)}_i$. It is needed only for the Schwarzian results of \cref{sec:C2,sec:C1}, the Schwarzian being void for $d\ge3$ as \cref{rem:S3} explains.

Iterated along a finite word $\iii=i_1\cdots i_n$, the two affine cocycles compute the pre-Schwarzian and Schwarzian derivatives of $f_\iii$ in closed form, the empty composition $f_\varnothing$ being the identity with $T_{f_\varnothing}\equiv0$ and $S_{f_\varnothing}\equiv0$. Taking $g=f_{\sigma\iii}$ in the cocycle relations \cref{eq:Tcocycle} and \cref{eq:Scocycle}, the decomposition $f_\iii=f_{\sigma\iii}\circ\fii_{i_1}$ gives the recursions $T_{f_\iii}=\LL^{(1)}_{i_1}T_{f_{\sigma\iii}}+T_{\fii_{i_1}}$ and $S_{f_\iii}=\LL^{(2)}_{i_1}S_{f_{\sigma\iii}}+S_{\fii_{i_1}}$, and iterating them until the empty word is reached unfolds the two derivatives into the finite sums
\begin{align}
  T_{f_\iii} &= \sum_{k=1}^{n}(Df_{\iii|_{k-1}})^{\top}(T_{\fii_{i_k}}\circ f_{\iii|_{k-1}}), \label{eq:Tword} \\
  S_{f_\iii} &= \sum_{k=1}^{n}(Df_{\iii|_{k-1}})^{\top}(S_{\fii_{i_k}}\circ f_{\iii|_{k-1}})(Df_{\iii|_{k-1}}), \label{eq:Sword}
\end{align}
the multipliers accumulated along the iteration collecting, by the chain rule and $(AB)^{\top}=B^{\top}A^{\top}$, into the transposed derivatives of the prefix maps $f_{\iii|_{k-1}}$.

On the line the sums are classical. There the chain rule telescopes the derivative of $f_\iii$ into the product $f_\iii'=\prod_{k=1}^{n}\fii_{i_k}'\circ f_{\iii|_{k-1}}$ along the orbit of the prefixes, the logarithmic derivative $T_f=f''/f'$ recorded after \cref{eq:Tdef} turns this product into a sum, and $(Df)^{\top}$ acts by multiplication by $f'$, so \cref{eq:Tword} reads
\begin{equation} \label{eq:Tline}
  \frac{f_\iii''}{f_\iii'} = \sum_{k=1}^{n}\biggl(\frac{\fii_{i_k}''}{\fii_{i_k}'}\circ f_{\iii|_{k-1}}\biggr)f_{\iii|_{k-1}}'
\end{equation}
and \cref{eq:Sword} reads $S_{f_\iii}=\sum_{k=1}^{n}(S_{\fii_{i_k}}\circ f_{\iii|_{k-1}})(f_{\iii|_{k-1}}')^{2}$, the classical Schwarzian cocycle. The identity \cref{eq:Tline} is where the construction of B\'ar\'any, Kolossv\'ary, and Troscheit \cite{BaranyKolossvaryTroscheit} begins: they take the right-hand side of \cref{eq:Tline} as the definition of the projection attached to the word $\iii$ and formulate their separation hypothesis \cite[Theorem~1.5]{BaranyKolossvaryTroscheit} directly in terms of it, identifying the sum with the derivative ratio $f_\iii''/f_\iii'$ only afterward, whereas here both sides of \cref{eq:Tline} unpack the same definition \cref{eq:Tdef}.

In higher dimensions it is the conformal factor, not the derivative, that telescopes; on the line the two coincide up to sign, so the distinction is invisible there. The derivative of $f_\iii$ is a product of matrices, but conformality makes its norm multiplicative along the composition, by the identity $\lambda_{g\circ f}=(\lambda_g\circ f)\lambda_f$ of \cref{lem:chain}; hence $\|Df_\iii\|=\prod_{k=1}^{n}\|D\fii_{i_k}\|\circ f_{\iii|_{k-1}}$, the logarithm turns this product into a sum, and the gradient chain rule $\nabla(h\circ f)=(Df)^{\top}(\nabla h\circ f)$ differentiates the sum into \cref{eq:Tword} term by term, exactly as the logarithmic derivative turned the telescoped product into \cref{eq:Tline}. In the plane the classical reading persists verbatim in complex notation, the only trace of the ambient dimension being a complex conjugation at weight one, and not even that surviving at weight two. Indeed, $Df$ and its transpose act on $\C$ as multiplication by $f'$ and by $\ol{f'}$ respectively, and $T_{\fii_{i_k}}=\ol{\fii_{i_k}''/\fii_{i_k}'}$ by \cref{lem:scalar}, so each summand of \cref{eq:Tword} is the complex conjugate of its counterpart in \cref{eq:Tline}; since $T_{f_\iii}=\ol{f_\iii''/f_\iii'}$, again by \cref{lem:scalar}, conjugating termwise shows that \cref{eq:Tline} holds in the plane exactly as written, with the complex derivatives of the prefix maps in place of the real ones. 

One weight up the conjugation disappears altogether: the two derivative factors of \cref{eq:Sword} feed a vector $v\in\C$ into the quadratic form of $S_{\fii_{i_k}}\circ f_{\iii|_{k-1}}$ as the product $f_{\iii|_{k-1}}'v$, so the tensor reading $v^{\top}S_fv=\re(S_f^{\C}v^{2})$ of \cref{lem:scalar} turns the $k$-th summand, evaluated at $v$, into $\re((S_{\fii_{i_k}}^{\C}\circ f_{\iii|_{k-1}})(f_{\iii|_{k-1}}')^{2}v^{2})$, the weight-two multiplier being the holomorphic square $(f_{\iii|_{k-1}}')^{2}$ with no conjugation at all, and \cref{eq:Sword} becomes the classical complex Schwarzian cocycle
\begin{equation} \label{eq:Sline}
  S_{f_\iii}^{\C}=\sum_{k=1}^{n}(S_{\fii_{i_k}}^{\C}\circ f_{\iii|_{k-1}})(f_{\iii|_{k-1}}')^{2},
\end{equation}
formally the same identity as on the line. For $d\ge3$ the reading changes character rather than form: by \cref{lem:mobius} the pre-Schwarzian $T_{\fii_{i_k}}$ of each generator either vanishes, when $\fii_{i_k}$ is a similarity, or is the pole field of \cref{eq:pole}, so \cref{eq:Tword} exhibits $T_{f_\iii}$ as a superposition of pole fields seen through the prefix maps, while every summand of \cref{eq:Sword} vanishes outright, consistent with the vanishing of $S_{f_\iii}$ itself, the composition $f_\iii$ being the restriction of a M\"obius transformation by the Liouville classification recorded in \cref{sec:conformal-maps}.

\subsection{The dual maps and their projections} \label{sec:dual-maps}

For $d\neq2$ the \emph{dual iterated function system} of $\Phi=(\fii_i)_{i\in\II}$ is the tuple $\Phi^*=(F_i)_{i\in\II}$ of affine maps
\begin{equation} \label{eq:dualdef}
  F_ih = \LL^{(1)}_i h+T_{\fii_i} = (D\fii_i)^{\top}(h\circ\fii_i)+T_{\fii_i},
\end{equation}
whose linear part and translation are the two ingredients of the cocycle relation \cref{eq:Tcocycle}. Each $F_i$ maps $\mathcal{C}(\ol{\Omega};\R^d)$ into itself, the translation $T_{\fii_i}$ being real-analytic on $\Omega'$ and in particular a member of $\mathcal{C}(\ol{\Omega};\R^d)$; it preserves real-analyticity on $\Omega$, as $\LL^{(1)}_i$ does, and satisfies $\|F_ih-F_ih'\|\le\cmax\|h-h'\|$, the difference $F_ih-F_ih'=\LL^{(1)}_i(h-h')$ shedding the translation: the dual system is an iterated function system of affine $\cmax$-contractions of the complete space $\mathcal{C}(\ol{\Omega};\R^d)$. In this reading the recursion that unfolds into \cref{eq:Tword} is one application of a dual map, $T_{f_\iii}=\LL^{(1)}_{i_1}T_{f_{\sigma\iii}}+T_{\fii_{i_1}}=F_{i_1}T_{f_{\sigma\iii}}$, so induction on the length, from the empty word, gives $F_\iii0=T_{f_\iii}$ for every $\iii\in\II^*$, where $F_\iii=F_{i_1}\circ\cdots\circ F_{i_n}$ and $F_\varnothing$ is the identity: the orbit of the zero field under the dual maps runs through the pre-Schwarzian cocycle. The construction repeats one weight up: the \emph{weight-two dual iterated function system} of $\Phi$ is the tuple $\Phi^{*2}=(G_i)_{i\in\II}$ of affine maps 
\begin{equation} \label{eq:dualdefG}
  G_iH = \LL^{(2)}_iH+S_{\fii_i} = (D\fii_i)^{\top}(H\circ\fii_i)(D\fii_i)+S_{\fii_i},
\end{equation}
assembled from the two ingredients of \cref{eq:Scocycle}; each $G_i$ is an affine $\cmax^2$-contraction of the complete space of continuous symmetric matrix fields on $\ol{\Omega}$, and the same induction gives $G_\iii0=S_{f_\iii}$, the orbit unfolding along the Schwarzian cocycle \cref{eq:Sword}.

For $d=2$ the general weight-one action $h\mapsto(D\fii_i)^{\top}(h\circ\fii_i)+T_{\fii_i}$ is anti-holomorphic: as recorded after \cref{eq:Tline}, the multiplier $(D\fii_i)^{\top}$ acts as multiplication by $\ol{\fii_i'}$ and the translation is $T_{\fii_i}=\ol{T^{\C}_{\fii_i}}$, so it carries conjugates of members of $\mathcal{H}(\ol{\Omega})$ to conjugates of members of $\mathcal{H}(\ol{\Omega})$, where $\mathcal{H}(\ol{\Omega})$ is the space of $\C$-valued maps holomorphic on $\Omega$ and continuous on $\ol{\Omega}$, again a Banach space under the supremum norm, holomorphy passing to uniform limits by the Weierstrass convergence theorem; see \cite[Theorem 5.1]{Ahlfors}. To keep the dual system in $\mathcal{H}(\ol{\Omega})$ we conjugate the two coefficients: for $d=2$ the \emph{dual iterated function system} of $\Phi$ is the tuple $\Phi^*=(F_i)_{i\in\II}$ of affine self-maps
\begin{equation} \label{eq:dualdef2}
  F_ih = \ol{\LL^{(1)}_i\ol{h}+T_{\fii_i}} = \fii_i'(h\circ\fii_i)+T^{\C}_{\fii_i}
\end{equation}
of $\mathcal{H}(\ol{\Omega})$, with the same properties as its counterpart for $d\neq2$; no pre-Schwarzian changes meaning, the conjugation being applied to the two coefficients and written out, $\ol{T_{\fii_i}}=T^{\C}_{\fii_i}=\fii_i''/\fii_i'$. Since \cref{eq:dualdef2} is the conjugate of the general action, the interior conjugations cancel in pairs along a composition, and, the zero field being fixed, its orbit is the conjugate of the pre-Schwarzian: $F_\iii0=\ol{T_{f_\iii}}=T^{\C}_{f_\iii}$, holomorphic and unfolding along the sums \cref{eq:Tline}. Here the weight-one composition operator is the multiplication $\LL^{(1)}_ih=\fii_i'(h\circ\fii_i)$, the conjugate of its general form and of the same norms. One weight up no such adjustment is needed, but the fields narrow: we restrict to the closed subspace of symmetric trace-free fields, which contains every planar Schwarzian by \cref{lem:scalar} and is invariant under each $\LL^{(2)}_i$, the conjugation by $D\fii_i$ preserving trace-freeness. Reading this subspace through the identification of \cref{lem:scalar}, under which a symmetric trace-free matrix with quadratic form $v\mapsto\re(\mu v^2)$ corresponds to the scalar $\mu$, with eigenvalues $\pm|\mu|$ and hence operator norm $|\mu|$, the operator $\LL^{(2)}_i$ becomes the multiplication $\LL^{(2)}_ih=(\fii_i')^{2}(h\circ\fii_i)$ exactly, and the weight-two dual system $\Phi^{*2}$ becomes the tuple of affine self-maps $G_ih=\LL^{(2)}_ih+S^{\C}_{\fii_i}$ of $\mathcal{H}(\ol{\Omega})$, with $G_\iii0=S^{\C}_{f_\iii}$ unfolding along \cref{eq:Sline}. Since the conjugation and the identification preserve all supremum norms, every estimate below reads the same in either picture.

In every dimension we denote by $T_\iii$ and $S_\iii$ the orbit points of the zero field under the dual maps, the \emph{dual canonical projections} of $\iii$,
\begin{equation} \label{eq:dualrec}
  T_\iii = F_\iii0 = F_{i_1}T_{\sigma\iii}\qquad\text{and}\qquad S_\iii = G_\iii0 = G_{i_1}S_{\sigma\iii},
\end{equation}
for every $\iii\in\II^*$, with $T_\varnothing\equiv0$ and $S_\varnothing\equiv0$. Typewriter word subscripts thus name dual projections throughout the paper, while map subscripts keep naming the operators of \cref{sec:schwarzian}: by the above, for finite words $T_\iii=T_{f_\iii}$ and $S_\iii=S_{f_\iii}$ outright except in the plane, where $T_\iii=T^{\C}_{f_\iii}=\ol{T_{f_\iii}}$ and $S_\iii=S^{\C}_{f_\iii}$ are the holomorphic scalars. The dual canonical projections are the counterparts of the canonical projection $\pi$ of $\Phi$, and \cref{lem:dual,lem:dual2} below extend them to infinite words. In particular $\|T_\iii-T_\jjj\|=\|T_{f_\iii}-T_{f_\jjj}\|$ and $\|S_\iii-S_\jjj\|=\|S_{f_\iii}-S_{f_\jjj}\|$ for finite words in every dimension. Away from the plane the two sides of each identity are the same object; in the plane the first holds because conjugation preserves moduli, and the second because the correspondence of \cref{lem:scalar} between a symmetric trace-free matrix and its scalar is linear over $\R$ and isometric, the matrix attached to $\mu$ having eigenvalues $\pm|\mu|$ and hence operator norm $|\mu|$, so it carries the difference of the matrices to the difference of the scalars without changing any supremum. The finite-word part of the separation hypotheses \cref{eq:mainhyp,eq:mainhyp2} of \cref{thm:main} therefore reads the same whether the projections or the pre-Schwarzians and Schwarzians themselves are compared, the infinite-word part being stated in the dual projections alone.

Composed along a word, the dual maps stay affine: $F_\iii=F_{i_1}\circ\cdots\circ F_{i_n}$ has linear part $\LL^{(1)}_\iii=\LL^{(1)}_{i_1}\cdots\LL^{(1)}_{i_n}$ and translation $F_\iii0=T_\iii$, and $G_\iii$ has linear part $\LL^{(2)}_\iii=\LL^{(2)}_{i_1}\cdots\LL^{(2)}_{i_n}$ and translation $S_\iii$. By the chain rule and $(AB)^{\top}=B^{\top}A^{\top}$ the letter operators collect into
\begin{equation} \label{eq:wordops}
  \LL^{(1)}_\iii h=(Df_\iii)^{\top}(h\circ f_\iii)\qquad\text{and}\qquad\LL^{(2)}_\iii H=(Df_\iii)^{\top}(H\circ f_\iii)(Df_\iii),
\end{equation}
for $d=2$ the multiplications $\LL^{(1)}_\iii h=f_\iii'(h\circ f_\iii)$ and $\LL^{(2)}_\iii h=(f_\iii')^{2}(h\circ f_\iii)$, with $\LL^{(k)}_\varnothing$ being the identity.

\subsection{Limit projections and peeling} \label{sec:dual-properties}

The following lemmas gather the basic properties of the dual system at weights one and two, each Schwarzian claim being its pre-Schwarzian counterpart with the multiplier squared. At weight one, \cref{lem:dual,lem:dual-peel} generalize the construction of B\'ar\'any, Kolossv\'ary, and Troscheit \cite[Lemmas~2.1 and~2.4]{BaranyKolossvaryTroscheit} from the line to every dimension $d\ge1$: \cref{lem:dual} records the metric bounds and the real-analyticity of the limits along infinite words, and \cref{lem:dual-peel} records the peeling identity, the paragraph after its proof assembling the dual attractor $X^*$; at weight two, \cref{lem:dual2} records the metric bounds, the peeling identity, and the regularity needed in the plane, and the paragraphs after its proof assemble the weight-two dual attractor $X^{*2}$ and relate it to $X^*$.

The limits along infinite words are produced on a complex thickening of $\ol{\Omega}$, on which the generators and their pre-Schwarzians live holomorphically and still contract. The same thickening carries several later arguments, so we record it on its own before using it. Throughout, $E=\C$ for $d\in\{1,2\}$ and $E=\C^d$ for $d\ge3$, with $\ol{\Omega}$ embedded in its usual real slice when $d\ne2$, and $\Omega_s=\{z\in E : \dist(z,\ol{\Omega})<s\}$ denotes the open $s$-neighborhood of $\ol{\Omega}$ in $E$.

\begin{lemma} \label{lem:thicken}
  Let $\Phi=(\fii_i)_{i\in\II}$ be a conformal iterated function system on $\Omega \subset \R^d$. Then there are an open set $\Omega^{\C}\subseteq E$ containing $\ol{\Omega}$, a constant $\lambda$ with $\cmax<\lambda<1$, and a radius $r>0$ such that the following hold:
  \begin{enumerate}
    \item\label{it:thick-ext} Every generator $\fii_i$ and every pre-Schwarzian $T_{\fii_i}$ extends holomorphically to $\Omega^{\C}$, the extension of $\fii_i$ having nonvanishing derivative when $d\in\{1,2\}$.
    \item\label{it:thick-bound} The closure $\ol{\Omega_r}$ is a compact subset of $\Omega^{\C}$ and $\max_{i \in \II}\sup_{z\in\Omega_r}\|D\fii_i(z)\|\le \lambda$.
    \item\label{it:thick-contract} Every generator satisfies $\dist(\fii_i(z),\ol{\Omega})\le \lambda\dist(z,\ol{\Omega})$ on $\Omega_r$, so that $\fii_i(\ol{\Omega_s})\subseteq\Omega_s$ for every $0<s<r$; in particular every $\fii_i$ maps $\Omega_r$ into itself, all partial orbits of points of $\Omega_r$ stay in $\Omega_r$, and $\sup_{z\in\Omega_r}\|Df_\iii(z)\|\le \lambda^{|\iii|}$ for every $\iii\in\II^*$.
    \item\label{it:thick-slice} The real slice $\Omega_r\cap\R^d$ is an open neighborhood of $\ol{\Omega}$ contained in $\Omega'$ and mapped into itself by every generator, and for $d\ne2$ the extensions of the generators and of the pre-Schwarzians coincide there with the original real ones.
  \end{enumerate}
\end{lemma}

\begin{proof}
  We construct one analytic neighborhood for all words. For $d=1$, the real-analytic generators extend holomorphically to a complex neighborhood of $\ol{\Omega}$; see \cite[Section~6.4]{KrantzParks}. After shrinking the neighborhood their derivatives do not vanish, and we extend $T_{\fii_i}$ as $\fii_i''/\fii_i'$. For $d=2$, the generators and the scalar pre-Schwarzians $T_{\fii_i}^{\C}=\fii_i''/\fii_i'$ are already holomorphic on $\Omega'$, where $\fii_i'$ does not vanish by conformality. For $d\ge3$, the real-analytic maps $\fii_i$ and vector fields $T_{\fii_i}$ extend holomorphically, componentwise, to a complex neighborhood of $\ol{\Omega}$; see \cite[Section~6.4]{KrantzParks}. Since the alphabet is finite, in every case the extensions, denoted by the same symbols, are defined on one open set $\Omega^{\C}\subseteq E$ containing $\ol{\Omega}$; for $d=2$ we take $\Omega^{\C}=\Omega'$. This gives \cref{it:thick-ext}.

  At each point of $\ol{\Omega}$ the complex operator norm of $D\fii_i$ equals its real operator norm. Choose $\lambda$ with $\cmax<\lambda<1$. By continuity of the complexified derivatives and compactness of $\ol{\Omega}$, there is $r>0$ such that the neighborhood $\Omega_r$ has compact closure in $\Omega^{\C}$ and $\max_{i \in \II}\sup_{z\in\Omega_r}\|D\fii_i(z)\|\le \lambda$, which is \cref{it:thick-bound}. Shrinking $r$, we also assume that $r<\dist(\ol{\Omega},\R^d\setminus\Omega')$, so that the real slice $\Omega_r\cap\R^d$ lies in $\Omega'$; for $d=2$ this inclusion is already implied by $\ol{\Omega_r}\subseteq\Omega^{\C}=\Omega'$.

  For \cref{it:thick-contract}, let $z\in\Omega_r$ and let $x\in\ol{\Omega}$ be a nearest point, so that $|z-x|=\dist(z,\ol{\Omega})$. The segment from $x$ to $z$ lies in $\Omega_r$, each of its points $w$ satisfying $\dist(w,\ol{\Omega})\le|w-x|\le|z-x|$, and $\fii_i(x)\in\ol{\Omega}$. Integrating $D\fii_i$ along the segment, as in the proof of \cref{lem:tube} but in $E$ with the complex operator norm, gives $\dist(\fii_i(z),\ol{\Omega})\le|\fii_i(z)-\fii_i(x)|\le \lambda|z-x|=\lambda\dist(z,\ol{\Omega})$. If $0<s<r$ and $z\in\ol{\Omega_s}\subseteq\Omega_r$, then $\dist(z,\ol{\Omega})\le s$ and hence $\dist(\fii_i(z),\ol{\Omega})\le \lambda s<s$, which is the inclusion $\fii_i(\ol{\Omega_s})\subseteq\Omega_s$. Applied on $\Omega_r$ itself the estimate shows that every $\fii_i$ maps $\Omega_r$ into itself, so all partial orbits of points of $\Omega_r$ stay in $\Omega_r$ and the chain rule gives $\sup_{z\in\Omega_r}\|Df_\iii(z)\|\le \lambda^{|\iii|}$.

  The real slice $\Omega_r\cap\R^d$ is the open $r$-neighborhood of $\ol{\Omega}$ in $\R^d$, and it lies in $\Omega'$ by the choice of $r$. For $d\ne2$ the extended generators and pre-Schwarzians coincide on it with the original real ones by the identity theorem, the slice being connected, so the extensions send its points to points of $\R^d$; for $d=2$ the slice is $\Omega_r$ itself. With the invariance of $\Omega_r$ just proved, the slice is therefore mapped into itself by every generator, which is \cref{it:thick-slice}.
\end{proof}

For $\iii\in\II^\N$, we denote by $T_\iii$ the uniform limit on $\ol{\Omega}$ of the prefix projections $T_{\iii|_n}$, extending the dual canonical projection to infinite words; the following lemma verifies that this limit exists.

\begin{lemma} \label{lem:dual}
  Let $\Phi=(\fii_i)_{i\in\II}$ be a conformal iterated function system. Then there is a neighborhood $U$ of $\ol{\Omega}$ such that every projection $T_\iii$, for $\iii\in\II^\N\cup\II^*$, extends real-analytically to $U$; when $d=2$, these extensions are holomorphic, so that every $T_\iii$ belongs to $\mathcal{H}(\ol{\Omega})$. For $\iii\in\II^\N$, the prefix projections $T_{\iii|_n}$ converge uniformly on $\ol{\Omega}$. We have
  \begin{equation} \label{eq:dualbounds}
    \|T_\iii-T_\jjj\|\le\frac{2\max_{i \in \II}\|T_{\fii_i}\|\cmax^{|\iii\land\jjj|}}{1-\cmax}
  \end{equation}
  for all distinct $\iii,\jjj\in\II^\N\cup\II^*$. In particular, $\iii\mapsto T_\iii$ is continuous on $\II^\N\cup\II^*$.
\end{lemma}

\begin{proof}
  Write $\beta_1=\max_{i \in \II}\|T_{\fii_i}\|$. Each $T_{\fii_i}$ is real-analytic on the extension domain $\Omega'$ and hence continuous on the compact set $\ol{\Omega}$, so $\beta_1<\infty$. No distortion theorem is needed, the generators being finitely many fixed maps. Let $\Omega^{\C}$, $\lambda$, and $r$ be as in \cref{lem:thicken}, and denote the holomorphic extensions of the generators and of the pre-Schwarzians by the same symbols as the original maps. For $d=2$ we use the scalar holomorphic pre-Schwarzian $T_{\fii_i}^{\C}=\fii_i''/\fii_i'$, for which $\|T_{\fii_i}^{\C}\|=\|T_{\fii_i}\|$, since $T_{\fii_i}^{\C}=\ol{T_{\fii_i}}$.

  Put $B=\max_{i \in \II}\sup_{z\in\Omega_r}\|T_{\fii_i}(z)\|$, using the scalar holomorphic pre-Schwarzian when $d=2$; this supremum is finite because each $T_{\fii_i}$ is continuous on the compact closure $\ol{\Omega_r}\subseteq\Omega^{\C}$. For each infinite word $\iii$, form on $\Omega_r$ the series obtained from \cref{eq:Tword} with the extensions in place of the real maps, and for $d=2$ form the scalar series obtained from \cref{eq:Tline}; for a finite word $\iii$, use the corresponding truncation at $|\iii|$. Since transpose preserves the Euclidean operator norm, the $k$-th summand has norm at most $B\lambda^{k-1}$ on $\Omega_r$. The infinite series therefore converges normally, and its sum is holomorphic on $\Omega_r$ by the Weierstrass convergence theorem. On $\ol{\Omega}$ its finite partial sums are precisely the projections $T_{\iii|_n}$ by \cref{eq:Tword,eq:Tline}; hence for an infinite word $\iii$ its restriction is the uniform limit denoted by $T_\iii$. Taking $U=\Omega_r\cap\R^d$ when $d\ne2$ and $U=\Omega_r$ when $d=2$ proves the asserted common real-analytic extension and the planar holomorphic extension.

  On $\ol{\Omega}$ the $k$-th summand has the sharper bound $\beta_1\cmax^{k-1}$, so the sums converge absolutely and uniformly. If $\iii\ne\jjj$ and $m=|\iii\land\jjj|$, then the first $m$ summands of $T_\iii$ and $T_\jjj$ agree, while the two tails are together at most $2\beta_1\sum_{k>m}\cmax^{k-1}=2\beta_1\cmax^{m}/(1-\cmax)$, giving \cref{eq:dualbounds}; for identical pairs the difference vanishes. Since the word metric is $2^{-m}$, the bound is a H\"older estimate with exponent $\alpha=\log(1/\cmax)/\log 2>0$. In particular, this proves the prefix continuity.
\end{proof}

So far each projection has figured only as a point of the function space, bounded and prefix-continuous by \cref{lem:dual}. It was produced, however, by iterating the dual maps from the zero field, and the next lemma reads off the dynamical content of that construction. Two facts result: the iteration reaches the same limit $T_\iii$ from any starting field, matching the base-point independence of the canonical projection $\pi$, while concatenating words composes the dual maps, so a shared prefix of two projections peels off through one word-level operator, the identity used throughout the separation proofs. The same construction leaves the set of projections compact, and the paragraph after the lemma reads this compactness as the invariance of a dual attractor.

\begin{lemma} \label{lem:dual-peel}
  Let $\Phi=(\fii_i)_{i\in\II}$ be a conformal iterated function system. Then $F_{\iii|_n}h\to T_\iii$ uniformly on $\ol{\Omega}$ for all $\iii\in\II^\N$ and every starting point $h$ in the space carrying the dual system. The peeling identity
  \begin{equation} \label{eq:peeling}
    T_{\kkk\lll} = F_\kkk T_{\lll} = \LL^{(1)}_\kkk T_{\lll}+T_\kkk
  \end{equation}
  holds for all $\kkk\in\II^*$ and $\lll\in\II^\N\cup\II^*$. The set $\{T_\iii : \iii\in\II^\N\}$ is compact in the supremum norm.
\end{lemma}

\begin{proof}
  Differences of dual orbits forget the translations: iterating $F_ih-F_ih'=\LL^{(1)}_i(h-h')$ along a prefix gives $F_{\iii|_n}h-F_{\iii|_n}h'=\LL^{(1)}_{\iii|_n}(h-h')$, of norm at most $\cmax^{n}\|h-h'\|$, and taking $h'=0$, where $F_{\iii|_n}0=T_{\iii|_n}$, shows $F_{\iii|_n}h\to T_\iii$ uniformly on $\ol{\Omega}$ for every starting point $h$. The peeling identity is definitional for finite $\lll$: concatenation of words is composition of dual maps, $F_{\kkk\lll}=F_\kkk\circ F_\lll$, so $T_{\kkk\lll}=F_{\kkk\lll}0=F_\kkk T_{\lll}=\LL^{(1)}_\kkk T_{\lll}+T_\kkk$, the last equality by \cref{eq:wordops} and $F_\kkk0=T_\kkk$; for infinite $\lll$, apply the finite identity to $\lll|_n$. By \cref{lem:dual}, $T_{\lll|_n}\to T_\lll$ and $T_{\kkk(\lll|_n)}\to T_{\kkk\lll}$ uniformly on $\ol{\Omega}$, while $\|F_\kkk u-F_\kkk v\|\le\cmax^{|\kkk|}\|u-v\|$ implies $F_\kkk T_{\lll|_n}\to F_\kkk T_\lll$ uniformly. Passing to the limit proves \cref{eq:peeling}.

  Finally, the set $\{T_\iii : \iii\in\II^\N\}$ is compact as the image of the compact $\II^\N$ under the continuous $\iii\mapsto T_\iii$ of \cref{lem:dual}.
\end{proof}

We call the compact set $X^*=\{T_\iii : \iii\in\II^\N\}$ the \emph{dual attractor} of $\Phi$. The case $\kkk=i$ of \cref{eq:peeling} reads $F_iT_{\lll}=T_{i\lll}$, whence 
\begin{equation*}
  X^*=\bigcup_{i}F_iX^*;
\end{equation*}
since each $F_i$ is an affine $\cmax$-contraction of a complete space, the theorem of Hutchinson \cite{Hutchinson} makes $X^*$ the unique nonempty compact set with this property.

The Schwarzian projections satisfy the same metric and cocycle properties one weight up, with the multiplier squared, and the analogous dual attractor is assembled after the proof. For $\iii\in\II^\N$, we denote by $S_\iii$ the uniform limit on $\ol{\Omega}$ of the prefix projections $S_{\iii|_n}$, extending the dual canonical projection to infinite words; the following lemma verifies that this limit exists.

\begin{lemma} \label{lem:dual2}
  Let $\Phi=(\fii_i)_{i\in\II}$ be a conformal iterated function system. Then each finite-word projection $S_\iii$ extends real-analytically to a neighborhood of $\ol{\Omega}$, and for $\iii\in\II^\N$ the prefix projections $S_{\iii|_n}$ converge uniformly on $\ol{\Omega}$. When $d\in\{1,2\}$, every projection satisfies $S_\iii=T_\iii'-\tfrac12T_\iii^2$ on $\Omega$; when $d=2$, every projection also belongs to $\mathcal{H}(\ol{\Omega})$. We have
  \begin{equation} \label{eq:dualbounds2}
    \|S_\iii-S_\jjj\|\le\frac{2\max_{i \in \II}\|S_{\fii_i}\|\cmax^{2|\iii\land\jjj|}}{1-\cmax^2}
  \end{equation}
  for all distinct $\iii,\jjj\in\II^\N\cup\II^*$. In particular, $\iii\mapsto S_\iii$ is continuous on $\II^\N\cup\II^*$. The peeling identity
  \begin{equation} \label{eq:peeling2}
    S_{\kkk\lll} = G_\kkk S_{\lll} = \LL^{(2)}_\kkk S_{\lll}+S_\kkk
  \end{equation}
  holds for all $\kkk\in\II^*$ and $\lll\in\II^\N\cup\II^*$.
\end{lemma}

\begin{proof}
  Write $\beta_2=\max_{i \in \II}\|S_{\fii_i}\|$. The metric part of the proof is that of \cref{lem:dual} under the replacements $T\mapsto S$, $F\mapsto G$, and $\LL^{(1)}\mapsto\LL^{(2)}$, with the multiplier squared and $\max_{i \in \II}\|T_{\fii_i}\|$ replaced by $\beta_2$, and the peeling identity follows as in \cref{lem:dual-peel} under the same replacements.

  Each $S_{\fii_i}$ is real-analytic on $\Omega'$ and continuous on the compact $\ol{\Omega}$, so $\beta_2<\infty$, and on the invariant real slice of \cref{lem:thicken}\cref{it:thick-slice} each finite-word $S_\iii$ inherits the regularity of $f_\iii$. For $d\in\{1,2\}$ it satisfies $S_\iii=T_\iii'-\tfrac12T_\iii^2$ by the scalar formula of \cref{sec:schwarzian}; for $d=2$ it is holomorphic there.

  Unfolded along the word, $S_\iii$ is the sum \cref{eq:Sword}, for $d=2$ the sum \cref{eq:Sline}, whose $k$-th summand is bounded on $\ol{\Omega}$ by $\|Df_{\iii|_{k-1}}\|^{2}\beta_2\le\beta_2\cmax^{2(k-1)}$, the identification of \cref{lem:scalar} preserving all supremum norms; the geometric series of ratio $\cmax^2$ then yields the absolute and uniform convergence. For distinct $\iii,\jjj\in\II^\N\cup\II^*$, put $m=|\iii\land\jjj|$. The first $m$ summands agree, while the two tails have total norm at most $2\beta_2\sum_{k>m}\cmax^{2(k-1)}=2\beta_2\cmax^{2m}/(1-\cmax^2)$, giving \cref{eq:dualbounds2}; for identical pairs the difference vanishes. Since the word metric is $2^{-m}$, this is a H\"older estimate with exponent $2\log(1/\cmax)/\log 2>0$, so $\iii\mapsto S_\iii$ is continuous on $\II^\N\cup\II^*$.

  Let $d\in\{1,2\}$ and $\iii\in\II^\N$. The proof of \cref{lem:dual} shows that on the complex neighborhood of \cref{lem:thicken} the holomorphic extensions of $T_{\iii|_n}$ converge normally to the extension of $T_\iii$. Hence $T_{\iii|_n}'\to T_\iii'$ locally uniformly there by the Weierstrass convergence theorem \cite[Theorem 5.1]{Ahlfors}. Since $S_{\iii|_n}\to S_\iii$ uniformly on $\ol{\Omega}$ by the first part of the present proof, the finite-word identities $S_{\iii|_n}=T_{\iii|_n}'-\tfrac12T_{\iii|_n}^2$ pass to the limit on $\Omega$. When $d=2$, the same normal-convergence argument shows that $S_\iii$ is holomorphic on $\Omega$ and continuous on $\ol{\Omega}$, so $S_\iii\in\mathcal H(\ol{\Omega})$.

  The peeling identity \cref{eq:peeling2} is definitional for finite $\lll$: indeed, $S_{\kkk\lll}=G_{\kkk\lll}0=G_\kkk S_{\lll}=\LL^{(2)}_\kkk S_{\lll}+S_\kkk$ by \cref{eq:wordops} and $G_\kkk0=S_\kkk$. For infinite $\lll$, apply the finite identity to $\lll|_n$. By the uniform convergence above, $S_{\lll|_n}\to S_\lll$ and $S_{\kkk(\lll|_n)}\to S_{\kkk\lll}$ uniformly on $\ol{\Omega}$ as $n\to\infty$, while $\|G_\kkk H-G_\kkk H'\|\le\cmax^{2|\kkk|}\|H-H'\|$; passing to the limit proves \cref{eq:peeling2}.
\end{proof}

We call the set $X^{*2}=\{S_\iii : \iii\in\II^\N\}$ the \emph{weight-two dual attractor} of $\Phi$; it is compact, being the image of the compact $\II^\N$ under the continuous $\iii\mapsto S_\iii$ of \cref{lem:dual2}. The case $\kkk=i$ of \cref{eq:peeling2} reads $G_iS_{\lll}=S_{i\lll}$, whence
\begin{equation*}
  X^{*2}=\bigcup_{i}G_iX^{*2};
\end{equation*}
since each $G_i$ is an affine $\cmax^2$-contraction of a complete space, the theorem of Hutchinson \cite{Hutchinson} makes $X^{*2}$ the unique nonempty compact set with this property, exactly as at weight one.

The two attractors are related through the map $h\mapsto h'-\tfrac12h^2$ by which \cref{eq:Sdef} builds the Schwarzian from the pre-Schwarzian. When $d\in\{1,2\}$, the identity $S_\iii=T_\iii'-\tfrac12T_\iii^2$ of \cref{lem:dual2} realizes $X^{*2}$ as the image of $X^*$ under this map, each $S_\iii$ being the continuous extension to $\ol{\Omega}$ of $T_\iii'-\tfrac12T_\iii^2$. The correspondence is dynamical and not merely pointwise: writing $F_ih=\fii_i'(h\circ\fii_i)+T_i$ in the scalar reading of \cref{eq:dualdef,eq:dualdef2}, a direct computation gives
\begin{equation*}
  (F_ih)'-\tfrac12(F_ih)^2 = \LL^{(2)}_i(h'-\tfrac12h^2)+S_i = G_i(h'-\tfrac12h^2)
\end{equation*}
on $\Omega$ for every differentiable $h$, holomorphic when $d=2$, the cross terms canceling through $T_i\fii_i'=\fii_i''$, so the map intertwines the dual systems and exhibits $\Phi^{*2}$ as a factor of $\Phi^*$. The kernels of \cref{lem:kernel} locate the information each system retains: a coincidence $T_\iii=T_\jjj$ of finite words forces $f_\iii=A\circ f_\jjj$ for a similarity $A$, while $S_\iii=S_\jjj$ forces only $f_\iii=M\circ f_\jjj$ for a M\"obius transformation $M$, so $\Phi^{*2}$ records the reversed compositions modulo M\"obius maps as $\Phi^*$ records them modulo similarities, and $X^{*2}$ is the quotient of $X^*$ by exactly the M\"obius relations that the Schwarzian cannot see. For $d\ge3$ the relation trivializes: every summand of \cref{eq:Sword} vanishes, as recorded after \cref{eq:Sline}, so $X^{*2}=\{0\}$ carries no information, in accordance with \cref{rem:S3}. We do not take up here the systems for which either dual attractor collapses to a singleton.

\subsection{Reduction to a coincidence of limit projections} \label{sec:reduction}

The two dual systems record the reversed compositions exactly modulo their kernels: as observed after \cref{lem:dual2}, a coincidence $T_\iii=T_\jjj$ of finite words forces $f_\iii=A\circ f_\jjj$ for a similarity $A$, and a coincidence $S_\iii=S_\jjj$ forces $f_\iii=M\circ f_\jjj$ for a M\"obius transformation $M$. The separation hypotheses of \cref{thm:main}, \cref{eq:mainhyp} at weight one and \cref{eq:mainhyp2} at weight two, ask that no two distinct equal-length words, finite or infinite, share their dual projection. For finite words this forbids, in particular, the exact overlaps $f_\iii=f_\jjj$, and at weight two also the M\"obius overlaps $f_\iii=M\circ f_\jjj$; for infinite words it says that the dual canonical projection $\iii\mapsto T_\iii$, respectively $\iii\mapsto S_\iii$, is injective on $\II^\N$, so that $X^*$, respectively $X^{*2}$, is a homeomorphic copy of the symbolic space $\II^\N$. \Cref{thm:main} converts these pointwise separations of the dual projections into the strong exponential separation condition, plain at weight one and modulo M\"obius maps at weight two.

\Cref{thm:main} is proved in the next three sections: the proof of \Cref{thm:main}\cref{it:main-pre} is in \cref{sec:C2-pre} for the plane, in \cref{sec:C1-pre} for the line, and in \cref{sec:C3-pre} for the dimensions $d\ge3$; \cref{thm:main}\cref{it:main-schw} is proved in \cref{sec:C2-mob} for the plane and in \cref{sec:C1-mob} for the line. The hypotheses \cref{eq:mainhyp,eq:mainhyp2} are quantified over finite and infinite equal-length pairs alike and, reversal being a bijection of each $\II^n$, their finite-word parts are reversal invariant. The plane and the higher dimensions run by different tools on one \emph{common strategy}: weak super-exponential condensation of the maps is differentiated into super-exponential closeness of dual projections, the common prefix of the two words is peeled off by \cref{eq:peeling} at an exponential cost that the super-exponential smallness absorbs, and compactness of $\II^\N\cup\II^*$ together with \cref{lem:dual} produces distinct equal-length limit words $\ppp$ and $\qqq$ whose dual projections coincide where the estimates reach, through all derivatives at a single point in the plane and through the value at a single point in higher dimensions. Upgrading this to the full coincidence $T_{\ppp}\equiv T_{\qqq}$, contradicting \cref{eq:mainhyp}, is where the two regimes take their separate tools: in the plane an identity-theorem step uses the holomorphy supplied by \cref{lem:dual}, and for $d\ge3$ the M\"obius pole structure recovers the pole, hence the whole projection, from that single value. At weight two, where the higher dimensions contribute nothing, the plane runs the same strategy with the Schwarzian in place of the pre-Schwarzian, the condensation now modulo M\"obius maps and \cref{eq:peeling2} and \cref{lem:dual2} standing in for \cref{eq:peeling} and \cref{lem:dual}. The line rides on the plane at both weights: real-analyticity lifts the system to a complex neighborhood of $\ol{\Omega}$, the planar case applies there, and the two-constants theorem carries the separation back down to the interval.

\subsection{The one-point gap criterion} \label{sec:gencrit}

The hypotheses of \cref{thm:main} quantify over all equal-length pairs of words, finite and infinite alike, and are not checkable as they stand. This subsection replaces them by a condition on the generators alone, which is what makes \cref{thm:main} applicable to a concrete system; it is used in \cref{ex:C2} and again, through \cref{lem:gap}, in \cref{sec:gaps}.

The hypotheses \cref{eq:mainhyp,eq:mainhyp2} carry checkable content: a gap between the projections of the generators at a single point, large against the geometric tail of the cocycle, already forces the separation of all projections, in every dimension at weight one and for $d\in\{1,2\}$ at weight two. We call this the \emph{one-point gap criterion}; on the line and at weight one it is \cite[Proposition~1.8]{BaranyKolossvaryTroscheit}.

\begin{proposition} \label{prop:gencrit}
  Let $\Phi=(\fii_i)_{i\in\II}$ be a conformal iterated function system.
  \begin{enumerate}
    \item\label{it:gencrit-pre} If there exists $\alpha>\frac{2\cmax}{1-\cmax}\max_{i \in \II}\|T_{\fii_i}\|$ such that for all $i\ne j$ in $\II$ there is $x_{i,j}\in\ol{\Omega}$ with $|T_i(x_{i,j})-T_j(x_{i,j})|\ge\alpha$, then
    \begin{equation*}
      \sup_{x\in\ol{\Omega}}|T_\iii(x)-T_\jjj(x)|>0
    \end{equation*}
    for all distinct $\iii,\jjj\in\II^\N\cup\II^*$ with $|\iii|=|\jjj|$.
    \item\label{it:gencrit-schw} If $d\in\{1,2\}$ and there exists $\alpha>\frac{2\cmax^2}{1-\cmax^2}\max_{i \in \II}\|S_{\fii_i}\|$ such that for all $i\ne j$ in $\II$ there is $x_{i,j}\in\ol{\Omega}$ with $|S_i(x_{i,j})-S_j(x_{i,j})|\ge\alpha$, then
    \begin{equation*}
      \sup_{x\in\ol{\Omega}}|S_\iii(x)-S_\jjj(x)|>0
    \end{equation*}
    for all distinct $\iii,\jjj\in\II^\N\cup\II^*$ with $|\iii|=|\jjj|$.
  \end{enumerate}
\end{proposition}

\begin{proof}
  Let $P$ denote either $T$, in any dimension, or $S$, in dimensions one and two, and let $\beta$ and $w$ be $\max_{i \in \II}\|T_{\fii_i}\|$ and $1$ in the first case and $\max_{i \in \II}\|S_{\fii_i}\|$ and $2$ in the second, so that the hypothesis reads $\alpha>2\beta\cmax^w/(1-\cmax^w)$. First, let $\iii,\jjj\in\II^\N\cup\II^*$ have equal length and $i_1 \ne j_1$. The first summand of \cref{eq:Tword}, respectively \cref{eq:Sword}, in the plane of \cref{eq:Tline}, respectively \cref{eq:Sline}, is $P_{i_1}$, the empty composition being the identity, and the $k$-th summand is bounded on $\ol{\Omega}$ by $\beta\cmax^{w(k-1)}$, as in the proof of \cref{lem:dual}, respectively \cref{lem:dual2}, so at $x_0=x_{i_1,j_1}$,
  \begin{equation} \label{eq:gencritbound}
    |P_\iii(x_0)-P_\jjj(x_0)| \ge |P_{i_1}(x_0)-P_{j_1}(x_0)|-2\beta\sum_{k\ge2}\cmax^{w(k-1)} \ge \alpha-\frac{2\beta\cmax^{w}}{1-\cmax^{w}} > 0,
  \end{equation}
  and $P_\iii\not\equiv P_\jjj$ on $\ol{\Omega}$. For general distinct $\iii,\jjj$ of equal length, let $\kkk=\iii\land\jjj$ and $m=|\kkk|$, so that $\sigma^m\iii$ and $\sigma^m\jjj$ have equal length and distinct first letters, and $B=P_{\sigma^m\iii}-P_{\sigma^m\jjj}$ is real-analytic on $\Omega$ and continuous on $\ol{\Omega}$, by \cref{lem:dual} at weight one and, through the identity $S_\lll=T_\lll'-\tfrac12T_\lll^2$ of \cref{lem:dual2}, at weight two, and does not vanish identically by the previous case. The peeling identity \cref{eq:peeling}, respectively \cref{eq:peeling2}, gives $P_\iii-P_\jjj=\LL^{(w)}_\kkk B$ with $\LL^{(w)}_\kkk$ as in \cref{eq:wordops}. Were this identically zero on $\ol{\Omega}$, the invertibility of $Df_\kkk$ would imply that $B$ vanishes on $f_\kkk(\Omega)$. Since $f_\kkk$ is a local diffeomorphism, this is a nonempty open subset of $\Omega$. The identity principle for real-analytic functions, applied componentwise, would therefore give $B\equiv0$ on $\Omega$ and hence on $\ol{\Omega}$ by continuity, a contradiction. Hence $\|P_\iii-P_\jjj\|>0$, which is \cref{eq:mainhyp} at weight one and \cref{eq:mainhyp2} at weight two.
\end{proof}

By \cref{thm:main}, a system meeting \cref{prop:gencrit}\cref{it:gencrit-pre} therefore satisfies the strong exponential separation condition, and one meeting \cref{prop:gencrit}\cref{it:gencrit-schw} satisfies it modulo M\"obius maps; in \cref{ex:C2}, we will exhibit a planar system meeting \cref{prop:gencrit}\cref{it:gencrit-schw}. The estimate \cref{eq:gencritbound} is uniform over the pairs with distinct first letters, which is the form in which \cref{sec:gaps} uses it. The criterion is checkable at a price in the contraction range, and the threshold in its hypothesis fixes that range exactly.

\begin{remark} \label{rem:gencrit-range}
  Let $P$, $\beta$, and $w$ be as in the proof above, so that $\beta=\max_{i \in \II}\|P_{\fii_i}\|$ and the hypothesis of \cref{prop:gencrit} reads $\alpha>2\beta\cmax^w/(1-\cmax^w)$; \cref{eq:frev} identifies the one-letter reversed compositions with the generators. The criterion keeps the first summand of the cocycle and bounds the remaining terms by the geometric estimate of \cref{eq:gencritbound}. The same $\beta$ controls both parts: the triangle inequality gives $\|P_{\fii_i}-P_{\fii_j}\|\le2\beta$, while the two tails cost $2\beta\cmax^w/(1-\cmax^w)$. If $\beta=0$, then every $P_{\fii_i}$ vanishes identically and no $\alpha$ meets the hypothesis; by \cref{lem:kernel} this is the case in which every generator is a similarity at weight one and every generator is M\"obius at weight two. Assume henceforth that $\beta>0$ and put
  \begin{equation*}
    \kappa = \frac{1}{\beta}\min_{i \ne j}\|P_{\fii_i}-P_{\fii_j}\|,
  \end{equation*}
  a quantity in $[0,2]$. The hypothesis of \cref{prop:gencrit} is satisfiable if and only if
  \begin{equation} \label{eq:gencritrange}
    \cmax^w < \frac{\kappa}{2+\kappa}.
  \end{equation}
  Indeed, a system meeting the hypothesis has $\kappa\beta\ge\alpha>2\beta\cmax^w/(1-\cmax^w)$. Conversely, compactness of $\ol{\Omega}$ makes every norm in the definition of $\kappa$ attained, so any $\alpha$ strictly between $2\beta\cmax^w/(1-\cmax^w)$ and $\kappa\beta$ is witnessed at some point of $\ol{\Omega}$ for each pair. Since $\kappa\le2$, the weight-one criterion is unavailable unless $\cmax<\tfrac12$ in every dimension, and the weight-two criterion is unavailable unless $\cmax<2^{-1/2}$ in dimensions one and two. These necessary bounds hold for every $N\ge2$ and every domain. The bound \cref{eq:gencritrange} decides whether the hypothesis can hold for a single system and involves neither $d$ nor $N$. By contrast, when $d\ge3$, the threshold $\cmax<N^{-2/d}$ enters the system-space approximation statement of \cref{thm:genericity}\cref{it:gen-higher}, together with the boundedness of the generators on $\Omega'$ and the quantitative hypothesis \cref{eq:genhigher}.

  The range narrows when a generator lies in the kernel of the operator. If exactly one generator $\fii_{i_0}$ is a similarity, then $T_{\fii_{i_0}}\equiv0$ by \cref{lem:kernel}\cref{it:tkernel}, and hence $\|T_{\fii_{i_0}}-T_{\fii_j}\|=\|T_{\fii_j}\|\le\beta$ for every $j\ne i_0$. Thus $\kappa\le1$, and \cref{eq:gencritrange} forces $\cmax<\tfrac13$. If at least two generators are similarities, then either $\beta=0$, as treated above, or $\kappa=0$; in either case \cref{prop:gencrit}\cref{it:gencrit-pre} is unavailable whatever the contraction. At weight two the same argument, with \cref{lem:kernel}\cref{it:skernel} in place of \cref{lem:kernel}\cref{it:tkernel}, replaces similarities by M\"obius generators and gives $\cmax<3^{-1/2}$ when exactly one generator is M\"obius and unavailability when at least two are M\"obius.

  The ceilings $\tfrac12$, $\tfrac13$, $2^{-1/2}$, and $3^{-1/2}$ are sharp. On $\Omega=(0,1)$ with extension domain $\Omega'=\R$ and $\II=\{1,2\}$, the maps
  \begin{equation*}
    \fii_1(x) = \tfrac{1}{20}+\lambda(e^{x-1}-e^{-1}) \qquad \text{and} \qquad \fii_2(x) = \tfrac13+\lambda(1-e^{-x})
  \end{equation*}
  send $\ol{\Omega}$ into $\Omega$ and have derivatives $\fii_1'(x)=\lambda e^{x-1}$ and $\fii_2'(x)=\lambda e^{-x}$, so they form a conformal iterated function system with $\cmax=\lambda$ for every $\lambda\in(0,1)$. Their pre-Schwarzians are the constants $T_{\fii_1}\equiv1$ and $T_{\fii_2}\equiv-1$, so $\beta=1$ and $\kappa=2$, and \cref{eq:gencritrange} makes \cref{prop:gencrit}\cref{it:gencrit-pre} available exactly for $\lambda<\tfrac12$. Replacing $\fii_1$ by the similarity $s_\lambda(x)=(1-\lambda)/2+\lambda x$, which maps $\ol{\Omega}$ into $\Omega$ for every $\lambda\in(0,1)$, leaves $\cmax=\lambda$ and $\beta=1$ but gives $\kappa=1$, so the weight-one criterion is available exactly for $\lambda<\tfrac13$. At weight two, on the same $\Omega$ with $\Omega'=(-1,\tfrac32)$, the maps $\psi_1(x)=\tfrac{1}{50}+\lambda\cos^2(1)\tan x$ and $\psi_2(x)=\tfrac15+\lambda\tanh x$ form a conformal iterated function system with $\cmax=\lambda$ for every $\lambda\in(0,1)$. Their Schwarzians are the constants $2$ and $-2$, so $\beta=2$ and $\kappa=2$, and the weight-two criterion is available exactly for $\lambda<2^{-1/2}$. Replacing $\psi_1$ by $s_\lambda$ gives $\beta=2$ and $\kappa=1$, so the criterion is available exactly for $\lambda<3^{-1/2}$.

  In dimensions $d\ge3$, the pre-Schwarzian of a generator is either $0$ or the nonconstant pole field of \cref{eq:pole}, but the universal bound $\kappa\le2$ remains sharp. Write $x=(s,y)\in\R\times\R^{d-1}$, fix $0<\eps<1$, and consider the bounded domain
  \begin{equation*}
    \Omega_\eps = \{(s,y)\in\R\times\R^{d-1} : 0<|y|<\eps \text{ and } |s|<|y|^2/4\}.
  \end{equation*}
  Since $d-1\ge2$, the punctured $y$-ball is connected, and hence so is $\Omega_\eps$. Put $p_1=(1,0)$ and $p_2=(-1,0)$. For $x=(s,y)\in\ol{\Omega_\eps}$, the bound $2|s|\le|y|^2/2$ gives $|x-p_i|\ge1$ for both $i$, with equality at $x=0$. Choose a bounded extension domain $\Omega'\supset\ol{\Omega_\eps}$ avoiding $p_1$ and $p_2$, fix $a\in\Omega_\eps$, and take $0<\rho<\min\{1,\dist(a,\partial\Omega_\eps)\}$. For $i\in\{1,2\}$, define
  \begin{equation*}
    \fii_i(x) = a+\rho\frac{x-p_i}{|x-p_i|^2}.
  \end{equation*}
  These M\"obius maps send $\ol{\Omega_\eps}$ into $\Omega_\eps$ and have conformal factors $\rho/|x-p_i|^2$, so they form a conformal iterated function system with $\cmax=\rho$. By \cref{eq:pole}, both pre-Schwarzians have norm $2$, while $T_{\fii_1}(0)=2e_1$ and $T_{\fii_2}(0)=-2e_1$. Thus $\beta=2$, $\|T_{\fii_1}-T_{\fii_2}\|=4$, and $\kappa=2$.
\end{remark}

\section{The planar case} \label{sec:C2}

With the preliminaries in place, we take up the proof of \cref{thm:main} outlined in \cref{sec:reduction}, and it begins in the plane rather than on the line. The plane comes first because it has for free the holomorphy that the argument runs on, and the line, which must buy that holomorphy by complexification, draws its separation from the results of this section rather than earning it on its own terms. The dimensions $d\ge3$ stand beside the plane rather than under it, sharing its skeleton but none of its function theory, which M\"obius rigidity replaces.

Throughout this section $d=2$, and we identify $\R^2$ with $\C$, so that, by the classification of \cref{sec:conformal-maps}, the generators are holomorphic with nonvanishing derivative; by \cref{sec:dual-ifs}, the dual projection of a finite word $\iii$ is at weight one the holomorphic function $T_\iii=f_\iii''/f_\iii'$ and at weight two the classical Schwarzian $S_\iii=S^{\C}_{f_\iii}$, and the norm of $h\in\mathcal{H}(\ol{\Omega})$ is $\|h\|=\sup_{x\in\ol{\Omega}}|h(x)|$. The section carries the planar program in two parts, and the line is deduced from both in \cref{sec:C1}. The first part proves the case $d=2$ of \cref{thm:main}\cref{it:main-pre}, which generalizes the sufficient separation condition of B\'ar\'any, Kolossv\'ary, and Troscheit \cite[Theorem~1.5]{BaranyKolossvaryTroscheit} from the real line to the plane, on the common strategy of \cref{sec:reduction}, which is the skeleton of their proof \cite[Section~3]{BaranyKolossvaryTroscheit}; everything after the differentiation of the condensation is factored into \cref{lem:coincidence}, a coincidence lemma stated for an abstract family of holomorphic projections, continuous in the word and obeying a peeling identity, two properties that the pre-Schwarzian and the Schwarzian projections share, and the differentiation itself into \cref{lem:pre-diff} and \cref{lem:schw-diff}.

The second part runs the same program one group level up, replacing the pre-Schwarzian by the Schwarzian: through \cref{lem:coincidence}, applied to the Schwarzian projections, a pointwise separation of the Schwarzian cocycle yields exponential separation that no post-composition by a M\"obius transformation can defeat. The machinery that both parts use is collected in \cref{sec:C2-machinery}: the constants and the geometry, the Cauchy estimate of \cref{lem:cauchy}, the trapped disc of \cref{lem:disc}, the notion of an admissible family of projections, and \cref{lem:coincidence} itself. That subsection is stated for an abstract admissible family and is therefore free of the weight; \cref{sec:C2-pre} and \cref{sec:C2-mob} each verify that their projections form such a family, differentiate the condensation into the decay hypothesis of \cref{lem:coincidence}, and conclude. The section closes, in \cref{sec:C2-compare}, with a comparison of the two hypotheses and an example. The constants and the geometry fixed in \cref{sec:C2-machinery} are also the ones \cref{sec:C1,sec:gen-plane} import.

\subsection{Admissible families and the coincidence lemma} \label{sec:C2-machinery}

Differentiation in the plane costs one scale: through Cauchy's estimate, recorded in \cref{lem:cauchy} below, a holomorphic function controls all derivatives of itself on any slightly smaller set. Applied to differences of compositions and of projections, whose supremum norms the hypotheses control, this single inequality produces every derivative bound the section needs. Before stating the lemma, we fix the constants and the geometry that the estimates of this section use.

Every composition $f_\iii$ lies in $\mathcal{H}(\ol{\Omega})$, the generators being holomorphic self-maps of $\Omega$ that extend continuously to $\ol{\Omega}$, and the multiplicativity of the conformal factor recorded in \cref{sec:composition}, applied along an orbit that stays in $\ol{\Omega}$, gives $\cmin^{|\iii|}\le|f_\iii'|\le\cmax^{|\iii|}$ on $\ol{\Omega}$. Put $C_0=2\max_{i \in \II}\|\fii_i''/\fii_i'\|/(1-\cmax)$, which is finite since each $\fii_i''/\fii_i'$ is holomorphic on $\Omega'$ and hence continuous on the compact set $\ol{\Omega}$. For a finite word $\iii=i_1\cdots i_n$ the chain rule telescopes the derivative into the product $f_\iii'=\prod_{k=1}^{n}\fii_{i_k}'\circ f_{\iii|_{k-1}}$, whose logarithmic derivative is the sum $f_\iii''/f_\iii'=\sum_{k=1}^{n}((\fii_{i_k}''/\fii_{i_k}')\circ f_{\iii|_{k-1}})f_{\iii|_{k-1}}'$ recorded in \cref{eq:Tline}; the $k$-th summand is bounded by $\max_{i \in \II}\|\fii_i''/\fii_i'\|\cmax^{k-1}$ on $\ol{\Omega}$, the orbit of the prefixes staying in $\ol{\Omega}$, so $|f_\iii''/f_\iii'|\le\max_{i \in \II}\|\fii_i''/\fii_i'\|/(1-\cmax)\le C_0$ and hence $|f_\iii''|\le C_0\cmax^{|\iii|}$ on $\ol{\Omega}$ for every finite word $\iii$.

The attractor $X$ is a nonempty compact subset of the bounded domain $\Omega$, so $\dist(X,\partial\Omega)$ is positive and finite; for $0<r<\dist(X,\partial\Omega)$ put $K_r=\{x\in\ol{\Omega} : \dist(x,\partial\Omega)\ge r\}$, and fix $\rho=\tfrac12\dist(X,\partial\Omega)$ and $z_0\in X$. Then each $K_r$ is a compact subset of $\Omega$ containing $X$ and satisfying $\dist(K_r,\partial\Omega)\ge r$. Since a nearest point of the closed set $\C\setminus\Omega$ to a point of $\Omega$ lies on $\partial\Omega$, we have $\dist(z_0,\C\setminus\Omega)=\dist(z_0,\partial\Omega)\ge2\rho$, so $\Omega$ contains the neighborhood $B^o(z_0,2\rho)$ of $B(z_0,\rho)$; moreover $B(z_0,\rho)\subseteq K_\rho$, since every $z\in B(z_0,\rho)$ then lies in $\ol{\Omega}$ and has $\dist(z,\partial\Omega)\ge\dist(z_0,\partial\Omega)-\rho\ge\rho$. Furthermore, $f_\kkk(z_0)\in X$ for every $\kkk\in\II^*$: the invariance $X=\bigcup_{i \in \II}\fii_i(X)$ gives $\fii_i(X)\subseteq X$, and $f_\kkk$ is a composition of generators.

\begin{lemma} \label{lem:cauchy}
  Let $h$ be holomorphic and bounded on a disc $B^o(z,r)\subseteq\C$. Then
  \begin{equation} \label{eq:cauchy}
    |h^{(k)}(z)| \le \frac{k!}{r^k}\sup_{w\in B^o(z,r)}|h(w)|
  \end{equation}
  for all integers $k\ge0$. In particular, $\sup_{x\in K_r}|h^{(k)}(x)|\le k!r^{-k}\|h\|$ for all $h\in\mathcal{H}(\ol{\Omega})$, integers $k\ge0$, and $0<r<\dist(X,\partial\Omega)$.
\end{lemma}

\begin{proof}
  For $0<r'<r$ the Cauchy integral formula on the circle of radius $r'$ about $z$ gives $|h^{(k)}(z)|\le k!(r')^{-k}\sup_{w\in B^o(z,r)}|h(w)|$, which is Cauchy's estimate \cite[Chapter~4, (25)]{Ahlfors}, and letting $r'\to r$ gives \cref{eq:cauchy}. If $0<r<\dist(X,\partial\Omega)$ and $x\in K_r$, then $r>0$ implies $x\notin\partial\Omega$, and hence $x\in\Omega$. Therefore $\dist(x,\C\setminus\Omega)=\dist(x,\partial\Omega)\ge r$, so $B^o(x,r)\subseteq\Omega$. Every $h\in\mathcal{H}(\ol{\Omega})$ is holomorphic on this disc and satisfies $|h|\le\|h\|$ there, so \cref{eq:cauchy} applies at $x$ with this radius.
\end{proof}

The proof of \cref{lem:coincidence} below reads the smallness of a dual difference on the image of a fixed disc under a composition $f_\kkk$ and must convert it into derivative decay at the image of the center. For this the image must contain a disc of controlled radius about the image of the center. Being a long composition, $f_\kkk$ need not be univalent, so the Koebe one-quarter theorem does not apply; instead, a lower bound on the derivative at the center and an upper bound on the second derivative trap a disc through Rouch\'e's theorem. The radius is exponentially small in the length of the word, which is enough against super-exponential errors.

\begin{lemma} \label{lem:disc}
  Let $r,a,M>0$, let $z_0\in\C$, and let $g$ be holomorphic on a neighborhood of $B(z_0,r)$ with $|g'(z_0)|\ge a$ and $\sup_{z\in B(z_0,r)}|g''(z)|\le M$. Then
  \begin{equation} \label{eq:disc}
    B^o(g(z_0),as/2) \subseteq g(B^o(z_0,r)),
  \end{equation}
  where $s=\min\{r,a/M\}$.
\end{lemma}

\begin{proof}
  Integrating $g''$ along segments from $z_0$ gives $|g'(w)-g'(z_0)|\le M|w-z_0|$ for $|w-z_0|\le s$, and integrating this bound in turn gives $|g(z)-g(z_0)-g'(z_0)(z-z_0)|\le Ms^2/2$ on the circle $|z-z_0|=s$, whence $|g(z)-g(z_0)|\ge as-Ms^2/2\ge as/2$ there, since $Ms\le a$. Fix $w_0$ with $|w_0-g(z_0)|<as/2$. On the circle $|z-z_0|=s$ the constant $w_0-g(z_0)$ has modulus strictly smaller than $|g(z)-g(z_0)|$, so by Rouch\'e's theorem \cite[p.~153]{Ahlfors} the functions $g-g(z_0)$ and $g-w_0=(g-g(z_0))+(g(z_0)-w_0)$ have the same number of zeros, counted with multiplicity, in $B^o(z_0,s)$. Since $g-g(z_0)$ vanishes at $z_0$, it has at least one zero in $B^o(z_0,s)$, and hence so does $g-w_0$. Thus $w_0\in g(B^o(z_0,s))\subseteq g(B^o(z_0,r))$.
\end{proof}

The heart of the section is the following lemma, which runs the peeling, the passage to the limit pair, and the endgame once, for any family of projections with two properties, and says that a super-exponential decay of differences along a sequence of distinct equal-length pairs forces an exact coincidence of two limit projections. We call a family $(P_\lll)_{\lll\in\II^\N\cup\II^*}$ in $\mathcal{H}(\ol{\Omega})$ an \emph{admissible family of projections} with rate $\gamma\in(0,1]$ if
\begin{enumerate}
  \item\label{it:adm-cont} the map $\lll\mapsto P_\lll$ is continuous from $\II^\N\cup\II^*$ to $\mathcal{H}(\ol{\Omega})$ in the supremum norm, 
  \item\label{it:adm-peel} for every $\kkk\in\II^*$ there is a function $g_\kkk\colon\ol{\Omega}\to\C$ with $|g_\kkk|\ge\gamma^{|\kkk|}$ on $\ol{\Omega}$ such that $P_{\kkk\lll}=g_\kkk(P_\lll\circ f_\kkk)+P_\kkk$ on $\ol{\Omega}$ for all $\lll\in\II^\N\cup\II^*$.
\end{enumerate}
The pre-Schwarzian projections form such a family with $g_\kkk=f_\kkk'$, as \cref{lem:pre-admissible} records in \cref{sec:C2-pre}, and the Schwarzian projections form one with $g_\kkk=(f_\kkk')^2$, as \cref{lem:schw-admissible} records in \cref{sec:C2-mob}; both the planar case of \cref{thm:main}\cref{it:main-pre} and \cref{thm:main}\cref{it:main-schw} reduce to the lemma, each after its own differentiation step, \cref{lem:pre-diff} at weight one and \cref{lem:schw-diff} at weight two.

\begin{lemma} \label{lem:coincidence}
  Let $\Phi=(\fii_i)_{i\in\II}$ be a conformal iterated function system on $\C$, let $\rho$ and $K_\rho$ be as above, read for $\Phi$, and let $(P_\lll)_{\lll\in\II^\N\cup\II^*}$ be an admissible family of projections with rate $\gamma$. Suppose that there are a strictly increasing sequence $(n_\ell)_{\ell \ge 1}$ and distinct $\iii_\ell,\jjj_\ell\in\II^{n_\ell}$ with 
  \begin{equation*}
    \sup_{x\in K_\rho}|P_{\iii_\ell}(x)-P_{\jjj_\ell}(x)|\le\eta_{n_\ell}
  \end{equation*}
  for numbers $\eta_{n_\ell}\ge0$ satisfying $\log\eta_{n_\ell}/n_\ell\to-\infty$. Then there are distinct $\ppp,\qqq\in\II^\N\cup\II^*$ with $|\ppp|=|\qqq|$, $\ppp_1\ne\qqq_1$, and 
  \begin{equation*}
    P_{\ppp}\equiv P_{\qqq}
  \end{equation*}
  on $\ol{\Omega}$.
\end{lemma}

\begin{proof}
  The dependence of $\iii=\iii_\ell$ and $\jjj=\jjj_\ell$ on $\ell$ is suppressed below. Since the conclusion asks for words with distinct first letters, the common prefix is peeled off at an exponential cost. Let $\kkk=\iii\land\jjj$, $m=|\kkk|<n_\ell$, $\hi=\sigma^m\iii$, and $\hj=\sigma^m\jjj$, so that $\iii=\kkk\hi$ and $\jjj=\kkk\hj$, and $\hi$ and $\hj$ are nonempty words of common length $n_\ell-m$ with $\hi_1\ne\hj_1$. Subtracting the peeling identities $P_\iii=g_\kkk(P_{\hi}\circ f_\kkk)+P_\kkk$ and $P_\jjj=g_\kkk(P_{\hj}\circ f_\kkk)+P_\kkk$, the instances of condition \cref{it:adm-peel} with the prefix $\kkk$, and writing $\Xi_\ell=P_{\hi}-P_{\hj}$ gives $P_\iii-P_\jjj=g_\kkk(\Xi_\ell\circ f_\kkk)$, where $|g_\kkk|\ge\gamma^{m}\ge\gamma^{n_\ell}$ on $\ol{\Omega}$, so $\sup_{x\in K_\rho}|(\Xi_\ell\circ f_\kkk)(x)|\le\gamma^{-n_\ell}\eta_{n_\ell}=\eta_{n_\ell}'$, still with $\log\eta_{n_\ell}'/n_\ell\to-\infty$; in other words, $|\Xi_\ell|\le\eta_{n_\ell}'$ on the image $f_\kkk(K_\rho)$.

  The suffixes left by the peeling produce the limit pair $\ppp$ and $\qqq$. By compactness of $(\II^\N\cup\II^*)^2$, pass to a subsequence along which $\hi\to\ppp$ and $\hj\to\qqq$. Since the empty word is isolated and every $\hi,\hj$ is nonempty, both limits are nonempty. If one limit is finite, then the corresponding sequence is eventually constant, and the equality $|\hi|=|\hj|$ forces the other limit to be finite of the same length; otherwise both lengths tend to infinity. Thus $|\ppp|=|\qqq|$, and convergence of the first letters gives $p_1 \ne q_1$, so $\ppp$ and $\qqq$ are distinct words of equal length. Every projection $P_\lll$ belongs to $\mathcal{H}(\ol{\Omega})$ and, by condition \cref{it:adm-cont}, $\lll\mapsto P_\lll$ is continuous, so $\Xi_\ell\to \Xi_0=P_{\ppp}-P_{\qqq}$ uniformly on $\ol{\Omega}$, and applying \cref{lem:cauchy} to $\Xi_\ell-\Xi_0\in\mathcal{H}(\ol{\Omega})$ shows that $\Xi_\ell^{(k)}\to \Xi_0^{(k)}$ uniformly on $K_\rho$ for every $k$. Passing to a further subsequence, the points $x_\ell=f_\kkk(z_0)\in X$ converge to some $\zeta\in X$.

  Cauchy's estimate needs a disc of controlled radius about $x_\ell$ inside the image $f_\kkk(K_\rho)$, and the derivative bounds trap one there. The map $f_\kkk$ is holomorphic on $\Omega$, which contains a neighborhood of $B(z_0,\rho)$, and satisfies $|f_\kkk'(z_0)|\ge\cmin^{m}$ and $|f_\kkk''|\le\max\{C_0,1\}\cmax^{m}$ on $\ol{\Omega}$, so \cref{lem:disc} shows that $f_\kkk(B^o(z_0,\rho))\subseteq f_\kkk(K_\rho)$ contains the disc $B^o(x_\ell,r_\ell)$ of radius
  \begin{equation} \label{eq:discradius}
    r_\ell = \frac{\cmin^{m}}{2}\min\biggl\{\rho,\frac{\cmin^{m}}{\max\{C_0,1\}\cmax^{m}}\biggr\} \ge r_0b^{n_\ell},
  \end{equation}
  where $b=\cmin^2/\cmax<1$ and $r_0=\tfrac12\min\{\rho,1/\max\{C_0,1\}\}$, using $\min\{\rho,tc\}\ge t\min\{\rho,c\}$ for $0<t\le1$ together with $m\le n_\ell$. Since $|\Xi_\ell|\le\eta_{n_\ell}'$ on $B^o(x_\ell,r_\ell)\subseteq f_\kkk(K_\rho)\subseteq\Omega$ and $\Xi_\ell$ is holomorphic on $\Omega$, the estimate \cref{eq:cauchy} at the center gives
  \begin{equation} \label{eq:allders}
    |\Xi_\ell^{(k)}(x_\ell)| \le k!r_\ell^{-k}\eta_{n_\ell}' \le k!r_0^{-k}b^{-kn_\ell}\eta_{n_\ell}',
  \end{equation}
  which tends to $0$ as $\ell\to\infty$ for every fixed $k$, the factor $b^{-kn_\ell}$ being only exponential in $n_\ell$. For every integer $k\ge0$, the continuity of $\Xi_0^{(k)}$, the uniform convergence $\Xi_\ell^{(k)}\to\Xi_0^{(k)}$ on $K_\rho$, and the estimate $\Xi_\ell^{(k)}(x_\ell)\to0$ give $\Xi_0^{(k)}(\zeta)=0$.

  Finally, $\Xi_0=P_{\ppp}-P_{\qqq}$ is holomorphic on $\Omega$, the family lying in $\mathcal{H}(\ol{\Omega})$, and every derivative of $\Xi_0$ vanishes at $\zeta\in\Omega$. Hence the Taylor series of $\Xi_0$ at $\zeta$ vanishes identically, so $\Xi_0$ vanishes near $\zeta$ and therefore throughout the connected domain $\Omega$ by the identity theorem \cite[pp.~126--127]{Ahlfors}, and on $\ol{\Omega}$ by continuity. Thus $P_{\ppp}\equiv P_{\qqq}$ on $\ol{\Omega}$ for the distinct equal-length pair $\ppp$ and $\qqq$.
\end{proof}

\subsection{Separation by the pre-Schwarzian} \label{sec:C2-pre}

The pre-Schwarzian projections meet the two conditions of an admissible family, the multiplier being the derivative $f_\kkk'$ of the composition along the peeled prefix.

\begin{lemma} \label{lem:pre-admissible}
  Let $\Phi=(\fii_i)_{i\in\II}$ be a conformal iterated function system on $\C$. Then the pre-Schwarzian projections $(T_\lll)_{\lll\in\II^\N\cup\II^*}$ form an admissible family of projections with rate $\cmin$, the multipliers being $g_\kkk=f_\kkk'$.
\end{lemma}

\begin{proof}
  By \cref{lem:dual}, every $T_\lll$ belongs to $\mathcal{H}(\ol{\Omega})$ and $\lll\mapsto T_\lll$ is continuous on $\II^\N\cup\II^*$, which is condition \cref{it:adm-cont}. The peeling identity \cref{eq:peeling} of \cref{lem:dual-peel}, read through the planar word operator $\LL^{(1)}_\kkk h=f_\kkk'(h\circ f_\kkk)$ of \cref{eq:wordops}, gives $T_{\kkk\lll}=f_\kkk'(T_\lll\circ f_\kkk)+T_\kkk$ on $\ol{\Omega}$ for all $\kkk\in\II^*$ and $\lll\in\II^\N\cup\II^*$, and $|f_\kkk'|\ge\cmin^{|\kkk|}$ on $\ol{\Omega}$, which is condition \cref{it:adm-peel} with $g_\kkk=f_\kkk'$ and $\gamma=\cmin$.
\end{proof}

The differentiation step converts the closeness of a composition $f_\iii$ to a comparison map $\psi$ into the closeness of their pre-Schwarzians on the compact $K_\rho$, at an exponential cost in the length of the word. It is stated for a general $\psi\in\mathcal{H}(\ol{\Omega})$, since the Schwarzian counterpart in \cref{sec:C2-mob} compares $f_\iii$ with a M\"obius image of another composition; the smallness threshold on the error $\eta$ is met, in both applications, once the length of the word is large. In the following lemma, let $\rho$, $K_\rho$, and $C_0$ be as in \cref{sec:C2-machinery}.

\begin{lemma} \label{lem:pre-diff}
  Fix $\iii\in\II^n$ and let $\psi\in\mathcal{H}(\ol{\Omega})$ satisfy $\|f_\iii-\psi\|\le\eta$ for some $0\le\eta\le\min\{\tfrac{\rho}{2}\cmin^{n},\cmax^{n}\}$. Then $|\psi'|\ge\tfrac12\cmin^{n}$ and $|\psi''|\le C_1\cmax^{n}$ on $K_\rho$, with $C_1=C_0+2\rho^{-2}$, and
  \begin{equation} \label{eq:prediff}
    \sup_{x\in K_\rho}\biggl|\frac{f_\iii''(x)}{f_\iii'(x)}-\frac{\psi''(x)}{\psi'(x)}\biggr| \le \biggl(\frac{2C_0}{\rho}+\frac{4}{\rho^2}\biggr)\cmin^{-2n}\eta.
  \end{equation}
\end{lemma}

\begin{proof}
  The difference $h=f_\iii-\psi$ lies in $\mathcal{H}(\ol{\Omega})$ with $\|h\|\le\eta$, so \cref{lem:cauchy} gives $|h'|\le\rho^{-1}\eta$ and $|h''|\le2\rho^{-2}\eta$ on $K_\rho$. On $\ol{\Omega}$ the composition satisfies $\cmin^{n}\le|f_\iii'|\le\cmax^{n}$ and $|f_\iii''|\le C_0\cmax^{n}$, so on $K_\rho$ we get $|\psi'|\ge|f_\iii'|-|h'|\ge\cmin^{n}-\rho^{-1}\eta\ge\tfrac12\cmin^{n}$ and $|\psi''|\le|f_\iii''|+|h''|\le C_0\cmax^{n}+2\rho^{-2}\eta\le C_1\cmax^{n}$, by the two bounds in the hypothesis on $\eta$. Writing the difference of the two quotients over the common denominator as
  \begin{equation*}
    \frac{f_\iii''}{f_\iii'}-\frac{\psi''}{\psi'} = \frac{f_\iii''(\psi'-f_\iii')+f_\iii'(f_\iii''-\psi'')}{f_\iii'\psi'},
  \end{equation*}
  we obtain, on $K_\rho$,
  \begin{equation*}
    \biggl|\frac{f_\iii''}{f_\iii'}-\frac{\psi''}{\psi'}\biggr| \le \frac{|f_\iii''||h'|}{|f_\iii'||\psi'|}+\frac{|h''|}{|\psi'|} \le \frac{2C_0\cmax^{n}}{\rho\cmin^{2n}}\eta+\frac{4}{\rho^{2}\cmin^{n}}\eta,
  \end{equation*}
  and $\cmax^{n}\le1$ together with $\cmin^{-n}\le\cmin^{-2n}$ gives \cref{eq:prediff}.
\end{proof}

We turn to the proof of \cref{thm:main}\cref{it:main-pre}, the first of the two reductions to \cref{lem:coincidence}. A failure of the strong exponential separation condition condenses the compositions super-exponentially along a sequence of distinct equal-length pairs; \cref{lem:pre-diff} carries that smallness to the pre-Schwarzian differences $T_\iii-T_\jjj$, meeting the decay hypothesis of \cref{lem:coincidence} for the admissible family of \cref{lem:pre-admissible}; and the coincidence that \cref{lem:coincidence} returns contradicts \cref{eq:mainhyp}.

\begin{proof}[Proof of \cref{thm:main}\cref{it:main-pre} for $d=2$]
  Suppose to the contrary that $\Phi$ satisfies \cref{eq:mainhyp} for all distinct equal-length words but not the strong exponential separation condition. The plain version of \cref{lem:condensation-mobius}, with the identity in place of every M\"obius transformation, provides a strictly increasing sequence $(n_\ell)_{\ell \ge 1}$ and distinct $\uuu_\ell,\vvv_\ell\in\II^{n_\ell}$ such that $\log\|\fii_{\uuu_\ell}-\fii_{\vvv_\ell}\|/n_\ell\to-\infty$. Set $\iii_\ell=\overleftarrow{\uuu_\ell}$ and $\jjj_\ell=\overleftarrow{\vvv_\ell}$. Since reversal is a bijection of each $\II^n$ and $f_{\iii_\ell}=\fii_{\uuu_\ell}$ and $f_{\jjj_\ell}=\fii_{\vvv_\ell}$ by \cref{eq:frev}, the words $\iii_\ell$ and $\jjj_\ell$ are distinct and their errors $\eta_{n_\ell}=\|f_{\iii_\ell}-f_{\jjj_\ell}\|$ satisfy $\log\eta_{n_\ell}/n_\ell\to-\infty$, where the logarithm of $0$ is read as $-\infty$; the dependence of $\iii$ and $\jjj$ on $\ell$ is suppressed below.

  It remains to differentiate the condensation into the decay hypothesis of \cref{lem:coincidence} for the family $(T_\lll)_{\lll\in\II^\N\cup\II^*}$. Discarding finitely many $\ell$, we may assume $\eta_{n_\ell}\le\min\{\tfrac{\rho}{2}\cmin^{n_\ell},\cmax^{n_\ell}\}$, the errors being super-exponentially small, so \cref{lem:pre-diff}, applied to $\psi=f_\jjj\in\mathcal{H}(\ol{\Omega})$, gives $\sup_{x\in K_\rho}|T_\iii(x)-T_\jjj(x)|\le(2C_0/\rho+4/\rho^2)\cmin^{-2n_\ell}\eta_{n_\ell}=\eta_{n_\ell}'$, and $\log\eta_{n_\ell}'/n_\ell\le\log\eta_{n_\ell}/n_\ell+2\log(1/\cmin)+\log(2C_0/\rho+4/\rho^2)/n_\ell\to-\infty$. By \cref{lem:coincidence}, applied with the errors $\eta_{n_\ell}'$ to the family $(T_\lll)_{\lll\in\II^\N\cup\II^*}$, admissible by \cref{lem:pre-admissible}, there are distinct $\ppp,\qqq\in\II^\N\cup\II^*$ with $|\ppp|=|\qqq|$ and $T_{\ppp}\equiv T_{\qqq}$ on $\ol{\Omega}$, which is the desired contradiction with \cref{eq:mainhyp}.
\end{proof}

\subsection{Separation modulo M\"obius maps by the Schwarzian} \label{sec:C2-mob}

By \cref{lem:kernel}, post-composition by a similarity is invisible to the pre-Schwarzian and post-composition by a M\"obius map is invisible to the Schwarzian, and these are exactly the invisible families; a Schwarzian analogue of \cref{eq:mainhyp} is thus a stronger hypothesis, forbidding the larger family of coincidences $f_\iii=M\circ f_\jjj$, and it earns a stronger conclusion: exponential separation that no post-composition by a M\"obius transformation can defeat. The separation notion matching the Schwarzian kernel of \cref{lem:kernel}\cref{it:skernel} is the strong exponential separation condition modulo M\"obius maps of \cref{eq:sescmob}, which quantifies over the group that the Schwarzian cannot see. The Schwarzian projections need no new construction, since $S_\iii=S^{\C}_{f_\iii}$ is the sum \cref{eq:Sline} of \cref{sec:dual-ifs}, whose properties \cref{lem:dual2} records; the following lemma recasts those properties as the admissibility of the Schwarzian family, with the multiplier squared. The case $d\ge3$ is void for the Schwarzian, as \cref{rem:S3} explains; the line, in contrast, is covered by the same statement, deduced in \cref{sec:C1}.

\begin{lemma} \label{lem:schw-admissible}
  Let $\Phi=(\fii_i)_{i\in\II}$ be a conformal iterated function system on $\C$. Then the Schwarzian projections $(S_\lll)_{\lll\in\II^\N\cup\II^*}$ form an admissible family of projections with rate $\cmin^2$, the multipliers being $g_\kkk=(f_\kkk')^{2}$.
\end{lemma}

\begin{proof}
  By \cref{lem:dual2}, every $S_\lll$ belongs to $\mathcal{H}(\ol{\Omega})$ and $\lll\mapsto S_\lll$ is continuous on $\II^\N\cup\II^*$, which is condition \cref{it:adm-cont}. The peeling identity \cref{eq:peeling2} of \cref{lem:dual2}, read through the planar word operator $\LL^{(2)}_\kkk h=(f_\kkk')^{2}(h\circ f_\kkk)$ of \cref{eq:wordops}, gives $S_{\kkk\lll}=(f_\kkk')^{2}(S_\lll\circ f_\kkk)+S_\kkk$ on $\ol{\Omega}$ for all $\kkk\in\II^*$ and $\lll\in\II^\N\cup\II^*$, and $|f_\kkk'|^{2}\ge\cmin^{2|\kkk|}$ on $\ol{\Omega}$, which is condition \cref{it:adm-peel} with $g_\kkk=(f_\kkk')^{2}$ and $\gamma=\cmin^2$.
\end{proof}

The differentiation step now runs one derivative deeper: the closeness of a composition $f_\iii$ to a comparison map $\psi$ is carried to the closeness of their scalar Schwarzians on $K_\rho$, through the third derivative that the Schwarzian involves. The estimate rests on \cref{lem:pre-diff}, which supplies the lower bound on $\psi'$ and the closeness of the quotients $f_\iii''/f_\iii'$ and $\psi''/\psi'$, and on one further application of \cref{lem:cauchy}; the loss is again exponential in the length of the word, one factor $\cmin^{-n}$ more than in \cref{eq:prediff}. In the following lemma, let $\rho$, $K_\rho$, and $C_0$ be as in \cref{sec:C2-machinery}.

\begin{lemma} \label{lem:schw-diff}
  Fix $\iii\in\II^n$ and let $\psi\in\mathcal{H}(\ol{\Omega})$ satisfy $\|f_\iii-\psi\|\le\eta$ for some $0\le\eta\le\min\{\tfrac{\rho}{2}\cmin^{n},\cmax^{n}\}$. Then
  \begin{equation} \label{eq:schwdiff}
    \sup_{x\in K_\rho}|S^{\C}_{f_\iii}(x)-S^{\C}_\psi(x)| \le C_3\cmin^{-3n}\eta,
  \end{equation}
  where $C_3>0$ depends only on $\rho$ and $C_0$.
\end{lemma}

\begin{proof}
  By \cref{lem:pre-diff}, $|\psi'|\ge\tfrac12\cmin^{n}$ and $|\psi''|\le C_1\cmax^{n}$ on $K_\rho$, so $\psi$ has nonvanishing derivative on $K_\rho$ and $S^{\C}_\psi$ is defined there. The difference $h=f_\iii-\psi$ lies in $\mathcal{H}(\ol{\Omega})$ with $\|h\|\le\eta$, so \cref{lem:cauchy} gives $|h'|\le\rho^{-1}\eta$ and $|h'''|\le6\rho^{-3}\eta$ on $K_\rho$. On $\ol{\Omega}$ the composition satisfies $\cmin^{n}\le|f_\iii'|\le\cmax^{n}$ and $|f_\iii''|\le C_0\cmax^{n}$, and on $K_\rho$ also $|f_\iii'''|\le(\rho^{-1}C_0+C_0^2)\cmax^{n}=C_2\cmax^{n}$: the quotient $f_\iii''/f_\iii'$ lies in $\mathcal{H}(\ol{\Omega})$ with $|f_\iii''/f_\iii'|\le C_0$, so \cref{lem:cauchy} gives $|(f_\iii''/f_\iii')'|\le\rho^{-1}C_0$ on $K_\rho$, and $f_\iii'''=((f_\iii''/f_\iii')'+(f_\iii''/f_\iii')^2)f_\iii'$. Writing the scalar Schwarzian of \cref{lem:scalar} as $S^{\C}_f=f'''/f'-\tfrac32(f''/f')^2$, by $(f''/f')'=f'''/f'-(f''/f')^2$, we estimate the two parts separately. Over the common denominators,
  \begin{equation*}
    \biggl|\frac{f_\iii'''}{f_\iii'}-\frac{\psi'''}{\psi'}\biggr| \le \frac{|f_\iii'''-\psi'''|}{|\psi'|}+\frac{|f_\iii'''||f_\iii'-\psi'|}{|f_\iii'||\psi'|} \le \frac{12\eta}{\rho^{3}\cmin^{n}} + \frac{2C_2\cmax^{n}\eta}{\rho\cmin^{2n}},
  \end{equation*}
  and, by \cref{eq:prediff} for the first factor,
  \begin{align*}
    \biggl|\biggl(\frac{f_\iii''}{f_\iii'}\biggr)^{2}-\biggl(\frac{\psi''}{\psi'}\biggr)^{2}\biggr| &\le \biggl|\frac{f_\iii''}{f_\iii'}-\frac{\psi''}{\psi'}\biggr|\biggl|\frac{f_\iii''}{f_\iii'}+\frac{\psi''}{\psi'}\biggr| \\
    &\le \biggl(\frac{2C_0}{\rho}+\frac{4}{\rho^{2}}\biggr)\cmin^{-2n}\eta\biggl(C_0+2C_1\biggl(\frac{\cmax}{\cmin}\biggr)^{n}\biggr),
  \end{align*}
  the second factor being at most $C_0+2C_1(\cmax/\cmin)^{n}$ on $K_\rho$, since $|f_\iii''/f_\iii'|\le C_0$ and $|\psi''/\psi'|\le 2C_1(\cmax/\cmin)^{n}$ there. Since $\cmax^{n}\le1$, $\cmin^{-n}\le\cmin^{-2n}\le\cmin^{-3n}$, and $(\cmax/\cmin)^{n}\le\cmin^{-n}$, the first part is at most $(12\rho^{-3}+2C_2\rho^{-1})\cmin^{-3n}\eta$ and the second at most $(C_0+2C_1)(2C_0\rho^{-1}+4\rho^{-2})\cmin^{-3n}\eta$; combining them with the factor $\tfrac32$ gives \cref{eq:schwdiff} with $C_3=12\rho^{-3}+2C_2\rho^{-1}+\tfrac32(C_0+2C_1)(2C_0\rho^{-1}+4\rho^{-2})$.
\end{proof}

We now turn to the proof of \cref{thm:main}\cref{it:main-schw} for $d=2$, the second reduction to \cref{lem:coincidence}. A failure of the strong exponential separation condition modulo M\"obius maps condenses a composition $f_\iii$ super-exponentially onto a M\"obius image $M_\ell\circ f_\jjj$ of another; a first step rules out orientation-reversing $M_\ell$, \cref{lem:schw-diff} then carries the smallness to the Schwarzians, the M\"obius invariance of the Schwarzian identifies $S^{\C}_{M_\ell\circ f_\jjj}$ with $S_\jjj$, and \cref{lem:coincidence}, applied to the family of \cref{lem:schw-admissible}, returns a coincidence that contradicts \cref{eq:mainhyp2}.

\begin{proof}[Proof of \cref{thm:main}\cref{it:main-schw} for $d=2$]
  Suppose to the contrary that $\Phi$ satisfies \cref{eq:mainhyp2} but not the conclusion. Since reversal is a bijection of each $\II^n$, \cref{lem:condensation-mobius} gives a strictly increasing sequence $(n_\ell)_{\ell \ge 1}$ and, for each $\ell$, distinct $\iii_\ell,\jjj_\ell\in\II^{n_\ell}$ and a M\"obius transformation $M_\ell$ whose errors
  \begin{equation*}
    \eta_{n_\ell}=\|f_{\iii_\ell}-M_\ell\circ f_{\jjj_\ell}\|
  \end{equation*}
  satisfy $\log\eta_{n_\ell}/n_\ell\to-\infty$; the dependence of $\iii$ and $\jjj$ on $\ell$ is suppressed below. Write $\psi=M_\ell\circ f_\jjj$ and $h=f_\iii-\psi$. The finiteness of the supremum places the pole of $M_\ell$ off the compact set $f_\jjj(\ol{\Omega})$, so $\psi$ is continuous on $\ol{\Omega}$, and $\psi$ is holomorphic on $\Omega$ when $M_\ell$ preserves orientation and anti-holomorphic when it reverses orientation. Recall that an orientation-reversing M\"obius transformation is $z\mapsto m(\ol z)$ for a holomorphic M\"obius transformation $m$.

  Orientation-reversing $M_\ell$ cannot occur infinitely often. If $\psi$ is anti-holomorphic on $\Omega$, then $\partial_z\psi\equiv0$, so $\partial_zh=f_\iii'$ there. In this case $h$ is smooth but not holomorphic, so \cref{lem:cauchy} does not apply and the bound comes from Stokes' theorem for the form $h\dd\ol z$, whose exterior derivative is $\partial_zh\dd z\wedge\dd\ol z=-2i\partial_zh\dd\LL^2$. Let $x\in K_\rho$ and $0<r<\rho$, so that the closed disc $B(x,r)$ lies in $\Omega$. Rewriting $f_\iii'$ as $\partial_zh$ moves the estimate onto the boundary, where $|h|\le\eta_{n_\ell}$, and the mean value property of the holomorphic $f_\iii'$ over $B^o(x,r)$ and Stokes' theorem on the positively oriented $\partial B^o(x,r)$ give
  \begin{equation*}
    f_\iii'(x)=\frac{1}{\pi r^2}\int_{B^o(x,r)}f_\iii'\dd\LL^2=\frac{1}{\pi r^2}\int_{B^o(x,r)}\partial_zh\dd\LL^2=\frac{i}{2\pi r^2}\oint_{\partial B^o(x,r)}h\dd\ol{z},
  \end{equation*}
  whence $|f_\iii'(x)|\le r^{-1}\eta_{n_\ell}$, and letting $r\to\rho$ gives $|f_\iii'|\le\rho^{-1}\eta_{n_\ell}$ on $K_\rho$, contradicting $|f_\iii'|\ge\cmin^{n_\ell}$ for large $\ell$, the bound $\eta_{n_\ell}$ being super-exponentially small. Discarding finitely many $\ell$ and passing to a subsequence, every $M_\ell$ preserves orientation, so $\psi\in\mathcal{H}(\ol{\Omega})$.

  It remains to differentiate the condensation into the decay hypothesis of \cref{lem:coincidence} for the family $(S_\lll)_{\lll\in\II^\N\cup\II^*}$. Discarding finitely many $\ell$ once more, we may assume $\eta_{n_\ell}\le\min\{\tfrac{\rho}{2}\cmin^{n_\ell},\cmax^{n_\ell}\}$, so \cref{lem:schw-diff} gives $\sup_{x\in K_\rho}|S_\iii(x)-S^{\C}_\psi(x)|\le C_3\cmin^{-3n_\ell}\eta_{n_\ell}=\eta_{n_\ell}'$, and $\log\eta_{n_\ell}'/n_\ell\to-\infty$, the loss being only exponential. Moreover $S^{\C}_\psi=S_\jjj$ on $\Omega$: the classical composition law recorded after \cref{lem:chain}, applied to $\psi=M_\ell\circ f_\jjj$, with $S^{\C}_{M_\ell}\equiv0$ by \cref{lem:mobius,lem:scalar}, gives $S^{\C}_\psi=(S^{\C}_{M_\ell}\circ f_\jjj)(f_\jjj')^2+S^{\C}_{f_\jjj}=S_\jjj$. Hence $\sup_{x\in K_\rho}|S_\iii(x)-S_\jjj(x)|\le\eta_{n_\ell}'$, the compact $K_\rho$ lying in $\Omega$, and \cref{lem:coincidence}, applied with the errors $\eta_{n_\ell}'$ to the family $(S_\lll)_{\lll\in\II^\N\cup\II^*}$, admissible by \cref{lem:schw-admissible}, produces distinct $\ppp,\qqq\in\II^\N\cup\II^*$ with $|\ppp|=|\qqq|$ and $S_{\ppp}\equiv S_{\qqq}$ on $\ol{\Omega}$, which is the desired contradiction with \cref{eq:mainhyp2}.
\end{proof}

\subsection{Comparison of the hypotheses and an example} \label{sec:C2-compare}

The two hypotheses of \cref{thm:main} are related but not interchangeable; the following remark records the relations between them, and the proposition after it shows that in the plane the weaker hypothesis already earns a conclusion beyond the plain condition.

\begin{remark} \label{rem:S3}
  (1) For $d\ge3$ the hypothesis \cref{eq:mainhyp2} cannot hold: by the classification of \cref{sec:conformal-maps} every composition $f_\iii$ is the restriction of a M\"obius transformation, so every Schwarzian projection vanishes identically by \cref{lem:mobius}, and already the length-one words $1$ and $2$ violate \cref{eq:mainhyp2}, while the strong exponential separation condition modulo M\"obius maps fails for every system there by \cref{ex:esc-not-mobius}(1); the Schwarzian is blind exactly where the pre-Schwarzian retains the pole of \cref{eq:pole}, which is why \cref{thm:main} runs on the pre-Schwarzian. 
  
  (2) The hypothesis \cref{eq:mainhyp2} implies \cref{eq:mainhyp}: a coincidence $T_\iii\equiv T_\jjj$ on $\ol{\Omega}$ forces $S_\iii\equiv S_\jjj$ through the identity $S=T'-\tfrac12T^2$ of \cref{lem:dual2}. 
\end{remark}

The weaker hypothesis earns more than the plain condition in the plane. By \cref{lem:kernel}\cref{it:tkernel}, post-composition by a similarity is invisible to the pre-Schwarzian, exactly as post-composition by a M\"obius map is invisible to the Schwarzian, and the proof of \cref{thm:main}\cref{it:main-schw} for $d=2$ runs one weight down on this invariance, with \cref{lem:pre-diff} in place of \cref{lem:schw-diff} and \cref{lem:coincidence} applied to the pre-Schwarzian family of \cref{lem:pre-admissible}. The separation notion matching the kernel of \cref{lem:kernel}\cref{it:tkernel} is the strong exponential separation condition modulo similarities of \cref{sec:esc}, in which the quantifier of \cref{eq:sescmob} runs over the similarities of $\C$, orientation-reversing ones included.

\begin{proposition} \label{prop:similarities}
  Let $\Phi=(\fii_i)_{i\in\II}$ be a conformal iterated function system on $\C$. If \cref{eq:mainhyp} holds for all distinct $\iii,\jjj\in\II^\N\cup\II^*$ with $|\iii|=|\jjj|$, then $\Phi$ satisfies the strong exponential separation condition modulo similarities.
\end{proposition}

\begin{proof}
  Suppose to the contrary that $\Phi$ satisfies \cref{eq:mainhyp} for all distinct equal-length words but not the conclusion. Since reversal is a bijection of each $\II^n$, \cref{lem:condensation-mobius} gives, read with the similarities in place of the M\"obius transformations as recorded after its proof, a strictly increasing sequence $(n_\ell)_{\ell\ge1}$ and, for each $\ell$, distinct $\iii_\ell,\jjj_\ell\in\II^{n_\ell}$ and a similarity $A_\ell$ of $\C$ whose errors
  \begin{equation*}
    \eta_{n_\ell}=\|f_{\iii_\ell}-A_\ell\circ f_{\jjj_\ell}\|
  \end{equation*}
  satisfy $\log\eta_{n_\ell}/n_\ell\to-\infty$, the logarithm of $0$ read as $-\infty$; the dependence of $\iii$ and $\jjj$ on $\ell$ is suppressed below. Write $\psi=A_\ell\circ f_\jjj$ and $h=f_\iii-\psi$. A similarity of $\C$ is $z\mapsto az+b$ when it preserves orientation and $z\mapsto a\ol{z}+b$ when it reverses it, in both cases with $a\ne0$, so $\psi$ is continuous on $\ol{\Omega}$ and holomorphic or anti-holomorphic on $\Omega$ accordingly.

  Orientation-reversing $A_\ell$ cannot occur infinitely often. If $\psi$ is anti-holomorphic on $\Omega$, then $\partial_z\psi\equiv0$ there, and the argument in the proof of \cref{thm:main}\cref{it:main-schw} for $d=2$, which uses nothing about the outer map beyond that identity and the bound $\|h\|\le\eta_{n_\ell}$, gives $|f_\iii'|\le\rho^{-1}\eta_{n_\ell}$ on $K_\rho$, contradicting $|f_\iii'|\ge\cmin^{n_\ell}$ for large $\ell$, the bound $\eta_{n_\ell}$ being super-exponentially small. Discarding finitely many $\ell$, every $A_\ell$ is $z\mapsto a_\ell z+b_\ell$ with $a_\ell\ne0$, so $\psi\in\mathcal{H}(\ol{\Omega})$ with $\psi'=a_\ell f_\jjj'$ and $\psi''=a_\ell f_\jjj''$, and hence $\psi''/\psi'=f_\jjj''/f_\jjj'=T_\jjj$ on $\Omega$; this is the similarity invariance of the pre-Schwarzian, the instance of the composition law \cref{eq:Tchain} in which the outer map has vanishing pre-Schwarzian.

  It remains to differentiate the condensation into the decay hypothesis of \cref{lem:coincidence} for the family $(T_\lll)_{\lll\in\II^\N\cup\II^*}$. Discarding finitely many $\ell$ once more, we may assume $\eta_{n_\ell}\le\min\{\tfrac{\rho}{2}\cmin^{n_\ell},\cmax^{n_\ell}\}$, so \cref{lem:pre-diff}, applied to $\psi$, gives $\sup_{x\in K_\rho}|T_\iii(x)-T_\jjj(x)|=\sup_{x\in K_\rho}|f_\iii''(x)/f_\iii'(x)-\psi''(x)/\psi'(x)|\le(2C_0/\rho+4/\rho^2)\cmin^{-2n_\ell}\eta_{n_\ell}=\eta_{n_\ell}'$, and $\log\eta_{n_\ell}'/n_\ell\to-\infty$, the loss being only exponential. By \cref{lem:coincidence}, applied with the errors $\eta_{n_\ell}'$ to the family $(T_\lll)_{\lll\in\II^\N\cup\II^*}$, admissible by \cref{lem:pre-admissible}, there are distinct $\ppp,\qqq\in\II^\N\cup\II^*$ with $|\ppp|=|\qqq|$ and $T_{\ppp}\equiv T_{\qqq}$ on $\ol{\Omega}$, which is the desired contradiction with \cref{eq:mainhyp}.
\end{proof}

In particular, in the plane the weight-one gap criterion \cref{prop:gencrit}\cref{it:gencrit-pre} is a checkable sufficient condition for the strong exponential separation condition modulo similarities, whose range of contraction ratios \cref{rem:gencrit-range} determines exactly. The conclusion of \cref{prop:similarities} lies between the two conclusions of \cref{thm:main}, and it is not implied by the plain condition: the two similarities of \cref{ex:esc-not-mobius}(2) satisfy the strong exponential separation condition, while a translation collapses every equal-length pair.

We will next exhibit a planar system that meets the Schwarzian part of the one-point gap criterion of \cref{prop:gencrit}; by \cref{rem:S3}(2), it then satisfies both hypotheses of \cref{thm:main}. 

\begin{example} \label{ex:C2}
  On $\Omega=B^o(0,1)$, the open unit disc, with $\Omega'=B^o(0,\frac{3}{2})$, take 
  \begin{equation*}
    \fii_1(z)=\frac{z}{10}, \qquad \fii_2(z)=\frac{z}{10}+\frac{z^2}{100}, \qquad \fii_3(z)=\frac{z}{10}+\frac{iz^2}{100}+\frac{4}{5}.
  \end{equation*}
  The three maps are polynomials with $|\fii_1|\le\frac{1}{10}$, $|\fii_2|\le\frac{11}{100}$, and $|\fii_3|\le\frac{91}{100}$ on $B(0,1)$, so they map $B(0,1)$ into $B^o(0,1)$ and $X\subseteq B^o(0,1)$; their derivatives are $\fii_1'(z)=\frac{1}{10}$, $\fii_2'(z)=(5+z)/50$, and $\fii_3'(z)=(5+iz)/50$, which vanish nowhere on $\Omega'$, the zeros lying at $-5$ and $5i$, and satisfy $|\fii_k'|\le\frac{6}{50}$ on $B(0,1)$, so $\Phi=(\fii_1,\fii_2,\fii_3)$ is a conformal iterated function system with $\cmax=\frac{6}{50}$. The pre-Schwarzian projections of the generators are $T_1\equiv0$, $T_2(z)=1/(5+z)$, and $T_3(z)=i/(5+iz)$, and the Schwarzian projections are $S_1\equiv0$, $S_2(z)=-\tfrac32(5+z)^{-2}$, and $S_3(z)=\tfrac32(5+iz)^{-2}$, by $S=T'-\tfrac12T^2$. We have $\|S_2\|=\|S_3\|=\tfrac{3}{32}$ and $\|S_2-S_3\|\ge|S_2(0)-S_3(0)|=\tfrac{3}{25}$, so $\beta=\tfrac{3}{32}$ and $\kappa=1$. Thus the exact ceiling supplied by \cref{eq:gencritrange} is $3^{-1/2}$, which $\cmax=\tfrac{6}{50}$ meets. The threshold in \cref{prop:gencrit}\cref{it:gencrit-schw} is $2\cdot\tfrac{3}{32}\cdot\tfrac{(6/50)^2}{1-(6/50)^2}=\tfrac{27}{9856}$. At $x_{i,j}=0$ the Schwarzian values are $0$, $-\tfrac{3}{50}$, and $\tfrac{3}{50}$, every pairwise gap being at least $\tfrac{3}{50}>\tfrac{27}{9856}$, so \cref{prop:gencrit}\cref{it:gencrit-schw} applies with $\alpha=\tfrac{3}{50}$, and $\Phi$ satisfies \cref{eq:mainhyp2}, hence also \cref{eq:mainhyp} by \cref{rem:S3}(2). By \cref{thm:main}, $\Phi$ therefore satisfies the strong exponential separation condition, plain and modulo M\"obius maps. The system is genuinely complex, $\fii_3$ having a non-real quadratic coefficient, and it has $\fii_1$ and $\fii_2$ sharing the fixed point $0$; it is the planar counterpart of the example on the line that B\'ar\'any, Kolossv\'ary, and Troscheit give after their \cite[Proposition~1.8]{BaranyKolossvaryTroscheit}, of a similar design, a linear map and a quadratic map sharing the fixed point $0$ together with a third quadratic map. The dimension formulas of \cref{thm:dim} apply to $\Phi$ too; this is verified in \cref{ex:C2dim}, once the genericity machinery is in place.
\end{example}

\section{The one-dimensional case} \label{sec:C1}

We turn to the line, the home of \cite[Theorem~1.5]{BaranyKolossvaryTroscheit}; throughout this section $d=1$, so that $\Omega$ is a bounded open interval and $\ol{\Omega}$ a compact nondegenerate interval. Conformality by itself carries no regularity, so the real-analyticity of the generators on $\Omega'$ is part of the definition of \cref{sec:cifs-esc}, and it buys exactly what the plane had for free: a holomorphic extension to a complex neighborhood of $\ol{\Omega}$, uniform over the finitely many generators. Both planar results of \cref{sec:C2} descend along this extension. The case $d=1$ of \cref{thm:main}\cref{it:main-pre} is a corollary of the planar case: the system lifts to such a neighborhood, the planar case applies to the lift, and the separation descends back to the interval; the case $d=1$ of \cref{thm:main}\cref{it:main-schw} descends through the same lift and one classical determinant identity. The route through the plane also simplifies the original argument on the line: one Cauchy estimate performs at once, and at a loss linear in the error, the differentiations that a Taylor estimate on the interval carries out order by order at square-root losses, drawing its derivative bounds from a Fa\`a di Bruno apparatus, and the disc of \cref{lem:disc} makes the endgame free of cases, as \cref{rem:cleaner} sets out. In \cref{sec:C1-lift} we record the two preparations that serve both descents, the lift itself and the two-constants theorem of potential theory, which is where the price of the route is paid; \cref{sec:C1-pre} and \cref{sec:C1-mob} then carry the two planar results down, at weight one and at weight two, and \cref{sec:C1-compare} places the weight-one result and its proof beside those of B\'ar\'any, Kolossv\'ary, and Troscheit.

\subsection{The complex lift and the two-constants theorem} \label{sec:C1-lift}

The neighborhoods $\Omega_r=\{z\in\C : \dist(z,\ol{\Omega})<r\}$ of \cref{lem:thicken}, read here with $E=\C$, are taken about a compact interval, so each is a convex bounded domain whose closure is $\{z\in\C : \dist(z,\ol{\Omega})\le r\}$.

\begin{lemma} \label{lem:lift}
  Let $\Phi=(\fii_i)_{i\in\II}$ be a conformal iterated function system on $\R$ with self-conformal set $X$. There are $\delta>0$ and $\lambda\in(0,1)$ such that every generator extends holomorphically to a neighborhood of $\ol{\Omega_\delta}$, satisfies $0<|\fii_i'|\le \lambda$ on $\ol{\Omega_\delta}$, and maps $\ol{\Omega_\delta}$ into $\Omega_\delta$ and $\ol{\Omega_{\delta/2}}$ into $\Omega_{\delta/2}$. The extensions restricted to $\Omega_{\delta/2}$ form a conformal iterated function system $\Phi_\C$ on the domain $\Omega_{\delta/2} \subset \C$, conformal on $\Omega_\delta$, whose attractor is $X$. The compositions of $\Phi_\C$ restrict on $\ol{\Omega}$ to those of $\Phi$, together with their derivatives, and so do the dual projections at both weights: $T^{\Phi_\C}_\iii=T_\iii$ and $S^{\Phi_\C}_\iii=S_\iii$ on $\ol{\Omega}$ for all $\iii\in\II^\N\cup\II^*$.
\end{lemma}

\begin{proof}
  By \cref{lem:thicken}, there are $r>0$ and $\lambda$ with $\cmax<\lambda<1$ such that every generator extends holomorphically to a neighborhood of $\ol{\Omega_r}$, satisfies $0<|\fii_i'|\le \lambda$ on $\Omega_r$, agrees on the real slice $\Omega_r\cap\R$ with the original real generator, which it maps into itself, and satisfies $\fii_i(\ol{\Omega_s})\subseteq\Omega_s$ for every $0<s<r$. Choose $0<\delta<r$; the last inclusion, read at the radii $\delta$ and $\delta/2$, gives $\fii_i(\ol{\Omega_\delta})\subseteq\Omega_\delta$ and $\fii_i(\ol{\Omega_{\delta/2}})\subseteq\Omega_{\delta/2}$. The extensions restricted to $\Omega_{\delta/2}$ send $\ol{\Omega_{\delta/2}}$ into $\Omega_{\delta/2}$, satisfy $|\fii_i'|\le \lambda<1$ there, and extend conformally to $\Omega_\delta\supseteq\ol{\Omega_{\delta/2}}$. The neighborhood $\Omega_{\delta/2}$ is a convex bounded domain, so the derivative bound integrates along segments to the Lipschitz bound $|\fii_i(z)-\fii_i(w)|\le \lambda|z-w|$ on $\ol{\Omega_{\delta/2}}$, and each extension is a contraction of the complete compact space $\ol{\Omega_{\delta/2}}$ into itself. The original attractor $X$ is a nonempty compact invariant subset of this space, since $X=\bigcup_{i \in \II}\fii_i(X)$ and $X\subseteq\Omega\subseteq\Omega_{\delta/2}$, so uniqueness for contractive systems shows that the lifted system has attractor $X$; see \cite{Hutchinson}. Thus the restrictions form a conformal iterated function system $\Phi_\C$ with $d=2$.

  For a finite word the partial orbits of $\ol{\Omega}$ stay in $\ol{\Omega}$, so the reversed composition $f_\iii$ of $\Phi_\C$ restricts on $\ol{\Omega}$ to that of $\Phi$, together with its derivatives, and likewise the forward composition $\fii_\iii$. The $d=1$ and $d=2$ dual projections of \cref{sec:dual-ifs} are read by the same formula, so $T^{\Phi_\C}_\iii=f_\iii''/f_\iii'=T_\iii$ on $\ol{\Omega}$, and at weight two both projections are computed from the common restriction by the same formula $S_\iii=f_\iii'''/f_\iii'-\tfrac32(f_\iii''/f_\iii')^2$, the classical Schwarzian recorded in \cref{sec:schwarzian-defs} for the line and in the proof of \cref{lem:schw-diff} for the plane, so $S^{\Phi_\C}_\iii=S_\iii$ on $\ol{\Omega}$. For an infinite word both sides of either identity are the uniform limits of their finite-word prefixes, by \cref{lem:dual}, respectively \cref{lem:dual2}, applied to $\Phi_\C$ and to $\Phi$, so the identities pass to the limit.
\end{proof}

The lift carries the hypotheses of \cref{thm:main} upward at no cost: the dual projections of $\Phi_\C$ extend those of $\Phi$ by \cref{lem:lift}, so the suprema in \cref{eq:mainhyp,eq:mainhyp2} for $\Phi_\C$, taken over the larger set $\ol{\Omega_{\delta/2}}$, dominate those for $\Phi$, and positivity survives. The conclusion does not come down by restriction. The planar case bounds from below the supremum of $|\fii_\iii-\fii_\jjj|$ over $\ol{\Omega_{\delta/2}}$, the closure of the domain of $\Phi_\C$, whereas the separation condition for $\Phi$ concerns the supremum over the smaller set $\ol{\Omega}$, and passing to a subset can only decrease a supremum. No constant closes the gap either, a holomorphic function being possibly far smaller on an interval than at a fixed distance from it: for large $N$ the function $z\mapsto e^{-N\delta}\cos(Nz)$ is at most $e^{-N\delta}$ on $\R$ but of modulus at least $\tfrac12(1-e^{-2N\delta})$ at height $\delta$. What closes it is the two-constants theorem: it converts smallness on an interval into smallness on a surrounding complex neighborhood at the cost of a fixed positive power, so, read contrapositively, it returns a lower bound on that neighborhood to a lower bound on the interval, and that power is the price the route through the plane pays in the separation constant.

\begin{lemma} \label{lem:twoconst}
  Let $\delta>0$. There is $\theta=\theta(\Omega,\delta)\in(0,1]$ such that every bounded holomorphic $h$ on $\Omega_\delta$ satisfies
  \begin{equation} \label{eq:twoconst}
    \sup_{z\in\ol{\Omega_{\delta/2}}}|h(z)| \le m^{\theta}M^{1-\theta},
  \end{equation}
  where $m=\sup_{x\in\ol{\Omega}}|h(x)|$ and $M=\sup_{z\in\Omega_\delta}|h(z)|$.
\end{lemma}

\begin{proof}
  We may assume $m>0$ and $M>0$, since $M=0$ makes \cref{eq:twoconst} trivial and $m=0$ forces $h\equiv0$ on the connected $\Omega_\delta$ by the identity theorem; also $m\le M$, the interval $\ol{\Omega}$ lying in $\Omega_\delta$. Let $D=\Omega_\delta\setminus\ol{\Omega}$ and write $\ol{\Omega}=[a,b]$. Its closure is $\ol{\Omega_\delta}$ and its boundary is the disjoint union of $\ol{\Omega}$ and $\partial\Omega_\delta$, two closed sets a distance $\delta$ apart. Since $\ol{\Omega}$ lies on the real axis, the intersections of $D$ with the open upper and the open lower half-plane are the intersections of $\Omega_\delta$ with them, hence open and convex, and the remaining points of $D$ are the real ones, the two intervals $(a-\delta,a)$ and $(b,b+\delta)$. Each of these intervals lies in the closure of both halves, so adjoining it to either of them leaves a connected set, and the two halves are joined inside $D$ by passing around either endpoint of $\ol{\Omega}$. Thus $D$ is a bounded domain.

  In the Riemann sphere the complement $\RS^2\setminus D$ is the disjoint union of $\ol{\Omega}$ and $\RS^2\setminus\Omega_\delta$, each connected and neither a single point. A boundary point of a planar domain is regular for the Dirichlet problem in the sense of \cite[Definition~4.1.4]{Ransford} as soon as the connected component of the complement containing it has more than one point, a criterion special to the plane with no analogue in higher dimensions, where a component may be thin at one of its points; see \cite[Theorem~4.2.2]{Ransford}. Every point of $\partial D$ is therefore regular, and \cite[Theorem~4.1.5]{Ransford} then gives solvability of the Dirichlet problem on $D$: every continuous function on $\partial D$ is the restriction of a function harmonic on $D$ and continuous on $\ol{D}$. The two pieces of $\partial D$ are disjoint and closed, so prescribing the value $1$ on $\ol{\Omega}$ and the value $0$ on $\partial\Omega_\delta$ defines such a continuous function, and we write $\omega$ for the harmonic function it produces, the harmonic measure of $\ol{\Omega}$ in $D$. The weak maximum principle \cite[Theorem~1.1.8(b)]{Ransford} gives $0\le\omega\le1$ on $\ol{D}$. Were $\omega$ to take the value $0$ or the value $1$ somewhere in $D$, it would be constant there, against its two boundary values, so $0<\omega<1$ on $D$.

  Recall that a function with values in $[-\infty,\infty)$ is \emph{subharmonic} if it is upper semicontinuous and is at most its own mean over every sufficiently small circle about each point, that $\log|h|$ is subharmonic for holomorphic $h$, with the value $-\infty$ at the zeros of $h$, and that adding a harmonic function preserves subharmonicity. Accordingly $u=\log|h|-\omega\log m-(1-\omega)\log M$ is subharmonic on $D$, and $\log|h|\le\log M$ there bounds it above by $\log(M/m)$. The limit superior of $u$ is nonpositive at every point of $\partial D$. Every limit here is taken along points of $D$, where $\omega$ is strictly between $0$ and $1$, and it reaches the boundary value of $\omega$ because $\omega$ is continuous on $\ol{D}$. At a point of $\ol{\Omega}$ the function $h$ is continuous and bounded by $m$ on $\ol{\Omega}$, so $\log|h|$ has limit superior at most $\log m$, while $\omega\to1$; hence $u$ has limit superior at most $\log m-\log m=0$. At a point of $\partial\Omega_\delta$ the bound $|h|\le M$ holds throughout $\Omega_\delta$ and $\omega\to0$, so $u$ has limit superior at most $\log M-\log M=0$. The maximum principle for subharmonic functions, which makes the function $u$ nonpositive throughout a bounded domain, therefore gives $|h|\le m^{\omega}M^{1-\omega}$ on $D$; see \cite[Theorem~2.3.1(b)]{Ransford}. This bound, interpolating the two constants $m$ and $M$ by the harmonic measure, is the two-constants theorem.

  The function $\omega$ is already defined and continuous on $\ol{D}=\ol{\Omega_\delta}$, with the value $1$ on $\ol{\Omega}$, and it vanishes only on $\partial\Omega_\delta$. The compact $\ol{\Omega_{\delta/2}}$ does not meet that set, its points lying at distance at most $\delta/2$ from $\ol{\Omega}$, so $\omega$ is positive on it and $\theta=\min_{\ol{\Omega_{\delta/2}}}\omega$ is positive. Moreover $\theta<1$, because $\ol{\Omega_{\delta/2}}$ meets $D$, for instance at $a+i\delta/4$, where $\omega<1$. Since $m\le M$, the identity $m^{\omega}M^{1-\omega}=m^{\theta}M^{1-\theta}(m/M)^{\omega-\theta}$ shows that $m^{\omega}M^{1-\omega}\le m^{\theta}M^{1-\theta}$ wherever $\omega\ge\theta$, which covers $D\cap\ol{\Omega_{\delta/2}}$, and on $\ol{\Omega}$ itself $|h|\le m\le m^{\theta}M^{1-\theta}$ directly. These two sets cover $\ol{\Omega_{\delta/2}}$, and \cref{eq:twoconst} follows.
\end{proof}

\subsection{Separation by the pre-Schwarzian} \label{sec:C1-pre}

We now prove the one-dimensional case: the planar case of \cref{thm:main} applies to the lift of \cref{lem:lift}, and \cref{lem:twoconst} returns the resulting separation from the complex neighborhood to the line.

\begin{proof}[Proof of \cref{thm:main}\cref{it:main-pre} for $d=1$]
  Let $\delta$ and $\Phi_\C$ be as in \cref{lem:lift}. The hypothesis transfers upward: for all distinct $\iii,\jjj\in\II^\N\cup\II^*$ with $|\iii|=|\jjj|$, the identity $T^{\Phi_\C}_\iii=T_\iii$ on $\ol{\Omega}$ of \cref{lem:lift} gives $\sup_{x\in\ol{\Omega_{\delta/2}}}|T^{\Phi_\C}_\iii(x)-T^{\Phi_\C}_\jjj(x)|\ge\|T_\iii-T_\jjj\|>0$, so $\Phi_\C$ satisfies \cref{eq:mainhyp}, and the planar case of \cref{thm:main}, proved in \cref{sec:C2}, provides $c>0$ with $\sup_{z\in\ol{\Omega_{\delta/2}}}|\fii_\iii(z)-\fii_\jjj(z)|\ge c^n$ for all $n$ and all distinct $\iii,\jjj\in\II^n$, the compositions of $\Phi_\C$ extending those of $\Phi$.

  The conclusion transfers downward by \cref{lem:twoconst}. For distinct $\iii,\jjj\in\II^n$, let $h=\fii_\iii-\fii_\jjj$ for the lifted compositions. By \cref{lem:lift}, every lifted generator maps $\ol{\Omega_\delta}$ into $\Omega_\delta$, and hence every lifted composition does so and is holomorphic on $\Omega_\delta$. Thus, with $R=\sup_{z\in\Omega_\delta}|z|$, we have $M=\sup_{z\in\Omega_\delta}|h(z)|\le2R$. On $\ol{\Omega}$ the restriction identity gives $m=\sup_{x\in\ol{\Omega}}|h(x)|=\|\fii_\iii-\fii_\jjj\|$ for the original system. \Cref{lem:twoconst} gives $c^n\le\sup_{z\in\ol{\Omega_{\delta/2}}}|h(z)|\le m^{\theta}(2R)^{1-\theta}$, so in particular $m>0$, and rearranging gives $m\ge Ac_1^n$ with $c_1=c^{1/\theta}$ and $A=(2R)^{-(1-\theta)/\theta}$; since $Ac_1^n\ge(Ac_1)^n$ for $n\ge1$ when $A<1$, and $Ac_1^n\ge c_1^n$ otherwise, the constant $c_2=c_1\min\{A,1\}$ satisfies $\|\fii_\iii-\fii_\jjj\|\ge c_2^n$ for all $n$ and all distinct $\iii,\jjj\in\II^n$. This is the strong exponential separation condition for $\Phi$.
\end{proof}

\subsection{Separation modulo M\"obius maps by the Schwarzian} \label{sec:C1-mob}

The Schwarzian theorem descends as well. For $d=1$ the M\"obius transformations of $\RS^1=\R\cup\{\infty\}$, the finite compositions of similarities and inversions as in \cref{sec:schwarzian}, are exactly the real fractional-linear maps $x\mapsto(ax+b)/(cx+d)$ with $a,b,c,d\in\R$ and $ad-bc\ne0$, orientation-reversing ones included: similarities and inversions are of this form, the fractional-linear maps form a group under composition, and, conversely, for $c\ne0$ the map $x\mapsto(ax+b)/(cx+d)$ equals $x\mapsto a/c-(ad-bc)/(c^2(x+d/c))$, the composition of a similarity with an inversion, while for $c=0$ it is a similarity. The descent is more delicate than for \cref{thm:main}: the finiteness of the norm in \cref{eq:sescmob} keeps the pole $-d/c$ of a competitor $M$ off the real image $f_\jjj(\ol{\Omega})$ only, so the pole may approach, or enter, the complex neighborhood on which \cref{lem:twoconst} would need $f_\iii-M\circ f_\jjj$ to stay bounded, and competitors with $ad-bc$ close to $0$ defeat every bound on $M$ alone. Clearing denominators removes the pole: the numerator $cf_\iii f_\jjj+df_\iii-af_\jjj-b$ is holomorphic and uniformly bounded on the neighborhood, whatever the pole does, and the following identity converts its smallness, through three derivatives, into closeness of the Schwarzian projections, with no case distinction on $M$. The identity is a Wronskian one: the numerator is the linear combination of the four functions $f_\iii f_\jjj$, $f_\iii$, $f_\jjj$, and the identity with the coefficients $c$, $d$, $-a$, and $-b$ of $M$, and the determinant it evaluates is the Wronskian of these same four functions, formed from them and their first three derivatives. A relation $f_\iii=M\circ f_\jjj$ is exactly a linear dependence among the four, which makes their Wronskian vanish identically, and the identity locates that vanishing in the difference of the two Schwarzian projections.

\begin{lemma} \label{lem:wronskian}
  Let $U\subseteq\C$ be a domain and let $u,v\colon U\to\C$ be holomorphic with nonvanishing derivatives. Then
  \begin{equation} \label{eq:wronskian}
    \det
    \begin{pmatrix}
      uv & u & v & 1 \\
      (uv)' & u' & v' & 0 \\
      (uv)'' & u'' & v'' & 0 \\
      (uv)''' & u''' & v''' & 0
    \end{pmatrix}
    = -2(u')^2(v')^2(S^{\C}_u-S^{\C}_v)
  \end{equation}
  on $U$.
\end{lemma}

\begin{proof}
  Write $W$ for the matrix of \cref{eq:wronskian}. Expanding along the fourth column, whose only nonzero entry is the $1$ in the first row, gives $\det W=-\det W_3$, where $W_3$ consists of the last three rows of the first three columns. The product rule gives $(uv)'=u'v+uv'$, $(uv)''=u''v+2u'v'+uv''$, and $(uv)'''=u'''v+3u''v'+3u'v''+uv'''$, so subtracting from the first column of $W_3$ the second column multiplied by $v$ and the third column multiplied by $u$, which at every point of $U$ leaves the determinant unchanged, replaces it by $(0,2u'v',3(u''v'+u'v''))^{\top}$. Expanding along this column and evaluating the two $2\times2$ determinants,
  \begin{align*}
    \det W_3 &= -2u'v'(u'v'''-v'u''')+3(u''v'+u'v'')(u'v''-v'u'') \\
    &= 2(v')^2u'u'''-2(u')^2v'v'''+3(u')^2(v'')^2-3(u'')^2(v')^2 \\
    &= 2(v')^2(u'u'''-\tfrac32(u'')^2)-2(u')^2(v'v'''-\tfrac32(v'')^2),
  \end{align*}
  and since $S^{\C}_u=(u''/u')'-\tfrac12(u''/u')^2=u'''/u'-\tfrac32(u''/u')^2$, so that $(u')^2S^{\C}_u=u'u'''-\tfrac32(u'')^2$, the last line is $2(u')^2(v')^2(S^{\C}_u-S^{\C}_v)$. Hence $\det W=-2(u')^2(v')^2(S^{\C}_u-S^{\C}_v)$.
\end{proof}

The identity of \cref{lem:wronskian} is the conversion promised above. With it, the case $d=1$ of \cref{thm:main}\cref{it:main-schw} follows by the lift-and-descend of the one-dimensional case of \cref{thm:main}\cref{it:main-pre}, now read one weight higher.

\begin{proof}[Proof of \cref{thm:main}\cref{it:main-schw} for $d=1$]
  Suppose to the contrary that $\Phi$ satisfies \cref{eq:mainhyp2} for all distinct equal-length words but not the conclusion. Since reversal is a bijection of each $\II^n$, \cref{lem:condensation-mobius} provides, in the reversed notation, a strictly increasing sequence $(n_\ell)_{\ell \ge 1}$ and, for each $\ell$, distinct $\iii_\ell,\jjj_\ell\in\II^{n_\ell}$ and a M\"obius transformation $M_\ell$ of $\RS^1$ whose errors $\eta_\ell=\|f_{\iii_\ell}-M_\ell\circ f_{\jjj_\ell}\|$ satisfy $\log\eta_\ell/n_\ell\to-\infty$, the logarithm of $0$ read as $-\infty$; the dependence of $\iii$ and $\jjj$ on $\ell$ is suppressed below. After discarding finitely many indices, we may assume $\eta_\ell<\infty$, the logarithm of $\infty$ being $\infty$. By the fractional-linear form recorded before \cref{lem:wronskian}, $M_\ell(x)=(ax+b)/(cx+d)$ with $a,b,c,d\in\R$ and $ad-bc\ne0$, and, multiplying the four coefficients by a common nonzero real, we may assume $\max\{|a|,|b|,|c|,|d|\}=1$. The finiteness of $\eta_\ell$ places the pole of $M_\ell$ off the compact $f_\jjj(\ol{\Omega})$, so $cf_\jjj+d$ vanishes nowhere on $\ol{\Omega}$, for $c=0$ because $ad-bc\ne0$ then forces $d\ne0$.

  Let $\delta$, $\lambda$, and $\Phi_\C$ be as in \cref{lem:lift}, let $\theta=\theta(\Omega,\delta)\in(0,1]$ be the exponent of \cref{lem:twoconst}, and put $R=\sup_{z\in\Omega_\delta}|z|$. By \cref{lem:lift}, every lifted composition is holomorphic on $\Omega_\delta$ with values in $\Omega_\delta$, the generators being holomorphic on an open set containing $\ol{\Omega_\delta}$ and mapping $\ol{\Omega_\delta}$ into $\Omega_\delta$, and restricts on $\ol{\Omega}$ to the corresponding composition of $\Phi$; since $B^o(x,\delta)\subseteq\Omega_\delta$ for $x\in\ol{\Omega}$, \cref{lem:cauchy} gives $|f_\iii^{(k)}(x)|\le k!\delta^{-k}R\le C_\delta$ for $0\le k\le3$, every $x\in\ol{\Omega}$, and every finite word $\iii$, with $C_\delta=6\max\{1,\delta^{-3}\}\max\{1,R\}$. Moreover, by \cref{lem:lift}, the weight-two dual projections of $\Phi_\C$ restrict on $\ol{\Omega}$ to those of $\Phi$; we refer to this as the restriction identity. For a finite word the projection of $\Phi_\C$ is the classical Schwarzian of the lifted composition, $S^{\Phi_\C}_\iii=S^{\C}_{f_\iii}$ on $\ol{\Omega_{\delta/2}}$, by \cref{sec:dual-ifs} for $d=2$, so the restriction identity reads $S^{\C}_{f_\iii}=S_\iii$ on $\ol{\Omega}$.

  The error $f_\iii-M_\ell\circ f_\jjj$ need not be holomorphic on all of $\Omega_\delta$, so we clear denominators before differentiating. The numerator $g_\ell=cf_\iii f_\jjj+df_\iii-af_\jjj-b$ is holomorphic on $\Omega_\delta$ with $\sup_{z\in\Omega_\delta}|g_\ell(z)|\le R(R+1)+R+1=(R+1)^2$, the coefficients being at most $1$ in modulus, and on $\ol{\Omega}$ it factorizes as $g_\ell=(cf_\jjj+d)(f_\iii-M_\ell\circ f_\jjj)$, so $\|g_\ell\|\le(R+1)\eta_\ell$. By \cref{lem:twoconst} and the monotonicity of $(m,M)\mapsto m^{\theta}M^{1-\theta}$, $\sup_{z\in\ol{\Omega_{\delta/2}}}|g_\ell(z)|\le((R+1)\eta_\ell)^{\theta}(R+1)^{2(1-\theta)}$; since $B^o(x,\delta/2)\subseteq\Omega_{\delta/2}$ for $x\in\ol{\Omega}$, \cref{lem:cauchy} gives
  \begin{equation*}
    \max_{0\le k\le3}\|g_\ell^{(k)}\| \le 6\max\{1,(2/\delta)^{3}\}((R+1)\eta_\ell)^{\theta}(R+1)^{2(1-\theta)} = \eta_\ell',
  \end{equation*}
  and $\log\eta_\ell'/n_\ell\to-\infty$, the exponent $\theta$ being fixed.

  Through \cref{lem:wronskian}, this smallness becomes closeness of the Schwarzians $S_\iii$ and $S_\jjj$. Fix $x\in\ol{\Omega}$ and let $W(x)$ be the matrix of \cref{eq:wronskian} with $u=f_\iii$ and $v=f_\jjj$, the lifted compositions, whose derivatives are nonvanishing on $\Omega_\delta$, being products of nonvanishing generator derivatives along orbits that stay in $\Omega_\delta$. Differentiating $g_\ell=cuv+du-av-b$ three times gives $W(x)w=(g_\ell(x),g_\ell'(x),g_\ell''(x),g_\ell'''(x))$ for the coefficient vector $w=(c,d,-a,-b)$, and multiplying by the adjugate matrix of $W(x)$ gives $(\det W(x))w=\operatorname{adj}(W(x))(g_\ell(x),g_\ell'(x),g_\ell''(x),g_\ell'''(x))$. Some entry of $w$ has modulus $1$; every entry of $W(x)$ is at most $8C_\delta^2$ in modulus, by the bound $|f_\iii^{(k)}(x)|\le C_\delta$, the product rule, and $C_\delta\ge1$; and every entry of the adjugate is, up to sign, a $3\times3$ minor of $W(x)$, of modulus at most $6(8C_\delta^2)^3$. Hence $|\det W(x)|\le4\cdot6(8C_\delta^2)^3\eta_\ell'$. By the restriction identity, $S^{\C}_{f_\iii}=S_\iii$ and $S^{\C}_{f_\jjj}=S_\jjj$ on $\ol{\Omega}$. On the other hand, the multiplicativity of the conformal factor recorded in \cref{sec:composition}, applied along an orbit that stays in $\ol{\Omega}$, gives $|f_\iii'|\ge\cmin^{n_\ell}$ and $|f_\jjj'|\ge\cmin^{n_\ell}$ on $\ol{\Omega}$, so \cref{eq:wronskian} yields
  \begin{equation*}
    \|S_\iii-S_\jjj\| \le 12(8C_\delta^2)^3\cmin^{-4n_\ell}\eta_\ell' = \eta_\ell'',
  \end{equation*}
  still with $\log\eta_\ell''/n_\ell\to-\infty$.

  It remains to transfer the closeness to the lift and to reach the contradiction through \cref{lem:coincidence}. The scalar Schwarzians $S^{\C}_{f_\iii}$ and $S^{\C}_{f_\jjj}$ of the lifted compositions are holomorphic on $\Omega_\delta$ and bounded there by $B_2=\beta_2^\delta/(1-\lambda^2)$, where $\beta_2^\delta=\max_{i \in \II}\sup_{z\in\ol{\Omega_\delta}}|S^{\C}_{\fii_i}(z)|<\infty$: iterating the classical composition law $S^{\C}_{g\circ f}=(S^{\C}_g\circ f)(f')^2+S^{\C}_f$ recorded after \cref{lem:chain} along the word, whose partial orbits stay in $\ol{\Omega_\delta}$, exhibits $S^{\C}_{f_\iii}$ as a sum of at most $|\iii|$ terms, the $k$-th bounded by $\beta_2^\delta \lambda^{2(k-1)}$, the prefix derivatives being at most $\lambda^{k-1}$ in modulus on $\ol{\Omega_\delta}$. Since $N\ge2$ and the restriction identity transfers the length-one instances of \cref{eq:mainhyp2} to the lift, not all lifted generator Schwarzians vanish, so $\beta_2^\delta>0$ and $B_2>0$. The difference $S^{\C}_{f_\iii}-S^{\C}_{f_\jjj}$ therefore satisfies the hypotheses of \cref{lem:twoconst} with $M\le2B_2$, while, by the bound on $\|S_\iii-S_\jjj\|$ and the restriction identity, $m\le\eta_\ell''$, so $\sup_{z\in\ol{\Omega_{\delta/2}}}|S^{\C}_{f_\iii}(z)-S^{\C}_{f_\jjj}(z)|\le(\eta_\ell'')^{\theta}(2B_2)^{1-\theta}=\eta_\ell'''$, still with $\log\eta_\ell'''/n_\ell\to-\infty$. The scalars $S^{\C}_{f_\iii}$ are the weight-two dual projections of $\Phi_\C$ on $\ol{\Omega_{\delta/2}}$, as recorded above; denoting by $K_\C$ the compact $K_\rho$ of \cref{sec:C2-machinery} for $\Phi_\C$, we have $K_\C\subseteq\ol{\Omega_{\delta/2}}$. Hence \cref{lem:coincidence}, applied to $\Phi_\C$ and its Schwarzian projections, admissible by \cref{lem:schw-admissible} read for $\Phi_\C$, with the errors $\eta_\ell'''$, produces distinct $\ppp,\qqq\in\II^\N\cup\II^*$ with $|\ppp|=|\qqq|$ and $S^{\Phi_\C}_{\ppp}\equiv S^{\Phi_\C}_{\qqq}$ on $\ol{\Omega_{\delta/2}}$. Restricting to $\ol{\Omega}$ through the restriction identity gives $S_{\ppp}\equiv S_{\qqq}$ on $\ol{\Omega}$ for the distinct equal-length pair $\ppp$ and $\qqq$, which is the desired contradiction with \cref{eq:mainhyp2}.
\end{proof}

\subsection{Comparison with B\'ar\'any, Kolossv\'ary, and Troscheit} \label{sec:C1-compare}

The line is where \cite{BaranyKolossvaryTroscheit} works, and the weight-one half of \cref{thm:main} meets their sufficient condition there. The first remark below compares the two statements; the second compares their proofs.

\begin{remark} \label{rem:BKT-comparison}
  The class treated by B\'ar\'any, Kolossv\'ary, and Troscheit sits inside ours, so that \cref{thm:main} with $d=1$ contains \cite[Theorem~1.5]{BaranyKolossvaryTroscheit}. Their class consists, for a fixed $\eps>0$, of maps $\R\to\R$ that extend complex-analytically to the neighborhood $[0,1]_{2\eps}$, map $[0,1]$ into itself and $\ol{[0,1]_\eps}$ into $[0,1]_\eps$, and satisfy $0<|f'|<1$ on $\ol{[0,1]_\eps}$, where $[0,1]_r=\{z\in\C : \dist(z,[0,1])<r\}$ for $r>0$. We call $[0,1]_\eps$ the \emph{collar}, these requirements the \emph{collar conditions}, and their class the \emph{collar class}. Such a system is a conformal iterated function system with $d=1$, $\Omega=(-\tau,1+\tau)$, and $\Omega'=(-\eps,1+\eps)$ for any $0<\tau\le\eps/2$: the derivative bound $|f_i'|<1$ holds on $\ol{\Omega}=[-\tau,1+\tau]\subseteq\ol{[0,1]_\eps}$, each map sends $\ol{\Omega}$ into $\Omega$, the distance from $f_i(x)$ to $f_i(y)\in[0,1]$, for $y$ the point of $[0,1]$ nearest to $x\in\ol{\Omega}$, being at most $\max|f_i'|$ over the segment between them, a number less than $1$, times $|x-y|\le\tau$, hence strictly less than $\tau$, the attractor lies in $[0,1]\subset\Omega$, and the dual projections restrict on $[0,1]$ to theirs, their lifted operator and cocycle sums being \cref{eq:dualdef} in its scalar form and the sums \cref{eq:Tline}. Their separation hypothesis, a supremum over $[0,1]$, thus implies \cref{eq:mainhyp}, a supremum over the larger $\ol{\Omega}$, and \cref{thm:main} gives the strong exponential separation condition with the supremum over $\ol{\Omega}$; one further application of \cref{lem:twoconst}, with $[0,1]$ in place of $\ol{\Omega}$ and $2\tau$ in place of $\delta$, its statement and proof using nothing about $\ol{\Omega}$ beyond its being a compact nondegenerate interval, returns the supremum to $[0,1]$ exactly as in the proof above and recovers their conclusion, the differences of their compositions being holomorphic and uniformly bounded on $[0,1]_{2\tau}\subseteq[0,1]_\eps$, since their maps send $\ol{[0,1]_\eps}$ into $[0,1]_\eps$, and the supremum over $\ol{\Omega}\subseteq\ol{[0,1]_\tau}$ being at most the left side of \cref{eq:twoconst}.
\end{remark}

\begin{remark} \label{rem:cleaner}
  The proof of \cite[Theorem~1.5]{BaranyKolossvaryTroscheit} and the proof of \cref{thm:main}\cref{it:main-pre} for $d=1$ above, through the planar case it invokes, follow the common strategy of \cref{sec:reduction} and part ways at its first step, the differentiation of the condensation. On the line, the smallness of $f_\iii-f_\jjj$ that a failure of separation supplies, the input of the argument by contradiction, is a smallness on a real interval only, where no disc is available for Cauchy's estimate, even though the maps extend holomorphically to a complex neighborhood of the interval. B\'ar\'any, Kolossv\'ary, and Troscheit meet this by differentiating on the interval, by real-variable means. The proof above moves the problem into that complex neighborhood instead: the planar case of \cref{thm:main} differentiates there through Cauchy's estimate, and \cref{lem:twoconst} carries the separation back to the interval, its two-constants estimate being the price of the move. The move simplifies the argument at two points.

  First, on the interval the super-exponential closeness of two compositions is passed to their first two derivatives by the Taylor estimate \cite[Lemma~3.1]{BaranyKolossvaryTroscheit}, which loses a square root at each order of differentiation and needs a uniform bound one order higher, and from the derivatives to the pre-Schwarzians by the quotient rule. The uniform bounds that the Taylor estimate needs, beyond the bound $|f_\iii''|\le C_0\cmax^{|\iii|}$ from the cocycle sum that both proofs use, come from the Fa\`a di Bruno formula for $T_\iii^{(k)}$ in \cite[Lemma~2.7]{BaranyKolossvaryTroscheit} and its consequences \cite[Lemma~2.8, Corollary~2.9, and Lemma~2.10]{BaranyKolossvaryTroscheit}, which also supply the continuity of the derivatives $T_\iii^{(k)}$ in the point and in the word that the endgame requires. In the plane, Cauchy's estimate, \cref{lem:cauchy}, does both jobs at once and at a loss linear in the error: it passes the closeness of the compositions to their pre-Schwarzians in \cref{lem:pre-diff}, and it turns the norm continuity of $\lll\mapsto T_\lll$ from \cref{lem:dual} into the uniform convergence of every derivative in the proof of \cref{lem:coincidence}; no formula for the higher derivatives is needed.

  Second, the endgame of the proof of \cite[Theorem~1.5]{BaranyKolossvaryTroscheit} splits into two cases according to whether the common tail of the condensing pair of words, which in the reversed notation \cref{eq:frev} is the common prefix $\kkk$, stays bounded or grows. A tail of bounded length is constant along a subsequence, so the projections of the limit pair agree on the image of $[0,1]$ under the tail composition, a nondegenerate interval, and the identity theorem finishes. A growing tail shrinks that image to a single point of the attractor, where all derivatives have to be compared; this is done, in effect, by applying \cite[Lemma~3.1]{BaranyKolossvaryTroscheit} $k$ times to the pre-Schwarzians composed with the tail and inverting the chain rule order by order, dividing out at each order the derivative of the tail composition, at the cost of raising the super-exponential error to the power $2^{-k}$, together with an exponential factor $\cmin^{-kn_\ell}$. In the proof of \cref{lem:coincidence}, the disc of \cref{lem:disc}, of radius at least $r_0b^{n_\ell}$ about $f_\kkk(z_0)$ inside $f_\kkk(K_\rho)$, is available whether the prefix length $m$ stays bounded or grows, and one Cauchy estimate on it, \cref{eq:allders}, delivers all derivatives together, so the case distinction disappears. The analyticity of the infinite-word projections, which the identity theorem needs in both endgames, is obtained alike in the two proofs, as the analyticity of a uniform limit of holomorphic functions on a complex neighborhood, by Morera's theorem in \cite[Lemma~2.4]{BaranyKolossvaryTroscheit} and by the Weierstrass convergence theorem in \cref{lem:dual}.
\end{remark}

\section{The higher-dimensional case} \label{sec:C3}

Throughout this section $d\ge3$, so that, by the classification of \cref{sec:conformal-maps}, the generators are restrictions of M\"obius transformations of $\RS^d$; by \cref{sec:dual-ifs}, the dual projection of a finite word $\iii$ is the pre-Schwarzian $T_\iii=T_{f_\iii}=\nabla\log\lambda_{f_\iii}$ of the reversed composition, a vector field on $\ol{\Omega}$, and the norm of $h\in\mathcal{C}(\ol{\Omega};\R^d)$ is $\|h\|=\sup_{x\in\ol{\Omega}}|h(x)|$. The section proves the case $d\ge3$ of \cref{thm:main}\cref{it:main-pre} on the skeleton of the planar case, but the function theory disappears: the disc of \cref{lem:disc} and the identity-theorem endgame of \cref{lem:coincidence} are replaced by the rigidity of the M\"obius class, and the Cauchy estimate of \cref{lem:cauchy} by the derivative bounds that this rigidity supplies together with an elementary Landau-type trade-off.

The Schwarzian layer is void here: every $f_\iii$ is itself M\"obius, so, as \cref{ex:esc-not-mobius}(1) records, the M\"obius relation $f_\iii=M\circ f_\jjj$ holds outright for every equal-length pair, the very relation whose failure the Schwarzian detects, by \cref{lem:kernel}, and \cref{lem:wronskian} quantifies on the line, and every Schwarzian projection vanishes by \cref{lem:mobius}. As in \cref{rem:S3}(1), the separation must therefore be read one level down, from what the Schwarzian forgets and the pre-Schwarzian retains: by \cref{lem:mobius}, a M\"obius transformation is either a similarity, with vanishing pre-Schwarzian, or has a unique finite pole, and its pre-Schwarzian is then the explicit field \cref{eq:pole}, which records the pole and nothing else. The section stands on two pillars: a \emph{pole gap}, by which contraction pushes every pole a definite distance away from $\ol{\Omega}$, and a \emph{pole dictionary}, by which a pole field recovers its pole, hence itself, from its value at a single point.

The machinery is collected in \cref{sec:C3-poles}: the constants and the M\"obius reading of the compositions, the pole gap of \cref{lem:CU}, which also bounds every derivative of every composition, the pole dictionary of \cref{lem:poledict}, and the Landau-type trade-off of \cref{lem:landau}, which differentiates the condensation in place of the Cauchy estimate. In \cref{sec:C3-pre} the proof of \cref{thm:main}\cref{it:main-pre} for $d\ge3$ then runs the common strategy of \cref{sec:reduction} on these three lemmas, and closes with \cref{rem:rigid}, on the economy that M\"obius rigidity buys relative to the planar proof and on the role of the pole at infinity. There is no counterpart here of \cref{sec:C2-mob} or \cref{sec:C1-mob}, and its absence is a theorem rather than an omission: by the preceding paragraph the weight-two hypothesis \cref{eq:mainhyp2} cannot hold for $d\ge3$, and the conclusion it would earn fails for every system there, as \cref{rem:S3}(1) records.

\subsection{The pole gap, the pole dictionary, and a Landau estimate} \label{sec:C3-poles}

The rigidity of the M\"obius class replaces the function theory that carried the planar proof, and this subsection assembles what that rigidity yields. Before turning to the pole gap, we fix the setting on which the estimates rest.

The M\"obius transformations, being the finite compositions of similarities and inversions, form a group of bijections of $\RS^d$. Each generator is the restriction to $\Omega'$ of a M\"obius transformation, unique since two M\"obius transformations that agree on a nonempty open subset of $\R^d$ agree on all of $\RS^d$ by the identity principle, and each $f_\iii$ agrees on $\ol{\Omega}$ with the composition of the corresponding transformations, the partial orbits staying in $\ol{\Omega}$; we write $f_\iii$ also for this M\"obius transformation and note that, by the multiplicativity of the conformal factor recorded in \cref{sec:composition}, applied along the orbit, $\cmin^{|\iii|}\le\lambda_{f_\iii}\le\cmax^{|\iii|}$ on $\ol{\Omega}$. Since $T_\varnothing\equiv0$, taking $\jjj=\varnothing$ in \cref{eq:dualbounds} gives $\|T_\iii\|\le2\max_{i \in \II}\|T_{\fii_i}\|/(1-\cmax)$ for every $\iii\in\II^\N\cup\II^*$. Recall from \cref{lem:mobius} that a M\"obius transformation $f$ that is not a similarity has a unique finite pole $p=f^{-1}(\infty)$ and, by \cref{eq:pole}, pre-Schwarzian $T_f(x)=-2(x-p)/|x-p|^2$, while a similarity has $T_f\equiv0$. Accordingly, for $p\in\R^d$ we call the vector field $T_p(x)=-2(x-p)/|x-p|^2$, real-analytic and nowhere zero on $\R^d\setminus\{p\}$, the \emph{pole field} of $p$, and we set $T_\infty\equiv0$; as in \cref{sec:dual-maps}, typewriter subscripts of dual projections name words, while italic subscripts of pole fields name points of $\RS^d$. Every finite word $\iii$ thus satisfies $T_\iii=T_{p_\iii}$ on $\ol{\Omega}$, where $p_\iii\in\RS^d$ is the \emph{pole} of $f_\iii$, read as $\infty$ when $f_\iii$ is a similarity; a finite pole lies off $\ol{\Omega}$, since $f_\iii$ maps $\ol{\Omega}$ into $\ol{\Omega}\subseteq\R^d$ and a M\"obius transformation is finite exactly off its pole, so the pole field $T_{p_\iii}$ is defined on all of $\ol{\Omega}$ and agrees there with $T_\iii=T_{f_\iii}$ by \cref{lem:mobius}. Finally, for $k\ge1$ and a map $f$ of an open subset of $\R^d$ into a finite-dimensional normed space, we write $\|D^kf(x)\|$ for the operator norm of the $k$-linear map $D^kf(x)$, the supremum of $|D^kf(x)[v_1,\dots,v_k]|$ over unit vectors $v_1,\dots,v_k$, which for $k=1$ is the operator norm already in use; for $f$ with values in $\R^d$ it is invariant under orthogonal maps, $\|D^k(O'\circ f\circ O)(x)\|=\|D^kf(Ox)\|$ for orthogonal $O$ and $O'$.

The pole gap comes from the dual bound: a pole close to $\ol{\Omega}$ would make the pre-Schwarzian large somewhere on $\ol{\Omega}$, while every dual projection is capped at $2\max_{i \in \II}\|T_{\fii_i}\|/(1-\cmax)$. Through the scaling structure of the inversion, the gap in turn bounds every derivative of every composition $f_\iii$, the counterpart of the planar derivative bounds that the same cap, together with the Cauchy estimate of \cref{lem:cauchy}, supplied in \cref{sec:C2}.

\begin{lemma} \label{lem:CU}
  Let $d\ge3$. If $\iii\in\II^*$ and $f_\iii$ is not a similarity, then $\max_{i \in \II}\|T_{\fii_i}\|>0$ and
  \begin{equation} \label{eq:polegap}
    \dist(p_\iii,\ol{\Omega}) \ge \frac{1-\cmax}{\max_{i \in \II}\|T_{\fii_i}\|}.
  \end{equation}
  Moreover, for every integer $k\ge1$ there is a constant $L_k>0$, depending only on $k$, $d$, $\cmax$, and $\max_{i \in \II}\|T_{\fii_i}\|$, such that
  \begin{equation} \label{eq:scaling}
    \sup_{x\in\ol{\Omega}}\|D^kf_\iii(x)\| \le L_k\cmax^{|\iii|}
  \end{equation}
  for every $\iii\in\II^*$.
\end{lemma}

\begin{proof}
  Write $\beta_1=\max_{i \in \II}\|T_{\fii_i}\|$ and $C_0=2\beta_1/(1-\cmax)$, so that the cap recorded above reads $\|T_\iii\|\le C_0$ for every $\iii\in\II^\N\cup\II^*$. If $f_\iii$ is not a similarity, then at least one letter $\fii_{i_r}$ occurring in $\iii$ is not a similarity, since similarities are closed under composition; by \cref{lem:mobius} the pre-Schwarzian of $\fii_{i_r}$ is then a pole field, which vanishes nowhere on $\ol{\Omega}$, so $\beta_1\ge\|T_{\fii_{i_r}}\|>0$. For such a word, the pole $p_\iii$ lies outside $\ol{\Omega}$ and the identity $T_\iii=T_{p_\iii}$ on $\ol{\Omega}$ gives $\|T_\iii\|=\sup_{x\in\ol{\Omega}}2/|x-p_\iii|=2/\dist(p_\iii,\ol{\Omega})$. The bound $\|T_\iii\|\le C_0$ therefore gives $\dist(p_\iii,\ol{\Omega})\ge2/C_0=(1-\cmax)/\beta_1$, which is \cref{eq:polegap}.

  For \cref{eq:scaling}, let first $f_\iii$ be a non-similarity and let $\iota_\iii(x)=p_\iii+(x-p_\iii)/|x-p_\iii|^2$ be the inversion in the unit sphere centered at $p_\iii$, an involution of $\RS^d$ interchanging $p_\iii$ and $\infty$. The composition $f_\iii\circ\iota_\iii$ is a M\"obius transformation fixing $\infty$, hence a similarity $x\mapsto\lambda Ox+v$, so $f_\iii(x)=\lambda O\iota_\iii(x)+v$. The unit inversion $\iota_0(y)=y/|y|^2$ satisfies $\iota_0(ty)=t^{-1}\iota_0(y)$ and $\iota_0(O'y)=O'\iota_0(y)$ for $t>0$ and orthogonal $O'$. Differentiating these identities with respect to $y$ gives $\|D^k\iota_0(y)\|=c_k|y|^{-(k+1)}$, where $c_k=\|D^k\iota_0(e)\|$ is independent of the choice of unit vector $e$, and $c_1=1$. Differentiating $t\mapsto\iota_0(te)=t^{-1}e$ at $t=1$ gives $D^k\iota_0(e)[e,\dots,e]=(-1)^k k!e$, so $c_k\ge k!>0$. Since $\iota_\iii(x)=p_\iii+\iota_0(x-p_\iii)$, the chain rule gives $\|D^kf_\iii(x)\|=\lambda c_k|x-p_\iii|^{-(k+1)}$. Comparing at $k=1$ identifies $\lambda_{f_\iii}(x)=\|Df_\iii(x)\|=\lambda|x-p_\iii|^{-2}$, whence
  \begin{equation} \label{eq:scalingid}
    \|D^kf_\iii(x)\| = c_k\lambda_{f_\iii}(x)|x-p_\iii|^{-(k-1)},
  \end{equation}
  and on $\ol{\Omega}$ the two factors are bounded by $\lambda_{f_\iii}\le\cmax^{|\iii|}$ and, by \cref{eq:polegap}, $|x-p_\iii|^{-(k-1)}\le(C_0/2)^{k-1}$, which gives \cref{eq:scaling} with $L_1=1$ and $L_k=c_k\max\{1,(C_0/2)^{k-1}\}$ for $k\ge2$. If $f_\iii$ is a similarity, then $D^kf_\iii\equiv0$ for $k\ge2$ and $\|Df_\iii\|=\lambda_{f_\iii}\le\cmax^{|\iii|}$, so \cref{eq:scaling} holds trivially.
\end{proof}

In the plane, the endgame of \cref{lem:coincidence} makes all derivatives of the difference of two limit projections vanish at a point and upgrades the vanishing to the whole domain by the identity theorem. For $d\ge3$ the projections range over the finite-dimensional family of pole fields, and membership in the family survives uniform limits, so the value at a single point already determines everything: no derivatives and no identity theorem are needed. The dictionary strengthens the common analyticity of \cref{lem:dual}: every dual projection is the restriction of an explicit rational vector field on a quantitatively controlled neighborhood. Whenever $\max_{i \in \II}\|T_{\fii_i}\|>0$, we define the set of \emph{admissible poles} by
\begin{equation} \label{eq:admissible-poles-def}
  \mathcal P = \biggl\{p \in \R^d : \dist(p,\ol{\Omega}) \ge \frac{1-\cmax}{\max_{i \in \II}\|T_{\fii_i}\|}\biggr\} \cup \{\infty\}.
\end{equation}

\begin{lemma} \label{lem:poledict}
  Let $d\ge3$ and suppose that $\max_{i \in \II}\|T_{\fii_i}\|>0$. Then for every $\iii\in\II^\N\cup\II^*$ there is a unique $p_\iii\in\mathcal P$ with $T_\iii=T_{p_\iii}$ on $\ol{\Omega}$; in particular $T_\iii$ extends real-analytically to the neighborhood $\{x\in\R^d : \dist(x,\ol{\Omega})<(1-\cmax)/\max_{i \in \II}\|T_{\fii_i}\|\}$ of $\ol{\Omega}$. Moreover, for every finite $p\in\mathcal P$ and every $x_0\in\ol{\Omega}$,
  \begin{equation} \label{eq:polerecover}
    p = x_0+\frac{2T_p(x_0)}{|T_p(x_0)|^2}
  \end{equation}
  so that, for $\iii,\jjj\in\II^\N\cup\II^*$, the equality $T_\iii(x_0)=T_\jjj(x_0)$ at a single $x_0\in\ol{\Omega}$ already forces $p_\iii=p_\jjj$ and $T_\iii\equiv T_\jjj$ on $\ol{\Omega}$.
\end{lemma}

\begin{proof}
  Write $C_0=2\max_{i \in \II}\|T_{\fii_i}\|/(1-\cmax)$, so that $\mathcal P=\{p\in\R^d : \dist(p,\ol{\Omega})\ge2/C_0\}\cup\{\infty\}$. For finite $p\in\mathcal P$ and $x_0\in\ol{\Omega}$ the vector $v=T_p(x_0)$ satisfies $|v|=2/|x_0-p|>0$ and $2v/|v|^2=-(x_0-p)$, which rearranges to \cref{eq:polerecover}. Hence the map $p\mapsto T_p(x_0)$ is injective on $\mathcal P$ for every fixed $x_0\in\ol{\Omega}$: distinct finite poles are told apart by \cref{eq:polerecover}, and every finite pole is told apart from $\infty$ by $T_p(x_0)\ne0=T_\infty(x_0)$. This proves the uniqueness of $p_\iii$ and the final claim, granted the existence of $p_\iii$, since $T_\iii(x_0)=T_\jjj(x_0)$ then reads $T_{p_\iii}(x_0)=T_{p_\jjj}(x_0)$.

  For finite $\iii$, the M\"obius classification and \cref{lem:mobius,lem:CU} give existence: a non-similarity word has $T_\iii=T_{p_\iii}$ with $\dist(p_\iii,\ol{\Omega})\ge2/C_0$, and a similarity word has $T_\iii\equiv0=T_\infty$ with $p_\iii=\infty$. Let then $\iii\in\II^\N$ and consider the poles $q_n=p_{\iii|_n}\in\mathcal P$ of the prefixes. Passing to a subsequence convergent in $\RS^d$, either $q_{n_j}\to\infty$, or $q_{n_j}\to q$ for a finite $q$, which then lies in $\mathcal P$, the finite part of $\mathcal P$ being closed. In the first case $\dist(q_{n_j},\ol{\Omega})\to\infty$ along the finite terms because $\ol{\Omega}$ is bounded, and hence $\|T_{q_{n_j}}\|\to0$, the supremum being $2/\dist(q_{n_j},\ol{\Omega})$ for finite $q_{n_j}$ and $0$ for $q_{n_j}=\infty$, so $T_{q_{n_j}}\to T_\infty$ uniformly on $\ol{\Omega}$. In the second case, $q_{n_j}$ is finite for all sufficiently large $j$, and once $|q_{n_j}-q|\le1/C_0$, every point $w$ on the segment from $q_{n_j}$ to $q$ satisfies $\dist(w,\ol{\Omega})\ge1/C_0$, the function $\dist(\,\cdot\,,\ol{\Omega})$ being $1$-Lipschitz, and on the region $|x-w|\ge1/C_0$ the $w$-gradient of $w\mapsto-2(x-w)/|x-w|^2$ is, by the second identity of \cref{eq:pole}, $2|x-w|^{-2}$ times a reflection, hence of norm $2/|x-w|^2\le2C_0^2$, so the mean value inequality along the segment gives $\|T_{q_{n_j}}-T_q\|\le2C_0^2|q_{n_j}-q|\to0$. In both cases $T_{\iii|_{n_j}}=T_{q_{n_j}}$ converges uniformly on $\ol{\Omega}$ to a pole field with pole in $\mathcal P$, while $T_{\iii|_n}\to T_\iii$ uniformly on $\ol{\Omega}$ by \cref{lem:dual}, so $T_\iii$ is that pole field. The asserted real-analytic extension is clear for $T_\infty$ and is given by $T_{p_\iii}$ for finite $p_\iii$, whose only singularity is $p_\iii$; this point lies outside the asserted neighborhood because $\dist(p_\iii,\ol{\Omega})\ge2/C_0$, while the neighborhood is defined by the strict inequality $\dist(x,\ol{\Omega})<2/C_0$.
\end{proof}

It remains to transfer the super-exponential closeness of the compositions to their first two derivatives, the data from which the pre-Schwarzian is built. In the plane the Cauchy estimate of \cref{lem:cauchy} did this; in $\R^d$ the elementary Landau-type trade-off below does, a higher-dimensional counterpart of \cite[Lemma~3.1]{BaranyKolossvaryTroscheit}: a function small in value and bounded in second derivative has a small first derivative one geometric-mean scale inward.

\begin{lemma} \label{lem:landau}
  Let $U\subseteq\R^d$ be open, let $E$ be a finite-dimensional normed space, and let $h\colon U\to E$ be a $\mathcal{C}^2$-map with $\sup_{x\in U}|h(x)|\le\eta$ and $\sup_{x\in U}\|D^2h(x)\|\le Q$ for some $0<\eta,Q<\infty$. Then
  \begin{equation} \label{eq:landau}
    \|Dh(x)\| \le \tfrac32\sqrt{\eta Q}
  \end{equation}
  for every $x\in U$ with $\dist(x,\R^d\setminus U)\ge\sqrt{\eta/Q}$.
\end{lemma}

\begin{proof}
  Put $s=\sqrt{\eta/Q}$ and fix $x$ as in the statement, a unit vector $e\in\R^d$, and $0<t<s$. The ball $B^o(x,s)$ lies in $U$, so the segment from $x-te$ to $x+te$ does as well, and subtracting the Taylor expansions with integral remainder $h(x\pm te)=h(x)\pm tDh(x)e+\int_0^t(t-u)D^2h(x\pm ue)[e,e]\dd u$ gives $|h(x+te)-h(x-te)-2tDh(x)e|\le t^2Q$. Since $|h(x+te)-h(x-te)|\le2\eta$, it follows that $2t|Dh(x)e|\le2\eta+t^2Q$, whence $|Dh(x)e|\le\eta/t+tQ/2$. Letting $t\uparrow s$ and taking the supremum over unit vectors $e$ gives $\|Dh(x)\|\le\eta/s+sQ/2=\tfrac32\sqrt{\eta Q}$.
\end{proof}

\subsection{Separation by the pre-Schwarzian} \label{sec:C3-pre}

With the pole gap, the pole dictionary, and the Landau trade-off assembled, we turn to the higher-dimensional case. A failure of the strong exponential separation condition again condenses the compositions super-exponentially along a sequence of distinct equal-length pairs; two applications of \cref{lem:landau} carry that smallness to the first two derivatives of the difference, the trace formula for the pre-Schwarzian converts it into closeness of the differences $T_\iii-T_\jjj$, and, once the common prefix is peeled, the value at a single point feeds \cref{lem:poledict}, the coincidence it returns contradicting \cref{eq:mainhyp}.

\begin{proof}[Proof of \cref{thm:main}\cref{it:main-pre} for $d\ge3$]
  The hypothesis \cref{eq:mainhyp} forces $\max_{i \in \II}\|T_{\fii_i}\|>0$: if $\max_{i \in \II}\|T_{\fii_i}\|=0$, then every generator has vanishing pre-Schwarzian on $\ol{\Omega}$ and is a similarity, since a non-similarity generator would have a nowhere-vanishing pole field as its pre-Schwarzian, so every composition is a similarity and $T_{12}\equiv T_{21}\equiv0$ for the distinct equal-length words $12$ and $21$, violating \cref{eq:mainhyp}. In particular \cref{lem:poledict} applies.

  Suppose now, contrary to the claim, that the strong exponential separation condition fails. Since reversal is a bijection of $\II^n$ and $f_\iii=\fii_{\overleftarrow{\iii}}$, the level-$n$ collection of pairwise supremum distances for the reversed compositions agrees, after relabelling, with that for the forward compositions. Hence the reversed compositions also fail the condition, so the plain version of \cref{lem:condensation-mobius}, with the identity in place of every M\"obius transformation as in \cref{sec:cifs-esc}, provides a strictly increasing sequence $(n_\ell)_{\ell \ge 1}$ and distinct $\iii_\ell,\jjj_\ell\in\II^{n_\ell}$ whose errors $\eta_{n_\ell}=\|f_{\iii_\ell}-f_{\jjj_\ell}\|$ satisfy $\log\eta_{n_\ell}/n_\ell\to-\infty$; the dependence of $\iii$ and $\jjj$ on $\ell$ is suppressed below. Every $\eta_{n_\ell}$ is positive, since $\eta_{n_\ell}=0$ would give $f_\iii=f_\jjj$ on $\ol{\Omega}$ and hence $T_\iii=T_\jjj$, contradicting \cref{eq:mainhyp}. As in \cref{sec:C2}, put $\rho=\tfrac12\dist(X,\partial\Omega)>0$ and $K_\rho=\{x\in\ol{\Omega} : \dist(x,\partial\Omega)\ge\rho\}$, and fix $z_0\in X$. Since $\dist(X,\partial\Omega)=2\rho$, we have $X\subseteq K_\rho$, so $z_0\in K_\rho$. Recall also that $f_\kkk(z_0)\in X$ for every $\kkk\in\II^*$ and that $\dist(x,\R^d\setminus\Omega)=\dist(x,\partial\Omega)$ for $x\in\Omega$.

  Each application of \cref{lem:landau} costs one scale inward, so the two together must still reach $K_\rho$. The difference $h=f_\iii-f_\jjj$ is real-analytic on $\Omega$ with $\sup_{x\in\Omega}|h(x)|\le\eta_{n_\ell}$. Since $\Omega\subseteq\ol{\Omega}$, \cref{eq:scaling} gives $\sup_{x\in\Omega}\|D^2h(x)\|\le2L_2\cmax^{n_\ell}=Q_2$ and $\sup_{x\in\Omega}\|D^3h(x)\|\le2L_3\cmax^{n_\ell}=Q_3$. Applying \cref{lem:landau} on $U=\Omega$ gives $\|Dh(x)\|\le2\sqrt{\eta_{n_\ell}Q_2}=\eta_{(1)}$ whenever $\dist(x,\partial\Omega)\ge s_1=\sqrt{\eta_{n_\ell}/Q_2}$. Apply it again to $Dh$, viewed as a map into the space of linear maps of $\R^d$ under the operator norm, on the open set $U_1=\{x\in\Omega : \dist(x,\partial\Omega)>s_1\}$; here $\|D(Dh)(x)\|=\|D^2h(x)\|$ and $\|D^2(Dh)(x)\|=\|D^3h(x)\|$, the supremum over the last unit vector being taken first, so the second-derivative bound for $Dh$ is $Q_3$. If $\dist(x,\partial\Omega)\ge s_1+s_2$, where $s_2=\sqrt{\eta_{(1)}/Q_3}$, then $\dist(x,\R^d\setminus U_1)\ge s_2$: for $y\notin\Omega$ this follows from $\dist(x,\R^d\setminus\Omega)=\dist(x,\partial\Omega)$, while for $y\in\Omega\setminus U_1$ the $1$-Lipschitz property of $\dist(\,\cdot\,,\partial\Omega)$ gives $|x-y|\ge\dist(x,\partial\Omega)-\dist(y,\partial\Omega)\ge s_2$. Therefore $\|D^2h(x)\|\le2\sqrt{\eta_{(1)}Q_3}=\eta_{(2)}$ whenever $\dist(x,\partial\Omega)\ge s_1+s_2$. Since $\eta_{(1)}\le2\sqrt{2L_2}\eta_{n_\ell}^{1/2}$ and $\eta_{(2)}\le4\sqrt{L_3}(2L_2)^{1/4}\eta_{n_\ell}^{1/4}$, both are super-exponentially small, and $s_1$ and $s_2$ tend to $0$; discarding finitely many $\ell$, we may assume $s_1+s_2\le\rho$, so that $\sup_{x\in K_\rho}\|Dh(x)\|\le\eta_{(1)}$ and $\sup_{x\in K_\rho}\|D^2h(x)\|\le\eta_{(2)}$.

  Once the pre-Schwarzian is written as a trace, the two derivative bounds pass to $T_\iii-T_\jjj$ through the resolvent identity. Conformality gives $\|Df\|=|\det Df|^{1/d}$ by \cref{sec:conformal-maps}, so the pre-Schwarzian is the determinant field $\tfrac1d\nabla\log|\det D\,\cdot\,|$, as noted in \cref{sec:kernels}, and Jacobi's formula in the form of the expansion $\det(A+tB)=\det A(1+t\tr(A^{-1}B)+O(t^2))$ for invertible $A$ gives, componentwise,
  \begin{equation} \label{eq:trace}
    (T_\iii(x))_j = \tfrac1d\tr((Df_\iii(x))^{-1}\partial_jDf_\iii(x)).
  \end{equation}
  For $x\in K_\rho$ abbreviate $A_\iii=Df_\iii(x)$ and $B_{\iii,j}=\partial_jDf_\iii(x)$, and likewise for $\jjj$. Conformality gives $\|A_\iii^{-1}\|=\lambda_{f_\iii}(x)^{-1}\le\cmin^{-n_\ell}$, so the resolvent identity $A_\iii^{-1}-A_\jjj^{-1}=A_\iii^{-1}(A_\jjj-A_\iii)A_\jjj^{-1}$ gives $\|A_\iii^{-1}-A_\jjj^{-1}\|\le\cmin^{-2n_\ell}\eta_{(1)}$. Combining this with $|\tr M|\le d\|M\|$, with $\|B_{\iii,j}\|\le\|D^2f_\iii(x)\|\le L_2\cmax^{n_\ell}$, and with $\|B_{\iii,j}-B_{\jjj,j}\|\le\|D^2h(x)\|\le\eta_{(2)}$, the difference of \cref{eq:trace} for $\iii$ and $\jjj$ is bounded by
  \begin{equation} \label{eq:C3quot}
    \sup_{x\in K_\rho}|T_\iii(x)-T_\jjj(x)| \le \sqrt{d}(L_2\cmax^{n_\ell}\cmin^{-2n_\ell}\eta_{(1)}+\cmin^{-n_\ell}\eta_{(2)}) = \eta_{n_\ell}',
  \end{equation}
  and $\log\eta_{n_\ell}'/n_\ell\to-\infty$, the new factors being only exponential in $n_\ell$.

  Peeling the common prefix, as in the common strategy of \cref{sec:reduction}, costs an exponential factor that this closeness absorbs. With $\kkk=\iii\land\jjj$, $m=|\kkk|<n_\ell$, $\hi=\sigma^m\iii$, $\hj=\sigma^m\jjj$, and $\Xi_\ell=T_{\hi}-T_{\hj}$, subtracting the peeling identities \cref{eq:peeling} gives $T_\iii-T_\jjj=(Df_\kkk)^{\top}(\Xi_\ell\circ f_\kkk)$, where $Df_\kkk(x)^{\top}$ is $\lambda_{f_\kkk}(x)$ times an orthogonal matrix and $\lambda_{f_\kkk}(x)\ge\cmin^m\ge\cmin^{n_\ell}$, so $|\Xi_\ell|\le\cmin^{-n_\ell}\eta_{n_\ell}'=\eta_{n_\ell}''$ on $f_\kkk(K_\rho)$, still with $\log\eta_{n_\ell}''/n_\ell\to-\infty$.

  It remains to pass from $\hi$ and $\hj$ to their limit pair and to finish from the value at a single point. The suffixes $\hi$ and $\hj$ are nonempty, have the common length $n_\ell-m$, and have distinct first letters. Since the alphabet is finite, pass first to a subsequence on which the ordered pair $(i_1',j_1')$ is constant, and then, by compactness of $(\II^\N\cup\II^*)^2$, to a further subsequence along which $\hi\to\ppp$ and $\hj\to\qqq$. If one limit is finite, then the corresponding suffix sequence is eventually constant, and the equality $|\hi|=|\hj|$ forces the other limit to be finite of the same length; otherwise both limits are infinite. Thus $|\ppp|=|\qqq|$ and $\ppp_1\ne\qqq_1$. Passing to a further subsequence, the points $x_\ell=f_\kkk(z_0)\in X$ converge to some $\zeta\in X$, while the prefix continuity of \cref{lem:dual} gives $\Xi_\ell\to \Xi_0=T_{\ppp}-T_{\qqq}$ uniformly on $\ol{\Omega}$. Then $|\Xi_0(\zeta)|\le|\Xi_0(\zeta)-\Xi_0(x_\ell)|+\|\Xi_0-\Xi_\ell\|+|\Xi_\ell(x_\ell)|$, where the first term tends to $0$ by the continuity of $\Xi_0$ on $\ol{\Omega}$, the second by uniform convergence, and the third by $|\Xi_\ell(x_\ell)|\le\eta_{n_\ell}''$, since $z_0\in K_\rho$ and hence $x_\ell=f_\kkk(z_0)\in f_\kkk(K_\rho)$. Hence $T_{\ppp}(\zeta)=T_{\qqq}(\zeta)$ at the single point $\zeta\in\ol{\Omega}$, so \cref{lem:poledict} gives $T_{\ppp}\equiv T_{\qqq}$ on $\ol{\Omega}$, and the distinct equal-length pair $\ppp$ and $\qqq$ violates \cref{eq:mainhyp}. This contradiction proves the theorem for $d\ge3$.
\end{proof}

We close the section with two observations on the higher-dimensional argument: the economy that M\"obius rigidity buys relative to the planar proof, and the role of the pole at infinity, which the hypothesis of \cref{thm:main} must accommodate.

\begin{remark} \label{rem:rigid}
  M\"obius rigidity pays twice: the pole gap of \cref{lem:CU} supplies through \cref{eq:scaling} the derivative bounds that the dual cap and \cref{lem:cauchy} supplied in the plane, and the pole dictionary of \cref{lem:poledict} finishes from the value at one point alone, where the endgame of \cref{lem:coincidence} needed every derivative at the point and the identity theorem; the corresponding comparison with the original argument on the line is recorded in \cref{rem:cleaner}. The point $\infty$ is a genuine member of $\mathcal P$: two distinct equal-length similarity words $\iii$ and $\jjj$ share the pole $\infty$ and the projection $T_\iii\equiv T_\jjj\equiv0$, so \cref{eq:mainhyp} fails outright for a system admitting such a pair, and along an infinite word the poles of the prefixes may escape to $\infty$, in which case the dual projection of the word is the zero field. The hypothesis of \cref{thm:main}, quantified over finite and infinite words alike, prices this in.
\end{remark}

\section{Uniform dual gaps and a transversality principle} \label{sec:gaps}

We now turn to \cref{thm:genericity}, whose three cases are proved one at a time in \cref{sec:gen-line-dense,sec:gen-plane-dense,sec:gen-higher-dense}. What the three proofs need in common is prepared here, so that each of them is left with what belongs to its own dimension.

For a fixed dimension $d \ge 1$, domains $\Omega \subset \R^d$ and $\Omega' \supset \overline{\Omega}$, and alphabet $\II$, write $\mathcal S$ for the corresponding space of conformal iterated function systems and use the $\mathcal C^2$-metric
\begin{equation} \label{eq:d2metric}
  \begin{split}
    d_2(\Phi,\Psi) &= \max_{i \in \II}\|\fii_i-\psi_i\|_{\mathcal C^2(\ol\Omega)} \\
    &= \max_{i \in \II}(\|\fii_i-\psi_i\|+\|D\fii_i-D\psi_i\|+\|D^2\fii_i-D^2\psi_i\|),
  \end{split}
\end{equation}
where the scalar derivatives are used in dimensions one and two. All norms are over $\ol{\Omega}$ unless another set is indicated, and all openness and density statements are relative to the indicated system space. Although $d_2$ sees only the restrictions of the generators to $\ol{\Omega}$, it is a metric on $\mathcal S$: two members agreeing on $\ol{\Omega}$ agree on the connected $\Omega'$ by the identity principle, the generators being real-analytic there. The two requirements of \cref{sec:cifs} are stable under small perturbations, so that a tuple $\Psi=(\psi_i)_{i \in \II}$ conformal on $\Omega'$ and $d_2$-close to a member $\Phi=(\fii_i)_{i \in \II}$ of $\mathcal S$ is again a member of $\mathcal S$: each $\fii_i(\ol{\Omega})$ is a compact subset of $\Omega$, at positive distance $\rho_i$ from $\partial\Omega$, so $d_2(\Phi,\Psi)<\min_{i\in\II}\rho_i$ gives $\psi_i(\ol{\Omega})\subseteq\Omega$, and $\|D\psi_i\|\le\cmax+d_2(\Phi,\Psi)<1$ once $d_2(\Phi,\Psi)<1-\cmax$. The line admits localized real-analytic perturbations and gives an unconditional result. In the plane the Cauchy estimates rule out localization at interior points, so the Schwarzian gap class is made dense by finite-dimensional transversality, with perturbations localized at boundary points by peaking functions, and the pre-Schwarzian gap class inherits the density through the inclusion \cref{eq:GSsubsetGT}; the openness there is unconditional, the density assumes that $\Omega$ is a Jordan domain and that $\Omega'$ is simply connected, and \cref{ex:nonconvex} shows that the method fails on multiply connected domains. In dimensions at least three the conformal family is finite dimensional and the perturbations act directly on the poles; the density there is conditional, and \cref{prop:nondense} shows that it must be.

The three cases of the program run by different tools, but they share their objects and part of their machinery. This section fixes the gap classes once and proves the five results that serve more than one of the cases: \cref{lem:gap} identifies membership in a gap class both with the absence of coincident projections indexed by pairs with distinct first letters and with the corresponding pointwise hypothesis of \cref{thm:main}; \cref{lem:Sclosed} supplies the openness of $\mathcal G_S$ in the two dimensions where the Schwarzian is defined, its hypothesis being verified in \cref{sec:gen-line-open} on the line and in \cref{sec:gen-plane-open} in the plane; \cref{lem:avoidance} isolates the counting step behind both transversality arguments, those of \cref{sec:gen-plane-dense,sec:gen-higher-dense}; \cref{lem:paircover} supplies the metric on the pair space and the covering numbers that the counting consumes in both; and \cref{lem:smallball} derives the small-ball hypothesis of the counting from a lower bound on the least singular value of the parameter derivative, again in both. The three cases themselves occupy \cref{sec:gen-line,sec:gen-plane,sec:gen-higher}. They come in the reverse of the order of \cref{sec:C2,sec:C1,sec:C3}. There the line was reached through the plane, its real-analytic generators lifting to a complex neighborhood by \cref{lem:lift}; here the line needs no such lift, and the plane and the dimensions $d\ge3$ follow it as successive retreats.

Let 
\begin{equation} \label{eq:Zdef}
  Z = \{(\iii,\jjj) \in (\II^\N\cup\II^*)^2 : |\iii|=|\jjj| \ge 1 \text{ and } i_1 \ne j_1\}.
\end{equation}
Since a convergent sequence of words is either eventually constant or has lengths tending to infinity, lengths and first letters persist in limits, and $Z$ is a closed, hence compact, subset of $(\II^\N\cup\II^*)^2$. Define
\begin{equation*}
  \Delta_T(\Phi) = \inf_{(\iii,\jjj)\in Z}\|T_\iii-T_\jjj\| \qquad\text{and}\qquad \mathcal G_T = \{\Phi\in\mathcal S : \Delta_T(\Phi)>0\}.
\end{equation*}
When $d\in\{1,2\}$, define similarly
\begin{equation*}
  \Delta_S(\Phi) = \inf_{(\iii,\jjj)\in Z}\|S_\iii-S_\jjj\| \qquad\text{and}\qquad \mathcal G_S = \{\Phi\in\mathcal S : \Delta_S(\Phi)>0\}.
\end{equation*}
The infima defining $\Delta_T$ and $\Delta_S$ range only over pairs with distinct first letters. Nothing is lost by the restriction: peeling the common prefix turns a coincidence between two equal-length words into one indexed by a pair in $Z$, so a positive gap is exactly the hypothesis of \cref{thm:main}, now carried by a single number attached to the system.

\begin{lemma} \label{lem:gap}
  Let $\Phi$ be a conformal iterated function system on $\Omega \subset \R^d$.
  \begin{enumerate}
    \item\label{it:gap-pre} The system $\Phi$ belongs to $\mathcal G_T$ if and only if $T_\iii\not\equiv T_\jjj$ on $\ol{\Omega}$ for every $(\iii,\jjj)\in Z$, and if and only if
    \begin{equation} \label{eq:t-mainhyp}
      \sup_{x\in\ol{\Omega}}|T_\iii(x)-T_\jjj(x)|>0
    \end{equation}
    for all distinct $\iii,\jjj\in\II^\N\cup\II^*$ with $|\iii|=|\jjj|$, which is the hypothesis \cref{eq:mainhyp} of \cref{thm:main}.
    \item\label{it:gap-schw} If $d\in\{1,2\}$, then $\Phi$ belongs to $\mathcal G_S$ if and only if $S_\iii\not\equiv S_\jjj$ on $\ol{\Omega}$ for every $(\iii,\jjj)\in Z$, and if and only if
    \begin{equation} \label{eq:s-mainhyp2}
      \sup_{x\in\ol{\Omega}}|S_\iii(x)-S_\jjj(x)|>0
    \end{equation}
    for all distinct $\iii,\jjj\in\II^\N\cup\II^*$ with $|\iii|=|\jjj|$, which is the hypothesis \cref{eq:mainhyp2} of \cref{thm:main}.
  \end{enumerate}
  Moreover,
  \begin{equation} \label{eq:GSsubsetGT}
    \mathcal G_S\subseteq\mathcal G_T
  \end{equation}
  when $d\in\{1,2\}$. Consequently members of $\mathcal G_T$ satisfy the strong exponential separation condition, while members of $\mathcal G_S$ satisfy it modulo M\"obius maps.
\end{lemma}

\begin{proof}
  The maps $(\iii,\jjj)\mapsto\|T_\iii-T_\jjj\|$ and, in dimensions one and two, $(\iii,\jjj)\mapsto\|S_\iii-S_\jjj\|$ are continuous on $Z$ by \cref{lem:dual,lem:dual2}. Their infima are therefore attained, and positivity is equivalent to the asserted absence of a coincidence.

  Let $P$ denote either $T$, in any dimension, or $S$, in dimensions one and two, and suppose that no two $P$-projections indexed by a pair in $Z$ coincide. Given distinct equal-length words $\iii,\jjj$, put $\kkk=\iii\land\jjj$ and $m=|\kkk|$. The peeled pair $(\sigma^m\iii,\sigma^m\jjj)$ lies in $Z$, so $B=P_{\sigma^m\iii}-P_{\sigma^m\jjj}$ is real-analytic on $\Omega$ and does not vanish identically. By \cref{eq:peeling,eq:peeling2}, we have $P_\iii-P_\jjj=\LL^{(w)}_\kkk B$, where $w=1$ when $P=T$ and $w=2$ when $P=S$. Suppose that this difference vanishes identically on $\ol{\Omega}$. At each $x\in\Omega$, put $A=Df_\kkk(x)$. In the vector and tensor formulations, the defining formulas give $(\LL^{(1)}_\kkk B)(x)=A^{\top}B(f_\kkk(x))$ and $(\LL^{(2)}_\kkk B)(x)=A^{\top}B(f_\kkk(x))A$; in the scalar formulation, the corresponding multipliers are $f_\kkk'(x)$ and $(f_\kkk'(x))^2$. Since $A$ is invertible, each of these pointwise maps is injective, and therefore $B$ vanishes on $f_\kkk(\Omega)$. Since $f_\kkk$ is a local diffeomorphism, this image is a nonempty open subset of $\Omega$, and the identity principle for real-analytic functions, applied componentwise, gives $B\equiv0$ on $\Omega$ and hence on $\ol{\Omega}$ by continuity, a contradiction. Thus $P_\iii\not\equiv P_\jjj$ on $\ol{\Omega}$. This proves that the absence of such coincidences implies \cref{eq:t-mainhyp}, respectively \cref{eq:s-mainhyp2}. Conversely, every pair in $Z$ consists of two distinct words of equal length, so \cref{eq:t-mainhyp}, respectively \cref{eq:s-mainhyp2}, excludes a coincidence of two $P$-projections indexed by a pair in $Z$. The separation conclusions follow from \cref{thm:main}.

  Finally, for $d\in\{1,2\}$ the identity $S_\lll=T_\lll'-\tfrac12T_\lll^2$ of \cref{lem:dual2}, valid also for infinite words, shows that a coincidence of two pre-Schwarzian projections forces a coincidence of the corresponding Schwarzian projections. Its contrapositive gives \cref{eq:GSsubsetGT}.
\end{proof}

For a system meeting the hypotheses of \cref{prop:gencrit}, membership in $\mathcal G_T$, respectively $\mathcal G_S$, follows outright. The estimate \cref{eq:gencritbound} bounds $\|T_\iii-T_\jjj\|$, respectively $\|S_\iii-S_\jjj\|$, from below at a single point of $\ol{\Omega}$ for every pair in $Z$, so with $\alpha$ as in \cref{prop:gencrit}\cref{it:gencrit-pre}, respectively \cref{prop:gencrit}\cref{it:gencrit-schw},
\begin{align*}
  \Delta_T(\Phi) &\ge \alpha-\frac{2\cmax}{1-\cmax}\max_{i \in \II}\|T_{\fii_i}\|, \\
  \Delta_S(\Phi) &\ge \alpha-\frac{2\cmax^2}{1-\cmax^2}\max_{i \in \II}\|S_{\fii_i}\|,
\end{align*}
the second bound when $d\in\{1,2\}$, and the hypothesis on $\alpha$ makes both right-hand sides positive.

The openness of $\mathcal G_S$ in the $\mathcal C^2$-topology does not follow from continuity of the Schwarzian, which involves a third derivative. It follows instead from the closedness of its coincidence relation.

\begin{lemma} \label{lem:Sclosed}
  Let $d\in\{1,2\}$, let $\Phi_n\to\Phi$ in $d_2$, and suppose that
  \begin{equation} \label{eq:Tparametercontinuity}
    \sup_{\lll\in\II^\N\cup\II^*}\|T^{\Phi_n}_\lll-T^\Phi_\lll\|\to0.
  \end{equation}
  If there are $(\iii_n,\jjj_n)\in Z$ such that $S^{\Phi_n}_{\iii_n}\equiv S^{\Phi_n}_{\jjj_n}$ for every $n$, then a subsequential limit $(\iii,\jjj)\in Z$ satisfies $S^\Phi_\iii\equiv S^\Phi_\jjj$. In particular, whenever \cref{eq:Tparametercontinuity} holds locally uniformly in the system, $\mathcal G_S$ is open.
\end{lemma}

\begin{proof}
  Pass to a subsequence for which $(\iii_n,\jjj_n)\to(\iii,\jjj)\in Z$. Then
  \begin{equation*}
    A_n=T^{\Phi_n}_{\iii_n}\to A=T^\Phi_\iii \qquad\text{and}\qquad B_n=T^{\Phi_n}_{\jjj_n}\to B=T^\Phi_\jjj
  \end{equation*}
  uniformly on $\ol{\Omega}$. By the triangle inequality, the first convergence reduces to $\|T^{\Phi_n}_{\iii_n}-T^\Phi_{\iii_n}\|\to0$ and $\|T^\Phi_{\iii_n}-T^\Phi_\iii\|\to0$. The first term tends to zero by \cref{eq:Tparametercontinuity} and the second by \cref{eq:dualbounds}; the argument for $B_n$ is identical. The identity of \cref{lem:dual2} turns the assumed Schwarzian coincidence into
  \begin{equation*}
    (A_n-B_n)'=\tfrac12(A_n+B_n)(A_n-B_n).
  \end{equation*}
  Fix $z_1\in\Omega$ and $\delta>0$ with $B^o(z_1,\delta)\subseteq\Omega$, the ball being an interval in the line case. For $z\in B^o(z_1,\delta)$ the segment from $z_1$ to $z$ lies in $B^o(z_1,\delta)$. Write $D_n=A_n-B_n$ and $C_n=\tfrac12(A_n+B_n)$, parametrize the segment by $\gamma(t)=z_1+t(z-z_1)$ for $t\in[0,1]$, and put $u_n(t)=D_n(\gamma(t))$, a $\mathcal C^1$ function on $[0,1]$ by \cref{lem:dual}. Differentiating along the path, by the complex chain rule when $d=2$, gives $u_n'(t)=D_n'(\gamma(t))(z-z_1)$, so the differential equation already recorded yields $u_n'(t)=a_n(t)u_n(t)$, where the coefficient $a_n(t)=C_n(\gamma(t))(z-z_1)$ is continuous on $[0,1]$. The integrating factor $E_n(t)=\exp(-\int_0^t a_n(s)\dd s)$ therefore satisfies $E_n'(t)=-a_n(t)E_n(t)$, and
  \begin{equation*}
    (E_n u_n)' = E_n'u_n+E_n u_n' = -a_n E_n u_n+E_n a_n u_n = 0.
  \end{equation*}
  Thus $E_n u_n$ is constant, so $u_n(t)=u_n(0)\exp(\int_0^t a_n(s)\dd s)$. The change of variables $w=\gamma(s)$ rewrites the integral as the segment integral of $C_n$ from $z_1$ to $\gamma(t)$. Evaluating at $t=1$ gives
  \begin{equation*}
    A_n(z)-B_n(z)=(A_n(z_1)-B_n(z_1))\exp\biggl(\frac12\int_{z_1}^{z}A_n(w)+B_n(w)\dd w\biggr).
  \end{equation*}
  Uniform convergence lets this identity pass to the limit on $B^o(z_1,\delta)$, where the exponent is a primitive of $\tfrac12(A+B)$, the integrand being real-analytic on $\Omega$ by \cref{lem:dual} and holomorphic when $d=2$. Differentiating the resulting identity with respect to $z$ gives $(A-B)'=\tfrac12(A+B)(A-B)$ on $B^o(z_1,\delta)$ and, since $z_1\in\Omega$ was arbitrary, on all of $\Omega$, which is precisely $S^\Phi_\iii=S^\Phi_\jjj$ on $\Omega$ by \cref{lem:dual2}, and on $\ol{\Omega}$ by continuity.

  For the final assertion, suppose that $\Phi_n\notin\mathcal G_S$ and $\Phi_n\to\Phi$. The local hypothesis gives \cref{eq:Tparametercontinuity}, and \cref{lem:gap} yields pairs $(\iii_n,\jjj_n)\in Z$ satisfying $S^{\Phi_n}_{\iii_n}\equiv S^{\Phi_n}_{\jjj_n}$. The preceding argument gives a subsequential limit $(\iii,\jjj)\in Z$ such that $S^\Phi_\iii\equiv S^\Phi_\jjj$, so another application of \cref{lem:gap} shows that $\Phi\notin\mathcal G_S$. Thus the complement of $\mathcal G_S$ is closed, and $\mathcal G_S$ is open.
\end{proof}

The following lemma isolates the counting step that the transversality arguments in the planar and pole settings have in common. The $k$-dimensional Lebesgue measure is denoted $\LL^k$.

\begin{lemma} \label{lem:avoidance}
  Let $(Y,\varrho)$ be compact, let $B\subseteq\R^M$ be a ball, and let $F\colon B\times Y\to\R^r$. Suppose that, for all sufficiently small $\delta>0$, the space $Y$ is covered by at most $C_3\delta^{-s}$ sets of $\varrho$-diameter at most $\delta$, that $|F(t,y)-F(t,y')|\le C_1\varrho(y,y')$, and that
  \begin{equation} \label{eq:smallball}
    \LL^M(\{t\in B : |F(t,y)|\le u\})\le C_2u^r
  \end{equation}
  for every $y\in Y$ and $u>0$. If $r>s$, then
  \begin{equation*}
    \LL^M(\{t\in B : F(t,y)=0\text{ for some }y\in Y\})=0.
  \end{equation*}
\end{lemma}

\begin{proof}
  If $Y$ is empty, there is nothing to prove. Let $E = \{t\in B : F(t,y) = 0\text{ for some }y\in Y\}$. Fix a sufficiently small $\delta > 0$, and let $U_1,\ldots,U_L$ cover $Y$, where $L \le C_3\delta^{-s}$ and each $U_j$ has $\varrho$-diameter at most $\delta$. After discarding empty sets, choose $y_j\in U_j$. If $t\in E$, then $F(t,y) = 0$ for some $y\in U_j$, and hence $|F(t,y_j)| \le C_1\varrho(y,y_j) \le C_1\delta < (C_1+1)\delta$. Therefore
  \begin{equation*}
    E\subseteq A_\delta = \bigcup_{j=1}^L\{t\in B : |F(t,y_j)|\le (C_1+1)\delta\}.
  \end{equation*}
  The set $A_\delta$ is measurable by \cref{eq:smallball}, and the union bound gives
  \begin{equation*}
    \LL^M(A_\delta) \le LC_2(C_1+1)^r\delta^r \le C_3C_2(C_1+1)^r\delta^{r-s}.
  \end{equation*}
  Since $r>s$, the set $E$ has Lebesgue measure zero.
\end{proof}

In both transversality arguments the compact space of \cref{lem:avoidance} is the pair space $Z$ of \cref{eq:Zdef}, metrized by a fixed base raised to the length of the shorter of the two common prefixes, so that the prefix estimates for the dual projections become Lipschitz bounds. The base is a power of a rate $\lambda\in(\cmax,1)$ that bounds the contraction of the system and of its perturbations: the square $\lambda^2$ in the plane, where the Schwarzian projections carry two derivative factors, and $\lambda$ itself in higher dimensions, where the poles are compared. The covering numbers that the counting needs are the same computation for every base, and we record it once.

\begin{lemma} \label{lem:paircover}
  Let $\beta\in(0,1)$, and for $w=(\iii,\jjj)$ and $w'=(\iii',\jjj')$ in $Z$ put $\varrho_\beta(w,w')=\beta^{\min\{|\iii\land\iii'|,|\jjj\land\jjj'|\}}$ when $w\ne w'$ and $\varrho_\beta(w,w)=0$. Then $\varrho_\beta$ is a metric inducing the topology of $Z$, so $(Z,\varrho_\beta)$ is compact, and for every $\delta\in(0,1)$ the space $Z$ is covered by at most $2\beta^{-s}\delta^{-s}$ sets of $\varrho_\beta$-diameter at most $\delta$, where $s=2\log N/\log(1/\beta)$.
\end{lemma}

\begin{proof}
  Distinct words have a finite common prefix, so $\varrho_\beta(w,w')>0$ for $w\ne w'$, and $\varrho_\beta$ is symmetric. Let $d_\beta$ be the ultrametric of \cref{sec:cifs-esc} with base $\beta$. For $w\ne w'$ we have $\varrho_\beta(w,w')=\max\{d_\beta(\iii,\iii'),d_\beta(\jjj,\jjj')\}$: when both pairs of words are distinct this is the identity $\beta^{\min\{a,b\}}=\max\{\beta^a,\beta^b\}$. When $\iii=\iii'$, the first coordinate distance is zero. Since $w\ne w'$, we have $\jjj\ne\jjj'$, and $|\jjj\land\jjj'|\le|\jjj|=|\iii|$ because the two coordinates of each point of $Z$ have equal length. Hence both sides equal $\beta^{|\jjj\land\jjj'|}$; the case $\jjj=\jjj'$ is symmetric. As the maximum of the pullbacks of $d_\beta$ under the two coordinate projections, $\varrho_\beta$ satisfies the ultrametric inequality and induces on $Z$ the topology inherited from $(\II^\N\cup\II^*)^2$, in which $Z$ is compact by the remark following \cref{eq:Zdef}. For the covering, let $p\ge1$. The pairs whose words have length at least $p$ are covered by the sets of pairs whose words have prescribed prefixes $\uuu$ and $\vvv$ with $|\uuu|=|\vvv|=p$, at most $N^{2p}$ sets, any two members of which have both common prefixes of length at least $p$, so that their $\varrho_\beta$-diameter is at most $\beta^p$; the pairs of words of length less than $p$ number at most $\sum_{k<p}N^{2k}\le N^{2p}$, as $N\ge2$ by the convention of \cref{sec:cifs-esc}, and are covered by singletons. Thus $Z$ is covered by at most $2N^{2p}=2(\beta^p)^{-s}$ sets of diameter at most $\beta^p$. Given $\delta\in(0,1)$, choose $p\ge1$ with $\beta^p\le\delta<\beta^{p-1}$; then $\beta^p>\beta\delta$, so the cover just described consists of at most $2(\beta\delta)^{-s}=2\beta^{-s}\delta^{-s}$ sets of diameter at most $\delta$.
\end{proof}

In the two transversality arguments the map $F$ of \cref{lem:avoidance} is differentiable in $t\in B$, and the small-ball hypothesis \cref{eq:smallball} comes from a lower bound on the least singular value of $D_tF$; we fix the meaning of that quantity here. For a linear map $L$ between finite-dimensional inner product spaces, with adjoint $L^{*}$, let $\sigma_{\min}(L)=\min_{|\zeta|=1}|L^{*}\zeta|$, the minimum being over the unit vectors $\zeta$ of the target space; this is the least singular value of $L$ when the singular values are taken to be the eigenvalues of $(LL^{*})^{1/2}$, so that there are as many of them as the dimension of the target. Thus $\sigma_{\min}(L)>0$ if and only if $L$ is onto, and
\begin{equation} \label{eq:sigmamin}
  |\sigma_{\min}(L+E)-\sigma_{\min}(L)| \le \|E\|
\end{equation}
for every linear map $E$ between the same spaces, by the triangle inequality and $\|E^{*}\|=\|E\|$.

The passage from a lower bound on the least singular value of $D_tF$ to the small-ball hypothesis is the same in both arguments, and the next lemma records it once; it is the mechanism of \cite[Lemma~7.7 and proof of Lemma~7.10]{PeresSchlag}. Its constant depends only on the dimensions, the radius of the ball, a bound on the derivative and on its Lipschitz constant, and the lower bound on its least singular value, which is what makes the resulting estimate uniform over the pair space.

\begin{lemma} \label{lem:smallball}
  Let $r\le M$ be positive integers, let $B\subseteq\R^M$ be an open ball of radius $r_B$, let $\sigma_*>0$ and $\Lambda<\infty$, and let $F\colon B\to\R^r$ be continuously differentiable with $\|DF(t)\|\le\Lambda$, $\|DF(t)-DF(t')\|\le\Lambda|t-t'|$, and $\sigma_{\min}(DF(t))\ge\sigma_*$ for all $t,t'\in B$. Then there is a constant $C<\infty$, depending only on $M$, $r$, $r_B$, $\Lambda$, and $\sigma_*$, such that
  \begin{equation*}
    \LL^M(\{t\in B : |F(t)|\le u\}) \le Cu^r
  \end{equation*}
  for every $u>0$.
\end{lemma}

\begin{proof}
  Since $\sigma_*\le\sigma_{\min}(DF(t))\le\|DF(t)^{*}\|=\|DF(t)\|\le\Lambda$ for $t\in B$, the number $\Lambda$ is positive. Put $a=\sigma_*^r\Lambda^{1-r}\binom{M}{r}^{-1/2}/2$ and $\eta=a/\Lambda$, let $\omega_k$ be the Lebesgue measure of the unit ball of $\R^k$ for $k\ge1$, and let $\omega_0=1$. The sets $\{t\in B : |F(t)|\le u\}$ are relatively closed in $B$, since $F$ is continuous, and hence Borel.

  Fix $t_0\in B$ and write $L=DF(t_0)$ as an $r\times M$ matrix, so that $L^{*}=L^{\top}$. Every unit vector $\zeta\in\R^r$ satisfies $\zeta^{\top}LL^{\top}\zeta=|L^{\top}\zeta|^2\ge\sigma_*^2$, so every eigenvalue of $LL^{\top}$ is at least $\sigma_*^2$, and the Cauchy--Binet formula gives
  \begin{equation*}
    \sum_J\det(L_J)^2 = \det(LL^{\top}) \ge \sigma_*^{2r},
  \end{equation*}
  where $J$ ranges over the $\binom{M}{r}$ subsets of $\{1,\ldots,M\}$ with $r$ elements and $L_J$ is the $r\times r$ matrix formed by the columns of $L$ indexed by $J$. Choose $J$ with $|\det L_J|\ge\sigma_*^r\binom{M}{r}^{-1/2}$ and put $A=L_J$. The number $|\det A|$ is the product of the $r$ singular values of $A$, each of which is at most $\|A\|\le\|L\|\le\Lambda$, so the least of them is at least $\Lambda^{1-r}|\det A|\ge2a$, and $|Av|\ge2a|v|$ for every $v\in\R^r$.

  After a permutation of the coordinates, which preserves $\LL^M$, we write $t=(x,z)$, where $x\in\R^r$ consists of the coordinates indexed by $J$ and $z\in\R^{M-r}$ of the others, and $t_0=(x_0,z_0)$. The set $W=B\cap B^o(t_0,\eta)$ is convex. Let $(x,z)$ and $(x',z)$ lie in $W$, and for $s\in[0,1]$ let $\tau(s)=(x'+s(x-x'),z)$ and let $A_s$ be the matrix formed by the columns of $DF(\tau(s))$ indexed by $J$, so that $\|A_s-A\|\le\|DF(\tau(s))-DF(t_0)\|\le\Lambda|\tau(s)-t_0|<a$. Then $F(x,z)-F(x',z)=\int_0^1A_s(x-x')\dd s$, and hence
  \begin{equation*}
    |F(x,z)-F(x',z)| \ge |A(x-x')|-a|x-x'| \ge a|x-x'|.
  \end{equation*}
  Consequently, for fixed $z$ and $u>0$, any two points of the slice $\{x\in\R^r : (x,z)\in W\text{ and }|F(x,z)|\le u\}$ lie at distance at most $2u/a$. If this slice is nonempty, fixing one of its points shows that it is contained in a ball of radius $2u/a$, so its $\LL^r$-measure is at most $\omega_r(2u/a)^r$. The slice is empty unless $|z-z_0|<\eta$, and Fubini's theorem, integrating this bound over the $(M-r)$-dimensional ball of radius $\eta$ centered at $z_0$, gives
  \begin{equation*}
    \LL^M(\{t\in W : |F(t)|\le u\}) \le \omega_{M-r}\eta^{M-r}\omega_r\biggl(\frac{2u}{a}\biggr)^r,
  \end{equation*}
  which for $M=r$ is the bound on the single slice.

  Finally, if $t_1,\ldots,t_n$ are points of $B$ at mutual distance at least $\eta$, then the balls $B^o(t_j,\eta/2)$ are pairwise disjoint and contained in the open ball of radius $r_B+\eta/2$ concentric with $B$, so that $n\le(1+2r_B/\eta)^M$. We may therefore choose such points with $n$ as large as possible, and the maximality of $n$ gives $B\subseteq\bigcup_{j=1}^nB^o(t_j,\eta)$. Applying the preceding estimate with $t_0=t_j$, summing over this cover, and using the bound on $n$ proves the lemma with $C=(1+2r_B/\eta)^M\omega_{M-r}\omega_r\eta^{M-r}(2/a)^r$.
\end{proof}

\section{Genericity on the line} \label{sec:gen-line}

We now let $d=1$ and show that $\mathcal G_S$, and with it $\mathcal G_T$, is open and dense in $(\mathcal S,d_2)$. The metric \cref{eq:d2metric} and the formulas defining $\Delta_T$ and $\Delta_S$ see only the restrictions of the generators to $\ol{\Omega}$, and the formulas make sense for every conformal iterated function system on $\Omega$, whatever its extension domain; we use them in that generality, while $\mathcal G_T$ and $\mathcal G_S$ remain subsets of $\mathcal S$. The localization below is purely real: its holomorphic extensions may grow away from the real axis, which is why no planar analogue is possible. \Cref{sec:gen-line-open} proves the openness, \cref{sec:gen-line-cyl} reduces the gap condition to a condition on words of one fixed length, \cref{sec:gen-line-loc} builds the localized perturbation that meets that condition, and \cref{sec:gen-line-dense} assembles the two into \cref{thm:C1denseS}. \Cref{thm:genericity}\cref{it:gen-line} is proved at the end of that subsection.

\subsection{Openness} \label{sec:gen-line-open}

The pre-Schwarzian projections vary uniformly in $d_2$, with a constant independent of the word; that is precisely the hypothesis of \cref{lem:Sclosed}, so the openness of $\mathcal G_S$ follows from the estimate, and that of $\mathcal G_T$ follows from it directly. Since the intermediate systems of the density proof need not extend conformally to the fixed $\Omega'$, the estimate is stated for arbitrary conformal iterated function systems on $\Omega$, each with its own extension domain. For a system $\Phi$ in this class, the projections $T^\Phi_\iii$ of infinite words are those supplied by \cref{lem:dual} applied to $\Phi$. On this class $d_2$ is only a pseudometric, two systems with the same restrictions to $\ol{\Omega}$ being at distance zero, and \cref{eq:cont1} is to be read as an estimate between tuples; the same reading applies to the distances in \cref{lem:entire} and in the proof of \cref{thm:C1denseS}, where only numerical bounds, convergence, and the triangle inequality are used, all of which hold for a pseudometric.

\begin{proposition} \label{prop:cont1}
  Let $\Phi=(\fii_i)_{i \in \II}$ be a conformal iterated function system on $\Omega \subset \R$. Then there are $r_0>0$ and $C<\infty$ such that
  \begin{equation} \label{eq:cont1}
    \|T^\Phi_\iii-T^\Psi_\iii\|\le Cd_2(\Phi,\Psi)
  \end{equation}
  for all $\iii\in\II^\N\cup\II^*$ and every conformal iterated function system $\Psi$ on $\Omega$ with $d_2(\Phi,\Psi)<r_0$. In particular, if $\Delta_S(\Phi)>0$, then there is $r_1\in(0,r_0]$ such that 
  \begin{equation*}
    \Delta_S(\Psi) > 0
  \end{equation*}
  for all conformal iterated function systems $\Psi$ on $\Omega$ with $d_2(\Phi,\Psi)<r_1$, the same implication holds with $\Delta_T$ in place of $\Delta_S$, and the sets $\mathcal G_S$ and $\mathcal G_T$ are open in $(\mathcal S,d_2)$.
\end{proposition}

\begin{proof}
  Let $\cmax < \lambda < 1$ and $m=\min_{i \in \II}\inf_{\ol{\Omega}}|\fii_i'|>0$. Choose $r_0>0$ so small that every conformal iterated function system $\Psi=(\psi_i)_{i \in \II}$ on $\Omega$ with $d_2(\Phi,\Psi)<r_0$ has $|\psi_i'|\le \lambda$ and $|\psi_i'|\ge m/2$ on $\ol{\Omega}$; below $C$ denotes a finite constant depending only on $\Phi$, whose value may change. Fix such a system $\Psi$. For $\iii=\varnothing$, both projections vanish and the claim is immediate. Let $\iii$ be a nonempty finite or infinite word. Write $T^\Phi_\iii=\sum_k A_k^\Phi$ and $T^\Psi_\iii=\sum_k A_k^\Psi$, using the finite sums of \cref{eq:Tword} for finite words and the uniformly convergent series of \cref{lem:dual} for infinite words, where $A_k^\Phi=(f^\Phi_{\iii|_{k-1}})'(T^\Phi_{i_k}\circ f^\Phi_{\iii|_{k-1}})$ and $A_k^\Psi$ is defined analogously. All partial orbits stay in $\ol{\Omega}$.

  Fix $x\in\ol{\Omega}$ and an admissible $k$. For $l\in\{1,\ldots,k\}$, put $u_l=f^\Phi_{\iii|_{l-1}}(x)$ and $v_l=f^\Psi_{\iii|_{l-1}}(x)$, with the empty prefix interpreted as the identity, and set $u=u_k$ and $v=v_k$. Thus $u_1=v_1=x$, and, for $l\in\{1,\ldots,k-1\}$, the mean value theorem on the interval $\ol{\Omega}$ gives $|u_{l+1}-v_{l+1}| \le |\fii_{i_l}(u_l)-\psi_{i_l}(u_l)|+|\psi_{i_l}(u_l)-\psi_{i_l}(v_l)| \le d_2(\Phi,\Psi) + \lambda|u_l-v_l|$. Induction yields $|u_l-v_l|\le(1-\lambda)^{-1}d_2(\Phi,\Psi)$ for all $l\in\{1,\ldots,k\}$: the case $l=1$ is $0\le(1-\lambda)^{-1}d_2$, and if the bound holds at $l$ then $|u_{l+1}-v_{l+1}|\le \lambda(1-\lambda)^{-1}d_2+d_2=(1-\lambda)^{-1}d_2$. In particular,
  \begin{equation*}
    |u-v|\le(1-\lambda)^{-1}d_2(\Phi,\Psi).
  \end{equation*}
  The derivative cocycles are $(f^\Phi_{\iii|_{k-1}})'(x)=\prod_{l<k}\fii_{i_l}'(u_l)$ and $(f^\Psi_{\iii|_{k-1}})'(x)=\prod_{l<k}\psi_{i_l}'(v_l)$, both equal $1$ when $k=1$. When $k\ge2$, replacing one factor at a time produces
  \begin{equation*}
    (f^\Phi_{\iii|_{k-1}})'(x)-(f^\Psi_{\iii|_{k-1}})'(x) = \sum_{j=1}^{k-1}\Biggl(\prod_{l=1}^{j-1}\psi_{i_l}'(v_l)\Biggr)(\fii_{i_j}'(u_j)-\psi_{i_j}'(v_j))\Biggl(\prod_{l=j+1}^{k-1}\fii_{i_l}'(u_l)\Biggr),
  \end{equation*}
  empty products again being $1$. The generators of $\Phi$ are real-analytic on a neighborhood of the compact set $\ol{\Omega}$, so each $\|\fii_i''\|<\infty$. Using $|\fii_{i_l}'(u_l)-\psi_{i_l}'(v_l)|\le\|\fii_{i_l}''\||u_l-v_l|+d_2(\Phi,\Psi)\le Cd_2(\Phi,\Psi)$ we see that the two cocycles differ by 
  \begin{equation} \label{eq:cocycle-diff}
    |(f^\Phi_{\iii|_{k-1}})'(x)-(f^\Psi_{\iii|_{k-1}})'(x)| \le C(k-1)\lambda^{k-2}d_2(\Phi,\Psi)\le Ck\lambda^{k-1}d_2(\Phi,\Psi).
  \end{equation}
  The real-analyticity of the fixed generators and compactness of $\ol{\Omega}$ give $\|(T^\Phi_i)'\|<\infty$. Since the derivative lower bounds and \cref{eq:d2metric} give
  \begin{align*}
    \|T^\Phi_i-T^\Psi_i\| &= \biggl\|\frac{\fii_i''}{\fii_i'}-\frac{\psi_i''}{\psi_i'}\biggr\| \le \biggl\|\frac{\fii_i''(\psi_i'-\fii_i')}{\fii_i'\psi_i'}\biggr\| + \biggl\|\frac{\fii_i''-\psi_i''}{\psi_i'}\biggr\| \\
    &\le \frac{2\|\fii_i''\|\|\fii_i'-\psi_i'\|}{m^2}+\frac{2\|\fii_i''-\psi_i''\|}{m}\le Cd_2(\Phi,\Psi).
  \end{align*}
  Therefore $|T^\Phi_ i(u)-T^\Psi_i(v)|\le\|(T^\Phi_i)'\||u-v|+\|T^\Phi_i-T^\Psi_i\|\le Cd_2(\Phi,\Psi)$, and also $\|T^\Psi_i\|\le\|T^\Phi_i\|+Cd_2(\Phi,\Psi)\le C$. 
  
  Since $A_k^\Phi(x)=(f^\Phi_{\iii|_{k-1}})'(x)T^\Phi_{i_k}(u)$ and $A_k^\Psi(x)=(f^\Psi_{\iii|_{k-1}})'(x)T^\Psi_{i_k}(v)$. Splitting the difference as $(f^\Phi_{\iii|_{k-1}})'(x)(T^\Phi_{i_k}(u)-T^\Psi_{i_k}(v))+((f^\Phi_{\iii|_{k-1}})'(x)-(f^\Psi_{\iii|_{k-1}})'(x))T^\Psi_{i_k}(v)$ and using $|(f^\Phi_{\iii|_{k-1}})'(x)|\le\cmax^{k-1}\le \lambda^{k-1}$, together with \cref{eq:cocycle-diff}, therefore gives
  \begin{equation*}
    |A_k^\Phi(x)-A_k^\Psi(x)| \le \lambda^{k-1}Cd_2(\Phi,\Psi)+Ck\lambda^{k-1}d_2(\Phi,\Psi)C \le Ck\lambda^{k-1}d_2(\Phi,\Psi)
  \end{equation*}
  on $\ol{\Omega}$. Taking the supremum and summing over the admissible $k$, with a finite sum for finite words, yields \cref{eq:cont1} since $\sum_{k\ge1}k\lambda^{k-1}=(1-\lambda)^{-2}$.

  Finally, let $\Delta_S(\Phi)>0$ and suppose that no such $r_1$ exists. Since $\Delta_S$ is nonnegative, there are then conformal iterated function systems $\Psi_n$ on $\Omega$ with $d_2(\Phi,\Psi_n)<\min\{r_0,1/n\}$ and $\Delta_S(\Psi_n)=0$, so that $\Psi_n\to\Phi$ in $d_2$. The infimum defining $\Delta_S(\Psi_n)$ is attained, as in the proof of \cref{lem:gap}, so there are $(\iii_n,\jjj_n)\in Z$ with $S^{\Psi_n}_{\iii_n}\equiv S^{\Psi_n}_{\jjj_n}$, and \cref{eq:cont1} gives the hypothesis \cref{eq:Tparametercontinuity} of \cref{lem:Sclosed}, which yields a pair $(\iii,\jjj)\in Z$ with $S^\Phi_\iii\equiv S^\Phi_\jjj$, contradicting $\Delta_S(\Phi)>0$. For $\Delta_T$ no closedness step is needed, the pre-Schwarzian being assembled from a second derivative, which $d_2$ controls through \cref{eq:cont1} directly. Let $\delta=\Delta_T(\Phi)>0$, take the constant $C$ of \cref{eq:cont1} positive, and let $\Psi$ be a conformal iterated function system on $\Omega$ with $d_2(\Phi,\Psi)<\min\{r_0,\delta/(3C)\}$. For every $(\iii,\jjj)\in Z$ the triangle inequality and two applications of \cref{eq:cont1} give
  \begin{align*}
    \|T^\Psi_\iii-T^\Psi_\jjj\| &\ge \|T^\Phi_\iii-T^\Phi_\jjj\| - \|T^\Phi_\iii-T^\Psi_\iii\| - \|T^\Phi_\jjj-T^\Psi_\jjj\| \\
    &\ge \|T^\Phi_\iii-T^\Phi_\jjj\| - 2Cd_2(\Phi,\Psi) > \delta/3,
  \end{align*}
  the last inequality because $\|T^\Phi_\iii-T^\Phi_\jjj\|\ge\delta$ over $Z$ and $2Cd_2(\Phi,\Psi)<2\delta/3$. Taking the infimum over $Z$ yields $\Delta_T(\Psi)\ge\delta/3>0$. Specializing to members of $\mathcal S$ proves the last claim.
\end{proof}

\subsection{The cylinder criterion} \label{sec:gen-line-cyl}

The gap condition quantifies over infinite words, and the density argument can only control finitely many. This subsection removes that gap: the projections of all words with a given prefix lie in an interval attached to that prefix, so a separation checked on the words of one fixed length already separates every pair.

For the density proof, write
\begin{equation*}
  G_ih=(\fii_i')^2(h\circ\fii_i)+S_i
\end{equation*}
for the weight-two dual maps \cref{eq:dualdefG} read on the line, where $D\fii_i$ is the scalar $\fii_i'$ and a symmetric matrix field is a member of $\mathcal C(\ol{\Omega})$. Their linear parts preserve order and have norms at most $\cmax^2$. Choose $L>0$ so large that
\begin{equation} \label{eq:lineL}
  \cmax^2L+\max_{i \in \II}\|S_{\fii_i}\|<L.
\end{equation}
Then every $G_i$ maps the closed order interval $\{h : -L\le h\le L\}$ into its interior. For a word $\iii$ and a point $x$, the values $(G_\iii h)(x)$ with $-L\le h\le L$ form the interval
\begin{equation} \label{eq:linecylinder}
  I_\iii(x)=[S_\iii(x)-L(f_\iii'(x))^2,S_\iii(x)+L(f_\iii'(x))^2].
\end{equation}
Whatever the continuation of a word, its Schwarzian projection lies in the interval \cref{eq:linecylinder} of each of its prefixes, since the projection of the tail lies in the invariant order interval. Disjointness of two such intervals at a single point therefore separates every pair of words extending those prefixes, and the infimum over infinitely many pairs defining $\Delta_S$ is reduced to finitely many conditions on words of bounded length.

\begin{lemma} \label{lem:linecylinders}
  Suppose that, for some $n\ge1$, every pair $\iii,\jjj\in\II^n$ with $i_1 \ne j_1$ has a point $x_{\iii,\jjj}\in\ol{\Omega}$ such that $I_\iii(x_{\iii,\jjj})$ and $I_\jjj(x_{\iii,\jjj})$ are disjoint. Suppose also that $S_\iii\not\equiv S_\jjj$ for all $m \in \{1,\ldots,n-1\}$ and all pairs $\iii,\jjj\in\II^m$ with $i_1 \ne j_1$. Then $\Delta_S(\Phi)>0$ and, in particular, $\Phi\in\mathcal G_S$ whenever $\Phi\in\mathcal S$.
\end{lemma}

\begin{proof}
  The distances between the finitely many level-$n$ interval pairs at their respective witness points have a positive minimum, denoted by $\delta_n$. If $\iii,\jjj\in\II^\N\cup\II^*$ satisfy $|\iii|=|\jjj|\ge n$ and $i_1 \ne j_1$, write $\iii = \uuu\iii'$ and $\jjj = \vvv\jjj'$ with $|\uuu| = |\vvv| = n$. The projections $S_{\iii'}$ and $S_{\jjj'}$ lie in the invariant order interval, with the empty tail interpreted as $0$, so \cref{lem:dual2} and, in particular, \cref{eq:peeling2} gives $S_\iii(x_{\uuu,\vvv}) \in I_\uuu(x_{\uuu,\vvv})$ and $S_\jjj(x_{\uuu,\vvv}) \in I_\vvv(x_{\uuu,\vvv})$. Consequently $|S_\iii(x_{\uuu,\vvv})-S_\jjj(x_{\uuu,\vvv})| \ge \delta_n$, and hence $\|S_\iii-S_\jjj\| \ge \delta_n$. If $n = 1$, this covers every pair in $Z$. If $n \ge 2$, the finitely many shorter pairs have positive norms by hypothesis and hence a positive minimum, denoted by $\delta_n'$. Therefore $\Delta_S(\Phi) \ge \delta_n > 0$ when $n = 1$, while $\Delta_S(\Phi) \ge \min\{\delta_n,\delta_n'\} > 0$ when $n \ge 2$. The final assertion follows from the definition of $\mathcal G_S$.
\end{proof}

\subsection{Localized real-analytic perturbations} \label{sec:gen-line-loc}

The distinctive feature of the line is that a real-analytic generator can be perturbed to prescribe its Schwarzian at finitely many points while leaving it untouched at finitely many others. \Cref{lem:localizedS} supplies such a perturbation, and the three lemmas after it prepare the ground on which it acts: \cref{lem:identityfree} removes exact overlaps, \cref{lem:lineorbits} selects base points with disjoint orbits, and \cref{lem:entire} returns the perturbed system to $\mathcal S$. No planar analogue of \cref{lem:localizedS} exists, as \cref{sec:gen-plane} explains.

The next lemma is the localized perturbation behind the density proof. The Schwarzian depends only on the first three derivatives of a map, linearly on the third, so a real-analytic correction with a zero of order three at a point moves the Schwarzian there by a prescribed amount, while a zero of order four moves nothing. Summing one correction for each target therefore gives prescribed shifts at finitely many target points, no change at finitely many guard points, and an arbitrarily small cost in the $\mathcal C^2$-distance.

\begin{lemma} \label{lem:localizedS}
  Let $f$ be real-analytic on a neighborhood of a compact interval $J$ with $\inf_J|f'|>0$. Let $\mathcal X,\mathcal Y\subseteq\inter(J)$ be finite disjoint sets, let $(\delta_y)_{y\in\mathcal Y}$ be real numbers, and set $M_\delta = \max(\{0\}\cup\{|\delta_y| : y \in \mathcal Y\})$. Then for every $\eta>0$ there is a real-analytic function $g$ on a neighborhood of $J$ such that
  \begin{equation} \label{eq:localizedStargets}
    g^{(k)}(y)=f^{(k)}(y) \qquad\text{and}\qquad S_g(y)-S_f(y)=\delta_y
  \end{equation}
  for all $y\in\mathcal Y$ and $k \in \{0,1,2\}$, while
  \begin{equation} \label{eq:localizedSguards}
    g^{(k)}(z)=f^{(k)}(z)
  \end{equation}
  for all $z\in\mathcal X$ and $k \in \{0,1,2,3\}$. Moreover $\|g-f\|_{\mathcal C^2(J)}<\eta$, and
  \begin{equation} \label{eq:localizedSbound}
    \|S_g-S_f\|_{\mathcal C(J)}\le C_fM_\delta+\eta,
  \end{equation}
  where $C_f$ depends on $f$ and $J$ but not on the finite sets. In particular, if $0<|f'|<1$ on $J$ and $f(J)$ lies at positive distance from $\R\setminus\inter(J)$, then $g$ may be chosen with the same two properties.
\end{lemma}

\begin{proof}
  If $\mathcal Y$ is empty, take $g=f$. Then \cref{eq:localizedStargets} is vacuous, \cref{eq:localizedSguards} holds, $\|g-f\|_{\mathcal C^2(J)}=0<\eta$, and $M_\delta=0$, so \cref{eq:localizedSbound} holds. The final assertion is immediate. Henceforth assume $\mathcal Y$ is nonempty. For $y\in\mathcal Y$ define
  \begin{equation*}
    Q_y(x)=\prod_{p\in(\mathcal Y\setminus\{y\})\cup\mathcal X}\biggl(\frac{x-p}{y-p}\biggr)^4
  \end{equation*}
  and, for parameters $\sigma_y>0$, put
  \begin{equation} \label{eq:Hy}
    H_y(x)=\frac{f'(y)\delta_y}{6}(x-y)^3Q_y(x)e^{-(x-y)^2/\sigma_y^2},\qquad H=\sum_{y\in\mathcal Y}H_y,\qquad g=f+H.
  \end{equation}
  At its center the summand $H_y$ has a cubic zero, and both $Q_y$ and $e^{-(x-y)^2/\sigma_y^2}$ equal $1$ there. In the third derivative only the term carrying all three derivatives on $(x-y)^3$ survives at $x=y$, and the $3!$ it contributes cancels the denominator $6$, so $H_y'''(y)=f'(y)\delta_y$. Every summand $H_{y'}$ has a fourth-order zero at each point of $(\mathcal Y\setminus\{y'\})\cup\mathcal X$, since each factor of $Q_{y'}$ enters to the fourth power while the remaining factors of $H_{y'}$ do not vanish there. Thus $H^{(k)}(y)=0$ for $k \in \{0,1,2\}$, $H'''(y)=f'(y)\delta_y$, and $H^{(k)}(z)=0$ for $k \in \{0,1,2,3\}$. Differentiating the quotient gives $(f''/f')'=f'''/f'-(f''/f')^2$, so the classical Schwarzian $S_f=(f''/f')'-\tfrac12(f''/f')^2$ of \cref{eq:Sdef} is
  \begin{equation*}
    S_f = \frac{f'''}{f'}-\biggl(\frac{f''}{f'}\biggr)^2-\frac12\biggl(\frac{f''}{f'}\biggr)^2 = \frac{f'''}{f'}-\frac32\biggl(\frac{f''}{f'}\biggr)^2,
  \end{equation*}
  which at a target $y$ leaves the term $-\tfrac32(f''/f')^2$ unchanged, since $g$ and $f$ agree there to second order, and gives $S_g(y)-S_f(y)=(g'''(y)-f'''(y))/f'(y)=\delta_y$, the derivative $g'(y)=f'(y)$ being nonzero so that $S_g(y)$ is defined. This is \cref{eq:localizedStargets}, while the vanishing of $H$ and its first three derivatives at the guard points is \cref{eq:localizedSguards}.

  It remains to choose the widths. Put $m = \inf_J|f'|$. For $u\in\mathcal C^3(J)$, direct subtraction in the quotient formula for the Schwarzian gives
  \begin{equation*}
    S_{f+u}-S_f = \frac{u'''}{f'+u'} - \frac{f'''u'}{f'(f'+u')} - \frac32\biggl(\biggl(\frac{f''+u''}{f'+u'}\biggr)^2 - \biggl(\frac{f''}{f'}\biggr)^2\biggr).
  \end{equation*}
  If $\|u\|_{\mathcal C^2(J)}<\min\{1,m/2\}$, then $|f'+u'|\ge m/2$, and factoring the difference of squares shows that there is a constant $A_f\ge1$, depending only on $f$ and $J$, for which
  \begin{align*}
    \|S_{f+u}-S_f\|_{\mathcal C(J)} &\le \frac{2\|u'''\|_{\mathcal C(J)}}{m} + \frac{2}{m^2}\|f'''\|_{\mathcal C(J)}\|u\|_{\mathcal C^2(J)} \\
    &\qquad\qquad+ \frac{3(m+\|f''\|_{\mathcal C(J)})(3\|f''\|_{\mathcal C(J)}+2)\|u\|_{\mathcal C^2(J)}}{m^3} \\
    &\le A_f(\|u\|_{\mathcal C^2(J)}+\|u'''\|_{\mathcal C(J)}).
  \end{align*}
  Choose $\theta>0$ such that
  \begin{equation*}
    \theta<\min\biggl\{\eta,1,\frac{m}{2},\frac{\eta}{2A_f}\biggr\}.
  \end{equation*}
  Under the additional hypotheses in the final assertion, require also that $\theta<1-\sup_J|f'|$ and $\theta<\dist(f(J),\R\setminus\inter(J))$. Continuity of $f'$ on the compact interval $J$ makes $\sup_J|f'|<1$ whenever $0<|f'|<1$ on $J$, so a positive $\theta$ remains available.

  Write $B_k$ for the $k$-th derivative of $t\mapsto t^3e^{-t^2}$, a polynomial times $e^{-t^2}$ and hence bounded on $\R$ with limit zero at infinity. Since $b_\sigma(x)=x^3e^{-x^2/\sigma^2}=\sigma^3B_0(x/\sigma)$ and each differentiation in $x$ contributes a factor $\sigma^{-1}$, the derivatives of $b_\sigma$ satisfy $b_\sigma^{(k)}(x)=\sigma^{3-k}B_k(x/\sigma)$ for $k \in \{0,1,2,3\}$. Set $N = \max\{1,\#\mathcal Y\}$. Choose pairwise disjoint neighborhoods $U_y$ of the target points that avoid all other target and guard points and are small enough that $|Q_y|\le2$ on $U_y$. The product rule applied to $H_y(x)=\frac{1}{6}f'(y)\delta_yQ_y(x)b_{\sigma_y}(x-y)$ gives
  \begin{equation*}
    H_y^{(k)}(x)=\frac{f'(y)\delta_y}{6}\sum_{j=0}^k\binom{k}{j}Q_y^{(j)}(x)\sigma_y^{3-k+j}B_{k-j}\biggl(\frac{x-y}{\sigma_y}\biggr)
  \end{equation*}
  for all $k \in \{0,1,2,3\}$. On $U_y$, every term with $k\le2$ tends uniformly to zero as $\sigma_y\to0$. For $k=3$, the term with $j=0$ is bounded by $C_B|f'(y)\delta_y|$ with the absolute constant $C_B=\tfrac13\sup_\R|B_3|$, while every term with $j\ge1$ tends uniformly to zero. Since the derivatives of $Q_y$ are bounded on $U_y$, the widths may therefore be chosen so that
  \begin{equation*}
    \sup_{U_y}\max_{k \in \{0,1,2\}}|H_y^{(k)}|<\frac{\theta}{3N} \qquad\text{and}\qquad \sup_{U_y}|H_y'''|\le C_B|f'(y)\delta_y|+\frac{\theta}{N}.
  \end{equation*}
  Outside $U_y$, the argument of each $B_k$ tends uniformly to infinity as $\sigma_y\to0$. Shrinking the widths further makes the absolute values of $H_y$ and its first three derivatives smaller than $\theta/(3N)$ there. Consequently each summand has $\mathcal C^2(J)$-norm smaller than $\theta/N$. Since the neighborhoods $U_y$ are disjoint, at each point at most one third derivative has a central contribution. If $x\in U_y$, then the remaining third derivatives contribute less than $(N-1)\theta/(3N)$, and together with the error $\theta/N$ from $H_y'''$ this is at most $\theta$; outside $\bigcup_{y\in\mathcal Y}U_y$, the total contribution is less than $\theta/3$. Hence
  \begin{equation*}
    \|H\|_{\mathcal C^2(J)}<\theta \qquad\text{and}\qquad \|H'''\|_{\mathcal C(J)}\le C_B\|f'\|_{\mathcal C(J)}M_\delta+\theta.
  \end{equation*}
  Since $\theta<\min\{1,m/2\}$, the derivative $g'=f'+H'$ does not vanish on $J$, and the estimate above applies with $u=H$. Setting $C_f = A_fC_B\|f'\|_{\mathcal C(J)}$ gives
  \begin{equation*}
    \|S_g-S_f\|_{\mathcal C(J)}\le A_f(\|H\|_{\mathcal C^2(J)}+\|H'''\|_{\mathcal C(J)})\le C_fM_\delta+2A_f\theta\le C_fM_\delta+\eta,
  \end{equation*}
  which proves \cref{eq:localizedSbound}, while $\|g-f\|_{\mathcal C^2(J)}<\theta<\eta$. Each summand $H_y$ is a polynomial times $e^{-(x-y)^2/\sigma_y^2}$, so the correction $H$ is the restriction to $\R$ of an entire function and $g$ is real-analytic on the same neighborhood of $J$ as $f$. Under the additional hypotheses, the further restrictions on $\theta$ give $0<|g'|<1$ on $J$ and, writing $\rho=\dist(f(J),\R\setminus\inter(J))$, also $\dist(g(J),\R\setminus\inter(J))\ge\rho-\|H\|_{\mathcal C(J)}>\rho-\theta>0$.
\end{proof}

The density argument begins by removing exact overlaps. That the systems without them are dense has nothing to do with separation: any system is joined inside $\mathcal S$ to an affine one whose generator images are pairwise disjoint, and along that segment an identity between two fixed distinct words is an analytic condition on the parameter which fails at the affine endpoint.

\begin{lemma} \label{lem:identityfree}
  The systems $\Phi\in\mathcal S$ with $\fii_\uuu\not\equiv\fii_\vvv$ on $\ol{\Omega} \subset \R$ for all distinct $\uuu,\vvv\in\II^*$ are dense in $(\mathcal S,d_2)$.
\end{lemma}

\begin{proof}
  Fix $\Phi=(\fii_i)_{i \in \II}\in\mathcal S$ and $\eta>0$. Each $\fii_i'$ is continuous and vanishes nowhere on the interval $\Omega'$, so it has a constant sign there. Choose pairwise disjoint nondegenerate compact intervals $K_i\subseteq\Omega$, $i \in \II$, and affine bijections $\alpha_i \colon \ol{\Omega}\to K_i$ whose derivative has the same sign as $\fii_i'$, each extended affinely to $\R$; every $K_i$ is a proper subinterval of $\ol{\Omega}$, so $|\alpha_i'|<1$. Write $A=(\alpha_i)_{i \in \II}$, a member of $\mathcal S$ with extension domain $\Omega'$, and write $\alpha_\uuu$ for its forward compositions.

  The system $A$ has no identities. If exactly one of two words is empty, equality is impossible because every nonempty composition maps $\ol{\Omega}$ into a proper subinterval. Every nonempty word $\uuu$ has $\alpha_\uuu(\ol{\Omega})\subseteq K_{u_1}$, the inner compositions mapping $\ol{\Omega}$ into itself, and the intervals $K_i$ are pairwise disjoint, so $\alpha_\uuu\equiv\alpha_\vvv$ forces $u_1=v_1$ for two nonempty words. Canceling the injective $\alpha_{u_1}$ and iterating either exhausts both words at once, giving $\uuu=\vvv$, or exhausts one of them and equates the identity of $\ol{\Omega}$ with a composition along a nonempty word, which is impossible because such a composition maps $\ol{\Omega}$ into one of the proper subintervals $K_i$.

  For $s\in[0,1]$ put $\fii_i^{(s)}=(1-s)\fii_i+s\alpha_i$ and $\Phi_s=(\fii_i^{(s)})_{i \in \II}$, so that $\Phi_0=\Phi$ and $\Phi_1=A$. The derivatives $\fii_i'$ and $\alpha_i'$ share their sign on $\Omega'$, so $(\fii_i^{(s)})'$ carries that sign and has modulus $(1-s)|\fii_i'|+s|\alpha_i'|$ there; it therefore vanishes nowhere on $\Omega'$, and on $\ol{\Omega}$ its modulus is less than $1$, both $|\fii_i'|$ and $|\alpha_i'|$ being less than $1$ there. Since $\Omega$ is an interval containing both $\fii_i(\ol{\Omega})$ and $K_i$, convexity gives $\fii_i^{(s)}(\ol{\Omega})\subseteq\Omega$. Hence $\Phi_s\in\mathcal S$, with the extension domain $\Omega'$ of $\Phi$, for every $s\in[0,1]$.

  Fix distinct $\uuu,\vvv\in\II^*$. Since $A$ has no identities, there is $x_{\uuu,\vvv}\in\ol{\Omega}$ with $\alpha_\uuu(x_{\uuu,\vvv})\ne\alpha_\vvv(x_{\uuu,\vvv})$, and we set
  \begin{equation*}
    h_{\uuu,\vvv}(s) = \fii^{(s)}_\uuu(x_{\uuu,\vvv})-\fii^{(s)}_\vvv(x_{\uuu,\vvv}).
  \end{equation*}
  Only finitely many maps are composed for this fixed pair, and each $(s,y)\mapsto\fii_i^{(s)}(y)$ is real-analytic on $\R\times\Omega'$. For every $s\in[0,1]$, all intermediate points in both finite compositions lie in $\ol{\Omega}$. Since $\ol{\Omega}$ is at positive distance from $\R\setminus\Omega'$, the finite number of intermediate maps and the compactness of $[0,1]$ imply by continuity that these points remain in $\Omega'$ for $s$ in an open interval $J$ containing $[0,1]$. The function $h_{\uuu,\vvv}$ is therefore real-analytic on $J$, and $h_{\uuu,\vvv}(1)\ne0$, so it does not vanish identically and its zero set in $[0,1]$ is finite.

  There are countably many pairs of distinct finite words, so the union $E$ of these zero sets over all such pairs is countable, and there are $s\notin E$ arbitrarily close to $0$. For such an $s$ no identity $\fii^{(s)}_\uuu\equiv\fii^{(s)}_\vvv$ between distinct finite words can hold, since it would give $h_{\uuu,\vvv}(s)=0$. Finally $\fii_i^{(s)}-\fii_i=s(\alpha_i-\fii_i)$ gives $d_2(\Phi_s,\Phi)=sd_2(A,\Phi)$, so an $s\notin E$ with $sd_2(A,\Phi)<\eta$ produces a system without identities within $\eta$ of $\Phi$, as claimed.
\end{proof}

The points at which the localized perturbation is applied are selected by the induction of \cite[Lemma~4.1]{BaranyKolossvaryTroscheit}, which we restate in the form that the proof of \cref{thm:C1denseS} requires. For a finite word $\iii$ and $x\in\ol{\Omega}$ write $\mathcal O_\iii(x)=\{f_{\iii|_k}(x) : k \in \{0,\ldots,|\iii|\}\}$ for the finite orbit of $x$ along the prefixes of $\iii$, where $f_{\iii|_0}$ is the identity. 

\begin{lemma} \label{lem:lineorbits}
  Let $\Phi$ be a conformal iterated function system on $\Omega \subset \R$ such that $f_\uuu\not\equiv f_\vvv$ for all distinct $\uuu,\vvv \in \II^*$, let $\mathcal P$ be a finite set of pairs $(\iii,\jjj)$ of nonempty finite words satisfying $i_1\ne j_1$, and for each $(\iii,\jjj)\in\mathcal P$ let $U_{\iii,\jjj}\subseteq\Omega$ be a nonempty open set. Then there are points $x_{\iii,\jjj}\in U_{\iii,\jjj}$, one for each $(\iii,\jjj)\in\mathcal P$, such that
  \begin{enumerate}
    \item\label{it:orbits-distinct} the points $f_{\iii|_k}(x_{\iii,\jjj})$, $k \in \{0,\ldots,|\iii|\}$, are pairwise distinct, as are the points $f_{\jjj|_k}(x_{\iii,\jjj})$, $k \in \{0,\ldots,|\jjj|\}$, and $\mathcal O_\iii(x_{\iii,\jjj})\cap\mathcal O_\jjj(x_{\iii,\jjj})=\{x_{\iii,\jjj}\}$,
    \item\label{it:orbits-disjoint} $(\mathcal O_\iii(x_{\iii,\jjj})\cup\mathcal O_\jjj(x_{\iii,\jjj}))\cap(\mathcal O_\kkk(x_{\kkk,\lll})\cup\mathcal O_\lll(x_{\kkk,\lll}))=\varnothing$ for all distinct $(\iii,\jjj),(\kkk,\lll)\in\mathcal P$.
  \end{enumerate}
\end{lemma}

\begin{proof}
  This is the induction of \cite[Lemma~4.1]{BaranyKolossvaryTroscheit}; we give the argument because the common-length restriction there is unnecessary and the prescribed open set here is arbitrary. Order the pairs in $\mathcal P$, suppose that the points for the first $r$ pairs have been chosen, and let $Y_r$ be the union of their orbit sets, with $Y_0=\varnothing$. For the next pair $(\iii,\jjj)$, let $\mathcal F_{\iii,\jjj}$ consist of the maps $f_{\iii|_k}$, $k \in \{0,\ldots,|\iii|\}$, and $f_{\jjj|_\ell}$, $\ell \in \{0,\ldots,|\jjj|\}$, with the two copies of the identity identified. The maps in $\mathcal F_{\iii,\jjj}$ are pairwise distinct. Indeed, within either prefix list the nonempty words have different lengths, while nonempty prefixes from the two lists have different first letters. The absence of exact overlaps therefore distinguishes the corresponding nonempty compositions, and no such composition is the identity because its derivative has absolute value at most $\cmax^m<1$, where $m$ is its positive length.

  Every map in $\mathcal F_{\iii,\jjj}$ is strictly monotone, so the preimage under it of the finite set $Y_r$ is finite. For distinct $h,g\in\mathcal F_{\iii,\jjj}$, the function $h-g$ is nonzero and real-analytic on a common neighborhood of the compact interval $\ol{\Omega}$. Its zero set in $\ol{\Omega}$ is therefore finite. Thus only finitely many points of $U_{\iii,\jjj}$ either map into $Y_r$ under a member of $\mathcal F_{\iii,\jjj}$ or make two distinct members agree. Choose $x_{\iii,\jjj}$ outside this finite set. The values of the maps in $\mathcal F_{\iii,\jjj}$ at $x_{\iii,\jjj}$ are pairwise distinct and avoid $Y_r$; since the identity is the one map shared by the two prefix lists, this gives \cref{it:orbits-distinct,it:orbits-disjoint}. The induction completes the proof.
\end{proof}

The last ingredient returns a system to $\mathcal S$. The localized perturbation controls a generator only on the interval on which it is measured, and the derivative of the perturbed generator may vanish somewhere in the fixed $\Omega'$ even when it does not vanish near $\ol{\Omega}$. The following lemma replaces the generators, at an arbitrarily small cost in $d_2$, by restrictions to $\R$ of entire functions whose derivatives vanish nowhere in $\C$. Such a generator is real-analytic with nonvanishing derivative on all of $\R$, hence conformal on every extension domain $\Omega'$; the system even extends holomorphically to the whole plane, with nonvanishing derivatives there, although only its restriction to $\R$ is used below.

\begin{lemma} \label{lem:entire}
  Let $\Phi=(\fii_i)_{i \in \II}$ be a conformal iterated function system on $\Omega \subset \R$. For every $\delta>0$ there is a conformal iterated function system $\Psi=(\psi_i)_{i \in \II}$ on $\Omega$ with $d_2(\Phi,\Psi)<\delta$ whose generators are the restrictions to $\R$ of entire functions $\C\to\C$ having real Taylor coefficients and nowhere vanishing derivative in $\C$. In particular, the generators are real-analytic with nowhere vanishing derivative on $\R$, so $\Psi\in\mathcal S$ for every choice of the extension domain $\Omega'$.
\end{lemma}

\begin{proof}
  Fix $i\in\II$. The derivative $\fii_i'$ is continuous and nonvanishing on the interval $\ol{\Omega}$, hence of constant sign $\sigma\in\{-1,1\}$ there, and $u=\log(\sigma \fii_i')$ is continuously differentiable on $\ol{\Omega}$ with $u'=\fii_i''/\fii_i'$, $\fii_i'=\sigma e^{u}$, and $\fii_i''=\sigma u'e^{u}$. Put $K=\|u\|$, $M=\|u'\|$, and $D=\diam(\ol{\Omega})$, all norms being over $\ol{\Omega}$, and fix $x_1\in\Omega$. By the Stone--Weierstrass approximation theorem there is a real polynomial $q$ with $\eta=\|q-u'\|\le1/D$ as small as we please. Let $p(z)=u(x_1)+\int_{x_1}^{z}q(w)\dd w$ for $z\in\C$, a polynomial with real coefficients, and define the entire function
  \begin{equation*}
    \psi_i(z)=\fii_i(x_1)+\sigma\int_{x_1}^{z}e^{p(w)}\dd w,
  \end{equation*}
  the integral being along any path in $\C$, since the integrand is entire. Since $p$ has real coefficients and $x_1,\fii_i(x_1)\in\R$, the entire function $\psi_i$ has real Taylor coefficients at every real center and maps $\R$ into $\R$. Moreover, $\psi_i'=\sigma e^{p}$ vanishes nowhere in $\C$; in particular, the restriction of $\psi_i$ to $\R$ is real-analytic with nonvanishing derivative, that is, conformal, on every domain $\Omega'\subseteq\R$ containing $\ol{\Omega}$. Below $\psi_i$ denotes this restriction. On $\ol{\Omega}$, integrating $q-u'$ from $x_1$ gives $\|p-u\|\le D\eta\le1$, so $\|p\|\le K+1$, and the mean value theorem applied to the exponential gives $|e^{p}-e^{u}|\le e^{K+1}|p-u|$. Hence
  \begin{align*}
    \|\psi_i'-\fii_i'\| &= \|e^{p}-e^{u}\|\le e^{K+1}D\eta, \\
    \|\psi_i''-\fii_i''\| &= \|qe^{p}-u'e^{u}\|\le\|q-u'\|\|e^{p}\|+M\|e^{p}-e^{u}\|\le e^{K+1}(1+MD)\eta,
  \end{align*}
  and, since $\psi_i(x_1)=\fii_i(x_1)$, also $\|\psi_i-\fii_i\|\le D\|\psi_i'-\fii_i'\|\le e^{K+1}D^2\eta$. Each of the three norms is therefore at most a constant multiple of $\eta$, the constant depending only on $\fii_i$ and $\Omega$. Each $\fii_i(\ol{\Omega})$ is a compact subset of the open interval $\Omega$, so $\rho=\min_{i \in \II}\dist(\fii_i(\ol{\Omega}),\R\setminus\Omega)$ is positive, and $\cmax(\Phi)<1$. Since $\II$ is finite and the error $\eta$ can be made arbitrarily small for each $i$, choose the polynomials so that $d_2(\Phi,\Psi)<\min\{\delta,\rho,1-\cmax(\Phi)\}$. Then every $\psi_i$ sends $\ol{\Omega}$ into $\Omega$ and satisfies $|\psi_i'|<1$ on $\ol{\Omega}$, and the conformal extension to any $\Omega'$ is the restriction of the entire function. The final assertion holds because membership in $\mathcal S$ asks, beyond these two properties, only for conformality on $\Omega'$.
\end{proof}

\subsection{Density on the line} \label{sec:gen-line-dense}

The pieces assembled above now give the unconditional density of $\mathcal G_S$, and with it of $\mathcal G_T$. \Cref{rem:BKT} places the result beside the genericity theorem of B\'ar\'any, Kolossv\'ary, and Troscheit, whose open and dense set carries the plain condition rather than the condition modulo M\"obius maps. \Cref{rem:gencrit-necessity} applies it to the one-point criterion of \cref{prop:gencrit}: the contraction ceiling of \cref{rem:gencrit-range} is an artifact of word length one, and the criterion is not necessary for the separation it delivers.

We are now ready to prove the main result of this section. The argument starts from the system without exact overlaps that \cref{lem:identityfree} supplies, applies the perturbation of \cref{lem:localizedS} at the points that \cref{lem:lineorbits} selects, verifies the criterion of \cref{lem:linecylinders}, and returns to $\mathcal S$ through \cref{lem:entire}.

\begin{theorem} \label{thm:C1denseS}
  The sets $\mathcal G_S$ and $\mathcal G_T$ are dense in $(\mathcal S,d_2)$.
\end{theorem}

\begin{proof}
  Fix $\Phi_0 \in \mathcal S$ and $\eps>0$. The system produced by the localized perturbation below is a conformal iterated function system on $\Omega$ whose extension domain may be smaller than $\Omega'$; the final step returns to $\mathcal S$.

  We first make a small perturbation without exact overlaps. By \cref{lem:identityfree} there is a system $\Psi_0=(f_i)_{i \in \II}\in\mathcal S$ with $d_2(\Phi_0,\Psi_0)<\eps/4$ whose distinct finite words have distinct forward compositions. Here and below $f_\iii$ denotes the reversed composition \cref{eq:frev} of the generators of $\Psi_0$, so that $f_{\iii|_k}=f_{i_k}\circ f_{\iii|_{k-1}}$, the one-letter words returning the generators. Since reversal is a length-preserving involution of $\II^*$, it follows that $f_\uuu\not\equiv f_\vvv$ for all distinct $\uuu,\vvv\in\II^*$, which is the hypothesis of \cref{lem:lineorbits}.

  Put $c_0 = \max_{i \in \II}\|f_i'\|$ and $M_0 = \max_{i \in \II}\|S_{f_i}\|$. Choose $\lambda \in (c_0,1)$ and $L>0$ such that $\lambda^2L+M_0<L$, and put $\gamma = L-\lambda^2L-M_0>0$. For each generator let $C_{f_i}$ be the constant of \cref{lem:localizedS} for $f_i$ and $J=\ol{\Omega}$, and put $C = \max_{i \in \II}C_{f_i}$. Choose $\tau,\eta_0>0$ so that
  \begin{equation*}
    C\tau+\eta_0<\gamma \qquad\text{and}\qquad \eta_0<\min\{\lambda-c_0,\eps/4\}.
  \end{equation*}
  Then choose $n$ so large that
  \begin{equation} \label{eq:linewidth}
    2L\lambda^{2n}<\tau.
  \end{equation}
  For $m\ge1$ let $Z_m=Z\cap(\II^m\times\II^m)$, where $Z$ is defined in \cref{eq:Zdef}, and put $\ol{Z}_n=\bigcup_{m=1}^{n}Z_m$; the pairs are ordered, and $(\iii,\jjj)$ and $(\jjj,\iii)$ are treated independently. Write $S_\iii^{\Psi_0}$ and $I_\iii^{\Psi_0}$ for the Schwarzian projections and the intervals \cref{eq:linecylinder} of $\Psi_0$, where $L$ satisfies \cref{eq:lineL} for $\Psi_0$ since $c_0<\lambda$. Call a pair $(\iii,\jjj)\in Z_n$ good if $I_\iii^{\Psi_0}(x)\cap I_\jjj^{\Psi_0}(x)=\varnothing$ for some $x\in\ol{\Omega}$, and bad otherwise, and call a pair $(\iii,\jjj)\in Z_m$ with $m<n$ coincident if $S_\iii^{\Psi_0}\equiv S_\jjj^{\Psi_0}$ on $\ol{\Omega}$. For $(\iii,\jjj)\in\ol{Z}_n$ define
  \begin{equation*}
    U_{\iii,\jjj}=
    \begin{cases}
      \{x\in\Omega : I_\iii^{\Psi_0}(x)\cap I_\jjj^{\Psi_0}(x)=\varnothing\}, &\text{if } (\iii,\jjj)\in Z_n \text{ is good}, \\
      \{x\in\Omega : S_\iii^{\Psi_0}(x)\ne S_\jjj^{\Psi_0}(x)\}, &\text{if } (\iii,\jjj)\in Z_m \text{ with } m<n \text{ is not coincident}, \\
      \Omega, &\text{otherwise}.
    \end{cases}
  \end{equation*}
  Each $U_{\iii,\jjj}$ is open, the endpoints of the intervals \cref{eq:linecylinder} and the projections being continuous in $x$, and nonempty: the defining condition holds at some point of $\ol{\Omega}$ by the definition of a good or a noncoincident pair, and, being open in $\ol{\Omega}$, it holds at points of $\Omega$ as well. Let $x_{\iii,\jjj}\in U_{\iii,\jjj}$, $(\iii,\jjj)\in\ol{Z}_n$, be the points provided by \cref{lem:lineorbits}, whose hypothesis holds because $\Psi_0$ has no exact overlaps.

  Let $B\subseteq\ol{Z}_n$ consist of the bad pairs in $Z_n$ and the coincident pairs in $Z_m$, $m<n$. For $(\iii,\jjj)\in B$ with $x=x_{\iii,\jjj}$ define the shift
  \begin{equation*}
    \delta_{\iii,\jjj}=
    \begin{cases}
      \tau, &\text{if } S_\iii^{\Psi_0}(x)\ge S_\jjj^{\Psi_0}(x), \\
      -\tau, &\text{if } S_\iii^{\Psi_0}(x)<S_\jjj^{\Psi_0}(x).
    \end{cases}
  \end{equation*}
  Along the selected orbits the generator $f_i$ acts at the points $f_{\iii|_{k-1}}(x_{\iii,\jjj})$ with $i_k=i$ and at the points $f_{\jjj|_{k-1}}(x_{\iii,\jjj})$ with $j_k=i$; let $\mathcal W_i\subseteq\Omega$ be the set of all such points, over all $(\iii,\jjj)\in\ol{Z}_n$ and all admissible $k$. For $i\in\II$ let
  \begin{equation*}
    \mathcal Y_i=\{x_{\iii,\jjj} : (\iii,\jjj)\in B \text{ and } i_1=i\} \qquad\text{and}\qquad \mathcal X_i=\mathcal W_i\setminus\mathcal Y_i
  \end{equation*}
  be the targets and the guards of $f_i$, and for $y=x_{\iii,\jjj}\in\mathcal Y_i$ put $\delta_y=\delta_{\iii,\jjj}$; the shifts are well-defined because distinct pairs have distinct starting points by \cref{lem:lineorbits}\cref{it:orbits-disjoint}, and $\mathcal Y_i\subseteq\mathcal W_i$ because $f_{i_1}$ acts at $x_{\iii,\jjj}$. By \cref{lem:lineorbits}\cref{it:orbits-distinct,it:orbits-disjoint}, a starting point $x=x_{\iii,\jjj}$ occurs along the selected orbits only as the initial point of $\mathcal O_\iii(x)$ and of $\mathcal O_\jjj(x)$, where the acting generators are $f_{i_1}$ and $f_{j_1}$, and no other point of $\mathcal O_\iii(x)\cup\mathcal O_\jjj(x)$ is a starting point. Moreover $x\notin\mathcal Y_{j_1}$, since $x\in\mathcal Y_{j_1}$ would require $x=x_{\iii',\jjj'}$ for a pair $(\iii',\jjj')\in B$ with $i'_1=j_1$, whence $(\iii',\jjj')=(\iii,\jjj)$ and $i_1=j_1$. Explicitly, for $k\ge2$ the points $f_{\iii|_{k-1}}(x)$ and $f_{\jjj|_{k-1}}(x)$ lie in $\mathcal X_{i_k}$ and $\mathcal X_{j_k}$, the point $x$ lies in $\mathcal X_{j_1}$, and $x$ lies in $\mathcal Y_{i_1}$ if $(\iii,\jjj)\in B$ and in $\mathcal X_{i_1}$ otherwise.

  Apply \cref{lem:localizedS} to every generator $f_i$ of $\Psi_0$, with $J=\ol{\Omega}$, the guards $\mathcal X_i$, the targets $\mathcal Y_i$, the shifts $(\delta_y)_{y\in\mathcal Y_i}$, for which $M_\delta\le\tau$, and the tolerance $\eta_0$, and let $g_i$ be the function it provides. The hypotheses of the lemma hold: each $f_i$ is real-analytic on a neighborhood of $\ol{\Omega}$ with $\inf_{\ol{\Omega}}|f_i'|>0$, the sets $\mathcal Y_i$ and $\mathcal X_i$ are finite, disjoint, and contained in $\Omega$, the orbit points of a point of $\Omega$ staying in $\Omega$, and $0<|f_i'|<1$ on $\ol{\Omega}$ with $f_i(\ol{\Omega})$ at positive distance from $\R\setminus\Omega$. The perturbed generators $g_i$ therefore send $\ol{\Omega}$ into $\Omega$, satisfy $0<|g_i'|<1$ there, and are real-analytic on a neighborhood of $\ol{\Omega}$, the perturbations of \cref{lem:localizedS} being restrictions to $\R$ of entire functions; as $g_i'$ is continuous and nonvanishing on the compact $\ol{\Omega}$, it vanishes nowhere on a neighborhood of $\ol{\Omega}$, which we take as the extension domain. The perturbed tuple $\Psi_1=(g_i)_{i \in \II}$ is thus a conformal iterated function system on $\Omega$, with an extension domain that may be smaller than $\Omega'$; we write $g_\iii$ for the reversed compositions of its generators. We have
  \begin{equation*}
    d_2(\Psi_0,\Psi_1)<\eta_0<\eps/4 \qquad\text{and}\qquad d_2(\Phi_0,\Psi_1)<\eps/2.
  \end{equation*}
  Moreover $\max_{i \in \II}\|g_i'\|<c_0+\eta_0<\lambda$ and, by \cref{eq:localizedSbound},
  \begin{equation*}
    \max_{i \in \II}\|S_{g_i}-S_{f_i}\|\le C\tau+\eta_0<\gamma.
  \end{equation*}
  Hence $\lambda^2L+\max_{i \in \II}\|S_{g_i}\|<L$, so \cref{eq:lineL} holds for $\Psi_1$ with the same $L$. Write $S^{\Psi_1}_\iii$ and $I^{\Psi_1}_\iii$ for the Schwarzian projections and the intervals \cref{eq:linecylinder} of $\Psi_1$, with this $L$.

  Fix $(\iii,\jjj)\in\ol{Z}_n$ and put $x=x_{\iii,\jjj}$. Since $g_i$ and $f_i$ agree together with their first derivatives at every point of $\mathcal Y_i\cup\mathcal X_i=\mathcal W_i$, by \cref{eq:localizedStargets,eq:localizedSguards}, and since $f_{\iii|_{k-1}}(x)\in\mathcal W_{i_k}$, induction on $k$ along $g_{\iii|_k}=g_{i_k}\circ g_{\iii|_{k-1}}$ gives
  \begin{equation*}
    g_{\iii|_k}(x)=f_{\iii|_k}(x) \qquad\text{and}\qquad g_{\iii|_k}'(x)=f_{\iii|_k}'(x)
  \end{equation*}
  for all $k\in\{0,\ldots,|\iii|\}$, and likewise along $\jjj$. In particular $I^{\Psi_1}_\iii(x)$ and $I_\iii^{\Psi_0}(x)$ have the same radius $L(f_\iii'(x))^2$, and so do $I^{\Psi_1}_\jjj(x)$ and $I_\jjj^{\Psi_0}(x)$. The cocycle \cref{eq:Sword}, read on the line for $\Psi_0$ and for $\Psi_1$ and combined with the agreement just established, gives
  \begin{equation*}
    S^{\Psi_1}_\iii(x)-S_\iii^{\Psi_0}(x)=\sum_{k=1}^{|\iii|}(f_{\iii|_{k-1}}'(x))^2(S_{g_{i_k}}-S_{f_{i_k}})(f_{\iii|_{k-1}}(x)),
  \end{equation*}
  and likewise for $\jjj$. The $k$-th summand vanishes when $f_{\iii|_{k-1}}(x)\in\mathcal X_{i_k}$, by \cref{eq:localizedSguards}, and equals $\delta_{\iii,\jjj}$ when $k=1$ and $x\in\mathcal Y_{i_1}$, by \cref{eq:localizedStargets} and $f_{\iii|_0}'=1$. Hence, by the three cases recorded above, $S^{\Psi_1}_\jjj(x)=S_\jjj^{\Psi_0}(x)$ for every pair, $S^{\Psi_1}_\iii(x)=S_\iii^{\Psi_0}(x)$ if $(\iii,\jjj)\notin B$, and $S^{\Psi_1}_\iii(x)=S_\iii^{\Psi_0}(x)+\delta_{\iii,\jjj}$ if $(\iii,\jjj)\in B$.

  We verify the hypotheses of \cref{lem:linecylinders} for $\Psi_1$. Let $(\iii,\jjj)\in Z_n$ and $x=x_{\iii,\jjj}$. If the pair is good, then $I^{\Psi_1}_\iii(x)=I_\iii^{\Psi_0}(x)$ and $I^{\Psi_1}_\jjj(x)=I_\jjj^{\Psi_0}(x)$ are disjoint since $x\in U_{\iii,\jjj}$. If the pair is bad, then $I^{\Psi_1}_\iii(x)=I_\iii^{\Psi_0}(x)+\delta_{\iii,\jjj}$ and $I^{\Psi_1}_\jjj(x)=I_\jjj^{\Psi_0}(x)$. The radii of these intervals are at most $Lc_0^{2n}<L\lambda^{2n}$, so their sum is less than $\tau$ by \cref{eq:linewidth}. By the choice of sign, the distance of the new centers is $|S_\iii^{\Psi_0}(x)+\delta_{\iii,\jjj}-S_\jjj^{\Psi_0}(x)|=|S_\iii^{\Psi_0}(x)-S_\jjj^{\Psi_0}(x)|+\tau\ge\tau$, which exceeds the sum of the radii, so the translated intervals are disjoint. Let $(\iii,\jjj)\in Z_m$ with $m<n$ and $x=x_{\iii,\jjj}$. If the pair is not coincident, then $S^{\Psi_1}_\iii(x)-S^{\Psi_1}_\jjj(x)=S_\iii^{\Psi_0}(x)-S_\jjj^{\Psi_0}(x)\ne0$ since $x\in U_{\iii,\jjj}$, and if it is coincident, then $S^{\Psi_1}_\iii(x)-S^{\Psi_1}_\jjj(x)=\delta_{\iii,\jjj}=\tau\ne0$. Thus $S^{\Psi_1}_\iii\not\equiv S^{\Psi_1}_\jjj$ for every $(\iii,\jjj)\in Z_m$ with $m<n$, and \cref{lem:linecylinders} gives $\Delta_S(\Psi_1)>0$.

  The system $\Psi_1$ need not belong to $\mathcal S$, its generators having been certified conformal only near $\ol{\Omega}$. By \cref{prop:cont1} there is $r_1>0$ such that every conformal iterated function system $\Psi$ on $\Omega$ with $d_2(\Psi_1,\Psi)<r_1$ has $\Delta_S(\Psi)>0$, and \cref{lem:entire} provides such a $\Psi$ in $\mathcal S$ with $d_2(\Psi_1,\Psi)<\min\{r_1,\eps/2\}$. Then $\Psi \in \mathcal G_S$ and
  \begin{equation*}
    d_2(\Phi_0,\Psi)\le d_2(\Phi_0,\Psi_1)+d_2(\Psi_1,\Psi)<\eps/2+\eps/2=\eps,
  \end{equation*}
  which proves the density of $\mathcal G_S$. By \cref{lem:gap}, the inclusion \cref{eq:GSsubsetGT} gives $\mathcal G_S\subseteq\mathcal G_T$ when $d=1$, and a set that contains a dense set is dense, so $\mathcal G_T$ is dense as well.
\end{proof}

We are now ready to prove the first item of \cref{thm:genericity}.

\begin{proof}[Proof of \cref{thm:genericity}\cref{it:gen-line}]
  The sets $\mathcal G_S$ and $\mathcal G_T$ are open in $(\mathcal S,d_2)$ by \cref{prop:cont1} and dense by \cref{thm:C1denseS}, and $d_2$ is the $\mathcal C^2$-distance of the statement. By \cref{lem:gap}, the members of $\mathcal G_S$ satisfy the strong exponential separation condition modulo M\"obius maps and those of $\mathcal G_T$ the strong exponential separation condition.
\end{proof}

The following remark compares the open-dense theorem of B\'ar\'any, Kolossv\'ary, and Troscheit with \cref{thm:C1denseS} in the space, the conclusion, and the method, and delimits what is new.

\begin{remark} \label{rem:BKT}
  The plain half of \cref{thm:genericity}\cref{it:gen-line} is the counterpart in $(\mathcal S,d_2)$ of the open-dense theorem of B\'ar\'any, Kolossv\'ary, and Troscheit \cite[Theorem~1.4]{BaranyKolossvaryTroscheit}, the companion of their sufficient condition \cite[Theorem~1.5]{BaranyKolossvaryTroscheit} that \cref{thm:main} contains on the line by \cref{rem:BKT-comparison}, and the line case of the program arrives at it independently of that theorem. The two differ in three respects.

  The first difference is the space. Their systems form the collar class of \cref{rem:BKT-comparison}, which is contained in $\mathcal S$ for $\Omega=(-\tau,1+\tau)$ and $\Omega'=(-\eps,1+\eps)$ with $0<\tau\le\eps/2$, the collar conditions being additional assumptions; the inclusion is proper, since a member of $\mathcal S$ has a complex neighborhood of its own, by \cref{lem:thicken}, but need not satisfy the collar conditions for the given $\eps$. Our result covers their openness result: their metric is the $\mathcal C^2$-distance over $[0,1]$ and ours the same distance over $\ol{\Omega}$, so ours dominates theirs, and on their class the two induce the same topology when $\tau<\eps/2$, the two-constants estimate of \cref{rem:BKT-comparison} together with Cauchy's estimates on discs inside the collar turning smallness on $[0,1]$ into smallness of the first two derivatives on $\ol{\Omega}$, so the openness of $\mathcal G_T$ and $\mathcal G_S$ in $\mathcal S$ restricts to openness in their class. Their density result, on the other hand, is proved inside the collar class, their approximants keeping the collar conditions because their perturbation is holomorphic and controlled on the collar. We do not need the collar for density and do not use it: our density statement is relative to the larger space $\mathcal S$, and the approximants in the proof of \cref{thm:C1denseS} are controlled on $\ol{\Omega}$ only and need not satisfy the collar conditions, so neither density statement is a consequence of the other.

  The second difference is the conclusion. Their open and dense set carries the strong exponential separation condition, which is the pre-Schwarzian statement, whereas \cref{thm:C1denseS} makes the smaller class $\mathcal G_S$ dense, whose members satisfy the separation modulo M\"obius maps, and the pre-Schwarzian statement follows through the inclusion \cref{eq:GSsubsetGT}.

  The third difference is the method. Their perturbation multiplies a generator by the exponential of a sum of Gaussians, a holomorphic bump on the collar that keeps the derivative nonvanishing there but is tied to that collar; ours is the localized real third-derivative perturbation of \cref{lem:localizedS}, whose holomorphic extensions grow away from the real axis, followed by the entire smoothing of \cref{lem:entire}, which is what returns the perturbed system to the fixed $\Omega'$. Their theorem is not used in the proof of \cref{thm:C1denseS}, the system without exact overlaps from which the density argument starts being supplied by \cref{lem:identityfree}.
\end{remark}

The density just proved bears on the one-point criterion of \cref{sec:gencrit}, whose contraction ceiling \cref{rem:gencrit-range} located.

\begin{remark} \label{rem:gencrit-necessity}
  That ceiling is a feature of testing the generators at word length one. The criterion of \cref{prop:gencrit} weighs the separation of two one-letter projections against the whole tail of the cocycle, which costs $2\beta\cmax^w/(1-\cmax^w)$ and forces $\cmax<\tfrac12$ at weight one and $\cmax<2^{-1/2}$ at weight two by \cref{eq:gencritrange}, whereas \cref{lem:linecylinders} runs the same comparison at word length $n$, weighing the separation against the squared derivative of a word of length $n$ rather than against $\cmax^2$. The density theorem that \cref{lem:linecylinders} feeds, \cref{thm:C1denseS}, carries no hypothesis on $\cmax$ at all.

  Nor is the criterion necessary for the separation it delivers. On $\Omega=(0,1)$ with $\Omega'=\R$ the pair of similarities $x\mapsto\tfrac{9}{10}x+\tfrac{1}{50}$ and $x\mapsto\tfrac{9}{10}x+\tfrac{2}{25}$ is a conformal iterated function system of contraction ratio $\tfrac{9}{10}$, and by \cref{thm:C1denseS}, read through \cref{lem:gap}\cref{it:gap-pre}, the systems satisfying \cref{eq:mainhyp} are $\mathcal C^2$-dense. Since $\cmax$ changes by at most the $\mathcal C^2$-distance, arbitrarily close to that pair there are systems with $\cmax>\tfrac12$ that satisfy \cref{eq:mainhyp} and for which no $\alpha$ makes \cref{prop:gencrit}\cref{it:gencrit-pre} available.
\end{remark}

\section{Genericity in the plane} \label{sec:gen-plane}

Throughout this section $d=2$. We keep the conventions of \cref{sec:C2-machinery}; in particular, put $C_0=2\max_{i \in \II}\|\fii_i''/\fii_i'\|/(1-\cmax)$ and recall the bounds $\cmin^{|\iii|}\le|f_\iii'|\le\cmax^{|\iii|}$ and $|f_\iii''|\le C_0\cmax^{|\iii|}$ on $\ol{\Omega}$. \Cref{sec:gen-plane-open} proves both gap classes open for every bounded $\Omega$, in \cref{prop:cont}, on a neighborhood of $\ol{\Omega}$ that \cref{lem:collar} builds and the density argument then reuses. The density is where the plane parts from the line: the localized perturbation of \cref{sec:gen-line-loc} has no interior analogue here, by Cauchy's estimates, so the perturbations must act from the boundary. \Cref{sec:gen-plane-chart} builds the finite-dimensional family in which one generator is perturbed and computes the derivative of the Schwarzian projections along it, and \cref{sec:gen-plane-peak}, where the hypothesis that $\Omega$ is a Jordan domain enters, supplies the sources of the perturbations, peaking polynomials localized at a boundary point, and packages them in \cref{lem:peakdata}. \Cref{sec:gen-plane-dense} assembles the two in \cref{prop:transversal} and adds the counting, proving in \cref{thm:C2dense} that both classes are dense whenever $\Omega$ is a Jordan domain and $\Omega'$ is simply connected; \cref{thm:genericity}\cref{it:gen-plane} is proved there, from that theorem and the openness. The hypotheses on the domains belong to the method: \cref{ex:nonconvex} shows that the parametrization through which the perturbations act is unavailable on an annulus, and \cref{q:nonconvex} asks whether the conclusion holds there nonetheless.

\subsection{Openness} \label{sec:gen-plane-open}

Openness in the plane needs no hypothesis on the domain, but it does need the dual projections on a neighborhood of $\ol{\Omega}$. \Cref{lem:collar} builds that neighborhood and \cref{prop:cont} is the resulting estimate; the neighborhood carries the whole density argument below, and the estimate is used again in \cref{sec:dimension}.

The perturbation theory needs the dual projections on a neighborhood of $\ol{\Omega}$, with bounds one derivative deep; the invariant thickening of \cref{lem:thicken} provides it. Let $r_\Phi$ be the radius \cref{lem:tube} provides for $\lambda \in (\cmax,1)$, so that $|\fii_i'|\le\lambda$ on the neighborhood of $\ol{\Omega}$ of radius $r_\Phi$ for every $i \in \II$. The same $\lambda$ thus serves twice, as a contraction rate on $\ol{\Omega}$ shared by every system close to $\Phi$ and as a derivative bound for $\Phi$ itself on a neighborhood of $\ol{\Omega}$, no relation between the two radii being needed.

\begin{lemma} \label{lem:collar}
  There are $\eps>0$ and $r'\in(2\eps,r_\Phi)$, with $r_\Phi$ as above, such that the open neighborhoods
  \begin{equation} \label{eq:collar}
    U=\{z\in\C : \dist(z,\ol{\Omega})<\eps\} \qquad\text{and}\qquad V=\{z\in\C : \dist(z,\ol{\Omega})<r'\}
  \end{equation}
  have compact closures satisfying $\ol{U}\subseteq V\subseteq\ol{V}\subseteq\Omega'$, and there are constants $c_U<1$, $c_V<1$, $K_U<\infty$, and $D_U<\infty$ with the following properties:
  \begin{enumerate}
    \item\label{it:collar-derivatives} The derivatives satisfy 
    \begin{equation*}
      \max_{i \in \II}\sup_{z\in\ol{U}}|\fii_i'(z)|\le c_U \qquad \text{and} \qquad \max_{i \in \II}\sup_{z\in\ol{V}}|\fii_i'(z)|\le c_V.
    \end{equation*}
    \item\label{it:collar-invariance} Every generator maps $\ol{U}$ into $U$, with the margin 
    \begin{equation*}
      \dist(\fii_i(\ol{U}),\partial U)\ge(1-c_U)\eps
    \end{equation*}
    for every $i \in \II$.
    \item\label{it:collar-paths} Every two points of $\ol{U}$ are joined by a polygonal path in $V$ of length at most $D_U$.
    \item\label{it:collar-extension} Every dual projection extends to a function, again denoted $T_\iii$, holomorphic on $U$ and continuous on $\ol{U}$, and these extensions satisfy
    \begin{equation} \label{eq:collarbounds}
      \sup_{z\in\ol{U}}|T_\iii(z)|\le K_U \qquad\text{and}\qquad \sup_{z\in\ol{U}}|T_\iii(z)-T_\jjj(z)|\le2K_Uc_U^{|\iii\land\jjj|}
    \end{equation}
    for all distinct $\iii,\jjj\in\II^\N\cup\II^*$.
  \end{enumerate}
  In particular,
  \begin{equation} \label{eq:collarderiv}
    \|T_\iii'\|\le K_U/\eps \qquad\text{and}\qquad \|T_\iii'-T_\jjj'\|\le(2K_U/\eps)c_U^{|\iii\land\jjj|}
  \end{equation}
  for all distinct $\iii,\jjj\in\II^\N\cup\II^*$.
\end{lemma}

\begin{proof}
  The proof of \cref{lem:tube}, run with $\lambda \in (\cmax,1)$, shows that the neighborhood $W=\{z\in\C : \dist(z,\ol{\Omega})<r_\Phi\}$ lies in $\Omega'$, that every generator maps $W$ into $W$, and that $|\fii_i'|\le\lambda$ on $W$. \Cref{lem:thicken} provides $\delta>0$ such that the closed $\delta$-neighborhood of $\ol{\Omega}$ lies in the common holomorphic neighborhood constructed there and $|\fii_i'|<1$ on it for every $i\in\II$. Choose $0<\eps<\min\{\delta,r_\Phi/4\}$ and $r'\in(2\eps,r_\Phi)$. The sets \cref{eq:collar} then have compact closures and satisfy $\ol{U}\subseteq V\subseteq\ol{V}\subseteq W\subseteq\Omega'$.

  Let us first prove \cref{it:collar-derivatives}. The neighborhood supplied by \cref{lem:dual} was chosen so that the generators have derivatives of modulus less than $1$ on it, and shrinking $\eps$ only shrinks $\ol{U}$, so $c_U=\max_{i \in \II}\sup_{z\in\ol{U}}|\fii_i'(z)|<1$. Since $\ol{V}\subseteq W$ and $|\fii_i'|\le\lambda$ on $W$, also $c_V=\max_{i \in \II}\sup_{z\in\ol{V}}|\fii_i'(z)|\le\lambda<1$.

  For \cref{it:collar-invariance}, let $i \in \II$ and $z\in\ol{U}$, and let $x\in\ol{\Omega}$ be a nearest point of $z$. The segment $[x,z]$ has length at most $\eps$ and stays in $\ol{U}$, so $|\fii_i(z)-\fii_i(x)|\le c_U\eps$, and $\fii_i(x)$ lies in $\ol{\Omega}$, whence $\dist(\fii_i(z),\ol{\Omega})\le c_U\eps<\eps$. This gives the invariance $\fii_i(\ol{U})\subseteq U$. Since $w\mapsto\dist(w,\ol{\Omega})$ is $1$-Lipschitz and equals $\eps$ on $\partial U$, the same estimate gives the margin $\dist(\fii_i(\ol{U}),\partial U)\ge(1-c_U)\eps$.

  The set $V$ is open and connected, being the union of the connected $\Omega$ with the balls $B^o(x,r')$, $x\in\ol{\Omega}$, and, exactly as in the proof of \cref{lem:tube}, the infimum $\ell_V(z,w)$ of the lengths of polygonal paths in $V$ joining $z$ to $w$ is continuous on $V\times V$. Choose $D_U>\max\{\ell_V(z,w) : z,w\in\ol{U}\}$, which is finite by the compactness of $\ol{U}\times\ol{U}$. By the definition of $\ell_V$, every two points of $\ol{U}$ are joined in $V$ by a polygonal path of length less than $D_U$. This proves \cref{it:collar-paths}.

  It remains to prove \cref{it:collar-extension}. Since the partial orbits of $\ol{U}$ stay in $\ol{U}$, the sum \cref{eq:Tline} converges absolutely and uniformly on $\ol{U}$, its $k$-th summand being bounded there by $\beta_Uc_U^{k-1}$ with $\beta_U=\max_{i \in \II}\sup_{z\in\ol{U}}|T_{\fii_i}(z)|<\infty$, and, for every word, defines a function holomorphic on $U$, continuous on $\ol{U}$, and equal to $T_\iii$ on $\ol{\Omega}$. The bound $K_U=\beta_U/(1-c_U)$ and the prefix bound of \cref{eq:collarbounds} follow exactly as in \cref{lem:dual}, the first $|\iii\land\jjj|$ summands agreeing. This shows \cref{it:collar-extension}, and, every $x\in\ol{\Omega}$ having $B^o(x,\eps)\subseteq U$, the claims in \cref{eq:collarderiv} follow from \cref{eq:cauchy} applied on $B^o(x,r)$ with $r<\eps$ and letting $r\to\eps$.
\end{proof}

The dual projections vary Lipschitz-continuously in $d_2$, with a constant independent of the word. A gap over the pair space $Z$ of \cref{eq:Zdef} is an infimum of distances between such projections, so this single estimate is all that the openness of the two gap classes needs: for $\mathcal G_T$ it preserves the gap outright, and for $\mathcal G_S$ it supplies the hypothesis of \cref{lem:Sclosed}, exactly as on the line.

\begin{proposition} \label{prop:cont}
  Let $\Phi=(\fii_i)_{i \in \II}$ be a conformal iterated function system on $\Omega \subset \R^2$. Then there are $r_0 > 0$ and $C < \infty$ such that
  \begin{equation} \label{eq:cont}
    \|T^{\Phi}_\iii-T^{\Psi}_\iii\| \le Cd_2(\Phi,\Psi)
  \end{equation}
  for all $\iii\in\II^\N\cup\II^*$ and every $\Psi\in\mathcal S$ with $d_2(\Phi,\Psi)<r_0$. In particular, if $\Delta_T(\Phi)>0$, then there is $r_1\in(0,r_0]$ such that
  \begin{equation*}
    \Delta_T(\Psi) > 0
  \end{equation*}
  for all $\Psi\in\mathcal S$ with $d_2(\Phi,\Psi)<r_1$, the same implication holds with $\Delta_S$ in place of $\Delta_T$, and the sets $\mathcal G_T$ and $\mathcal G_S$ are open in $(\mathcal S,d_2)$.
\end{proposition}

\begin{proof}
  With $U$, $\eps$, and $c_U$ from \cref{lem:collar}, let $\lambda\in(\cmax,1)$ and choose $r_0>0$ so small that $r_0\le(1-c_U)\eps/2$ and that every $\Psi=(\psi_i)_{i \in \II}\in\mathcal S$ with $d_2(\Phi,\Psi)<r_0$ has $\cmax(\Psi)\le \lambda$ and $\cmin(\Psi)\ge\cmin/2$; below $C$ denotes a finite constant depending only on $\Phi$, whose value may change. Fix such a $\Psi$ and a nonempty word $\iii$, and write $T^\Phi_\iii=\sum_kA_k^\Phi$ and $T^\Psi_\iii=\sum_kA_k^\Psi$ as in \cref{eq:Tline}, with $A_k^\Phi=(f^\Phi_{\iii|_{k-1}})'(T^\Phi_{i_k}\circ f^\Phi_{\iii|_{k-1}})$ and $A_k^\Psi$ defined analogously; for the empty word both projections vanish, and all partial orbits stay in $\ol{\Omega}$.

  Fix $x\in\ol{\Omega}$ and an admissible $k$, and put $u_l=f^\Phi_{\iii|_{l-1}}(x)$ and $v_l=f^\Psi_{\iii|_{l-1}}(x)$, points of $\ol{\Omega}$ with $u_1=v_1=x$. Then $|u_{l+1}-v_{l+1}|\le|\fii_{i_l}(u_l)-\fii_{i_l}(v_l)|+|\fii_{i_l}(v_l)-\psi_{i_l}(v_l)|\le c_U|u_l-v_l|+d_2(\Phi,\Psi)$ by \cref{lem:collar}\cref{it:collar-derivatives}: the perturbed generator is evaluated only at the point $v_l$ of $\ol{\Omega}$, where $d_2$ controls it, and the fixed generator is integrated along the segment $[u_l,v_l]$, which stays in $U$ because every point of it is within $|u_l-v_l|$ of $\ol{\Omega}$ and induction gives $|u_l-v_l|\le(1-c_U)^{-1}d_2(\Phi,\Psi)\le\eps/2$. When $k\ge2$, replacing one factor at a time in the derivative cocycles $(f^\Phi_{\iii|_{k-1}})'(x)=\prod_{l<k}\fii_{i_l}'(u_l)$ and $(f^\Psi_{\iii|_{k-1}})'(x)=\prod_{l<k}\psi_{i_l}'(v_l)$, whose factors are bounded by $\cmax$ and $\lambda$, and using $|\fii_{i_l}'(u_l)-\psi_{i_l}'(v_l)|\le\sup_U|\fii_{i_l}''||u_l-v_l|+d_2(\Phi,\Psi)\le Cd_2(\Phi,\Psi)$, the second derivatives being bounded on the compact $\ol{U}\subseteq\Omega'$, shows that the two cocycles differ by at most $Ck\lambda^{\,k-1}d_2(\Phi,\Psi)$. Moreover $|T^\Phi_{i_k}(u_k)-T^\Psi_{i_k}(v_k)|\le\sup_U|(T^\Phi_{i_k})'||u_k-v_k|+\|T^\Phi_{i_k}-T^\Psi_{i_k}\|\le Cd_2(\Phi,\Psi)$, since $T^\Phi_i=\fii_i''/\fii_i'$ is holomorphic on $\Omega'\supseteq U$ and $\|T^\Phi_i-T^\Psi_i\|\le2\cmin^{-2}\|\fii_i''\|\|\fii_i'-\psi_i'\|+2\cmin^{-1}\|\fii_i''-\psi_i''\|\le Cd_2(\Phi,\Psi)$ on $\ol{\Omega}$, and also $\|T^\Psi_i\|\le C$. Combining the bounds gives $|A_k^\Phi-A_k^\Psi|\le Ck\lambda^{\,k-1}d_2(\Phi,\Psi)$ on $\ol{\Omega}$, and summation over $k$ with $\sum_{k\ge1}k\lambda^{\,k-1}=(1-\lambda)^{-2}$ yields \cref{eq:cont}.

  Finally, let $\Delta_T(\Phi)>0$, write $\delta=\Delta_T(\Phi)$, enlarge $C$ if necessary so that $C>0$, and let $\Psi\in\mathcal S$ satisfy $d_2(\Phi,\Psi)<r_1=\min\{r_0,\delta/(3C)\}$. For every $(\iii,\jjj)\in Z$, two applications of \cref{eq:cont} give $\|T^\Psi_\iii-T^\Psi_\jjj\|\ge\|T^\Phi_\iii-T^\Phi_\jjj\|-2Cd_2(\Phi,\Psi)>\delta/3$, and taking the infimum over $Z$ yields $\Delta_T(\Psi)\ge\delta/3>0$. For the Schwarzian gap, let $\Delta_S(\Phi)>0$ and suppose that there is no $r\in(0,r_0]$ with $\Delta_S(\Psi)>0$ for every $\Psi\in\mathcal S$ satisfying $d_2(\Phi,\Psi)<r$. Since $\Delta_S$ is nonnegative, there are then $\Psi_n\in\mathcal S$ with $d_2(\Phi,\Psi_n)<\min\{r_0,1/n\}$ and $\Delta_S(\Psi_n)=0$, so that $\Psi_n\to\Phi$ in $d_2$. Since $Z$ is compact and the Schwarzian projections are continuous in the word variable by \cref{lem:dual2}, the infimum defining $\Delta_S(\Psi_n)$ is attained, so there are $(\iii_n,\jjj_n)\in Z$ with $S^{\Psi_n}_{\iii_n}\equiv S^{\Psi_n}_{\jjj_n}$, and \cref{eq:cont} gives the hypothesis \cref{eq:Tparametercontinuity} of \cref{lem:Sclosed}, which yields a pair $(\iii,\jjj)\in Z$ with $S^\Phi_\iii\equiv S^\Phi_\jjj$, contradicting $\Delta_S(\Phi)>0$. Specializing to members of $\mathcal S$ proves the last claim.
\end{proof}

\subsection{Linearizing the Schwarzian along a perturbation} \label{sec:gen-plane-chart}

We turn to the density of $\mathcal G_T$ and $\mathcal G_S$. The density argument perturbs one generator inside a finite-dimensional family and reads the effect on the Schwarzian projections. This subsection builds that family and computes the derivative of the projections along it: \cref{lem:perturbed} keeps the perturbed systems in $\mathcal S$, \cref{lem:Scollar} controls the projections on the neighborhood $U$ of \cref{lem:collar}, and \cref{lem:linearizeS} is the linearization itself.

From here to the end of the proof of \cref{thm:C2dense} we assume the two hypotheses of that theorem, that $\Omega$ is a Jordan domain and that $\Omega'$ is simply connected; the openness of \cref{prop:cont} needs neither. The first enters only through \cref{lem:access,lem:peaking}, and the second makes every primitive below single-valued. We perturb the nonlinearities globally. Fix $z_0\in\Omega$. Given $i\in\II$ and a holomorphic function $\theta_i$ on $\Omega'$, let $\Theta_i$ be the primitive of $\theta_i$ on the simply connected $\Omega'$ with $\Theta_i(z_0)=0$, the integrals from $z_0$ below being along arbitrary paths in $\Omega'$ with path-independent values, and define $g_i$ on $\Omega'$ by
\begin{equation} \label{eq:gdef}
  g_i(z) = \fii_i(z_0)+\int_{z_0}^{z}\fii_i'(w)e^{\Theta_i(w)}\dd w,
\end{equation}
so that $g_i'=\fii_i'e^{\Theta_i}$ vanishes nowhere and $g_i''/g_i'=\fii_i''/\fii_i'+\theta_i$: the dual projection of the letter $i$ shifts by $\theta_i$. For a tuple $\theta=(\theta_i)_{i\in\II}$ and small $t\in\R$ write $\Phi_t=(g_i^t)$ for the tuple built by \cref{eq:gdef} with $t\theta_i$ in place of $\theta_i$, so that $T^{\Phi_t}_i=T_i+t\theta_i$ and $\Phi_0=\Phi$. The construction makes sense for complex $t$ as well, and the next lemma records that the perturbed generators stay under control on the neighborhood $U$ of \cref{lem:collar} for all small complex $t$; the complex parameter is used to prove that the projections of $\Phi_t$ depend analytically on $t$.

\begin{lemma} \label{lem:perturbed}
  Let $\theta=(\theta_i)_{i\in\II}$ be a tuple of functions holomorphic on $\Omega'$, let $\Theta_i$ be the normalized primitive of $\theta_i$ on $\Omega'$ for each $i\in\II$, and let $U$, $V$, $\eps$, $c_U$, $c_V$, and $D_U$ be as in \cref{lem:collar}. For complex $t$ the tuple $t\theta$ is again a tuple of functions holomorphic on $\Omega'$, so \cref{eq:gdef} builds $\Phi_t=(g_i^t)$ for such $t$ as well, jointly holomorphically in $(t,z)$, with the nonvanishing derivative $(g_i^t)'=\fii_i'e^{t\Theta_i}$, and with $M_U=\max_{i\in\II}\sup_{z\in\ol{V}}|\Theta_i(z)|$ the generators satisfy
  \begin{equation} \label{eq:gbounds}
    |(g_i^t)'|\le c_Ue^{|t|M_U} \qquad\text{and}\qquad |g_i^t-\fii_i|\le D_Uc_V(e^{|t|M_U}-1)
  \end{equation}
  on $U$. Moreover, there are $t_0>0$ and $q_0<1$ such that, for all complex $t$ with $|t|\le t_0$, every $g_i^t$ sends $U$ into $U$ with $\sup_{z\in U}|(g_i^t)'(z)|\le q_0$, and $\Phi_t\in\mathcal S$ for real $|t|\le t_0$.
\end{lemma}

\begin{proof}
  For complex $t$ the function $t\Theta_i$ is holomorphic on $\Omega'$, so \cref{eq:gdef} defines $g_i^t$ on $\Omega'$, jointly holomorphically in $(t,z)$, the integrand being jointly holomorphic, with $(g_i^t)'=\fii_i'e^{t\Theta_i}$, which vanishes nowhere. The constant $M_U$ is finite, the primitives being continuous on the compact $\ol{V}\subseteq\Omega'$. On $U$ the first bound in \cref{eq:gbounds} follows from \cref{lem:collar}\cref{it:collar-derivatives}. For the second, integrate $(g_i^t)'-\fii_i'=\fii_i'(e^{t\Theta_i}-1)$ from $z_0$ to $z\in U$ along a polygonal path in $V$ of length at most $D_U$, which \cref{lem:collar}\cref{it:collar-paths} supplies; the two maps agree at $z_0$, and $|\fii_i'|\le c_V$ and $|e^{t\Theta_i}-1|\le e^{|t|M_U}-1$ on $\ol{V}$. Choose $t_0>0$ with $q_0=c_Ue^{t_0M_U}<1$ and $D_Uc_V(e^{t_0M_U}-1)<(1-c_U)\eps$. For complex $|t|\le t_0$ the first bound gives $\sup_{z\in U}|(g_i^t)'(z)|\le q_0$, and the second, together with the margin $\dist(\fii_i(\ol{U}),\partial U)\ge(1-c_U)\eps$ of \cref{lem:collar}\cref{it:collar-invariance}, shows that $g_i^t$ sends $U$ into $U$. Since each $\fii_i(\ol{\Omega})$ is a compact subset of the open $\Omega$, decreasing $t_0$ makes the displacement bound also give $g_i^t(\ol{\Omega})\subseteq\Omega$, and, $|(g_i^t)'|\le q_0<1$ holding on $\ol{\Omega}\subseteq U$, the system $\Phi_t$ lies in $\mathcal S$ for real $|t|\le t_0$.
\end{proof}

Varying a generator changes not only that generator's own Schwarzian but also the points and the derivatives at which the tail Schwarzian projections are read. This calls for bounds on the Schwarzian projections and their derivatives that are uniform in the word, and we collect these first. Put
\begin{equation} \label{eq:KSCS}
  K_S = \frac{\max_{i \in \II}\|S_i\|}{1-\cmax^2} \qquad \text{and} \qquad C_S = \frac{2C_0\max_{i \in \II}\|S_i\|}{1-\cmax^2}+\frac{\max_{i \in \II}\|S_i'\|}{1-\cmax^3}.
\end{equation}

\begin{lemma} \label{lem:Scollar}
  Every Schwarzian projection satisfies $\|S_\iii\|\le K_S$ and $\|S_\iii'\|\le C_S$. With $U$ and $c_U$ from \cref{lem:collar}, there are constants $K_{S,U},C_{S,U}<\infty$ such that
  \begin{equation} \label{eq:Scollar}
    \sup_{z\in U}|S_\iii(z)-S_\jjj(z)|\le2K_{S,U}c_U^{2|\iii\land\jjj|} \qquad \text{and} \qquad \|S_\iii'-S_\jjj'\|\le C_{S,U}c_U^{2|\iii\land\jjj|}
  \end{equation}
  for all distinct $\iii,\jjj\in\II^\N\cup\II^*$.
\end{lemma}

\begin{proof}
  The series \cref{eq:Sline} and $|f_{\iii|_{k-1}}'|\le\cmax^{k-1}$ give $\|S_\iii\|\le\sum_{k\ge1}\max_{i \in \II}\|S_i\|\cmax^{2(k-1)}=K_S$. Differentiating its $k$-th summand gives
  \begin{equation*}
    (S_{i_k}'\circ f_{\iii|_{k-1}})(f_{\iii|_{k-1}}')^3+2(S_{i_k}\circ f_{\iii|_{k-1}})f_{\iii|_{k-1}}'f_{\iii|_{k-1}}'',
  \end{equation*}
  whose norm is at most $\max_{i \in \II}\|S_i'\|\cmax^{3(k-1)}+2C_0\max_{i \in \II}\|S_i\|\cmax^{2(k-1)}$. These majorants are summable, so the differentiated series converges uniformly on $\ol{\Omega}$; together with the convergence of the original series, this justifies termwise differentiation for infinite words. Summing gives $\|S_\iii'\|\le C_S$. On the neighborhood $U$, invariant by \cref{lem:collar}\cref{it:collar-invariance}, put $\beta_{S,U}=\max_{i \in \II}\sup_{z\in\ol{U}}|S_i(z)|$ and $K_{S,U}=\beta_{S,U}/(1-c_U^2)$. The same series with the bound $c_U$ of \cref{lem:collar}\cref{it:collar-derivatives} in place of $\cmax$ converges normally, and if two distinct words have a common prefix of length $q$, their first $q$ summands agree while their two tails have total norm at most $2K_{S,U}c_U^{2q}$. Cauchy's estimate on the discs $B^o(x,\eps)\subseteq U$, with $x\in\ol{\Omega}$ and $\eps$ from \cref{lem:collar}, gives the derivative estimate in \cref{eq:Scollar} with $C_{S,U}=2K_{S,U}/\eps$.
\end{proof}

The parametrization \cref{eq:gdef} can realize arbitrary infinitesimal changes of the generator Schwarzians. For $i\in\II$ and $h$ holomorphic on $\Omega'$ define
\begin{equation} \label{eq:Qa}
  Q_ih(z) = \fii_i'(z)\int_{z_0}^{z}\frac{h(w)}{\fii_i'(w)}\dd w,
\end{equation}
the integral being single-valued since $\fii_i'$ has no zeros on the simply connected $\Omega'$; differentiating $Q_ih$ gives $(Q_ih)'-T_iQ_ih=h$. The paths of \cref{lem:collar}\cref{it:collar-paths} lie in $V$ and have length at most $D_U$, and $|\fii_i'|\le c_V$ on $\ol{V}$, so estimating the integral in \cref{eq:Qa} along them gives
\begin{equation} \label{eq:CQ}
  \sup_{z\in U}|Q_ih(z)|\le C_Q\sup_{z\in\ol{V}}|h(z)|,
\end{equation}
where $C_Q = c_VD_U\max_{j \in \II}\sup_{z\in\ol{V}}|1/\fii_j'(z)|$. For a tuple $h=(h_i)_{i\in\II}$ of functions holomorphic on $\Omega'$, taking $\theta_i=Q_ih_i$ in \cref{eq:gdef} yields
\begin{equation} \label{eq:Sgenshift}
  S_{g_i^t} = S_i+th_i-\tfrac12t^2\theta_i^2.
\end{equation}

\begin{lemma} \label{lem:linearizeS}
  Let $h=(h_i)_{i\in\II}$, $\theta_i=Q_ih_i$, and $\Phi_t$ be as above. For each nonempty $\iii\in\II^\N\cup\II^*$ the map $t\mapsto S^{\Phi_t}_\iii\in\mathcal{H}(\ol{\Omega})$ is real-analytic near $t=0$, and its derivative is the sum of the convergent series,
  \begin{equation} \label{eq:Slinseries}
    \partial_tS^{\Phi_t}_\iii|_{t=0} = \sum_{k=0}^{|\iii|-1}\LL^{(2)}_{\iii|_k}(h_{i_{k+1}}+\Gamma_{i_{k+1},\sigma^{k+1}\iii}),
  \end{equation}
  where
  \begin{equation*}
    \Gamma_{i,\lll} = 2(\fii_i')^2\Theta_i(S_\lll\circ\fii_i)+(\fii_i')^2(S_\lll'\circ\fii_i)\xi_i,
  \end{equation*}
  the function $\Theta_i$ is the primitive of $\theta_i$ on $\Omega'$ with $\Theta_i(z_0)=0$, the function $\xi_i$ is the primitive of $\fii_i'\Theta_i$ on $\Omega'$ with $\xi_i(z_0)=0$, and $\LL^{(2)}_{\iii|_k}$ is the linear part of the dual map $G_{\iii|_k}$, given by \cref{eq:wordops}. Moreover,
  \begin{equation*}
    |\Gamma_{i,\lll}(w)|\le\cmax^2(2K_S|\Theta_i(w)|+C_S|\xi_i(w)|)
  \end{equation*}
  for all $w\in\ol{\Omega}$, $i\in\II$, and $\lll\in\II^\N\cup\II^*$.
\end{lemma}

\begin{proof}
  For complex $t$, interpret $S_\lll^{\Phi_t}$ as the sum of the series \cref{eq:Sline} formed from the generators of $\Phi_t$, the compositions $f^{\Phi_t}_{\lll|_{k-1}}$ being the reversed compositions \cref{eq:frev} of these generators. Let $t_0$ and $q_0$ be as in \cref{lem:perturbed}, so that for complex $|t|\le t_0$ the partial orbits of $U$ under the generators of $\Phi_t$ stay in $U$ and $|(f^{\Phi_t}_{\lll|_{k-1}})'|\le q_0^{k-1}$ on $U$. Put
  \begin{equation*}
    B_S = \max_{i\in\II}\sup_{z\in\ol{U}}(|S_i(z)|+t_0|h_i(z)|+\tfrac12t_0^2|\theta_i(z)|^2).
  \end{equation*}
  This constant is finite because $S_i$, $h_i$, and $\theta_i$ are holomorphic on $\Omega'$ and $\ol{U}\subseteq\Omega'$. By \cref{eq:Sgenshift}, $|S_{g_i^t}|\le B_S$ on $\ol{U}$ for $|t|\le t_0$, so the $k$-th summand of the series is bounded by $B_Sq_0^{2(k-1)}$ on $\{|t|\le t_0\}\times U$, and it is jointly holomorphic in $(t,z)$ there by \cref{lem:perturbed}. A bounded function jointly holomorphic in $(t,z)$ on $\{|t|<t_0\}\times U$ is, as a function of $t$, a holomorphic $\mathcal{H}(\ol{\Omega})$-valued function: its $t$-Taylor coefficients are holomorphic in $z$ by the Cauchy integral formula, and the Taylor series converges in $\mathcal{H}(\ol{\Omega})$ by the Cauchy estimates. Consequently the series converges in $\mathcal H(\ol{\Omega})$ uniformly in $|t|\le t_0$, and its sum is a holomorphic $\mathcal{H}(\ol{\Omega})$-valued function of $t$ on $|t|<t_0$ by the Weierstrass convergence theorem, whose Cauchy-integral proof \cite[Theorem 5.1]{Ahlfors} applies to Banach-space-valued functions unchanged, in particular real-analytic in real $t$; moreover, with $M_S=B_S/(1-q_0^2)$, one has $\sup_{|t|\le t_0}\|S^{\Phi_t}_\lll\|\le M_S$ for every nonempty word $\lll$. Write $v_\lll=\partial_tS^{\Phi_t}_\lll|_{t=0}$ for the derivative at $t=0$. For every $r\in(0,t_0)$ the Cauchy estimate \cite[Chapter~4, (25)]{Ahlfors} gives $\|v_\lll\|\le M_S/r$, and letting $r\to t_0$ gives $\|v_\lll\|\le M_S/t_0$. The recursion
  \begin{equation*}
    S_\iii^{\Phi_t} = ((g_{i_1}^t)')^2(S_{\sigma\iii}^{\Phi_t}\circ g_{i_1}^t)+S_{g_{i_1}^t},
  \end{equation*}
  the weight-two half of \cref{eq:dualrec} for $\Phi_t$, holds for the series-defined sums by the reindexing of \cref{eq:Sline}, and the same series define $S^{\Phi_t}_{\sigma\iii}$ as a function jointly holomorphic in $(t,z)$ on $\{|t|<t_0\}\times U$; differentiating the recursion at $t=0$, using $\partial_tg_i^t|_{t=0}=\xi_i$, $\partial_t(g_i^t)'|_{t=0}=\fii_i'\Theta_i$, and \cref{eq:Sgenshift}, gives the recursion $v_\iii=h_{i_1}+\LL^{(2)}_{i_1}v_{\sigma\iii}+\Gamma_{i_1,\sigma\iii}$. Unfolding it $n$ times leaves the remainder $\LL^{(2)}_{\iii|_n}v_{\sigma^n\iii}$, which is $0$ for a finite word once $n=|\iii|$, the empty-word projection vanishing identically in $t$, and has norm at most $\cmax^{2n}M_S/t_0$ in general; letting $n\to\infty$ gives \cref{eq:Slinseries}. Finally, $|\fii_i'|\le\cmax$ on $\ol{\Omega}$, and $|S_\lll\circ\fii_i|\le K_S$ and $|S_\lll'\circ\fii_i|\le C_S$ there by \cref{lem:Scollar}, which gives the pointwise bound on $\Gamma_{i,\lll}$.
\end{proof}

\subsection{Boundary access and peaking polynomials} \label{sec:gen-plane-peak}

This is where the hypothesis that $\Omega$ is a Jordan domain enters, through \cref{lem:access,lem:peaking} below; \cref{lem:peakdata} then packages the polynomials and their evaluation points, together with the bounds on the primitives that \cref{lem:linearizeS} attaches to them through the operator \cref{eq:Qa}, in the form the transversality argument of \cref{sec:gen-plane-dense} consumes. The localized perturbation of \cref{sec:gen-line-loc} has no interior analogue in the plane: by Cauchy's estimates a holomorphic function small on the boundary of a disc in $\Omega$ is small at its center with all its derivatives. A substitute exists at the boundary: a \emph{peaking function} at $x\in\partial\Omega$ is one holomorphic on $\Omega$, continuous on $\ol{\Omega}$, equal to $1$ at $x$ and of smaller modulus on the rest of $\ol{\Omega}$, and its powers still equal $1$ at $x$, stay bounded by $1$, and are exponentially small on every compact subset of $\ol{\Omega}$ omitting $x$. \Cref{lem:peaking} makes them polynomials, as the parametrizations \cref{eq:gdef,eq:Qa} need sources holomorphic on $\Omega'$, and these localize where the cocycle reads them: in \cref{eq:Slinseries} every term but the first reads its source at a point of the compact set $\bigcup_{i\in\II}\fii_i(\ol{\Omega})\subseteq\Omega$, only the first at an evaluation point near $x$. The Jordan hypothesis enters twice, through \cref{lem:access} for these functions and the rectifiable paths along which the primitives are estimated there, and through Mergelyan's theorem for the polynomials of \cref{lem:peaking}.

\begin{lemma} \label{lem:access}
  Let $\Omega$ be a Jordan domain and let $z_0\in\Omega$ and $m\in\N$. Then there are distinct points $x_1,\ldots,x_m\in\partial\Omega$, peaking functions $u_1,\ldots,u_m$ at $x_1,\ldots,x_m$, and injective rectifiable access paths $\alpha_l\colon[0,1]\to\Omega\cup\{x_l\}$ with $\alpha_l(0)=z_0$, $\alpha_l(1)=x_l$, and $\alpha_l([0,1))\subseteq\Omega$.
\end{lemma}

\begin{proof}
  Let $\fii\colon B^o(0,1)\to\Omega$ be a conformal bijection with $\fii(0)=z_0$, provided by the Riemann mapping theorem; see \cite[Theorem~6.1]{Ahlfors}. Since $\partial\Omega$ is a Jordan curve, $\fii$ extends to a homeomorphism of $B(0,1)$ onto $\ol{\Omega}$ by Carath\'eodory's theorem; see \cite[Theorem~3.1]{GarnettMarshall}. Let $\psi\colon\ol{\Omega}\to B(0,1)$ be the inverse homeomorphism, holomorphic on $\Omega$. For $\theta\in[0,2\pi)$ the radial path $t\mapsto\fii(te^{i\theta})$, $t\in[0,1]$, is injective, starts at $z_0$, stays in $\Omega$ for $t<1$, and ends at the boundary point $\fii(e^{i\theta})$; distinct angles give distinct endpoints. Its length, possibly infinite, is $\int_0^1|\fii'(te^{i\theta})|\dd t$. To justify the endpoint $t=1$, apply the Cauchy--Schwarz inequality first on $[1/2,r]$ for $r<1$ and then let $r\to1$ by monotone convergence. This gives
  \begin{equation*}
    \int_0^{2\pi}\biggl(\int_{1/2}^{1}|\fii'(te^{i\theta})|\dd t\biggr)^2\dd\theta\le\log2\int_0^{2\pi}\int_{1/2}^{1}|\fii'(te^{i\theta})|^2t\dd t\dd\theta\le\LL^2(\Omega)\log2<\infty,
  \end{equation*}
  Since $\fii$ is injective and its Jacobian is $|\fii'|^2$, the change-of-variables formula identifies the middle integral with the area of the image of the annulus $\{\tfrac12<|w|<1\}$. It follows that $\int_0^1|\fii'(te^{i\theta})|\dd t<\infty$ for almost every $\theta$, the integral over $[0,\tfrac12]$ being bounded by the maximum of $|\fii'|$ on $B(0,\tfrac12)$. Choose $m$ distinct angles with finite radial integrals, put $x_l=\fii(e^{i\theta_l})$, and let $\alpha_l$ be the corresponding radial paths, which are rectifiable by the finiteness of their radial integrals. Finally, put $w_l=e^{i\theta_l}$ and
  \begin{equation*}
    u_l = \tfrac12(1+\ol{w_l}\psi).
  \end{equation*}
  Each $u_l$ is continuous on $\ol{\Omega}$ and holomorphic on $\Omega$ with $|u_l|\le\tfrac12(1+|\psi|)\le1$, and $|u_l(z)|=1$ forces $|1+\ol{w_l}\psi(z)|=2$ with $|\ol{w_l}\psi(z)|\le1$, hence $\ol{w_l}\psi(z)=1$ and $z=\fii(w_l)=x_l$; conversely $u_l(x_l)=1$.
\end{proof}

The powers of a peaking function already peak and decay, but they are holomorphic on $\Omega$ alone. Mergelyan's theorem buys polynomials, entire and so admissible as sources, at the price of a fixed loss at the peak, where the value $1$ becomes a modulus at least $e^{-3}$, and of a halved rate of decay. Two forms of decay are recorded because the access paths of \cref{lem:access} run into $x$: a uniform bound on each compact set omitting $x$, and a bound on each sublevel set of $|u|$, whose complementary peak region decreases to $\{x\}$ and therefore meets such a path in a piece of vanishing length.

\begin{lemma} \label{lem:peaking}
  Let $\Omega$ be a Jordan domain, let $x\in\partial\Omega$, and let $u$ be a peaking function at $x$. Then there is a sequence of polynomials $(P_n)_{n\in\N}$ such that $|P_n|\le1$ on $\ol{\Omega}$ for every $n\in\N$ and $|P_n(x)|\ge e^{-3}$ for every $n\ge2$, and, for every nonempty compact set $K\subseteq\ol{\Omega}\setminus\{x\}$ and every $n\ge1/c_K$, where $c_K>0$ is defined by $\sup_K|u|=1-2c_K$,
  \begin{equation} \label{eq:peakcompact}
    \sup_K|P_n|\le e^{-c_Kn},
  \end{equation}
  and, for every $\vartheta\in(0,1)$, every $n\ge2/\vartheta$, and every $z\in\ol{\Omega}$ satisfying $|u(z)|\le1-\vartheta$,
  \begin{equation} \label{eq:peaklocal}
    |P_n(z)|\le e^{-\vartheta n/2}.
  \end{equation}
\end{lemma}

\begin{proof}
  A Jordan curve is the boundary of both of its complementary regions, so no point of $\partial\Omega$ is interior to $\ol{\Omega}$ and $\inter(\ol{\Omega})=\Omega$; moreover $\C\setminus\ol{\Omega}$, the exterior region, is connected. Mergelyan's theorem \cite[Theorem~20.5]{Rudin} therefore applies on $\ol{\Omega}$ and provides polynomials $p_n$ with $\|p_n-u\|\le1/n$. Put $P_n=(p_n/(1+1/n))^n$. On $\ol{\Omega}$ we have $|p_n|\le|u|+1/n\le1+1/n$, so $|P_n|\le1$; and $|p_n(x)|\ge1-1/n$, so $|P_n(x)|\ge((n-1)/(n+1))^n\ge((n-1)/n)^{2n}\ge\tfrac1{16}\ge e^{-3}$ for $n\ge2$, since $n^2\ge(n-1)(n+1)$ and $((n-1)/n)^n$ increases with $n$ from its value $\tfrac14$ at $n=2$. For \cref{eq:peakcompact}, compactness gives $\sup_K|u|<1$, so $c_K>0$, and on $K$ we have $|p_n|\le1-2c_K+1/n\le1-c_K$ once $n\ge1/c_K$, whence $\sup_K|P_n|\le(1-c_K)^n\le e^{-c_Kn}$. For \cref{eq:peaklocal}, suppose that $\vartheta\in(0,1)$, $n\ge2/\vartheta$, and $|u(z)|\le1-\vartheta$. Then $1/n\le\vartheta/2$ and $|p_n(z)|\le1-\vartheta+1/n\le1-\vartheta/2$, so $|P_n(z)|\le((1-\vartheta/2)/(1+1/n))^n\le(1-\vartheta/2)^n\le e^{-\vartheta n/2}$.
\end{proof}

The transversality argument of \cref{sec:gen-plane-dense} reads the peaking polynomials at $m$ evaluation points, one on each access path near its peak, and feeds them into the parametrization \cref{eq:gdef} through the operator $Q_i$ of \cref{eq:Qa}. What it needs from them is collected in the next lemma: each polynomial has modulus at most one on $\ol{\Omega}$ and modulus at least $\tfrac12e^{-3}$ at its own evaluation point, it is exponentially small on the images $\fii_i(\ol{\Omega})$ and at the other evaluation points, and the primitives that \cref{lem:linearizeS} attaches to it are small wherever the linearization reads them. The last property is the reason for the two forms of decay in \cref{lem:peaking}: the primitives are integrals along paths, and along the access path to the peak the polynomial is not small, but the part of the path on which it is not small has vanishing length.

\begin{lemma} \label{lem:peakdata}
  Let $\Omega$ be a Jordan domain, let $\Omega'$ be simply connected, let $\Phi\in\mathcal S$ and $m\in\N$, let $z_0\in\Omega$ be the base point of \cref{eq:gdef}, and let $K_\Phi=\bigcup_{i\in\II}\fii_i(\ol{\Omega})$. Then there are $c\in(0,\tfrac15]$ and numbers $\kappa_n>0$ for $n\ge1/c$ with $\kappa_n\to0$ as $n\to\infty$ such that for every integer $n\ge1/c$ there are polynomials $P_{1,n},\ldots,P_{m,n}$ and points $y_{1,n},\ldots,y_{m,n}\in\Omega$ with the following properties.
  \begin{enumerate}
    \item\label{it:peakdata-size} For every $l\in\{1,\ldots,m\}$, we have $|P_{l,n}|\le1$ on $\ol{\Omega}$ and $|P_{l,n}(y_{l,n})|\ge\tfrac12e^{-3}$.
    \item\label{it:peakdata-decay} For every $l\in\{1,\ldots,m\}$, we have $|P_{l,n}|\le e^{-cn}$ on $K_\Phi$, and $|P_{l,n}(y_{l',n})|\le e^{-cn}$ for all $l'\in\{1,\ldots,m\}\setminus\{l\}$.
    \item\label{it:peakdata-prim} For all $i\in\II$ and $l\in\{1,\ldots,m\}$, the primitives $\Theta$ of $Q_iP_{l,n}$ and $\xi$ of $\fii_i'\Theta$ on $\Omega'$ with $\Theta(z_0)=\xi(z_0)=0$, which are the functions $\Theta_i$ and $\xi_i$ of \cref{lem:linearizeS} for the source tuple with $i$-th entry $P_{l,n}$ and all other entries $0$, satisfy $|\Theta|\le\kappa_n$ and $|\xi|\le\kappa_n$ on $K_\Phi\cup\{y_{1,n},\ldots,y_{m,n}\}$.
  \end{enumerate}
\end{lemma}

\begin{proof}
  Fix a compact set $K\subseteq\Omega$ containing $K_\Phi\cup\{z_0\}$ and a constant $D\ge1$ such that every point of $K_\Phi$ is joined to $z_0$ by a path in $K$ of length at most $D$: cover the compact $K_\Phi\cup\{z_0\}$ by finitely many discs with closures in $\Omega$, join their centers to $z_0$ by polygonal paths in $\Omega$, and let $K$ be the union of the closed discs and the path images. By \cref{lem:access}, applied with the base point $z_0$ of \cref{eq:gdef}, fix distinct boundary points $x_1,\ldots,x_m$, peaking functions $u_l$, and access paths $\alpha_l$ from $z_0$, whose lengths we may assume to be at most $D$ by enlarging $D$, and by \cref{lem:peaking} polynomials $P_{l,n}$ for the pairs $(x_l,u_l)$. For each $l$ the set $K_l=K\cup\bigcup_{l'\ne l}\alpha_{l'}([0,1])$ is a compact subset of $\ol{\Omega}$ avoiding $x_l$, so $\max_l\sup_{K_l}|u_l|<1$, and $c=\min\{\tfrac15,\tfrac12(1-\max_l\sup_{K_l}|u_l|)\}$ is positive and at most the constant $c_{K_l}$ of \cref{lem:peaking} for $K_l$ and $u_l$ for every $l$, so that \cref{eq:peakcompact} gives $|P_{l,n}|\le e^{-cn}$ on $K_l$ for every $l$ and $n\ge1/c$. For each $n\ge1/c$ choose $y_{l,n}\in\alpha_l([0,1))\subseteq\Omega$ with $|P_{l,n}(y_{l,n})|\ge\tfrac12e^{-3}$, possible since $n\ge5$, the polynomial $P_{l,n}$ is continuous at $x_l=\alpha_l(1)$, and $|P_{l,n}(x_l)|\ge e^{-3}$. This gives \cref{it:peakdata-size}, as $|P_{l,n}|\le1$ on $\ol{\Omega}$ by \cref{lem:peaking}, and \cref{it:peakdata-decay} follows from the decay on $K_l$, since $K_\Phi\subseteq K\subseteq K_l$ and $y_{l',n}\in\alpha_{l'}([0,1))\subseteq K_l$ for $l'\ne l$. All estimates below are uniform in these choices, and we suppress the $n$-dependence by writing $y_l=y_{l,n}$.

  For \cref{it:peakdata-prim}, fix $i\in\II$ and $l\in\{1,\ldots,m\}$, let $n\ge1/c$, put $g=P_{l,n}$ and $\theta=Q_ig$, and let $\Theta$ and $\xi$ be the primitives of the statement. The primitives are path-independent, so they may be evaluated along paths chosen per point, on which $g$ is controlled. Along a path in $\ol{\Omega}$ starting at $z_0$, where $\cmin\le|\fii_i'|\le\cmax$, the representation \cref{eq:Qa} bounds $|\theta|$ at the endpoint by $\cmax/\cmin$ times the integral of $|g|$ along the path. For $z\in K_\Phi$ take the path of length at most $D$ from $z_0$ to $z$ in $K$, where $|g|\le e^{-cn}$: the subpaths obeying the same bound, $|\theta|\le(\cmax/\cmin)De^{-cn}$ at every point of the path, hence $|\Theta(z)|\le(\cmax/\cmin)D^2e^{-cn}$ and $|\xi(z)|\le(\cmax^2/\cmin)D^3e^{-cn}$. For $l'\ne l$ take the initial part of $\alpha_{l'}$ ending at $y_{l'}$, a path in $\Omega$ of length at most $D$ on which likewise $|g|\le e^{-cn}$, so that the same three bounds hold at $y_{l'}$. Along $\alpha_l$ itself $|g|\le1$, and for $\vartheta\in(0,1)$ and $n\ge2/\vartheta$ the bound \cref{eq:peaklocal} gives $|g|\le e^{-\vartheta n/2}$ outside $R_\vartheta=\{z\in\ol{\Omega} : |u_l(z)|>1-\vartheta\}$, so the integral of $|g|$ along $\alpha_l$, or along any of its initial parts, is at most $De^{-\vartheta n/2}+\ell_l(\vartheta)$, where $\ell_l(\vartheta)$ is the length of the part of $\alpha_l$ inside $R_\vartheta$. Let $\mu_l$ be the finite arc-length measure on $[0,1]$ induced by $\alpha_l$. Its distribution function is the arc-length function of the continuous rectifiable path $\alpha_l$, hence it is continuous and $\mu_l(\{1\})=0$. As $\vartheta\to0$, the sets $R_\vartheta$ decrease to $\{x_l\}$ because $|u_l|<1$ on $\ol{\Omega}\setminus\{x_l\}$, and the injectivity of $\alpha_l$ gives $\alpha_l^{-1}(R_\vartheta)\downarrow\{1\}$. Therefore, by continuity from above, $\ell_l(\vartheta)=\mu_l(\alpha_l^{-1}(R_\vartheta))\to0$. Taking $\vartheta=2/\sqrt n$, admissible as $n\ge5$, gives $|\theta|\le(\cmax/\cmin)\tau_n$ along $\alpha_l$ with $\tau_n=De^{-\sqrt n}+\max_l\ell_l(2/\sqrt n)\to0$, hence $|\Theta(y_l)|\le(\cmax/\cmin)D\tau_n$ and $|\xi(y_l)|\le(\cmax^2/\cmin)D^2\tau_n$. Since $D\ge1$ and $\cmax<1$, every bound on $\Theta$ and $\xi$ obtained, on $K_\Phi$ and at the points $y_1,\ldots,y_m$, is at most
  \begin{equation*}
    \kappa_n = \frac{\cmax}{\cmin}D^3(e^{-cn}+\tau_n),
  \end{equation*}
  which tends to $0$ as $n\to\infty$ and depends on neither $i$ nor $l$, the paths $\alpha_l$ and the constants $c$ and $D$ having been fixed before $n$. This proves \cref{it:peakdata-prim}.
\end{proof}

\subsection{Density on Jordan domains} \label{sec:gen-plane-dense}

The parametrization of \cref{sec:gen-plane-chart} and the peaking polynomials of \cref{sec:gen-plane-peak} now combine, through the counting lemma \cref{lem:avoidance}, into the density of both gap classes. \Cref{ex:nonconvex} shows that the method, and not merely this proof of it, fails on a multiply connected domain.

The density is proved by transversality for the Schwarzian projections; the pre-Schwarzian gap class then inherits it through the inclusion \cref{eq:GSsubsetGT}, as on the line in \cref{thm:C1denseS}. The argument has two halves. \Cref{prop:transversal} builds from the data of \cref{lem:peakdata} a finite-dimensional family of systems through $\Phi$ along which the differences of the Schwarzian projections, read at the evaluation points, form a map that is Lipschitz in the pair of words and satisfies a small-ball estimate in the parameter, uniformly over the pair space. Perturbing one generator by a polynomial peaking at one of $m$ boundary points moves the difference of two Schwarzian projections, read at the corresponding evaluation point, by the value of that polynomial there, up to an error that vanishes as the peak sharpens; the $m\times m$ matrix of these derivatives is then a small perturbation of an invertible diagonal one, so the parameter map is a submersion, uniformly over the pair space, and the small-ball estimate follows from \cref{lem:smallball}. \Cref{thm:C2dense} then applies \cref{lem:avoidance}, which makes the parameters admitting a coincidence a null set as soon as $2m$ exceeds the Minkowski dimension of the pair space. The openness is \cref{prop:cont}, needing neither hypothesis.

The proposition below is the first half. Its statement records, besides the two estimates on $F$, that the family comes from the parametrization \cref{eq:gdef} with a shift depending linearly on the parameter, and that it converges to $\Phi$ in $\mathcal C^2$ on the neighborhood $U$ of \cref{lem:collar} and not only on $\ol{\Omega}$; \cref{sec:dimension} needs the convergence on $U$.

\begin{proposition} \label{prop:transversal}
  Let $\Omega$ be a Jordan domain, let $\Omega'$ be simply connected, let $\Phi\in\mathcal S$, let $m\in\N$, let $\lambda\in(\cmax,1)$, and put $M=2mN$. Then there are points $y_1,\ldots,y_m\in\Omega$, a ball $B\subseteq\R^M$ centered at $0$, a constant $C<\infty$, and tuples $\theta^{(1)},\ldots,\theta^{(M)}$ of functions holomorphic on $\Omega'$ such that, for $t\in B$, the system $\Phi_t=(g_i^t)_{i\in\II}$ that \cref{eq:gdef} builds from $\theta^t=\sum_{j=1}^{M}t_j\theta^{(j)}$, so that $T^{\Phi_t}_i=T_i+\theta^t_i$ and $\Phi_0=\Phi$, belongs to $\mathcal S$ and satisfies $\|g_i^t-\fii_i\|_{\mathcal C^2(\ol{U})}\le C|t|$ for every $i\in\II$, where $U$ is the neighborhood of \cref{lem:collar}, so that in particular $d_2(\Phi_t,\Phi)\le C|t|$. Moreover, the map $F\colon B\times Z\to\C^m=\R^{2m}$ defined by
  \begin{equation} \label{eq:Fdef}
    F(t,(\iii,\jjj)) = (S^{\Phi_t}_\iii(y_l)-S^{\Phi_t}_\jjj(y_l))_{l=1}^{m}
  \end{equation}
  satisfies
  \begin{equation} \label{eq:Fsmallball}
    \LL^M(\{t\in B : |F(t,w)|\le u\}) \le Cu^{2m}
  \end{equation}
  for all $u>0$ and $w\in Z$, and $|F(t,w)-F(t,w')|\le C\varrho_{\lambda^2}(w,w')$ for all $t\in B$ and $w,w'\in Z$, where $\varrho_{\lambda^2}$ is the metric of \cref{lem:paircover}.
\end{proposition}

\begin{proof}
  Below $C$ denotes a finite constant, whose value may change from one occurrence to the next, depending only on $\Phi$, on $m$, and on the integer $n$ fixed below. Put $\rho_\Phi=\min_{i \in \II}\dist(\fii_i(\ol{\Omega}),\partial\Omega)$, which is positive, each $\fii_i(\ol{\Omega})$ being a compact subset of the open $\Omega$; hence every tuple $\Psi=(\psi_i)_{i\in\II}$ of maps conformal on $\Omega'$ with $\max_{i\in\II}\|\psi_i-\fii_i\|<\rho_\Phi/2$ and $|\psi_i'|<1$ on $\ol{\Omega}$ maps $\ol{\Omega}$ into $\Omega$, so $\Psi$ lies in $\mathcal S$. Let $c$, $\kappa_n$, $P_{l,n}$, and $y_{l,n}$ be as in \cref{lem:peakdata}, with $K_\Phi=\bigcup_{i\in\II}\fii_i(\ol{\Omega})$, and put $\sigma_0=\tfrac14e^{-3}$. All estimates below are uniform in $n\ge1/c$ until $n$ is fixed, and we suppress the $n$-dependence by writing $y_l=y_{l,n}$.

  Fix $l\in\{1,\ldots,m\}$ and $i\in\II$, put $g=P_{l,n}$, and let $h=h^{(i,l)}$ be the source tuple with $h_i=g$ and $h_j=0$ for $j\ne i$. A source $h_j$ determines the shift $\theta_j=Q_jh_j$ that \cref{eq:gdef} applies to the dual projection of the letter $j$, with $Q_j$ being the operator of \cref{eq:Qa}, and \cref{lem:linearizeS} builds from $\theta_j$ the primitives $\Theta_j$ of $\theta_j$ and $\xi_j$ of $\fii_j'\Theta_j$, both vanishing at $z_0$, and the correction $\Gamma_{j,\lll}$ that \cref{eq:Slinseries} adds to $h_j$. All four vanish when $h_j$ does, so that $\theta_i=Q_ig$ while $\theta_j=\Theta_j=\xi_j=0$ and $\Gamma_{j,\lll}=0$ for $j\ne i$, and $\Theta_i$ and $\xi_i$ are the primitives of \cref{lem:peakdata}\cref{it:peakdata-prim}, so that $|\Theta_i|\le\kappa_n$ and $|\xi_i|\le\kappa_n$ on $K_\Phi\cup\{y_1,\ldots,y_m\}$.

  Let $\iii$ be a nonempty word, let $l'\in\{1,\ldots,m\}$, and write $v_\iii$ for the derivative in \cref{lem:linearizeS} along the source tuple $h=h^{(i,l)}$ fixed above, the left-hand side of \cref{eq:Slinseries}. By \cref{eq:wordops}, the $k$-th term of \cref{eq:Slinseries} at $y_{l'}$ is $(f_{\iii|_k}'(y_{l'}))^2$ times the value of $h_{i_{k+1}}+\Gamma_{i_{k+1},\sigma^{k+1}\iii}$ at $f_{\iii|_k}(y_{l'})$, a point of $K_\Phi$ when $k\ge1$ and equal to $y_{l'}$ when $k=0$. There $|h_{i_{k+1}}|\le e^{-cn}$ when $k\ge1$ by \cref{lem:peakdata}\cref{it:peakdata-decay}, the value of $h_{i_1}$ at $y_{l'}$ is $\delta_{i,i_1}g(y_{l'})$ with $\delta_{i,i_1}$ being the Kronecker delta, and $|\Gamma_{i_{k+1},\sigma^{k+1}\iii}|\le\cmax^2(2K_S+C_S)\kappa_n$ for every $k\ge0$ by the pointwise bound of \cref{lem:linearizeS} and the bounds on $\Theta_i$ and $\xi_i$ just recorded. Since $|f_{\iii|_k}'|\le\cmax^k$, summing over $k$ gives
  \begin{equation} \label{eq:keyS}
    |v_\iii(y_{l'})-\delta_{i,i_1}g(y_{l'})|\le\gamma_n = \frac{\cmax^2(e^{-cn}+(2K_S+C_S)\kappa_n)}{1-\cmax^2},
  \end{equation}
  and $\gamma_n$, which depends on none of $\iii$, $i$, $l$, and $l'$, tends to $0$ as $n\to\infty$.

  Now let $w=(\iii,\jjj)\in Z$, so that the first letters satisfy $i_1\ne j_1$. By \cref{lem:linearizeS}, for every tuple $h=(h_j)_{j\in\II}$ of functions holomorphic on $\Omega'$ the derivative at $t=0$, along the one-parameter family that the lemma builds from $h$, of the difference of the Schwarzian projections of $\iii$ and $\jjj$ is a function $D_wh$, which is complex-linear in $h$ because the right-hand side of \cref{eq:Slinseries} is, and which equals $v_\iii-v_\jjj$ when $h=h^{(i,l)}$. Applying \cref{eq:keyS} with $i=i_1$ to $\iii$, whose first letter is $i_1$, and to $\jjj$, whose first letter is not, gives $|D_wh^{(i_1,l)}(y_{l'})-P_{l,n}(y_{l'})|\le2\gamma_n$. Hence the complex $m\times m$ matrix $M_w=(D_wh^{(i_1,l)}(y_{l'}))_{l',l}$ is $A+E_w$, where $A$ is the diagonal matrix with entries $P_{l,n}(y_l)$, of moduli at least $\tfrac12e^{-3}$ by \cref{lem:peakdata}\cref{it:peakdata-size}, and every entry of $E_w$ has modulus at most $2\gamma_n+e^{-cn}$ by \cref{lem:peakdata}\cref{it:peakdata-decay}. Fix $n\ge1/c$ so large that $m(2\gamma_n+e^{-cn})\le\sigma_0$; then $\|E_w\|\le\sigma_0$, the operator norm of an $m\times m$ matrix being at most $m$ times the largest modulus of its entries, while the least singular value of the diagonal matrix $A$ is the least modulus of its diagonal entries, which is at least $2\sigma_0$; the perturbation bound \cref{eq:sigmamin} therefore gives $\sigma_{\min}(M_w)\ge\sigma_0$ for every $w\in Z$. Let $\mathcal V$ be the real span of the tuples $h^{(k,l)}$ and their imaginary multiples $ih^{(k,l)}$, $k\in\II$ and $l\in\{1,\ldots,m\}$, of real dimension $M=2mN$, the tuples being linearly independent since those with distinct $k$ have disjoint supports and the matrix $(P_{l,n}(y_{l'}))_{l',l}$ is invertible by the same singular-value estimate, with the inner product making these tuples orthonormal, and consider $\Xi_w\colon\mathcal V\to\R^{2m}=\C^m$, $h\mapsto(D_wh(y_l))_{l=1}^m$. As $D_w$ is complex-linear, on the subspace $\mathcal V_{i_1}$ spanned by the $h^{(i_1,l)}$ and $ih^{(i_1,l)}$ the map $\Xi_w$ is the realification of $\zeta\mapsto M_w\zeta$ on $\C^m$, whose singular values are those of $M_w$, each repeated twice, so $\Xi_w|_{\mathcal V_{i_1}}$ is onto with least singular value at least $\sigma_0$. For a subspace $W\subseteq\mathcal V$ with orthogonal projection $P_W$ one has $(\Xi_w|_W)^{*}=P_W(\Xi_w|_{\mathcal V})^{*}$, so $\sigma_{\min}(\Xi_w|_W)\le\sigma_{\min}(\Xi_w|_{\mathcal V})$: enlarging the domain cannot lower the least singular value. Thus $\Xi_w|_{\mathcal V}$ is onto with least singular value at least $\sigma_0$, uniformly in $w\in Z$.

  Enumerate the tuples $h^{(k,l)}$ and $ih^{(k,l)}$ as $e_1,\ldots,e_M$, put $\theta^{(j)}=(Q_i(e_j)_i)_{i\in\II}$, and for $t\in\R^M$ let $\Phi_t=(g_i^t)_{i\in\II}$ be built by \cref{eq:gdef} from $\theta^t=\sum_{j=1}^{M}t_j\theta^{(j)}$, so that $T^{\Phi_t}_i=T_i+\theta^t_i$ and, by \cref{eq:Sgenshift} applied at parameter value $1$ to the tuple $\sum_{j=1}^{M}t_je_j$, whose shift is $\theta^t$ by the linearity of the operators $Q_i$, $S_{g_i^t}=S_i+\sum_{j=1}^{M}t_j(e_j)_i-\tfrac12(\theta^t_i)^2$; from now on $\Phi_t$ and $g_i^t$ refer to this $M$-parameter family, the one-parameter families of \cref{lem:linearizeS} entering only through the derivatives $D_w$ defined above. Let $F$ be the map \cref{eq:Fdef} with these points $y_l$. The sources being polynomials, each $\theta^{(j)}$ is holomorphic on $\Omega'$, and $\Phi_t$ is the system that \cref{lem:perturbed} builds from the tuple $\theta^t$ at parameter value $1$, whose constant $M_U$ is at most $C|t|$, the primitive $\Theta^t_i$ of $\theta^t_i$ vanishing at $z_0$ being linear in $t$; the construction and the bounds \cref{eq:gbounds} therefore extend to complex $t\in\C^M$ and read $|(g_i^t)'|\le c_Ue^{C|t|}$ and $|g_i^t-\fii_i|\le C(e^{C|t|}-1)$ on $U$, and on $\ol{\Omega}$ also $|(g_i^t)'|\le|\fii_i'|e^{C|t|}$. For real $t$ with $|t|\le1$ this gives $\|g_i^t-\fii_i\|_{\mathcal C^2(\ol{U})}\le C|t|$: the displacement is at most $C(e^{C|t|}-1)\le C|t|$, the derivatives differ by $(g_i^t)'-\fii_i'=\fii_i'(e^{\Theta^t_i}-1)$ with $|\Theta^t_i|\le M_U\le C|t|$ on $\ol{U}$, and the second derivatives by $((g_i^t)'-\fii_i')T_i+(g_i^t)'\theta^t_i$, through $(g_i^t)''=(g_i^t)'(T_i+\theta^t_i)$, where $T_i=\fii_i''/\fii_i'$ is bounded on $\ol{U}\subseteq\Omega'$ and $|\theta^t_i|\le C|t|$ on $U$ by \cref{eq:CQ}; in particular $d_2(\Phi_t,\Phi)\le C|t|$, as $\ol{\Omega}\subseteq U$.

  Fix a ball $B_0\subseteq\R^M$ centered at $0$ and of radius at most $1$ so small that, for $t\in B_0$, the system $\Phi_t$ lies in $\mathcal S$ with $\cmax(\Phi_t)\le\lambda$, which the margin $\rho_\Phi$ and these bounds guarantee, and so small that, by the margin $(1-c_U)\eps$ of \cref{lem:collar}\cref{it:collar-invariance}, on a complex polydisc $\Pi\subseteq\C^M$ containing a neighborhood of the closure of $B_0$ the maps $g^t_i$ send $U$ into $U$ with $\sup_{z\in U}|(g^t_i)'(z)|\le q_0$ for some $q_0<1$. Exactly as in the proof of \cref{lem:linearizeS}, the series \cref{eq:Sline} formed from the generators of $\Phi_t$ then converges uniformly on $\Pi\times\ol{\Omega}$ and uniformly in the word, so $F$ is holomorphic on $\Pi$ and bounded uniformly in $w$, and the Cauchy estimates give a constant $\Lambda<\infty$ bounding the operator norms of its first two parameter derivatives on $B_0\times Z$. By \cref{lem:linearizeS}, applied to the tuple $e_j$, the partial derivative of $F$ with respect to $t_j$ at $t=0$ is $(D_we_j(y_l))_{l=1}^m=\Xi_we_j$, so, $F$ being differentiable in $t$, $D_tF(0,w)=\Xi_w|_{\mathcal V}$ under the isometry $t\mapsto\sum_jt_je_j$ of $\R^M$ onto $\mathcal V$. The bound on the second parameter derivatives makes the map $t\mapsto D_tF(t,w)$ Lipschitz with constant $\Lambda$ on the convex set $B_0$ for every $w$, and the perturbation bound \cref{eq:sigmamin} shows that the least singular value of $D_tF(t,w)$ is at least $\sigma_0/2$ for every $(t,w)\in\ol{B}\times Z$, where $B$ is an open ball centered at $0$ with $\ol{B}\subseteq B_0$ and radius at most $\sigma_0/(2\Lambda)$; the polydisc and every constant below depend on $n$, which has been fixed. For every $w\in Z$ the map $t\mapsto F(t,w)$ on $B$ therefore satisfies the hypotheses of \cref{lem:smallball} with $2m$ in place of $r$ and $\sigma_0/2$ in place of $\sigma_*$, and this lemma gives \cref{eq:Fsmallball} for all $u>0$, with a constant that does not depend on $w$.

  Finally, the prefix estimate \cref{eq:dualbounds2} for $\Phi_t$, whose constant is bounded for $t\in B$ since $\cmax(\Phi_t)\le\lambda$ and $\max_i\|S_{g_i^t}\|$ is bounded, the sources satisfying $|P_{l,n}|\le1$ on $\ol{\Omega}$ by \cref{lem:peakdata}\cref{it:peakdata-size} and $\sup_U|Q_i(e_j)_i|<\infty$ by \cref{eq:CQ}, makes $F$ Lipschitz in $w$ for the metric $\varrho_{\lambda^2}$, uniformly in $t\in B$: for $w=(\iii,\jjj)$ and $w'=(\iii',\jjj')$ in $Z$, each coordinate of $F(t,w)-F(t,w')$ is bounded in modulus by $\|S^{\Phi_t}_\iii-S^{\Phi_t}_{\iii'}\|+\|S^{\Phi_t}_\jjj-S^{\Phi_t}_{\jjj'}\|\le C\lambda^{2\min\{|\iii\land\iii'|,|\jjj\land\jjj'|\}}$, each of the two norms vanishing when its own two words coincide, so that $|F(t,w)-F(t,w')|\le\sqrt{m}C\varrho_{\lambda^2}(w,w')$. The ball $B$, the points $y_l$, the tuples $\theta^{(j)}$, and the largest of the constants $C$ above therefore have the properties claimed.
\end{proof}

The second half is the counting, and it is where the number $m$ of peaks is fixed: \cref{lem:avoidance} needs the dimension $2m$ of the target of $F$ to exceed the covering exponent of the pair space in the metric $\varrho_{\lambda^2}$, which \cref{lem:paircover} computes to be $s=\log N/\log(1/\lambda)$. The second statement of the theorem is the form in which \cref{sec:dimension} uses the density.

\begin{theorem} \label{thm:C2dense}
  Let $\Omega$ be a Jordan domain and let $\Omega'$ be simply connected. Then the sets $\mathcal G_T$ and $\mathcal G_S$ are dense in $(\mathcal S,d_2)$. Moreover, let $\Phi\in\mathcal S$, let $\lambda\in(\cmax,1)$, let $m$ be an integer with $m>\log N/(2\log(1/\lambda))$, and put $M=2mN$. If $y_1,\ldots,y_m\in\Omega$, if $B\subseteq\R^M$ is a ball, if $C<\infty$, and if $\Phi_t\in\mathcal S$ for $t\in B$ are such that the map $F$ of \cref{eq:Fdef} satisfies \cref{eq:Fsmallball} for all $u>0$ and $w\in Z$, and 
  \begin{equation*}
    |F(t,w)-F(t,w')|\le C\varrho_{\lambda^2}(w,w')
  \end{equation*}
  for all $t\in B$ and $w,w'\in Z$, then $\Phi_t\in\mathcal G_S$ for $\LL^M$-almost every $t\in B$.
\end{theorem}

\begin{proof}
  Let $\Phi\in\mathcal S$, let $\lambda\in(\cmax,1)$, put $s=\log N/\log(1/\lambda)$, and let $m$, $y_1,\ldots,y_m$, $B$, $C$, and $\Phi_t$ be as in the second statement, so that $2m>s$. By \cref{lem:paircover} with $\beta=\lambda^2$, for which $2\log N/\log(1/\beta)=s$, the space $Z$ is compact in the metric $\varrho_{\lambda^2}$ and, for every $\delta\in(0,1)$, covered by at most $2\lambda^{-2s}\delta^{-s}$ sets of diameter at most $\delta$. Together with the assumed properties of $F$, this shows that the map $F$ of \cref{eq:Fdef} satisfies the hypotheses of \cref{lem:avoidance} with $Y=Z$, $r=2m$, $C_1=C_2=C$, and $C_3=2\lambda^{-2s}$, and $r>s$, so the set of $t\in B$ for which $F(t,w)=0$ for some $w\in Z$ is a null set. For every other $t\in B$ and every $(\iii,\jjj)\in Z$ the projections $S^{\Phi_t}_\iii$ and $S^{\Phi_t}_\jjj$ differ at one of the points $y_l\in\Omega$, so $\Phi_t\in\mathcal G_S$ by \cref{lem:gap}\cref{it:gap-schw}, the system $\Phi_t$ being a member of $\mathcal S$. This proves the second statement. For the first, fix an integer $m>s/2$, and let $y_1,\ldots,y_m$, $B$, $C$, and $\Phi_t$, $t\in B$, be the points, the ball, the constant, and the systems that \cref{prop:transversal} provides; they satisfy the hypotheses just made, and in addition $B$ is centered at $0$ and $d_2(\Phi_t,\Phi)\le C|t|$. By the second statement, $\Phi_t\in\mathcal G_S$ for almost every $t\in B$, so there are such $t$ arbitrarily close to $0$; thus members of $\mathcal G_S$ lie arbitrarily close to $\Phi$, and since $\mathcal G_S\subseteq\mathcal G_T$ by \cref{eq:GSsubsetGT}, the same holds for $\mathcal G_T$.
\end{proof}

We are now ready to prove the second item of \cref{thm:genericity}.

\begin{proof}[Proof of \cref{thm:genericity}\cref{it:gen-plane}]
  Let $\Omega$ be a Jordan domain and let $\Omega'$ be simply connected. The sets $\mathcal G_S$ and $\mathcal G_T$ are open in $(\mathcal S,d_2)$ by \cref{prop:cont} and dense by \cref{thm:C2dense}, and by \cref{lem:gap} the members of $\mathcal G_S$ satisfy the strong exponential separation condition modulo M\"obius maps and those of $\mathcal G_T$ the strong exponential separation condition.
\end{proof}

\Cref{thm:C2dense} assumes a Jordan domain with simply connected extension domain, and the example below locates where multiple connectivity breaks the argument.

\begin{example} \label{ex:nonconvex}
  The topological hypotheses of \cref{thm:C2dense} cannot be dropped without replacement. Let $\Omega = \{z\in\C : 1<|z|<2\}$ and $\Omega' = \{z\in\C : \tfrac12<|z|<3\}$, and consider, for small $\eps>0$, the M\"obius generators $\fii_1(z) = \tfrac32-\eps/z$ and $\fii_2(z) = -\tfrac32-\eps/(z-\tfrac3{10})$, conformal and injective on $\Omega'$, their poles $0$ and $\tfrac3{10}$ lying outside $\Omega'$. For $\eps<\tfrac18$ the image $\fii_i(\ol{\Omega})$ lies in the disc of radius $2\eps$ about $(-1)^{i-1}\tfrac32$, hence in $\Omega$, and $\cmax=\tfrac{100}{49}\eps<1$, so $\Phi=(\fii_1,\fii_2)$ is a conformal iterated function system on the pair $(\Omega,\Omega')$ to which the density theorem does not apply, both domains being multiply connected.

  On the annular $\Omega$ the parametrization \cref{eq:gdef} is unavailable. The argument principle gives $\oint_\gamma T_g\dd z = 2\pi ik$ for every conformal $g$ on $\Omega$ and every cycle $\gamma$ in $\Omega$, where $k\in\Z$ is the winding number of the zero-free $g'\circ\gamma$ about the origin, so the periods of pre-Schwarzians are locked in $2\pi i\Z$: along every family of conformal maps continuous in the metric \cref{eq:d2metric} the period is constant, and every infinitesimal direction $\theta$ realized by such a family satisfies $\oint_\gamma\theta\dd z=0$, so no such family realizes a direction with a nonzero period. The vanishing of the period does not suffice either: a conformal $g$ on $\Omega$ with $T_g=T_{\fii_1}+t$ for some $t\ne0$ would have $g'=c\fii_1'e^{tz}$ with $c\ne0$, the zero-free ratio $g'/\fii_1'$ having constant logarithmic derivative $t$, yet 
  \begin{equation*}
    \oint_{|z|=3/2}\fii_1'(z)e^{tz}\dd z=\eps\oint_{|z|=3/2}z^{-2}e^{tz}\dd z=2\pi i\eps t\ne0,  
  \end{equation*}
  so $g'$ admits no single-valued primitive and no such $g$ exists; the exact shifts $T^{\Phi_t}_i=T_i+t\theta_i$ that the parametrization \cref{eq:gdef} produces are therefore unavailable already for the constant polynomial direction $\theta_1\equiv1$, and \cref{thm:C2dense} rests on the same parametrization, which enters \cref{prop:transversal} directly and \cref{lem:peakdata} through the operator \cref{eq:Qa}.
\end{example}

\begin{question} \label{q:nonconvex}
  Does the conclusion of \cref{thm:C2dense} hold on multiply connected planar domains? \Cref{ex:nonconvex} refutes only the method, and only its parametrization: the annulus admits peaking functions at every boundary point $x$, namely $z\mapsto\exp(n(z\ol{x}/4-1))$ when $|x|=2$ and $z\mapsto\exp(n(x/z-1))$ when $|x|=1$, whereas the period constraints are locally constant in $(\mathcal S,d_2)$, being integer multiples of $2\pi i$ that vary continuously with the system, so they can never force two dual projections to coincide, and a parametrization corrected by finitely many period conditions might restore the transversality argument on finitely connected domains. The hypotheses on $\Omega$ are not free of charge either: for the slit annulus, a simply connected domain whose closure is the closed annulus, the complement of $\ol{\Omega}$ is disconnected, so Mergelyan's theorem is unavailable, and a function holomorphic on a simply connected $\Omega'$, which necessarily contains the hole, cannot peak at an inner-boundary point, by the maximum principle.
\end{question}

\section{Genericity in higher dimensions} \label{sec:gen-higher}

Throughout this section $d\ge3$. We keep the conventions of \cref{sec:C3-poles}, in particular $C_0=2\max_{i \in \II}\|T_{\fii_i}\|/(1-\cmax)$, the pole fields $T_p$, and the norm convention for higher derivatives. When $\max_{i \in \II}\|T_{\fii_i}\|>0$, the pole $p_\iii\in\mathcal P$ is the one supplied by \cref{lem:poledict}; when $\max_{i \in \II}\|T_{\fii_i}\|=0$, we set $p_\iii=\infty$ for every $\iii\in\II^\N\cup\II^*$.

By Liouville's theorem every member of $\mathcal S$ is a tuple of M\"obius transformations, so $\mathcal S$ is inside a finite-dimensional family, and the pre-Schwarzian gap of $\mathcal G_T$ is a pole-separation condition. \Cref{sec:gen-higher-open} proves that equivalence in \cref{lem:gap3} and the openness of $\mathcal G_T$ in \cref{prop:cont3}, and records how conjugation by a similarity of $\R^d$ acts on the gap class. \Cref{sec:gen-higher-chordal} then develops the chordal geometry and the pole cocycle on which the rest of the section runs. \Cref{sec:gen-higher-dense} proves the density under a threshold on the contraction ratio in \cref{thm:C3dense}, recasts its hypotheses in \cref{cor:C3dense}, and deduces \cref{thm:genericity}\cref{it:gen-higher}. Some hypothesis of this kind is unavoidable, and \cref{sec:gen-higher-obstruction} locates the obstruction in a packing bound on the sphere: for $d\ge3$ the gap class is not dense in general, and on convex domains it misses an open set of systems. What survives with no hypothesis at all is the gap for finite words, and \cref{sec:gen-higher-finite} proves it dense among the systems whose generators have pairwise distinct poles, then examines the range of contraction ratios that the density theorem and the obstruction leave between them.

\subsection{The gap as a pole separation, and openness} \label{sec:gen-higher-open}

\Cref{lem:gap3} is what makes every argument of this section a statement about finitely many points of $\RS^d$ rather than about functions on $\ol{\Omega}$, and \cref{prop:cont3} is the openness that follows, the counterpart of \cref{prop:cont1} of \cref{sec:gen-line-open} and of \cref{prop:cont} of \cref{sec:gen-plane-open}.

\begin{lemma} \label{lem:gap3}
  The system $\Phi$ belongs to $\mathcal G_T$ if and only if $p_\iii\ne p_\jjj$ for every $(\iii,\jjj)\in Z$. Moreover, $\max_{i \in \II}\|T_{\fii_i}\|>0$ for every $\Phi\in\mathcal G_T$ when $N\ge2$.
\end{lemma}

\begin{proof}
  Put $\beta_1=\max_{i \in \II}\|T_{\fii_i}\|$. If $\beta_1=0$, then every generator is a similarity by \cref{lem:kernel}\cref{it:tkernel}, and hence every dual projection vanishes. Since $N\ge2$, the pair $(1,2)$ belongs to $Z$ and satisfies $T_1=T_2$ and $p_1=p_2=\infty$. Thus $\Phi\notin\mathcal G_T$, the pole-separation condition fails, and the final assertion follows.

  Suppose that $\beta_1>0$. By the compactness of $Z$ and the continuity in \cref{eq:dualbounds}, the system $\Phi$ belongs to $\mathcal G_T$ if and only if $T_\iii\not\equiv T_\jjj$ for every $(\iii,\jjj)\in Z$. By \cref{lem:poledict}, this is equivalent to $p_\iii\ne p_\jjj$ for every $(\iii,\jjj)\in Z$.
\end{proof}

Openness rests, as in the plane, on a Lipschitz estimate for the dual projections whose constant does not depend on the word, the gap $\Delta_T$ being an infimum over $Z$ of distances between such projections; only $\mathcal G_T$ is at stake here, the Schwarzian gap class being defined only in dimensions one and two. What has to be replaced is the holomorphic neighborhood $U$ of \cref{lem:collar}, which has no counterpart in $\R^d$, and its two services are supplied separately: the two-point identity \cref{eq:twopoint} makes every generator contract distances on $\ol{\Omega}$ pointwise, so the orbits of two nearby systems stay close with no neighborhood to integrate over, while the pole gap of \cref{lem:CU} holds the generator poles off $\ol{\Omega}$ and so bounds the second derivatives of the generators on a fixed neighborhood of $\ol{\Omega}$, on which the mean value inequality then runs along the segments joining the two orbits.

\begin{proposition} \label{prop:cont3}
  Let $\Phi=(\fii_i)_{i \in \II}$ be a conformal iterated function system on $\Omega \subset \R^d$, $d \ge 3$. Then there are $r_0>0$ and $C<\infty$ such that 
  \begin{equation*}
    \|T^{\Phi}_\iii-T^{\Psi}_\iii\|\le Cd_2(\Phi,\Psi)
  \end{equation*}
  for every $\iii\in\II^\N\cup\II^*$ and every $\Psi\in\mathcal S$ with $d_2(\Phi,\Psi)<r_0$. Consequently, the set $\mathcal G_T$ is open in $(\mathcal S,d_2)$.
\end{proposition}

\begin{proof}
  Write $\beta_1=\max_{i \in \II}\|T_{\fii_i}\|$, let $\lambda \in (\cmax,1)$, and choose $r_0$ so small that $\cmax(\Psi)\le \lambda$ and $\cmin(\Psi)\ge\cmin/2$ on the $r_0$-ball; $C$ below depends only on $\Phi$. The proof of \cref{prop:cont} transfers with the following changes. The orbit points move by at most $(1-\lambda)^{-1}d_2$: the telescoping only needs every generator of either system to contract distances on $\ol{\Omega}$ with the constant $\lambda$, which the two-point identity \cref{eq:twopoint} supplies pointwise from the bound $\lambda$ on the conformal factors, no segment being involved. Shrink $r_0$ further so that $(1-\lambda)^{-1}r_0\le1/\beta_1$ when $\beta_1>0$. If $\beta_1>0$, then the segment joining any two points of $\ol{\Omega}$ at distance at most $(1-\lambda)^{-1}d_2$ stays in the closed $(1/\beta_1)$-neighborhood of $\ol{\Omega}$, which is compact, and this neighborhood keeps the distance $1/\beta_1$ from the pole of every non-similarity generator, whose distance from $\ol{\Omega}$ is at least $2/\beta_1$ by the identity $\|T_{\fii_i}\|=2/\dist(p_i,\ol{\Omega})$ from the proof of \cref{lem:CU} and the cap $\|T_{\fii_i}\|\le\beta_1$. If $\beta_1=0$, every generator is a similarity, so $D^2\fii_i\equiv0$ and $T_{\fii_i}\equiv0$ for every $i\in\II$; hence $\|D\fii_i(u)-D\fii_i(v)\|=0$ and $|T_{\fii_i}(u)-T_{\fii_i}(v)|=0$ for all $u,v\in\R^d$, and no neighborhood is needed.

  The matrix cocycles $Df_{\iii|_{k-1}}(x)=D\fii_{i_{k-1}}\cdots D\fii_{i_1}$ along the two orbits then differ by at most $Ck\lambda^{\,k-1}d_2$, replacing one factor at a time, the unreplaced factors contributing at most $\lambda^{\,k-2}\le(2/\cmin)\lambda^{\,k-1}$ and each replacement error being at most $\|D\fii_{i_l}(u_l)-D\fii_{i_l}(v_l)\|+d_2\le Cd_2$; when $\beta_1>0$, the mean value inequality applies along the segment from $u_l$ to $v_l$, on which $\|D^2\fii_{i_l}\|$ is bounded, the finitely many fixed generators being real-analytic on the neighborhood, off their poles, while for $\beta_1=0$ the fixed derivatives are constant. For the generator terms, first $|T_{\fii_i}(u)-T_{\fii_i}(v)|\le6\beta_1^2|u-v|$: for a non-similarity generator the derivative of the pole field, read off \cref{eq:pole}, is bounded in norm by $6/|x-p_i|^2\le6\beta_1^2$ on the neighborhood, and the mean value inequality applies along the segment from $u$ to $v$, while a similarity generator has $T_{\fii_i}\equiv0$. Second, $\|T_{\fii_i}-T_{\psi_i}\|\le Cd_2$ by the estimate of \cref{eq:C3quot} run on the single maps: with $A=D\fii_i(x)$, $\widehat A=D\psi_i(x)$, $B_j=\partial_jD\fii_i(x)$, and $\widehat B_j=\partial_jD\psi_i(x)$, the componentwise formula \cref{eq:trace} and the resolvent identity give $|T_{\fii_i}(x)-T_{\psi_i}(x)|\le\sqrt d(\|A^{-1}\|\|\widehat A^{-1}\|\|A-\widehat A\|\max_j\|B_j\|+\|\widehat A^{-1}\|\max_j\|B_j-\widehat B_j\|)\le Cd_2$, using $\|A^{-1}\|\le\cmin^{-1}$, $\|\widehat A^{-1}\|\le2\cmin^{-1}$, and $\max_j\|B_j\|\le L_2\cmax$ from \cref{eq:scaling}. Summing $\sum_kk\lambda^{\,k-1}<\infty$ over the summands of \cref{eq:Tword} gives the bound for finite words. For infinite words, apply the finite-word bound to the prefixes and pass to their uniform limits from \cref{lem:dual}.

  Finally, suppose that $\Phi\in\mathcal G_T$ and put $\delta=\Delta_T(\Phi)>0$. Enlarge $C$ if necessary so that $C>0$, and let $\Psi\in\mathcal S$ satisfy $d_2(\Phi,\Psi)<\min\{r_0,\delta/(3C)\}$. For every $(\iii,\jjj)\in Z$, two applications of the preceding bound give $\|T^\Psi_\iii-T^\Psi_\jjj\|\ge\|T^\Phi_\iii-T^\Phi_\jjj\|-2Cd_2(\Phi,\Psi)>\delta/3$. Taking the infimum over $Z$ gives $\Delta_T(\Psi)\ge\delta/3>0$, so $\mathcal G_T$ is open.
\end{proof}

Moving the domain by a similarity of $\R^d$ rescales the gap condition but does not otherwise affect it, and the next lemma records this. For a similarity $S$ of $\R^d$ and $\Phi=(\fii_i)_{i\in\II}\in\mathcal S$ write $\Phi^S=(S^{-1}\circ\fii_i\circ S)_{i\in\II}$, and for $c\in\R^d$ write $\tau_c(x)=x+c$.

\begin{lemma} \label{lem:similarity}
  Let $S(x)=\lambda Ox+b$ be a similarity of $\R^d$, with $\lambda>0$, $O\in O(d)$, and $b\in\R^d$. Then $\Phi\mapsto\Phi^S$ is a bijection of $\mathcal S$ onto the system space of the pair $(S^{-1}(\Omega),S^{-1}(\Omega'))$, the norms of the conjugated systems being taken over $S^{-1}(\ol{\Omega})$, with
  \begin{equation*}
    \min\{\lambda,\lambda^{-1}\}d_2(\Phi,\Psi) \le d_2(\Phi^S,\Psi^S) \le \max\{\lambda,\lambda^{-1}\}d_2(\Phi,\Psi)
  \end{equation*}
  for all $\Phi,\Psi\in\mathcal S$, and $\cmax(\Phi^S)=\cmax(\Phi)$ and $\max_{i\in\II}\|T_{\fii^S_i}\|=\lambda\max_{i\in\II}\|T_{\fii_i}\|$. Moreover,
  \begin{equation*}
    T^{\Phi^S}_\iii=\lambda O^{\top}(T^\Phi_\iii\circ S) \qquad\text{and}\qquad p^{\Phi^S}_\iii=S^{-1}(p^\Phi_\iii)
  \end{equation*}
  for every $\iii\in\II^\N\cup\II^*$, where $S^{-1}(\infty)=\infty$. In particular $\Delta_T(\Phi^S)=\lambda\Delta_T(\Phi)$, so $\Phi\mapsto\Phi^S$ carries $\mathcal G_T$ onto the gap class of $(S^{-1}(\Omega),S^{-1}(\Omega'))$, and for a translation $S=\tau_c$ it is an isometry.
\end{lemma}

\begin{proof}
  The generator $\fii^S_i=S^{-1}\circ\fii_i\circ S$ is conformal on $S^{-1}(\Omega')$ and sends $S^{-1}(\ol{\Omega})$ into $S^{-1}(\Omega)$, and since $DS=\lambda O$ and $DS^{-1}=\lambda^{-1}O^{\top}$ are constant, the chain rule gives $D\fii^S_i(x)=O^{\top}D\fii_i(Sx)O$ and $D^2\fii^S_i(x)[u,v]=\lambda O^{\top}D^2\fii_i(Sx)[Ou,Ov]$, while $\fii^S_i(x)-\psi^S_i(x)=\lambda^{-1}O^{\top}(\fii_i(Sx)-\psi_i(Sx))$ for $\Phi,\Psi\in\mathcal S$. Hence $\|D\fii^S_i\|=\|D\fii_i\|\circ S$, so $\Phi^S$ lies in the system space of $(S^{-1}(\Omega),S^{-1}(\Omega'))$ with $\cmax(\Phi^S)=\cmax(\Phi)$, and the three supremum norms in \cref{eq:d2metric} are multiplied by $\lambda^{-1}$, $1$, and $\lambda$, which gives the two-sided bound; the map is onto because the same construction with $S^{-1}$ in place of $S$, on the system space of $(S^{-1}(\Omega),S^{-1}(\Omega'))$, inverts it. Two applications of the composition law \cref{eq:Tchain}, with $T_S\equiv0$ and $T_{S^{-1}}\equiv0$ by \cref{lem:mobius}, give $T_{\fii^S_i}=\lambda O^{\top}(T_{\fii_i}\circ S)$. Since the reversed composition of $\Phi^S$ associated to $\iii$ is $S^{-1}\circ f_\iii\circ S$ by \cref{eq:frev}, the same calculation gives $T^{\Phi^S}_\iii=\lambda O^{\top}(T^\Phi_\iii\circ S)$ for every finite word and, by the uniform convergence of \cref{lem:dual}, for every infinite word. Since $O^{\top}$ is an isometry and the supremum norms over $S^{-1}(\ol{\Omega})$ and over $\ol{\Omega}$ correspond under $S$, the generator norms and $\Delta_T$ are multiplied by $\lambda$. It follows that $\Phi\mapsto\Phi^S$ carries $\mathcal G_T$ onto the corresponding gap class; if $S=\tau_c$, then $\lambda=1$ and the metric bounds are equalities, so the map is an isometry. For the poles, if $\max_{i\in\II}\|T_{\fii_i}\|=0$, then every pole of either system is $\infty$ by convention. Otherwise every pole field satisfies $\lambda O^{\top}(T_p\circ S)=T_{S^{-1}(p)}$, trivially for $p=\infty$ and for finite $p$ because $Sx-p=\lambda O(x-S^{-1}(p))$, so $T^{\Phi^S}_\iii=T_{S^{-1}(p_\iii)}$ on $S^{-1}(\ol{\Omega})$, where $S^{-1}(p_\iii)$ lies in the set of admissible poles of $\Phi^S$, which is $S^{-1}(\mathcal P)$ because $\dist(S^{-1}(p),S^{-1}(\ol{\Omega}))=\lambda^{-1}\dist(p,\ol{\Omega})$ while the bound in \cref{eq:admissible-poles-def} is multiplied by $\lambda^{-1}$; the uniqueness clause of \cref{lem:poledict}, applied to $\Phi^S$, identifies it as $p^{\Phi^S}_\iii$.
\end{proof}

\subsection{Chordal geometry and the pole cocycle} \label{sec:gen-higher-chordal}

Poles may lie at or near $\infty$, where the Euclidean metric is blind, and we measure them by the chordal metric $\varsigma$ on $\RS^d=\R^d\cup\{\infty\}$, given by
\begin{equation*}
  \varsigma(x,y) = \frac{2|x-y|}{\sqrt{(1+|x|^2)(1+|y|^2)}}
\end{equation*}
and $\varsigma(x,\infty)=2(1+|x|^2)^{-1/2}$, of diameter $2$, together with the \emph{chordal factor}
\begin{equation*}
  f^{\#}(x)=\|Df(x)\| \frac{(1+|x|^2)}{(1+|f(x)|^2)}
\end{equation*}
of a M\"obius transformation. Under the stereographic identification $f^{\#}$ is the conformal factor of $f$ acting on the round sphere, so it extends to a smooth positive function on $\RS^d$, and the chain rule gives $(f\circ h)^{\#}=(f^{\#}\circ h)h^{\#}$ and, for the inverse $g=f^{-1}$, the reciprocal rule $g^{\#}(y)f^{\#}(g(y))=1$. Every M\"obius transformation obeys the two-point identity \cref{eq:twopoint} away from its pole, as established in the proof of \cref{lem:osc}. Substituting \cref{eq:twopoint} into the definition of $\varsigma$ turns it into the chordal form
\begin{equation*}
  \varsigma(f(u),f(v))=\sqrt{f^{\#}(u)f^{\#}(v)}\varsigma(u,v),
\end{equation*}
which extends to $\RS^d$ by continuity. In particular $\sup_Af^{\#}$ is a $\varsigma$-Lipschitz constant for $f$ on $A$, for every $A\subseteq\RS^d$.

The chordal factor of a generator is small at its attracting fixed point, where it agrees with the Euclidean factor and is therefore at most $\cmax$, but it is never small everywhere: a M\"obius contraction has a second fixed point on $\RS^d$, at which the chordal factor is the reciprocal of that at the first. This is the fact behind the lower bound on the chordal stretch in \cref{sec:gen-higher-obstruction}.

\begin{lemma} \label{lem:secondfixed}
  Let $\Phi=(\fii_i)_{i\in\II}\in\mathcal S$, let $i\in\II$, and let $x_i\in\ol{\Omega}$ be the fixed point of $\fii_i$. Then $\fii_i$ has a fixed point $y_i\in\RS^d\setminus\{x_i\}$ with $\fii_i^{\#}(y_i)=1/\|D\fii_i(x_i)\|\ge1/\cmax$. In particular, $\sup_{x\in\RS^d}\fii_i^{\#}(x)\ge1/\cmax$ for every $i\in\II$.
\end{lemma}

\begin{proof}
  Put $k=\|D\fii_i(x_i)\|$. Since $\fii_i$ is conformal at $x_i$, we have $0<k\le\cmax<1$. Conjugate by the inversion $\iota_{x_i}(x)=x_i+(x-x_i)/|x-x_i|^2$ in the unit sphere centered at $x_i$, an involution of $\RS^d$ interchanging $x_i$ and $\infty$. The map $u=\iota_{x_i}\circ\fii_i\circ\iota_{x_i}$ is a M\"obius transformation fixing $\infty$, hence a similarity $u(x)=\lambda Ox+v$. Since $\iota_{x_i}$ is an involution, the reciprocal rule gives $\iota_{x_i}^{\#}(x_i)\iota_{x_i}^{\#}(\infty)=1$, while $\fii_i$ fixes $x_i$ and therefore $\fii_i^{\#}(x_i)=\|D\fii_i(x_i)\|=k$. The chain rule thus gives $u^{\#}(\infty)=\iota_{x_i}^{\#}(x_i)\fii_i^{\#}(x_i)\iota_{x_i}^{\#}(\infty)=k$. On the other hand $u^{\#}(\infty)=\lim_{x\to\infty}\lambda(1+|x|^2)/(1+|\lambda Ox+v|^2)=1/\lambda$. Thus $\lambda=1/k>1$. If $(I-\lambda O)z=0$, then $|z|=\lambda|Oz|=\lambda|z|$, and hence $z=0$. Thus $I-\lambda O$ is invertible and $u$ has the finite fixed point $y=(I-\lambda O)^{-1}v$, at which its chordal factor is $\|Du(y)\|=\lambda$. Set $y_i=\iota_{x_i}(y)$. Since $y$ is finite whereas $\iota_{x_i}^{-1}(x_i)=\infty$, we have $y_i\ne x_i$; moreover, $y_i=\infty$ exactly when $y=x_i$. The point $y_i$ is fixed by $\fii_i=\iota_{x_i}\circ u\circ\iota_{x_i}$, and the chain rule and the reciprocal rule give $\fii_i^{\#}(y_i)=\iota_{x_i}^{\#}(u(y))u^{\#}(y)\iota_{x_i}^{\#}(y_i)=\iota_{x_i}^{\#}(y)\iota_{x_i}^{\#}(\iota_{x_i}(y))\lambda=\lambda=1/k$, which is at least $1/\cmax$ and proves both claims.
\end{proof}

The supremum of the chordal factor over $\RS^d$ is a global quantity, whereas the metric $d_2$ sees a generator only on $\ol{\Omega}$. The next lemma bridges the two, a M\"obius transformation being determined by its value, its derivative, and its pre-Schwarzian at a single point; it is what makes the chordal stretch of \cref{sec:gen-higher-obstruction} continuous in $d_2$.

\begin{lemma} \label{lem:jet}
  Let $x_0\in\R^d$. A M\"obius transformation $f$ with $f(x_0)\ne\infty$ is determined by $f(x_0)$, $Df(x_0)$, and $T_f(x_0)$. On the set of triples arising in this way, equipped with the subspace topology inherited from $\R^d\times\R^{d\times d}\times\R^d$, the quantity $\sup_{x\in\RS^d}f^{\#}(x)$ depends continuously on the three entries.
\end{lemma}

\begin{proof}
  We show that
  \begin{equation} \label{eq:jet}
    f(x) = f(x_0)+Df(x_0)\frac{y+|y|^2w}{1+2\langle w,y\rangle+|w|^2|y|^2},
  \end{equation}
  for every $x\in\R^d$ other than the pole of $f$, where $y=x-x_0$ and $w=-\tfrac12 T_f(x_0)$; this proves the first claim, the right-hand side involving only the three data. Write $K_w$ for the map sending $y$ to the fraction in \cref{eq:jet}, which is the M\"obius transformation $\iota_0\circ\tau_w\circ\iota_0$ with $\iota_0(y)=y/|y|^2$ the unit inversion. Then \cref{eq:lambdachain} gives $\lambda_{K_w}(y)=1/\theta_w(y)$ for $y\ne0$ off the pole of $K_w$, hence at $y=0$ as well by continuity, where $\theta_w(y)=1+2\langle w,y\rangle+|w|^2|y|^2$; so $K_w(0)=0$, $DK_w(0)=I$ by the expansion $K_w(y)=y+O(|y|^2)$, and $T_{K_w}(0)=-\nabla\theta_w(0)=-2w=T_f(x_0)$. The M\"obius transformation $u=Df(x_0)^{-1}(f(\,\cdot\,+x_0)-f(x_0))$ satisfies $u(0)=0$, $Du(0)=I$, and $T_u(0)=T_f(x_0)$, the last by \cref{eq:Tchain}. Since $K_w(0)=0$ and $DK_w(0)=I$, the map $v=u\circ K_w^{-1}$ fixes $0$ with $Dv(0)=I$, and \cref{eq:Tchain}, applied to $u=v\circ K_w$ at $0$, gives $T_v(0)=T_u(0)-T_{K_w}(0)=0$. By \cref{eq:pole}, a M\"obius transformation whose pre-Schwarzian vanishes at a finite point where it is defined is a similarity. Thus $v$ is a similarity, and the identities $v(0)=0$ and $Dv(0)=I$ show that $v$ is the identity. Therefore $u=K_w$, which gives \cref{eq:jet}.

  Since $\|Df(x)\|=\|Df(x_0)\|/\theta_w(y)$ and $|y+|y|^2w|^2=|y|^2\theta_w(y)$, the chordal factor of $f$ at every $x\in\R^d$ other than the pole of $f$ is
  \begin{equation} \label{eq:jetfactor}
    \begin{split}
      f^{\#}(x) &= \frac{\|Df(x_0)\|(1+|x|^2)}{\theta_w(y)(1+|f(x)|^2)} \\
      &= \frac{\|Df(x_0)\|(1+|x|^2)}{(1+|f(x_0)|^2)\theta_w(y)+2\langle f(x_0),Df(x_0)(y+|y|^2w)\rangle+\|Df(x_0)\|^2|y|^2}.
    \end{split}
  \end{equation}
  The right-hand side of \cref{eq:jetfactor} is a quotient of two polynomials of degree two in $y$ whose coefficients depend continuously on $(f(x_0),Df(x_0),w)$. Its denominator is positive on $\R^d$: off the pole of $f$ it is $\theta_w(y)(1+|f(x)|^2)$, where $\theta_w(0)=1$ and $\theta_w(y)=|y|^2|\iota_0(y)+w|^2>0$ for $y\ne0$ off the pole, and at the pole, where $y+|y|^2w=0$, it is $\|Df(x_0)\|^2|y|^2$; so is its leading coefficient $(1+|f(x_0)|^2)|w|^2+2\langle f(x_0),Df(x_0)w\rangle+\|Df(x_0)\|^2\ge|w|^2+(|f(x_0)||w|-\|Df(x_0)\|)^2$, which is positive because $\|Df(x_0)\|>0$. The quadratic parts of the numerator and denominator are scalar multiples of $|y|^2$. To verify joint continuity at infinity, write $y=z/|z|^2$ and multiply both polynomials by $|z|^2$. The resulting numerator and denominator are polynomials in $z$ whose coefficients depend continuously on $(f(x_0),Df(x_0),w)$, and at $z=0$ the denominator is the positive leading coefficient displayed above. Their quotient therefore extends jointly continuously to $x=\infty$, with value equal to the ratio of the leading coefficients. Since $f^{\#}$ is continuous on $\RS^d$, it agrees with this continuous extension also at the pole of $f$ and at $\infty$. Consequently, the extended right-hand side of \cref{eq:jetfactor} is jointly continuous in $(f(x_0),Df(x_0),w,x)$. Since $w=-\tfrac12T_f(x_0)$, taking the supremum over the compact $\RS^d$ proves the claimed continuity in the three entries.
\end{proof}

The perturbations that prove the density live in the M\"obius family itself, as translations in the source, and their effect is read on the poles through the inverse maps $g_i=\fii_i^{-1}$, M\"obius transformations of $\RS^d$ that send $\RS^d\setminus\Omega$ into itself. The next lemma collects what the density argument of \cref{sec:gen-higher-dense,lem:packing,prop:finite} need about the poles: how a pole unfolds along a word, how it depends on the word, an exact formula for the conformal factor of an inverse composition, and a Euclidean bound on the distance between the poles of two words with a long common prefix. Only the pole gap of \cref{lem:CU} and the dual bound \cref{eq:dualbounds} enter; no contraction of the inverse maps is assumed.

\begin{lemma} \label{lem:polerec}
  Suppose that $\max_{i \in \II}\|T_{\fii_i}\|>0$, and let $g_i=\fii_i^{-1}$ for $i\in\II$. Then $g_i(\RS^d\setminus\Omega)\subseteq\RS^d\setminus\Omega$ for every $i\in\II$, the map $\iii\mapsto p_\iii$ is continuous from $\II^\N\cup\II^*$ to $(\RS^d,\varsigma)$, and
  \begin{equation} \label{eq:polerec}
    p_{\kkk\lll} = g_\kkk(p_\lll)
  \end{equation}
  for all $\kkk\in\II^*$ and $\lll\in\II^\N\cup\II^*$, where $g_\kkk=g_{k_1}\circ\cdots\circ g_{k_m}$ for $\kkk=k_1\cdots k_m$ and $g_\varnothing$ is the identity. In particular $p_\iii=g_{i_1}\circ\cdots\circ g_{i_n}(\infty)$ for every $\iii\in\II^n$, and the set $\widehat P=\{p_\iii : \iii\in\II^\N\}$ is compact with $g_\kkk(\widehat P)=\{p_{\kkk\lll} : \lll\in\II^\N\}$ for every $\kkk\in\II^*$. If $\kkk\in\II^*$ is nonempty and $f_\kkk$ is not a similarity, then the \emph{co-pole} $b_\kkk=f_\kkk(\infty)$ is finite and
  \begin{equation} \label{eq:cocycle}
    \lambda_{g_\kkk}(y) = \frac{\rho_{f_\kkk}}{|y-b_\kkk|^2} \qquad\text{and}\qquad \rho_{f_\kkk} = \lambda_{f_\kkk}(x)|x-p_\kkk|^2
  \end{equation}
  for all $y\in\R^d\setminus\{b_\kkk\}$ and $x\in\R^d\setminus\{p_\kkk\}$, where $\rho_{f_\kkk}$ is the constant of \cref{lem:mobius}. Finally, if $\iii,\jjj\in\II^\N\cup\II^*$ are distinct and $x_0\in\ol{\Omega}$ and $M>0$ satisfy $p_\iii,p_\jjj\in\R^d$, $|x_0-p_\iii|\le M$, and $|x_0-p_\jjj|\le M$, then
  \begin{equation} \label{eq:poleprefix}
    |p_\iii-p_\jjj| \le \frac{M^2}{2}|T_\iii(x_0)-T_\jjj(x_0)| \le \frac{M^2\max_{i \in \II}\|T_{\fii_i}\|}{1-\cmax}\cmax^{|\iii\land\jjj|}.
  \end{equation}
\end{lemma}

\begin{proof}
  The invariance of $\RS^d\setminus\Omega$ is immediate: if $y\in\RS^d\setminus\Omega$ had $g_i(y)\in\Omega$, then $y=\fii_i(g_i(y))\in\fii_i(\Omega)\subseteq\Omega$, a contradiction. For \cref{eq:polerec} with a finite $\lll$, the reversed compositions satisfy $f_{\kkk\lll}=f_\lll\circ f_\kkk$ by \cref{eq:frev}, so that $f_{\kkk\lll}^{-1}=g_\kkk\circ f_\lll^{-1}$; evaluating at $\infty$ gives the claim, the empty composition being the identity with $p_\varnothing=\infty$, and the choice $\lll=\varnothing$ gives the unfolded form.

  The continuity is read off the pole dictionary. Fix $x_0\in\ol{\Omega}$ and define $\Theta\colon\R^d\to\RS^d$ by $\Theta(v)=x_0+2v/|v|^2$ for $v\ne0$ and $\Theta(0)=\infty$. The map $\Theta$ is continuous, being a composition of continuous maps on $\R^d\setminus\{0\}$ and satisfying $|\Theta(v)|\ge2/|v|-|x_0|\to\infty$ as $v\to0$. By \cref{lem:poledict} we have $T_\iii=T_{p_\iii}$ on $\ol{\Omega}$ for every word, so $p_\iii=\Theta(T_\iii(x_0))$: for a finite pole this is \cref{eq:polerecover}, and for $p_\iii=\infty$ it is the vanishing of $T_\infty$. Since $\iii\mapsto T_\iii(x_0)$ is continuous by \cref{lem:dual}, so is $\iii\mapsto p_\iii$, and $\widehat P$ is compact as the image of the compact $\II^\N$. For an infinite $\lll$, the finite case gives $p_{\kkk(\lll|_n)}=g_\kkk(p_{\lll|_n})$ for every $n$, and letting $n\to\infty$ proves \cref{eq:polerec}, by the continuity just established and that of $g_\kkk$ on $\RS^d$; the description of $g_\kkk(\widehat P)$ follows.

  We turn to \cref{eq:cocycle}, where $\kkk$ is nonempty and $f_\kkk$ is not a similarity. By \cref{lem:mobius} the composition factors as $f_\kkk=S\circ\iota_{p_\kkk}$, where $\iota_{p_\kkk}(x)=p_\kkk+(x-p_\kkk)/|x-p_\kkk|^2$ is the inversion in the unit sphere centered at $p_\kkk$, of conformal factor $|x-p_\kkk|^{-2}$, and $S$ is a similarity of ratio $\rho_{f_\kkk}$; the second identity in \cref{eq:cocycle} is the formula $\lambda_{f_\kkk}(x)=\rho_{f_\kkk}/|x-p_\kkk|^2$ of that lemma. The co-pole is finite because $\iota_{p_\kkk}(\infty)=p_\kkk$, so that $b_\kkk=S(p_\kkk)$, and the inversion is an involution, so that $g_\kkk=\iota_{p_\kkk}\circ S^{-1}$. Now $S^{-1}$ scales distances by $\rho_{f_\kkk}^{-1}$ and has the constant conformal factor $\rho_{f_\kkk}^{-1}$, whence $|S^{-1}(y)-p_\kkk|=\rho_{f_\kkk}^{-1}|y-b_\kkk|$ and, by \cref{eq:lambdachain}, $\lambda_{g_\kkk}(y)=\rho_{f_\kkk}^{-1}|S^{-1}(y)-p_\kkk|^{-2}=\rho_{f_\kkk}/|y-b_\kkk|^2$.

  For the last claim, let $\iii$, $\jjj$, $x_0$, and $M$ be as in the statement and put $v_\iii=T_\iii(x_0)$ and $v_\jjj=T_\jjj(x_0)$. Both poles are finite, so \cref{eq:polerecover} gives $p_\iii=x_0+2\iota_0(v_\iii)$ and $p_\jjj=x_0+2\iota_0(v_\jjj)$ with the unit inversion $\iota_0(v)=v/|v|^2$, while $|v_\iii|=2/|x_0-p_\iii|\ge2/M$ and likewise $|v_\jjj|\ge2/M$. The two-point identity for $\iota_0$, computed in the proof of \cref{lem:osc}, reads $|\iota_0(u)-\iota_0(v)|=|u-v|/(|u||v|)$, so that
  \begin{equation*}
    |p_\iii-p_\jjj| = \frac{2|v_\iii-v_\jjj|}{|v_\iii||v_\jjj|} \le \tfrac12M^2|v_\iii-v_\jjj|,
  \end{equation*}
  which is the first inequality in \cref{eq:poleprefix}. The second is \cref{eq:dualbounds}.
\end{proof}

\subsection{Density under a contraction threshold} \label{sec:gen-higher-dense}

The density is proved by transversality for the poles. Translating the source of a generator moves the pole of that letter rigidly, and the pole of a longer word through a cocycle of the inverse maps. The terms of that cocycle decay geometrically, by \cref{eq:cocycle}, once the co-poles are trapped in $\Omega$ and the poles stay bounded. The first letter of a word therefore governs its pole, and a pair in $Z$, whose first letters differ, is pulled apart along a direction balanced between the two letters.

The argument has three steps, among which the hypotheses of \cref{thm:C3dense} divide. \Cref{lem:trapped} records what the hypothesis on the co-poles gives for a single system: the poles of nonempty words stay in a compact set off $\Omega$, and the inverse compositions contract at the poles at a geometric rate. \Cref{prop:transversal3} translates the sources of the generators, the hypothesis on the poles keeping the translated generators conformal on $\Omega'$, and uses a bound on $\cmax$ in terms of the pole geometry to make the two leading terms of the cocycle dominate its tail uniformly over the pair space; the balanced direction is then a right inverse of the derivative of the pole difference, and the parameter map is a submersion. The proof of \cref{thm:C3dense} then counts, which needs a further bound on $\cmax$, in terms of the cardinality of the alphabet. \Cref{cor:C3dense} below trades the hypothesis on the co-poles and the bound $\cmax(1+K)<1$ for a single lower bound on the pre-Schwarzians of the generators, read off on $\ol{\Omega}$, and restates the hypothesis on the poles as the boundedness of the generators on $\Omega'$.

The trapping is a property of a single system, which the density argument uses both for $\Phi$ and for its translates, and we record it with explicit constants. Suppose that $\fii_i(\infty)\in\Omega$ for every $i\in\II$. Then no generator fixes $\infty$, so none is a similarity, and \cref{lem:CU} gives $\max_{i \in \II}\|T_{\fii_i}\|>0$; thus $0<C_0<\infty$, and the poles $p_\iii$ of \cref{lem:poledict}, the inverse maps $g_i=\fii_i^{-1}$ and their compositions $g_\kkk$ of \cref{lem:polerec}, and the co-poles $b_i=\fii_i(\infty)$ are defined. With the convention $\dist(\infty,A)=\infty$ for $A\subseteq\R^d$, we measure the pole geometry of such a system by the constants
\begin{equation} \label{eq:polegeom}
  \begin{aligned}
    \rho_\Phi &= \min_{i \in \II}\dist(\fii_i(\ol{\Omega}),\partial\Omega), \qquad& \rho_b &= \min\{\min_{i \in \II}\dist(b_i,\partial\Omega),\rho_\Phi\}, \\
    R_* &= \max_{i \in \II}\sup_{y\in\RS^d\setminus\Omega}\dist(g_i(y),\ol{\Omega}), \qquad& R &= \sup\{\dist(p_\kkk,\ol{\Omega}) : \kkk\in\II^*\setminus\{\varnothing\}\}, \\
    K &= \frac{R^2}{\max\{\rho_b,2/C_0\}^2}.
  \end{aligned}
\end{equation}

\begin{lemma} \label{lem:trapped}
  Let $\Phi=(\fii_i)_{i \in \II}$ be a conformal iterated function system on $\Omega \subset \R^d$, $d \ge 3$, such that $b_i=\fii_i(\infty)\in\Omega$ for every $i\in\II$, and let $\rho_\Phi$, $\rho_b$, $R_*$, $R$, and $K$ be as in \cref{eq:polegeom}. Then $\rho_b>0$, and the following hold.
  \begin{enumerate}
    \item\label{it:trapped-poles} For every $i\in\II$ the set $g_i(\RS^d\setminus\Omega)$ is a compact subset of $\R^d\setminus\Omega$, so that $R_*<\infty$. For every nonempty $\iii\in\II^\N\cup\II^*$ the pole $p_\iii$ lies in $g_{i_1}(\RS^d\setminus\Omega)$ and satisfies $2/C_0\le\dist(p_\iii,\ol{\Omega})\le R_*$. In particular $2/C_0\le R\le R_*$, so that $0<K<\infty$, and $f_\kkk$ is not a similarity for any nonempty $\kkk\in\II^*$.
    \item\label{it:trapped-copoles} For every nonempty $\kkk=k_1\cdots k_m\in\II^*$ the co-pole $b_\kkk=f_\kkk(\infty)$ equals $b_{k_1}$ if $m=1$ and lies in $\fii_{k_m}(\ol{\Omega})$ if $m\ge2$. In particular $b_\kkk\in\Omega$ and $\dist(b_\kkk,\partial\Omega)\ge\rho_b$.
    \item\label{it:trapped-decay} For every nonempty $\kkk\in\II^*$ and every nonempty $\lll\in\II^\N\cup\II^*$ we have $\rho_{f_\kkk}\le R^2\cmax^{|\kkk|}$ and $|p_\lll-b_\kkk|\ge\max\{\rho_b,2/C_0\}$, and consequently
    \begin{equation} \label{eq:gdecay}
      \|Dg_\kkk(p_\lll)\| \le K\cmax^{|\kkk|}.
    \end{equation}
  \end{enumerate}
\end{lemma}

\begin{proof}
  As observed before \cref{eq:polegeom}, $\max_{i \in \II}\|T_{\fii_i}\|>0$, so that \cref{lem:poledict,lem:polerec} apply to $\Phi$. The set $\RS^d\setminus\Omega$ is compact, and each $g_i$ is a homeomorphism of $\RS^d$ that maps it into itself by \cref{lem:polerec} and sends only $b_i$ to $\infty$. Hence $g_i(\RS^d\setminus\Omega)$ is a compact subset of $\RS^d\setminus\Omega$ that omits $\infty$, that is, a compact subset of $\R^d\setminus\Omega$, and $R_*<\infty$ because $\dist(\,\cdot\,,\ol{\Omega})$ is continuous. Now let $\iii\in\II^\N\cup\II^*$ be nonempty. By \cref{eq:polerec}, $p_\iii=g_{i_1}(p_{\sigma\iii})$, where $p_{\sigma\iii}$ is $p_\varnothing=\infty$ if $\iii$ has length one and otherwise a point of $\mathcal P$, whose finite points lie at distance at least $2/C_0$ from $\ol{\Omega}$. In either case $p_{\sigma\iii}\in\RS^d\setminus\Omega$, so that $p_\iii\in g_{i_1}(\RS^d\setminus\Omega)$, and $p_\iii$ is a finite point of $\mathcal P$ with $2/C_0\le\dist(p_\iii,\ol{\Omega})\le R_*$. Taking the supremum over the nonempty finite words gives $2/C_0\le R\le R_*$, and since a finite word has the pole $\infty$ exactly when its composition is a similarity, no composition along a nonempty finite word is a similarity. This proves \cref{it:trapped-poles}.

  For \cref{it:trapped-copoles}, let $\kkk=k_1\cdots k_m$ be nonempty. The identity $f_\kkk=f_{\sigma\kkk}\circ\fii_{k_1}$ of \cref{eq:frev} gives $b_\kkk=f_{\sigma\kkk}(b_{k_1})$, which is $b_{k_1}$ when $m=1$. When $m\ge2$, we have $f_{\sigma\kkk}=\fii_{k_m}\circ f_{k_2\cdots k_{m-1}}$ by \cref{eq:frev}, the inner composition being the identity when $m=2$, and $f_{k_2\cdots k_{m-1}}(b_{k_1})\in\ol{\Omega}$, every composition sending $\ol{\Omega}$ into itself; so $b_\kkk\in\fii_{k_m}(\ol{\Omega})$. In both cases $b_\kkk\in\Omega$ and $\dist(b_\kkk,\partial\Omega)\ge\rho_b$. Moreover $\rho_b>0$, each $b_i$ lying in the open set $\Omega$ and each $\fii_i(\ol{\Omega})$ being a compact subset of it.

  For \cref{it:trapped-decay}, let $\kkk$ and $\lll$ be as there. By \cref{it:trapped-poles} the composition $f_\kkk$ is not a similarity, so \cref{eq:cocycle} applies to it. Evaluating its second identity at a point $x\in\ol{\Omega}$ nearest to $p_\kkk$ and using $\lambda_{f_\kkk}\le\cmax^{|\kkk|}$ on $\ol{\Omega}$ gives $\rho_{f_\kkk}=\lambda_{f_\kkk}(x)\dist(p_\kkk,\ol{\Omega})^2\le R^2\cmax^{|\kkk|}$. The pole $p_\lll$ lies in $\R^d\setminus\Omega$ by \cref{it:trapped-poles} and the co-pole $b_\kkk$ in $\Omega$ by \cref{it:trapped-copoles}, so the segment joining them meets $\partial\Omega$, and $|p_\lll-b_\kkk|\ge\dist(b_\kkk,\partial\Omega)\ge\rho_b$; moreover $|p_\lll-b_\kkk|\ge\dist(p_\lll,\ol{\Omega})\ge2/C_0$, as $b_\kkk\in\ol{\Omega}$. The first identity of \cref{eq:cocycle} at $y=p_\lll$ now gives $\|Dg_\kkk(p_\lll)\|=\lambda_{g_\kkk}(p_\lll)=\rho_{f_\kkk}/|p_\lll-b_\kkk|^2\le K\cmax^{|\kkk|}$, which is \cref{eq:gdecay}.
\end{proof}

The translated systems inherit the trapping with uniform constants, since translating the source of a generator moves its pole rigidly and leaves its co-pole in place. The next proposition, the counterpart of \cref{prop:transversal}, combines this with the domination of the pole cocycle by its two leading terms and derives the two estimates on the pole differences that the counting consumes.

\begin{proposition} \label{prop:transversal3}
  Let $\Phi=(\fii_i)_{i \in \II}$ be a conformal iterated function system on $\Omega \subset \R^d$, $d \ge 3$. Suppose that no generator $\fii_i$ is a similarity, that $\dist(p_i,\Omega')>0$ and $b_i=\fii_i(\infty)\in\Omega$ for every $i\in\II$, and that $\cmax(1+K)<1$, where $K$ is the constant of \cref{eq:polegeom}. Let $\lambda\in(\cmax,1)$. Then there are an open ball $B\subseteq\R^{Nd}$ centered at $0$ and a constant $C<\infty$ such that, for every $t=(t_i)_{i\in\II}\in B$, the system $\Phi_t=(\fii_i\circ\tau_{-t_i})_{i\in\II}$ belongs to $\mathcal S$ and satisfies $d_2(\Phi_t,\Phi)\le C|t|$ and $\max_{i \in \II}\|T_{\fii_i\circ\tau_{-t_i}}\|>0$, and its poles $p^t_\iii$ are finite for all nonempty $\iii\in\II^\N\cup\II^*$. Moreover, the map $F\colon B\times Z\to\R^d$ defined by
  \begin{equation} \label{eq:F3def}
    F(t,(\iii,\jjj)) = p^t_\iii-p^t_\jjj
  \end{equation}
  satisfies
  \begin{equation} \label{eq:F3smallball}
    \LL^{Nd}(\{t\in B : |F(t,w)|\le u\}) \le Cu^d
  \end{equation}
  for all $u>0$ and $w\in Z$, and $|F(t,w)-F(t,w')|\le C\varrho_\lambda(w,w')$ for all $t\in B$ and $w,w'\in Z$, where $\varrho_\lambda$ is the metric of \cref{lem:paircover}.
\end{proposition}

\begin{proof}
  Below $C$ denotes a finite constant depending only on $\Phi$ and $\lambda$, whose value may change from one occurrence to the next. By \cref{lem:trapped}, the constants of \cref{eq:polegeom} are finite and $\rho_b>0$. Each generator pole $p_i$ is finite and satisfies $\dist(p_i,\ol{\Omega})\ge2/C_0$ by \cref{eq:polegap}, and each $g_i$ is a M\"obius transformation with pole $b_i\in\Omega$.

  We begin with the perturbations. For $t=(t_i)_{i\in\II}\in\R^{Nd}$, with the Euclidean norm, so that $|t_i|\le|t|$, put $\fii_i^t=\fii_i\circ\tau_{-t_i}$ and $\Phi_t=(\fii_i^t)_{i\in\II}$. Then $\fii_i^t$ is a M\"obius transformation with pole $p_i+t_i$, co-pole $\fii_i^t(\infty)=b_i$, and inverse $g_i^t=\tau_{t_i}\circ g_i=g_i+t_i$, and it is conformal on $\Omega'$ whenever $|t|<\min_{i \in \II}\dist(p_i,\Omega')$, a positive bound by hypothesis. Let $U_0$ be the closed $(1/C_0)$-neighborhood of $\ol{\Omega}$, a compact set at distance at least $1/C_0$ from every $p_i$. The generators are real-analytic off their poles, so their derivatives up to order three are bounded on $U_0$, and for $|t|\le1/C_0$ the mean value inequality, applied to $D^k\fii_i(x-t_i)-D^k\fii_i(x)$ with $x\in\ol{\Omega}$ and $k\in\{0,1,2\}$ along the segment from $x$ to $x-t_i$, which lies in $U_0$, gives $d_2(\Phi_t,\Phi)\le C'|t|$ for a constant $C'$ depending only on $\Phi$. Moreover, $T_{\fii_i^t}$ is the pole field of $p_i+t_i$ by \cref{lem:mobius}, and its supremum over $\ol{\Omega}$ is $2/\dist(p_i+t_i,\ol{\Omega})$, so that $0<\max_{i \in \II}\|T_{\fii_i^t}\|\le2C_0$ for $|t|\le1/C_0$.

  Fix an open ball $B_0\subseteq\R^{Nd}$ centered at $0$ whose radius $r_0$ satisfies $r_0\le\min\{1,1/C_0\}$, $r_0<\min_{i \in \II}\dist(p_i,\Omega')$, and $C'r_0\le\min\{\rho_\Phi/2,\lambda-\cmax\}$, and let $t\in B_0$. Then $d_2(\Phi_t,\Phi)<\min\{\rho_\Phi,1-\cmax\}$, so $\Phi_t\in\mathcal S$ and $\cmax(\Phi_t)\le\cmax+d_2(\Phi_t,\Phi)\le\lambda$ by the stability of the two requirements of \cref{sec:cifs} observed after \cref{eq:d2metric}. Since $\fii_i^t(\infty)=b_i\in\Omega$ for every $i\in\II$, \cref{lem:trapped} applies to $\Phi_t$, and we write $p^t_\iii$ and $g^t_\kkk$ for the poles and the inverse compositions of $\Phi_t$, and $\rho_b(\Phi_t)$, $R_*(\Phi_t)$, $R(\Phi_t)$, and $K(\Phi_t)$ for its constants in \cref{eq:polegeom}. The co-poles of the generators do not move, and $\dist(\fii_i^t(\ol{\Omega}),\partial\Omega)\ge\rho_\Phi-C'|t|\ge\rho_\Phi/2$ for every $i$, so that $\rho_b(\Phi_t)\ge\rho_b/2$. The identity $g_i^t=g_i+t_i$ and the $1$-Lipschitz property of $\dist(\,\cdot\,,\ol{\Omega})$ give $R_*(\Phi_t)\le R_*+1=R_1$, and hence $R(\Phi_t)\le R_1$ by \cref{lem:trapped}\cref{it:trapped-poles}. By the same item for $\Phi_t$, every pole $p^t_\iii$ of a nonempty word $\iii\in\II^\N\cup\II^*$ lies in the compact set $P_*=\{y\in\R^d\setminus\Omega : \dist(y,\ol{\Omega})\le R_1\}$, which omits every $b_i$, and \cref{lem:trapped}\cref{it:trapped-decay} for $\Phi_t$, with $K(\Phi_t)\le R_1^2/(\rho_b/2)^2$ and $\cmax(\Phi_t)\le\lambda$, gives
  \begin{equation} \label{eq:Jbound}
    \|Dg^t_\kkk(p^t_\lll)\| \le K_1\lambda^{|\kkk|}, \qquad\text{where}\qquad K_1 = \frac{4R_1^2}{\rho_b^2},
  \end{equation}
  for all nonempty $\kkk\in\II^*$ and all nonempty $\lll\in\II^\N\cup\II^*$, while at $t=0$ the same item gives the sharper bound \cref{eq:gdecay}.

  Two further consequences of the translation will be used. Since $g_i^t=g_i+t_i$ and $g_i(\infty)=p_i$, the identity \cref{eq:polerec} for $\Phi_t$ reads $p^t_\iii=g_{i_1}(p^t_{\sigma\iii})+t_{i_1}$ for every nonempty $\iii\in\II^\N\cup\II^*$, with $p^t_\varnothing=\infty$. Moreover, $|x_0-p^t_\iii|\le M_1=R_1+\diam(\ol{\Omega})$ for $x_0\in\ol{\Omega}$ and every nonempty word $\iii$, so \cref{eq:poleprefix} for $\Phi_t$, with $M_1$ in place of $M$ and with the bounds $\max_{i \in \II}\|T_{\fii_i^t}\|\le2C_0$ and $\cmax(\Phi_t)\le\lambda$, its right-hand side being increasing in $\cmax$, gives
  \begin{equation} \label{eq:prefix3}
    |p^t_\iii-p^t_\jjj| \le C\lambda^{|\iii\land\jjj|}
  \end{equation}
  for all distinct nonempty $\iii,\jjj\in\II^\N\cup\II^*$ and all $t\in B_0$.

  We next differentiate the poles. Let $\iii$ be a nonempty finite word of length $n$. Unfolding the recursion expresses $p^t_\iii$ through the maps $g_{i_1},\ldots,g_{i_{n-1}}$, each real-analytic off its pole in $\Omega$, evaluated at the poles $p^t_{\sigma^l\iii}\in P_*$ with $1\le l<n$, starting from $p^t_{\sigma^{n-1}\iii}=p_{i_n}+t_{i_n}$, so $t\mapsto p^t_\iii$ is real-analytic on $B_0$. Differentiating the recursion and unfolding gives the \emph{pole cocycle}
  \begin{equation} \label{eq:polecocycle}
    \partial_tp^t_\iii[\theta] = \sum_{l=1}^{|\iii|}D(g^t_{\iii|_{l-1}})(p^t_{\sigma^{l-1}\iii})\theta_{i_l}
  \end{equation}
  for $\theta=(\theta_i)_{i\in\II}\in\R^{Nd}$, where $g^t_{\iii|_{l-1}}=g^t_{i_1}\circ\cdots\circ g^t_{i_{l-1}}$, the empty composition being the identity; indeed, $Dg_i^t=Dg_i$ and $g^t_{i_m}(p^t_{\sigma^m\iii})=p^t_{\sigma^{m-1}\iii}$, so the chain rule identifies $D(g^t_{\iii|_{l-1}})(p^t_{\sigma^{l-1}\iii})$ with the ordered product $Dg_{i_1}(p^t_{\sigma\iii})\cdots Dg_{i_{l-1}}(p^t_{\sigma^{l-1}\iii})$. For $2\le l\le n$ the word $\sigma^{l-1}\iii$ is nonempty, so the operator in the $l$-th term has norm at most $K_1\lambda^{l-1}$ by \cref{eq:Jbound}, and at most $K\cmax^{l-1}$ when $t=0$ by \cref{eq:gdecay}; in particular $\|\partial_tp^t_\iii\|\le\Lambda_1=1+K_1\lambda/(1-\lambda)$. Differentiating the recursion once more, with $g_i^t$ and the terminal pole $p^t_{\sigma^{n-1}\iii}$ affine in $t$, gives
  \begin{equation*}
    \partial_t^2p^t_\iii[\theta,\theta'] = \sum_{l=1}^{n-1}D(g^t_{\iii|_{l-1}})(p^t_{\sigma^{l-1}\iii})D^2g_{i_l}(p^t_{\sigma^l\iii})[A_l\theta,A_l\theta']
  \end{equation*}
  for $\theta,\theta'\in\R^{Nd}$, where $A_l=\partial_tp^t_{\sigma^l\iii}$. The number $\Lambda_2=\max_{i \in \II}\sup_{y\in P_*}\|D^2g_i(y)\|$ is finite, every $g_i$ being real-analytic on a neighborhood of the compact set $P_*$, so that $\|\partial_t^2p^t_\iii\|\le\Lambda_2\Lambda_1^3$. With $\Lambda=\max\{\Lambda_1,\Lambda_2\Lambda_1^3\}$, the derivative $\partial_tp^t_\iii$ is therefore bounded by $\Lambda$ and, $B_0$ being convex, $\Lambda$-Lipschitz on $B_0$.

  Let now $\iii\in\II^\N$. By \cref{eq:prefix3}, we have $|p^t_\iii-p^t_{\iii|_n}|\le C\lambda^n$ for $t\in B_0$, so the truncations converge to $p^t_\iii$ uniformly on $B_0$. Their derivatives converge uniformly on $B_0$ to the series \cref{eq:polecocycle} formed for $\iii$. Indeed, the terms with $l\ge2$ of all these sums and of the series have norm at most $K_1\lambda^{l-1}$ by \cref{eq:Jbound}, so the tails are uniformly small; and for fixed $l$ and $n\ge l$ the operator in the $l$-th term of the truncation of length $n$ is the ordered product of the derivatives $Dg_{i_m}$, $1\le m<l$, evaluated at the poles $p^t_{\sigma^m(\iii|_n)}$, which converge to $p^t_{\sigma^m\iii}$ uniformly on $B_0$ by \cref{eq:prefix3}, all these poles lying in the compact set $P_*$, on which each $Dg_i$ is bounded and uniformly continuous. Write $G_n$ for the derivative $\partial_tp^t_{\iii|_n}$, as a function of $t$ on $B_0$, and let $G$ be its uniform limit there. For $s,t\in B_0$, the convexity of $B_0$ and the fundamental theorem of calculus give
  \begin{equation*}
    p^t_{\iii|_n}-p^s_{\iii|_n} = \int_0^1 G_n(s+r(t-s))[t-s]\dd r.
  \end{equation*}
  Passing uniformly to the limit gives the same identity with $p^t_\iii$ and $p^s_\iii$ on the left-hand side and $G$ in the integrand. Since $G$ is continuous, this identity shows that $t\mapsto p^t_\iii$ is continuously differentiable on $B_0$ with derivative $G$, which is the series in \cref{eq:polecocycle}. Furthermore, $\partial_tp^t_\iii$, as a uniform limit of maps bounded by $\Lambda$ and $\Lambda$-Lipschitz, is itself bounded by $\Lambda$ and $\Lambda$-Lipschitz on $B_0$. For $w=(\iii,\jjj)\in Z$ and $t\in B_0$ we now define $F(t,w)=p^t_\iii-p^t_\jjj$ as in \cref{eq:F3def}; thus $t\mapsto F(t,w)$ is continuously differentiable on $B_0$ for every $w\in Z$, with derivative $D_tF(t,w)$ of norm at most $2\Lambda$ and $2\Lambda$-Lipschitz in $t$.

  We come to the transversality at $t=0$. Let $w=(\iii,\jjj)\in Z$. By \cref{eq:polecocycle}, $D_tF(0,w)\theta=(\theta_{i_1}-\theta_{j_1})+\mathcal E_w\theta$ with $\mathcal E_w\theta=\sum_{l\ge2}D(g_{\iii|_{l-1}})(p_{\sigma^{l-1}\iii})\theta_{i_l}-\sum_{l\ge2}D(g_{\jjj|_{l-1}})(p_{\sigma^{l-1}\jjj})\theta_{j_l}$, the sums running up to the common length of the two words, an operator built from the tail of the cocycle alone. For $h\in\R^d$, the balanced tuple $\mathcal Q_wh$ with $(\mathcal Q_wh)_{i_1}=h/2$, $(\mathcal Q_wh)_{j_1}=-h/2$, and zeros elsewhere therefore satisfies $D_tF(0,w)\mathcal Q_wh=h+\mathcal E_w\mathcal Q_wh$, the first letters being distinct, and every entry of $\mathcal Q_wh$ has norm at most $|h|/2$, so that \cref{eq:gdecay} gives
  \begin{align*}
    |\mathcal E_w\mathcal Q_wh| &\le \sum_{l\ge2}(\|D(g_{\iii|_{l-1}})(p_{\sigma^{l-1}\iii})\|+\|D(g_{\jjj|_{l-1}})(p_{\sigma^{l-1}\jjj})\|)\tfrac12|h| \\
    &\le \sum_{l\ge2}K\cmax^{\,l-1}|h| = \frac{K\cmax}{1-\cmax}|h|,
  \end{align*}
  and the factor is smaller than $1$ exactly because $\cmax(1+K)<1$. The operator $I+\mathcal E_w\mathcal Q_w$ is therefore invertible on $\R^d$ by the Neumann series, with inverse of norm at most $(1-K\cmax/(1-\cmax))^{-1}$, and $V_w=\mathcal Q_w(I+\mathcal E_w\mathcal Q_w)^{-1}$ is a right inverse of $D_tF(0,w)\colon\R^{Nd}\to\R^d$ of norm at most $(1-\cmax)/(1-\cmax(1+K))$, since $\|\mathcal Q_w\|=1/\sqrt2\le1$. A right inverse $V$ of a linear map $L$ satisfies $|\zeta|=|V^{*}L^{*}\zeta|\le\|V\||L^{*}\zeta|$ for every $\zeta$, and hence $\sigma_{\min}(D_tF(0,w))\ge\sigma_0$ for every $w\in Z$, where $\sigma_0=(1-\cmax(1+K))/(1-\cmax)>0$. Since $D_tF(\,\cdot\,,w)$ is $2\Lambda$-Lipschitz on $B_0$, \cref{eq:sigmamin} now gives $\sigma_{\min}(D_tF(t,w))\ge\sigma_0/2$ for all $w\in Z$ and all $t$ in the open ball $B$ centered at $0$ of radius $\min\{r_0,\sigma_0/(4\Lambda)\}$.

  It remains to derive the two estimates on $B$. For every $w\in Z$, the map $t\mapsto F(t,w)$ on $B$ satisfies the hypotheses of \cref{lem:smallball} with $Nd$, $d$, $2\Lambda$, and $\sigma_0/2$ in place of $M$, $r$, $\Lambda$, and $\sigma_*$, and this lemma gives \cref{eq:F3smallball} with a constant that does not depend on $w$. For the Lipschitz bound, let $w=(\iii,\jjj)$ and $w'=(\iii',\jjj')$ be distinct points of $Z$ and let $t\in B$. Each of the differences $p^t_\iii-p^t_{\iii'}$ and $p^t_\jjj-p^t_{\jjj'}$ vanishes if its two words coincide and is bounded by \cref{eq:prefix3} otherwise. If $\iii\ne\iii'$ and $\jjj\ne\jjj'$, this gives $|F(t,w)-F(t,w')|\le2C\lambda^{\min\{|\iii\land\iii'|,|\jjj\land\jjj'|\}}=2C\varrho_\lambda(w,w')$. If $\iii=\iii'$, then $\jjj\ne\jjj'$ and $|\jjj\land\jjj'|\le|\jjj|=|\iii|$, so that $\varrho_\lambda(w,w')=\lambda^{|\jjj\land\jjj'|}$ and $|F(t,w)-F(t,w')|\le C\varrho_\lambda(w,w')$; the case $\jjj=\jjj'$ is symmetric. The ball $B$, the restriction of $F$ to $B\times Z$, and the largest of $C'$ and the constants $C$ above therefore have the properties claimed, the remaining assertions having been established on $B_0\supseteq B$.
\end{proof}

The last step is the counting, and it is where the bound $\cmax<N^{-2/d}$ enters: \cref{lem:avoidance} needs the dimension $d$ of the target of the map $F$ of \cref{prop:transversal3} to exceed the covering exponent $2\log N/\log(1/\lambda)$ of the pair space in the metric $\varrho_\lambda$, which \cref{lem:paircover} computes, for some $\lambda\in(\cmax,1)$.

\begin{theorem} \label{thm:C3dense}
  Let $\Phi=(\fii_i)_{i \in \II}$ be a conformal iterated function system on $\Omega \subset \R^d$, $d \ge 3$. Suppose that no generator $\fii_i$ is a similarity, that $\dist(p_i,\Omega')>0$ and $b_i=\fii_i(\infty)\in\Omega$ for every $i\in\II$, and that
  \begin{equation} \label{eq:strongcontr}
    \cmax < \min\biggl\{\frac{1}{N^{2/d}},\frac{1}{1+K}\biggr\}, \qquad\text{where}\qquad K = \frac{R^2}{\max\{\rho_b,2/C_0\}^2},
  \end{equation}
  with $R$ and $\rho_b$ the constants of \cref{eq:polegeom}, which are finite and positive by \cref{lem:trapped}. Then $\Phi$ lies in the $d_2$-closure of $\mathcal G_T$.
\end{theorem}

\begin{proof}
  Fix $\lambda$ with $\cmax<\lambda<N^{-2/d}$, the interval being nonempty by \cref{eq:strongcontr}, and put $s=2\log N/\log(1/\lambda)$, so that $s<d$. The second bound in \cref{eq:strongcontr} is the hypothesis $\cmax(1+K)<1$ of \cref{prop:transversal3}, and we let $B$, $C$, and $F$ be the ball, the constant, and the map \cref{eq:F3def} that this proposition provides for $\lambda$. By \cref{lem:paircover} with $\beta=\lambda$, the space $Z$ is compact in the metric $\varrho_\lambda$ and, for every $\delta\in(0,1)$, covered by at most $2\lambda^{-s}\delta^{-s}$ sets of diameter at most $\delta$. Together with the two estimates of \cref{prop:transversal3}, this shows that $F$ satisfies the hypotheses of \cref{lem:avoidance} with $Y=Z$, with $Nd$ in place of $M$, and with $r=d$, $C_1=C_2=C$, and $C_3=2\lambda^{-s}$, and $r>s$, so the set of $t\in B$ for which $p^t_\iii=p^t_\jjj$ for some $(\iii,\jjj)\in Z$ is a null set. For every other $t\in B$ the system $\Phi_t$ is a member of $\mathcal S$ with $\max_{i \in \II}\|T_{\fii_i\circ\tau_{-t_i}}\|>0$ whose poles satisfy $p^t_\iii\ne p^t_\jjj$ for every $(\iii,\jjj)\in Z$, so $\Phi_t\in\mathcal G_T$ by \cref{lem:gap3}. Since the exceptional set is null, it contains no nonempty open subset of $B$. Hence there is a sequence $t_n\in B$ outside the exceptional set with $t_n\to0$, and $d_2(\Phi_{t_n},\Phi)\le C|t_n|\to0$. Thus members of $\mathcal G_T$ lie arbitrarily close to $\Phi$.
\end{proof}

\Cref{thm:C3dense} asks for three things that are not read off the generators on $\ol{\Omega}$: the position of the poles relative to $\Omega'$, that of the co-poles relative to $\Omega$, and the threshold $K$, into which every pole enters. The first is the boundedness of the generators on $\Omega'$, and the other two follow from a single lower bound on the pre-Schwarzians of the generators; this gives the form of the theorem stated in the introduction.

\begin{corollary} \label{cor:C3dense}
  Let $\Phi=(\fii_i)_{i \in \II}$ be a conformal iterated function system on $\Omega \subset \R^d$, $d \ge 3$, and let $\rho_\Phi=\min_{i \in \II}\dist(\fii_i(\ol{\Omega}),\partial\Omega)$. Suppose that every generator is bounded on $\Omega'$ and satisfies
  \begin{equation} \label{eq:C3neat}
    \|T_{\fii_i}\| > \frac{4}{\rho_\Phi}\sqrt{\frac{\cmax}{1-\cmax}},
  \end{equation}
  and that $\cmax<N^{-2/d}$. Then $\Phi$ lies in the $d_2$-closure of $\mathcal G_T$.
\end{corollary}

\begin{proof}
  We verify the hypotheses of \cref{thm:C3dense}. A similarity has vanishing pre-Schwarzian by \cref{lem:mobius}, so by \cref{eq:C3neat} no generator is a similarity, and each $\fii_i$ has a finite pole $p_i$, which lies off $\ol{\Omega}$, and a finite co-pole $b_i=\fii_i(\infty)$; in particular $\max_{i \in \II}\|T_{\fii_i}\|>0$ and \cref{lem:polerec} is available. Write $D_i=\dist(p_i,\ol{\Omega})$, choose $v_i\in\ol{\Omega}$ with $|v_i-p_i|=D_i$, and put $\theta_i=\cmax D_i/\rho_\Phi$. Since $T_{\fii_i}$ is the pole field of $p_i$, its supremum over $\ol{\Omega}$ is $\|T_{\fii_i}\|=2/D_i$, so \cref{eq:C3neat} reads
  \begin{equation*}
    \theta_i < \tfrac12\sqrt{\cmax(1-\cmax)} \le \tfrac14.
  \end{equation*}
  The pole lies off $\ol{\Omega'}$: a M\"obius transformation is continuous on $\RS^d$ and sends its pole to $\infty$, so a sequence in $\Omega'$ converging to $p_i$ would make $\fii_i$ unbounded on $\Omega'$. Hence $\dist(p_i,\Omega')>0$.

  The remaining hypotheses rest on one identity. For $x\in\R^d\setminus\{p_i\}$ and $u\in\R^d\setminus\{p_i\}$, the two-point identity \cref{eq:twopoint} for $\fii_i$, with $\lambda_{\fii_i}(u)=\rho_{\fii_i}/|u-p_i|^2$ from \cref{lem:mobius}, reads $|\fii_i(u)-\fii_i(x)|=\sqrt{\rho_{\fii_i}\lambda_{\fii_i}(x)}|u-x|/|u-p_i|$, and letting $u\to\infty$, with $\fii_i(u)\to b_i$ by the continuity on $\RS^d$, gives
  \begin{equation} \label{eq:copole}
    |\fii_i(x)-b_i| = \sqrt{\rho_{\fii_i}\lambda_{\fii_i}(x)} = \frac{\rho_{\fii_i}}{|x-p_i|},
  \end{equation}
  where $\rho_{\fii_i}=\lambda_{\fii_i}(x)|x-p_i|^2$ is the constant of \cref{lem:mobius}; evaluating it at $v_i$ gives $\rho_{\fii_i}\le\cmax D_i^2$. At $x=v_i$, \cref{eq:copole} gives $|\fii_i(v_i)-b_i|\le\cmax D_i=\theta_i\rho_\Phi<\rho_\Phi\le\dist(\fii_i(v_i),\partial\Omega)$, so the segment from $\fii_i(v_i)\in\Omega$ to $b_i$ does not meet $\partial\Omega$, whence $b_i\in\Omega$, and $\dist(b_i,\partial\Omega)\ge(1-\theta_i)\rho_\Phi$, the function $\dist(\,\cdot\,,\partial\Omega)$ being $1$-Lipschitz. Thus $\rho_b\ge(1-\theta)\rho_\Phi$, where $\theta=\max_{i \in \II}\theta_i=\cmax D/\rho_\Phi$ with $D=\max_{i \in \II}D_i$.

  It remains to bound the constant $K$ of \cref{eq:polegeom}. Since the co-poles lie in $\Omega$, \cref{lem:trapped} applies to $\Phi$. Let $\iii$ be a nonempty finite word. By \cref{eq:polerec}, $p_\iii=g_{i_1}(p_{\sigma\iii})$, where $p_{\sigma\iii}$ is $\infty$ or a point of $\R^d\setminus\Omega$, and $p_\iii$ is finite, by \cref{lem:trapped}\cref{it:trapped-poles}. If $p_{\sigma\iii}=\infty$, then $p_\iii=p_{i_1}$ and $\dist(p_\iii,\ol{\Omega})=D_{i_1}$. Otherwise $y=p_{\sigma\iii}$ is finite, $|y-b_{i_1}|\ge\dist(b_{i_1},\partial\Omega)\ge(1-\theta_{i_1})\rho_\Phi$, the segment from $b_{i_1}\in\Omega$ to $y\notin\Omega$ meeting $\partial\Omega$, and \cref{eq:copole} at $x=g_{i_1}(y)=p_\iii$ gives
  \begin{equation*}
    |p_\iii-p_{i_1}| = \frac{\rho_{\fii_{i_1}}}{|y-b_{i_1}|} \le \frac{\cmax D_{i_1}^2}{(1-\theta_{i_1})\rho_\Phi} = \frac{\theta_{i_1}D_{i_1}}{1-\theta_{i_1}},
  \end{equation*}
  so that $\dist(p_\iii,\ol{\Omega})\le D_{i_1}+|p_\iii-p_{i_1}|\le D_{i_1}/(1-\theta_{i_1})$. Since $t\mapsto t/(1-\cmax t/\rho_\Phi)$ is increasing on $[0,\rho_\Phi/\cmax)$, this gives $R\le D/(1-\theta)$, and hence
  \begin{equation*}
    K \le \frac{R^2}{\rho_b^2} \le \frac{D^2}{(1-\theta)^4\rho_\Phi^2} = \frac{\theta^2}{\cmax^2(1-\theta)^4} < \frac{1-\cmax}{4\cmax(1-\theta)^4} < \frac{1-\cmax}{\cmax},
  \end{equation*}
  the last inequality because $(1-\theta)^4\ge(3/4)^4>1/4$. Thus $\cmax(1+K)<1$, which together with $\cmax<N^{-2/d}$ is \cref{eq:strongcontr}, and \cref{thm:C3dense} applies.
\end{proof}

We are now ready to prove the third item of \cref{thm:genericity}.

\begin{proof}[Proof of \cref{thm:genericity}\cref{it:gen-higher}]
  The first assertion is \cref{ex:esc-not-mobius}(1). For the second, the set $\mathcal G_T$ is open in $(\mathcal S,d_2)$ by \cref{prop:cont3} and its members satisfy the strong exponential separation condition by \cref{lem:gap}, while its closure contains every system satisfying the hypotheses of \cref{cor:C3dense}. These are the hypotheses of \cref{it:gen-higher}: the norm $\|T_{\fii_i}\|$ is the supremum of $|T_{\fii_i}|$ over $\ol{\Omega}$ and $\rho_\Phi$ is the minimum written out in \cref{eq:genhigher}, so \cref{eq:C3neat} is \cref{eq:genhigher}.
\end{proof}

\subsection{The packing obstruction} \label{sec:gen-higher-obstruction}

On the line, \cref{thm:C1denseS} needs no hypothesis, and in the plane the hypotheses of \cref{thm:C2dense} concern the domain alone. The threshold of \cref{thm:C3dense} is of a different kind, a constraint on the system itself, and this subsection shows that some constraint of this kind is unavoidable: for $d\ge3$ the gap class $\mathcal G_T$, which by \cref{lem:gap} is the set of systems satisfying the hypothesis \cref{eq:mainhyp} of \cref{thm:main}, is not dense in general, and on convex domains it misses an open set of systems. The mechanism is the packing bound of \cref{lem:packing}. A system in $\mathcal G_T$ keeps the poles of any two infinite words with different first letters apart, and the finite area of the sphere then forces the generators to stretch it, in the chordal metric centered at any point of $\R^d$, by at least $N^{1/d}$. The systems that, for some center, stretch it less form an open set which the closure of $\mathcal G_T$ cannot meet, and on convex domains this set contains every tuple of homotheties of a common ratio above $N^{-1/d}$ whose centers lie close to one point. The obstruction thus sits at the exponent $1/d$ and the counting threshold of \cref{thm:C3dense} at $2/d$; the range between them, and what survives of the gap condition without any hypothesis, are taken up in \cref{sec:gen-higher-finite}. What is ruled out is the density of the gap condition, not that of the strong exponential separation condition itself, which we leave undecided.

For $\Phi=(\fii_i)_{i\in\II}\in\mathcal S$ and $c\in\R^d$, with the translation $\tau_c(x)=x+c$, we call
\begin{equation*}
  \varkappa_c(\Phi) = \max_{i \in \II}\sup_{x\in\RS^d}(\tau_{-c}\circ\fii_i\circ\tau_c)^{\#}(x)
\end{equation*}
the \emph{chordal stretch} of $\Phi$ centered at $c$. It is the maximal chordal factor of a generator for the chordal metric centered at $c$, that is, for the metric $(x,y)\mapsto\varsigma(x-c,y-c)$, and it is finite because each chordal factor is continuous on the compact $\RS^d$. Since the maps $\tau_{-c}\circ\fii_i\circ\tau_c$ are the generators of the conjugated system $\Phi^{\tau_c}$ of \cref{lem:similarity}, we have $\varkappa_c(\Phi)=\varkappa_0(\Phi^{\tau_c})$, where the right-hand side refers to the system space of $(\Omega-c,\Omega'-c)$.

The obstruction is a pigeonhole principle on the sphere. If, for a system $\Phi$, the poles of infinite words with different first letters were pairwise distinct, then for a center $c$, a fixed $\lll_0\in\II^\N$, and every $n$ the $N^n$ poles $p_{\iii\lll_0}$, $\iii\in\II^n$, would be pairwise separated, in the chordal metric centered at $c$, by at least a constant times $\varkappa_c(\Phi)^{-n}$, while the finite area of the sphere has room for only a constant times $\varkappa_c(\Phi)^{dn}$ such points.

\begin{lemma} \label{lem:packing}
  Let $N\ge2$. If $\Phi\in\mathcal S$ and $c\in\R^d$ satisfy $\varkappa_c(\Phi)<N^{1/d}$, then there are $\iii,\jjj\in\II^\N$ with $i_1\ne j_1$ and $p_\iii=p_\jjj$, and hence with $T_\iii\equiv T_\jjj$ on $\ol{\Omega}$. In particular,
  \begin{equation} \label{eq:packing}
    \varkappa_c(\Phi) \ge N^{1/d}
  \end{equation}
  for every $\Phi\in\mathcal G_T$ and $c\in\R^d$.
\end{lemma}

\begin{proof}
  By \cref{lem:similarity}, applied to the translation $\tau_c$, the system $\Phi^{\tau_c}$ lies in the system space of $(\Omega-c,\Omega'-c)$ and its poles are $p_\iii-c$, with $\infty-c=\infty$, while $\varkappa_0(\Phi^{\tau_c})=\varkappa_c(\Phi)$ by definition. Two poles of $\Phi$ therefore coincide if and only if the corresponding poles of $\Phi^{\tau_c}$ do, and, the pair of domains being arbitrary, it suffices to find the two words when $c=0$. Let now $c=0$ and $\varkappa_0(\Phi)<N^{1/d}$, and suppose, contrary to the claim, that $p_\iii\ne p_\jjj$ for all $\iii,\jjj\in\II^\N$ with $i_1\ne j_1$. Then $\max_{i \in \II}\|T_{\fii_i}\|>0$, since otherwise every pole is $\infty$ by convention and, as $N\ge2$, the words $111\cdots$ and $222\cdots$ share it; so \cref{lem:polerec} is available.

  Write $g_i=\fii_i^{-1}$ and recall that chordal factors multiply along compositions and that $g_i^{\#}(y)\fii_i^{\#}(g_i(y))=1$ for every $y\in\RS^d$; in particular $g_i^{\#}\ge1/\varkappa_0(\Phi)$ on $\RS^d$. By \cref{lem:polerec}, the sets $g_i(\widehat P)=\{p_{i\lll} : \lll\in\II^\N\}$, $i\in\II$, with $\widehat P=\{p_\iii : \iii\in\II^\N\}$, are compact, as continuous images of the compact $\widehat P$, and by the assumption they are pairwise disjoint, so that, with $\dist_\varsigma$ denoting the distance between sets in the chordal metric,
  \begin{equation*}
    \delta = \min\{\dist_\varsigma(g_i(\widehat P),g_j(\widehat P)) : i,j\in\II\text{ such that } i\ne j\} > 0.
  \end{equation*}
  Fix $\lll_0\in\II^\N$, let $n\ge1$, and put $q_\iii=p_{\iii\lll_0}=g_\iii(p_{\lll_0})$ for $\iii\in\II^n$, by \cref{eq:polerec}. We claim that these $N^n$ points are far apart in the chordal metric. Let $\iii,\jjj\in\II^n$ be distinct, let $\kkk=\iii\land\jjj$ and $m=|\kkk|\le n-1$, and let $a=i_{m+1}$ and $a'=j_{m+1}$, so that $a\ne a'$. Then, by \cref{eq:polerec} and the description of $g_a(\widehat P)$ in \cref{lem:polerec}, $q_\iii=g_\kkk(u)$ and $q_\jjj=g_\kkk(v)$ with $u=p_{(\sigma^m\iii)\lll_0}\in g_a(\widehat P)$ and $v=p_{(\sigma^m\jjj)\lll_0}\in g_{a'}(\widehat P)$, so that $\varsigma(u,v)\ge\delta$. The chordal two-point identity now gives $\varsigma(q_\iii,q_\jjj)=\sqrt{g_\kkk^{\#}(u)g_\kkk^{\#}(v)}\varsigma(u,v)\ge\varkappa_0(\Phi)^{-m}\delta$, the factor $g_\kkk^{\#}$ being a product of $m$ factors $g_{k_l}^{\#}$, each at least $1/\varkappa_0(\Phi)$. Since $\varkappa_0(\Phi)\ge1/\cmax>1$ by \cref{lem:secondfixed} and $m\le n-1$, the points $q_\iii$, $\iii\in\II^n$, are pairwise at chordal distance at least $\eta_n=\delta\varkappa_0(\Phi)^{-(n-1)}$.

  Under the stereographic identification $\varsigma$ is the Euclidean distance on the unit sphere of $\R^{d+1}$, so the open balls of radius $\eta_n/2$ about these points are pairwise disjoint, and each meets the sphere in a cap of $d$-dimensional area at least $\gamma_d\eta_n^d$, with $\gamma_d>0$ depending only on $d$, as $\eta_n\le\delta\le2$, the chordal diameter of $\RS^d$ being $2$. Indeed, a chordal ball of radius $r\le1$ corresponds to a spherical cap of angular radius $2\arcsin(r/2)$, whose $d$-dimensional area is bounded below by a constant depending only on $d$ times $r^d$. The sphere having finite area, $N^n\le C\eta_n^{-d}=C\delta^{-d}\varkappa_0(\Phi)^{d(n-1)}$ with $C$ independent of $n$. Taking $n$-th roots and letting $n\to\infty$ gives $N\le\varkappa_0(\Phi)^d$, contradicting $\varkappa_0(\Phi)<N^{1/d}$.

  This contradiction produces the two words for $c=0$ and hence, by the reduction, for the original system $\Phi$ and every center: if $\varkappa_c(\Phi)<N^{1/d}$, then there are $\iii,\jjj\in\II^\N$ with $i_1\ne j_1$ and $p_\iii=p_\jjj$. For such words $T_\iii\equiv T_\jjj$ on $\ol{\Omega}$, by \cref{lem:poledict} when $\max_{i \in \II}\|T_{\fii_i}\|>0$, both being the pole field of the common pole, and because every dual projection vanishes otherwise, as in the proof of \cref{lem:gap3}. Finally, if $\Phi\in\mathcal G_T$, then $\|T_\iii-T_\jjj\|\ge\Delta_T(\Phi)>0$ for every pair of infinite words with $i_1\ne j_1$, such a pair belonging to $Z$; so $\varkappa_c(\Phi)<N^{1/d}$ holds for no $c\in\R^d$, which is \cref{eq:packing}.
\end{proof}

The systems whose chordal stretch falls below $N^{1/d}$ for some center therefore lie outside $\mathcal G_T$. By \cref{lem:jet}, the chordal stretch is a continuous function of the values, derivatives, and pre-Schwarzians of the generators at a single point, so these systems form an open set, which on a convex domain contains the tuples of homotheties of a common ratio above $N^{-1/d}$ whose centers lie close to one point.

\begin{proposition} \label{prop:nondense}
  Let $d\ge3$ and $N\ge2$. The set
  \begin{equation*}
    \mathcal U = \bigcup_{c\in\R^d}\{\Psi\in\mathcal S : \varkappa_c(\Psi)<N^{1/d}\}
  \end{equation*}
  is open in $(\mathcal S,d_2)$, disjoint from the closure of $\mathcal G_T$, and contained in $\{\Psi\in\mathcal S : \cmax(\Psi)>N^{-1/d}\}$. If $\Omega$ is convex, $c\in\Omega$, and $N^{-1/d}<r<1$, then there is $\eps>0$, depending only on $r$, $N$, and $d$, such that the homotheties $\fii_i(x)=r(x-c_i)+c_i$, $i\in\II$, with $c_i\in\Omega$ and $|c_i-c|<\eps$ form a member of $\mathcal U$. In particular, $\mathcal G_T$ is not dense in $(\mathcal S,d_2)$ when $\Omega$ is convex.
\end{proposition}

\begin{proof}
  Fix $c\in\R^d$ and $x_0\in\Omega$. For $\Psi=(\psi_i)_{i\in\II}\in\mathcal S$ the M\"obius transformation $\tau_{-c}\circ\psi_i\circ\tau_c$ takes at $x_0-c$ the finite value $\psi_i(x_0)-c$, with derivative $D\psi_i(x_0)$ and, by \cref{eq:Tchain}, pre-Schwarzian $T_{\psi_i}(x_0)$, translations having the identity as derivative and vanishing pre-Schwarzian. If $\Psi_n=(\psi_{i,n})_{i\in\II}\to\Psi$ in $d_2$, then $\psi_{i,n}(x_0)$, $D\psi_{i,n}(x_0)$, and $T_{\psi_{i,n}}(x_0)$ converge to the corresponding data of $\psi_i$ for every $i$, the pre-Schwarzians by the trace formula \cref{eq:trace}, whose derivation applies to every conformal map; so $\varkappa_c(\Psi_n)\to\varkappa_c(\Psi)$ by \cref{lem:jet}, applied at the point $x_0-c$, a maximum of finitely many continuous functions being continuous. Thus $\varkappa_c$ is continuous on $(\mathcal S,d_2)$. For the lower bound, let $\Phi\in\mathcal S$. By \cref{lem:similarity}, $\Phi^{\tau_c}$ is a conformal iterated function system on $\Omega-c$ with extension domain $\Omega'-c$ and $\cmax(\Phi^{\tau_c})=\cmax(\Phi)$, so \cref{lem:secondfixed}, the pair of domains there being arbitrary, gives $\sup_{x\in\RS^d}(\tau_{-c}\circ\fii_i\circ\tau_c)^{\#}(x)\ge1/\cmax(\Phi)$ for every $i\in\II$, and hence $\varkappa_c(\Phi)\ge1/\cmax(\Phi)$.

  The set $\mathcal U$ is therefore open, as the union of the sets $\{\Psi\in\mathcal S : \varkappa_c(\Psi)<N^{1/d}\}$, $c\in\R^d$, each open by the continuity of $\varkappa_c$. It is disjoint from $\mathcal G_T$ by \cref{lem:packing}, hence, being open, from the closure of $\mathcal G_T$, and it lies in $\{\Psi\in\mathcal S : \cmax(\Psi)>N^{-1/d}\}$, since $\Psi\in\mathcal U$ gives $1/\cmax(\Psi)\le\varkappa_c(\Psi)<N^{1/d}$ for some $c\in\R^d$ by the lower bound.

  It remains to exhibit members of $\mathcal U$ on a convex domain. Let $\Omega$ be convex, let $c\in\Omega$ and $N^{-1/d}<r<1$, let $c_1,\ldots,c_N\in\Omega$, and put $\fii_i(x)=r(x-c_i)+c_i$. For $x\in\ol{\Omega}$ the point $\fii_i(x)=(1-r)c_i+rx$ lies in $\Omega$, being a convex combination of the interior point $c_i$ and the point $x$ of the closure with positive weight on the former, by \cite[Theorem~6.1]{Rockafellar}, the relative interior of the open convex set $\Omega$ being $\Omega$ itself. Moreover $\|D\fii_i\|=r<1$, and the pole of $\fii_i$ is $\infty\notin\Omega'$, so that $\fii_i$ is conformal on $\Omega'$; hence $\Phi=(\fii_i)_{i\in\II}\in\mathcal S$. The generators of $\Phi^{\tau_c}$ are the maps $x\mapsto rx+t_i$ with $t_i=(1-r)(c_i-c)$. Writing $u=rx+t_i$, so that $1+|x|^2=(r^2+|u-t_i|^2)/r^2$, and expanding $|u-t_i|^2$,
  \begin{equation*}
    (\tau_{-c}\circ\fii_i\circ\tau_c)^{\#}(x) = \frac{r^2+|u-t_i|^2}{r(1+|u|^2)} = \frac{r^2+|u|^2}{r(1+|u|^2)}+\frac{|t_i|^2-2\langle u,t_i\rangle}{r(1+|u|^2)} \le \frac{1+|t_i|^2+|t_i|}{r},
  \end{equation*}
  by $r^2\le1$ and $2|u|\le1+|u|^2$, and the bound persists at $x=\infty$ by continuity. Since $1/r<N^{1/d}$, every $\eps$ satisfying $0<\eps<(\sqrt{4rN^{1/d}-3}-1)/(2(1-r))$ also satisfies $(1+(1-r)^2\eps^2+(1-r)\eps)/r<N^{1/d}$, and then $|c_i-c|<\eps$ for every $i$ gives $\varkappa_c(\Phi)<N^{1/d}$, so that $\Phi\in\mathcal U$. The last claim follows, $\mathcal U$ being open, disjoint from the closure of $\mathcal G_T$, and nonempty, as the choice $c_1=\cdots=c_N=c$ with any $r\in(N^{-1/d},1)$ shows.
\end{proof}

By \cref{lem:packing}, every $\Psi\in\mathcal U$ fails the gap condition through a coincidence $T^\Psi_\iii\equiv T^\Psi_\jjj$ on $\ol{\Omega}$ between the dual projections of two infinite words with different first letters; for a tuple of similarities this is the trivial coincidence of vanishing projections. The upper endpoint for $\eps$ in the proof above tends to $\infty$ as $r\to1$, so $\eps$ may be taken larger than $\diam(\Omega)$ when $r$ is close enough to $1$, and every choice of centers in $\Omega$ is then admitted.

\subsection{The finite-word gap and the contraction range} \label{sec:gen-higher-finite}

\Cref{prop:nondense} leaves two questions, and this subsection takes up both. The first is how much of the gap condition survives without any hypothesis on the domains or on $\cmax$, and for finite words the answer is essentially all of it. For finite words, $p_\iii=p_\jjj$ says that $f_\iii\circ f_\jjj^{-1}$ fixes $\infty$, that is, that $f_\iii=A\circ f_\jjj$ for a similarity $A$, and by \cref{eq:frev} a common prefix $\kkk$ cancels from such a relation, $f_{\kkk\iii'}=A\circ f_{\kkk\jjj'}$ holding if and only if $f_{\iii'}=A\circ f_{\jjj'}$. The class of \cref{prop:finite} therefore consists of the systems in which no two distinct equal-length finite compositions differ by a similarity, which by \cref{lem:kernel} is the finite-word part of \cref{eq:mainhyp}, and it is dense in the open set of systems whose generators have pairwise distinct poles, with no hypothesis on $\Omega$, on $\Omega'$, or on $\cmax$. The perturbation that achieves this post-composes the generators by similarities, and it is available also where the perturbations of \cref{sec:gen-higher-dense} are not: it leaves the poles of the generators in place, so it asks nothing of $\Omega'$, and the poles of finite words then depend real-analytically on the perturbation parameter, so that a coincidence between two of them is an analytic condition, which fails for almost every parameter once it fails for one. 

The second question is the range of contraction ratios between the counting threshold $\cmax<N^{-2/d}$ of \cref{thm:C3dense} and the exponent $1/d$ at which \cref{prop:nondense} locates the obstruction. \Cref{rem:C3gen} examines the two bounds of \cref{eq:strongcontr} and shows by an explicit family that the threshold, under which \cref{thm:C3dense} places a system in the closure of $\mathcal G_T$, is not necessary for membership in $\mathcal G_T$ itself, just as the threshold of \cref{prop:gencrit} is not necessary for \cref{eq:mainhyp} by \cref{rem:gencrit-necessity}. \Cref{rem:C3domain} records how the density of $\mathcal G_T$ depends on the extension domain, and \cref{q:C3dense} states what remains open, which \cref{prop:finite} localizes to coincidences among the poles of infinite words.

\begin{proposition} \label{prop:finite}
  Let $d\ge3$ and $N\ge2$, let $\mathcal S^*$ be the set of $\Phi\in\mathcal S$ whose generators have pairwise distinct poles in $\RS^d$, and let $\mathcal G_T^{\mathrm{fin}}$ be the set of $\Phi\in\mathcal S$ with $p_\iii\ne p_\jjj$ for all $(\iii,\jjj)\in Z$ with $\iii,\jjj\in\II^*$. Then 
  \begin{equation*}
    \mathcal G_T\subseteq\mathcal G_T^{\mathrm{fin}}\subseteq\mathcal S^*,
  \end{equation*}
  the set $\mathcal S^*$ is open in $(\mathcal S,d_2)$, and $\mathcal G_T^{\mathrm{fin}}$ is dense in $\mathcal S^*$. More precisely, let $\mathcal A=(0,\infty)\times\mathrm{SO}(d)\times\R^d$ and, for $\beta=(\lambda,O,v)\in\mathcal A$, let $A_\beta(x)=\lambda Ox+v$. Then for every $\Phi\in\mathcal S^*$ there is a neighborhood $\mathcal N$ of the identity tuple $(1,I,0)$ in $\mathcal A^N$ such that $\Phi^\alpha=(A_{\alpha_i}\circ\fii_i)_{i\in\II}$ lies in $\mathcal S^*$ for every $\alpha=(\alpha_i)_{i\in\II}\in\mathcal N$, with $d_2(\Phi^\alpha,\Phi)\to0$ as $\alpha$ tends to the identity tuple, and the set of $\alpha\in\mathcal N$ with $\Phi^\alpha\notin\mathcal G_T^{\mathrm{fin}}$ has measure zero in every chart of $\mathcal A^N$.
\end{proposition}

\begin{proof}
  The inclusion $\mathcal G_T\subseteq\mathcal G_T^{\mathrm{fin}}$ is \cref{lem:gap3}, and $\mathcal G_T^{\mathrm{fin}}\subseteq\mathcal S^*$ because $(i,j)\in Z$ for all distinct $i,j\in\II$. The pole of a generator depends continuously on the system: fix $x_0\in\Omega$ and let $\Theta$ be the continuous map of the proof of \cref{lem:polerec} for this $x_0$. By \cref{eq:pole}, $T_{\fii_i}(x_0)=T_{p_i}(x_0)$ for every $i\in\II$, whether or not $\fii_i$ is a similarity, so that $p_i=\Theta(T_{\fii_i}(x_0))$, the computation behind \cref{eq:polerecover} applying to every finite $p_i\ne x_0$ and $\Theta(0)=\infty$ covering $p_i=\infty$. Moreover $T_{\fii_i}(x_0)$ is continuous in the $\mathcal C^2$-norm on $\ol{\Omega}$ by \cref{eq:trace}, whose derivation applies to every conformal map, $D\fii_i(x_0)$ being invertible with $\|D\fii_i(x_0)^{-1}\|\le\cmin^{-1}$. Hence $\Phi\mapsto p_i(\Phi)$ is continuous on $(\mathcal S,d_2)$ for every $i\in\II$, and $\mathcal S^*$ is open.

  Fix $\Phi\in\mathcal S^*$ and write $\alpha_i=(\lambda_i,O_i,v_i)$. Since $A_{\alpha_i}$ fixes $\infty$, the pole of $A_{\alpha_i}\circ\fii_i$ is $p_i$, so that the generator is conformal on $\Omega'$; moreover $\|D(A_{\alpha_i}\circ\fii_i)\|=\lambda_i\|D\fii_i\|\le\lambda_i\cmax$ on $\ol{\Omega}$ and $|A_{\alpha_i}(\fii_i(x))-\fii_i(x)|\le\|\lambda_iO_i-I\||\fii_i(x)|+|v_i|$. So for $\alpha$ in a neighborhood $\mathcal N$ of the identity tuple, on which $\lambda_i\cmax<1$ and $\|\lambda_iO_i-I\|\sup_{\ol{\Omega}}|\fii_i|+|v_i|<\rho_\Phi$ for every $i\in\II$, the system $\Phi^\alpha$ lies in $\mathcal S$, by the margin $\rho_\Phi=\min_{i \in \II}\dist(\fii_i(\ol{\Omega}),\partial\Omega)>0$, and in $\mathcal S^*$; and $d_2(\Phi^\alpha,\Phi)\le\max_i(\|\lambda_iO_i-I\|\|\fii_i\|_{\mathcal C^2(\ol{\Omega})}+|v_i|)$, which tends to $0$ as $\alpha$ tends to the identity tuple. Since $\Phi\in\mathcal S^*$ and $N\ge2$, at most one generator is a similarity, so that $\max_{i \in \II}\|T_{\fii_i}\|>0$, and the same holds for $\Phi^\alpha$, post-composition by a similarity preserving the pole.

  The inverse of $A_{\alpha_i}\circ\fii_i$ is $g_i\circ A_{\alpha_i}^{-1}$. For $\alpha\in\mathcal A^N$ and $\iii\in\II^n$ put
  \begin{equation*}
    p^\alpha_\iii = g_{i_1}\circ A_{\alpha_{i_1}}^{-1}\circ g_{i_2}\circ A_{\alpha_{i_2}}^{-1}\circ\cdots\circ g_{i_n}\circ A_{\alpha_{i_n}}^{-1}(\infty) = g_{i_1}\circ A_{\alpha_{i_1}}^{-1}\circ\cdots\circ A_{\alpha_{i_{n-1}}}^{-1}\circ g_{i_n}(\infty),
  \end{equation*}
  which is the point $(f^\alpha_\iii)^{-1}(\infty)$ for the composition $f^\alpha_\iii$ of the tuple $\Phi^\alpha$ of M\"obius transformations, by the unfolding of \cref{eq:polerec} in the proof of \cref{lem:polerec}, which uses only \cref{eq:frev}, and is therefore the pole of $\iii$ for $\Phi^\alpha$ whenever $\alpha\in\mathcal N$. It is a real-analytic map $\mathcal A^N\to\RS^d$: each $g_i$ is real-analytic on $\RS^d$, being rational in the charts given by the identity and by $\iota_0$, and $(\beta,x)\mapsto A_\beta^{-1}(x)$ is real-analytic on $\mathcal A\times\RS^d$. Near finite points it is given by $A_\beta^{-1}(x)=\lambda^{-1}O^{-1}(x-v)$, while in the chart $\iota_0$ at $\infty$ it reads $u\mapsto\lambda O^{-1}(u-|u|^2v)/(1-2\langle u,v\rangle+|u|^2|v|^2)$. For a pair $w=(\iii,\jjj)\in Z$ of finite words let $F_w(\alpha)=\varsigma(p^\alpha_\iii,p^\alpha_\jjj)^2$, a real-analytic function on $\mathcal A^N$, the square of the chordal distance being real-analytic on $\RS^d\times\RS^d$.

  We claim that $F_w$ does not vanish identically. The manifold $\mathcal A^N$ is connected because $\mathrm{SO}(d)$ is connected for $d\ge2$. Granting the claim, the zero set $E_w$ of $F_w$ is closed and null in every chart: on each connected component of a chart domain, the function $F_w$ does not vanish identically by the identity theorem, and the zero set of a real-analytic function $G$ that does not vanish identically on a connected open subset of $\R^n$ is Lebesgue null. This last fact is proved by induction on $n$. For $n=1$ the zeros are isolated. For $n\ge2$ cover the set by countably many open boxes $I\times J$ with $I\subseteq\R$ and $J\subseteq\R^{n-1}$, on none of which $G$ vanishes identically, by the identity theorem on the connected set. Fixing $x_1\in I$, there is $k\ge0$ for which $\partial_1^kG(x_1,\,\cdot\,)$ does not vanish identically on $J$, as otherwise every $G(\,\cdot\,,y)$ would vanish to infinite order at $x_1$ and hence on the interval $I$. So the set of $y\in J$ for which $G(\,\cdot\,,y)$ vanishes identically on $I$, being contained in the zero set of the real-analytic $\partial_1^kG(x_1,\,\cdot\,)$ on the connected $J$, is null by the induction hypothesis, while for every other $y$ the zeros of $G(\,\cdot\,,y)$ in $I$ are isolated; the zero set of $G$ in $I\times J$ is closed, hence measurable, and Fubini's theorem makes it null. Returning to the proposition, the set of $\alpha\in\mathcal N$ with $\Phi^\alpha\notin\mathcal G_T^{\mathrm{fin}}$ is the union of the sets $E_w\cap\mathcal N$ over the countably many pairs $w\in Z$ of finite words, hence null in every chart. Every other $\alpha\in\mathcal N$ gives $\Phi^\alpha\in\mathcal G_T^{\mathrm{fin}}$. Since every neighborhood of the identity tuple contained in $\mathcal N$ has positive Lebesgue measure in a chart, it contains such an $\alpha$; together with $d_2(\Phi^\alpha,\Phi)\to0$, this proves that $\mathcal G_T^{\mathrm{fin}}$ is dense in $\mathcal S^*$.

  To prove the claim, we drive the parameter to infinity along a curve of translations. Fix $v\in\R^d\setminus\{0\}$ and, for $s>0$, let $\alpha(s)\in\mathcal A^N$ be the tuple with $A_{\alpha(s)_i}=\tau_{sv}$ for every $i\in\II$, so that $A_{\alpha(s)_i}^{-1}=\tau_{-sv}$. Define $q_{n+1}(s)=\infty$ and $q_k(s)=g_{i_k}(q_{k+1}(s)-sv)$ for $k\in\{1,\ldots,n\}$, by downward recursion, reading $\infty-sv=\infty$, so that $p^{\alpha(s)}_\iii=q_1(s)$. Suppose first that no generator is a similarity. Then $q_n(s)=g_{i_n}(\infty)=p_{i_n}$ is finite, and if $q_{k+1}(s)\to p_{i_{k+1}}$, a finite point, as $s\to\infty$, then $q_{k+1}(s)-sv\to\infty$ and $q_k(s)\to g_{i_k}(\infty)=p_{i_k}$ by the continuity of $g_{i_k}$ on $\RS^d$. By downward induction $p^{\alpha(s)}_\iii\to p_{i_1}$, and likewise $p^{\alpha(s)}_\jjj\to p_{j_1}$, so that $F_w(\alpha(s))\to\varsigma(p_{i_1},p_{j_1})^2>0$, the poles $p_{i_1}$ and $p_{j_1}$ being distinct because $\Phi\in\mathcal S^*$ and $i_1\ne j_1$.

  If some generator $\fii_a$ is a similarity, it is the only one, and $g_a(x)=rOx+b$ with $r=1/\|D\fii_a\|>1$. We prove by downward induction, treating each maximal run of $a$ as one block, that $q_k(s)\to p_{i_k}$ in $\RS^d$ and $q_k(s)-sv\to\infty$ as $s\to\infty$. If $i_k\ne a$ and the assertion holds at $k+1$, then $q_{k+1}(s)-sv\to\infty$ gives $q_k(s)\to g_{i_k}(\infty)=p_{i_k}$; this pole is finite, so $q_k(s)$ is bounded for large $s$ and $q_k(s)-sv\to\infty$. Let now $i_k=\cdots=i_{k+m-1}=a$ be a maximal run. If the run ends the word, then $q_k(s)=\cdots=q_{k+m-1}(s)=\infty=p_a$, so both assertions hold throughout the run. Otherwise $i_{k+m}\ne a$ and the induction gives $q_{k+m}(s)\to p_{i_{k+m}}$, a finite point. Downward induction within the run then gives $q_{k+l}(s)=-s\sum_{j=1}^{m-l}(rO)^jv+O(1)$ for $l\in\{0,\ldots,m-1\}$ as $s\to\infty$. Consequently $q_{k+l}(s)-sv=-sS_{m-l}v+O(1)$ with $S_{m'}=\sum_{j=0}^{m'}(rO)^j=(I-(rO)^{m'+1})(I-rO)^{-1}$. The matrix $S_{m'}$ is invertible because every eigenvalue of $rO$ has modulus $r>1$, so neither $rO$ nor any positive power of it has $1$ as an eigenvalue. Since also $\sum_{j=1}^{m-l}(rO)^jv=rOS_{m-l-1}v\ne0$, we obtain $q_{k+l}(s)\to\infty=p_a$ and $q_{k+l}(s)-sv\to\infty$ throughout the run. This completes the induction, and the claim follows as before.
\end{proof}

\begin{remark} \label{rem:C3gen}
  The threshold \cref{eq:strongcontr} is not universal: the constants $R$ and $\rho_b$ of \cref{eq:polegeom} are features of the pole geometry of the particular system that $d$, $N$, and $\Omega$ do not bound, so the admissible range of $\cmax$ cannot be read off the ambient data. Its two bounds play different roles, $\cmax<N^{-2/d}$ being the counting threshold, under which the pair space has Minkowski dimension below $d$ in the metric of \cref{lem:paircover} used in the proof of \cref{thm:C3dense}, and $\cmax(1+K)<1$ making the balanced direction of \cref{prop:transversal3} a right inverse of the derivative of the pole difference. The counting threshold reads $\log N/\log(1/\cmax)<d/2$ and implies the bound $\dimm(X)<d/2$ on the Minkowski dimension of $X$, the $N^n$ sets $\fii_\iii(X)$, $\iii\in\II^n$, covering $X$ and having diameter at most $C_\Phi\lambda^n\diam(\ol{\Omega})$ for every $\lambda\in(\cmax,1)$ by \cref{eq:tube}; but the count runs over the pair space $Z$ and not over $X$, and whether the dimension bound $\dimm(X)<d/2$ alone suffices for the conclusion of \cref{thm:C3dense} we do not know. Both are invariant under similarities of $\R^d$ applied to $\Omega$, $\Omega'$, and the generators at once, as membership in $\mathcal G_T$ is by \cref{lem:similarity}: a similarity of ratio $\lambda$ preserves $\cmax$, multiplies $\max_{i \in \II}\|T_{\fii_i}\|$ by $\lambda$ and hence $2/C_0$ by $\lambda^{-1}$, and, sending poles to poles and co-poles to co-poles, multiplies $R$ and $\rho_b$ by $\lambda^{-1}$, so that $K$ is unchanged.

  \Cref{prop:nondense} shows that some hypothesis is necessary and locates an obstruction at the exponent $1/d$: on a convex domain the tuples of homotheties of a common ratio above $N^{-1/d}$ whose centers lie close to one point belong to the open set $\mathcal U$, which the closure of $\mathcal G_T$ misses, whereas $\mathcal U$ lies in $\{\Phi\in\mathcal S : \cmax>N^{-1/d}\}$, so that its closure, the function $\cmax$ being $1$-Lipschitz on $(\mathcal S,d_2)$, misses both the tuples of homotheties of ratio below $N^{-1/d}$ and the systems to which \cref{thm:C3dense} applies, which have $\cmax<N^{-2/d}$.

  The contraction threshold of \cref{thm:C3dense} is not necessary for membership in $\mathcal G_T$: for every $0<r<N^{-1/d}$ there are, on suitable domains, members of $\mathcal G_T$ with $\cmax$ arbitrarily close to $r$. Let $Q$ be the orthogonal map permuting the coordinate axes cyclically, $Qe_k=e_{k+1}$ for $k<d$ and $Qe_d=e_1$, let $z_i=(i-1)e_1+e_2$ for $i\in\II$, let $h_i$ be the affine map with $h_i^{-1}(y)=rQy+z_i$, and let $\fii_i=\iota_0\circ h_i\circ\iota_0$ with $\iota_0$ the unit inversion. Each $\fii_i$ is a M\"obius transformation fixing $0$ with $\lambda_{\fii_i}(0)=r$, the unit inversion being a chordal isometry and the chordal factor of $h_i$ at its fixed point $\infty$ being $r$, and with pole $\iota_0(z_i)$, of modulus at most $1$. So for $\Omega=B^o(0,\eta)$ and $\Omega'=B^o(0,2\eta)$ with $\eta$ small the tuple $\Phi=(\fii_i)_{i\in\II}$ is a conformal iterated function system, none of whose generators is a similarity, with $\cmax$ as close to $r$ as we please. Since $g_i=\iota_0\circ h_i^{-1}\circ\iota_0$ and $\iota_0(\infty)=0$, \cref{lem:polerec}, with its continuity clause for infinite words, gives $p_\iii=\iota_0(y_\iii)$ with $y_\iii=\sum_{l\le|\iii|}(rQ)^{l-1}z_{i_l}$ for every nonempty word, the series converging as $r<1$. The coefficient of $e_1$ in $(rQ)^{l-1}z_i$ is $r^{l-1}(i-1)$ when $l\equiv1$ modulo $d$, is $r^{l-1}$ when $l\equiv0$ modulo $d$, and vanishes otherwise. So for $(\iii,\jjj)\in Z$, the terms with $l\equiv0$ modulo $d$ agreeing because $|\iii|=|\jjj|$, the first coordinates of $y_\iii$ and $y_\jjj$ differ by at least $|i_1-j_1|-(N-1)\sum_{k\ge1}r^{dk}\ge1-(N-1)r^d/(1-r^d)>0$, as $Nr^d<1$; hence $y_\iii\ne y_\jjj$ and $p_\iii\ne p_\jjj$, the unit inversion being a bijection of $\RS^d$, and $\Phi\in\mathcal G_T$ by \cref{lem:gap3}.
\end{remark}

\begin{remark} \label{rem:C3domain}
  Whether $\mathcal G_T$ is dense in the open set $\{\Phi\in\mathcal S : \cmax<N^{-1/d}\}$, which the closure of $\mathcal U$ misses, we do not know in general; by \cref{prop:finite}, at a member of the open set $\mathcal S^*$ a failure of density can only come from coincidences of poles of infinite words. The localized perturbations of \cref{lem:localizedS} on the line and the peaking polynomials of \cref{lem:peaking} in the plane have no M\"obius analogue, the family being finite-dimensional.

  The density of $\mathcal G_T$ depends on $\Omega'$ through the closed set $\RS^d\setminus\Omega'$, in which every generator pole lies. The class $\mathcal G_T$ is empty when this set has fewer than $N$ points. When it has an isolated point $q$, consider the systems whose generators $1$ and $2$ have pole $q$, which exist for every $\Omega$: homotheties of small ratio about a point of $\Omega$ when $q=\infty$ and, for finite $q$, the maps $A\circ\iota_q$ with $A$ a similarity of small ratio carrying $\iota_q(x_1)$ to a point of $\Omega$ for some $x_1\in\Omega$. They form a nonempty open subset of $\mathcal S$ that misses $\mathcal G_T$ at the pair $(1,2)\in Z$: the poles depend continuously on the system, by the proof of \cref{prop:finite}, lie in $\RS^d\setminus\Omega'$ for every member of $\mathcal S$, and $q$ is isolated there, so that a pole equal to $q$ stays equal to $q$ under small perturbations. This open set meets $\{\Phi\in\mathcal S : \cmax<N^{-1/d}\}$, the ratio being small, and when $q=\infty$ it contains every tuple of homotheties, which is why the question below carries a hypothesis on $\Omega'$.
\end{remark}

\begin{question} \label{q:C3dense}
  Let $d\ge3$ and $N\ge2$, and suppose that the closed set $\RS^d\setminus\Omega'$ has no isolated points, as when $\Omega'$ is a ball. Is $\mathcal G_T$ dense in the open set $\{\Phi\in\mathcal S : \cmax<N^{-1/d}\}$? In particular, when $\Omega$ is convex, do the tuples of homotheties $x\mapsto r(x-c_i)+c_i$, $c_i\in\Omega$, with $r<N^{-1/d}$ belong to the $d_2$-closure of $\mathcal G_T$?
\end{question}

\section{Dimension of planar self-conformal sets and measures} \label{sec:dimension}

This section derives the dimensional consequences of the separation results, in the plane, where a theorem of Feng and Rapaport converts exponential separation into the sharp dimensions of a self-conformal set and of its self-conformal measures. In \cref{sec:dim-det} we record that theorem, together with the verifiable criterion of Feng and Rapaport for two of its geometric hypotheses, and in \cref{sec:dim-gen} we show that the criterion holds on a $\mathcal C^2$-open and dense subset of the systems with injective generators, so that the sharp dimensions of \cref{thm:dim} hold generically within this subspace. Throughout this section $d=2$, and we keep the conventions of \cref{sec:C2,sec:gen-plane}. In \cref{sec:dim-gen} we assume that $\Omega$ is a bounded Jordan domain and that $\Omega'$ is simply connected, as required by the planar density theorem.

\subsection{The theorem of Feng and Rapaport} \label{sec:dim-det}

For a positive probability vector $\mathbf p=(p_i)_{i\in\II}$, let $\mu_{\mathbf p}$ be the self-conformal measure of \cref{sec:cifs-measure}. Write 
\begin{equation*}
  H(\mathbf p)=-\sum_i p_i\log p_i \qquad\text{and}\qquad \chi(\mathbf p)=-\sum_i p_i\int\log|\fii_i'|\dd\mu_{\mathbf p}
\end{equation*}
for its entropy and Lyapunov exponent, respectively. The derivative bounds of \cref{sec:cifs} give $-\log\cmax \le \chi(\mathbf p) \le -\log\cmin$. We write $\dimh$ for Hausdorff dimension and $\dim(\mu_{\mathbf p})=\inf\{\dimh(A) : A\text{ is a Borel set such that } \mu_{\mathbf p}(A)>0\}$ for the exact-dimension of the self-conformal measure; see \cite{Falconer,FengHu}.

The chain rule gives $\|\fii_{\iii\jjj}'\|\le\|\fii_\iii'\|\|\fii_\jjj'\|$ for all $\iii,\jjj\in\II^*$, so subadditivity yields
\begin{equation} \label{eq:pressure}
  P(t) = \lim_{n\to\infty}\frac{1}{n}\log\sum_{\iii\in\II^n}\|\fii_\iii'\|^t = \inf_{n\ge1}\frac{1}{n}\log\sum_{\iii\in\II^n}\|\fii_\iii'\|^t
\end{equation}
for every $t\ge0$. The function $P$ is continuous and strictly decreasing, with $P(0)=\log N$ and $P(t+u)\le P(t)+u\log\cmax$ for $t,u\ge0$; its unique zero $\dimconf(\Phi)$ is the conformality dimension of $\Phi$; see \cite[Chapter~14]{BaranySimonSolomyak}.

The theorem of Feng and Rapaport \cite{FengRapaport} concerns conformal iterated function systems on a bounded domain $\Omega\subseteq\C$ whose generators $\fii_i$ are injective on $\ol\Omega$. By the classification in \cref{sec:conformal-maps}, a planar conformal map is holomorphic with nonvanishing derivative. Consequently, the derivative of a generator does not vanish on $\ol{\Omega}$, where $|\fii_i'|=\|D\fii_i\|$ and $0<\cmin\le|\fii_i'|\le\cmax<1$. Each generator also extends to an injection on a neighborhood of $\ol{\Omega}$. For every sufficiently small $\delta>0$, the set $\{z\in\C : \dist(z,\ol{\Omega})<\delta\}$ is a domain contained in $\Omega'$, since it is a union of discs centered on the connected set $\ol{\Omega}$. If $\fii_i$ were not injective on any such neighborhood, there would be points $x_n\ne y_n$ of $\Omega'$ with $\fii_i(x_n)=\fii_i(y_n)$ and $\dist(x_n,\ol{\Omega})+\dist(y_n,\ol{\Omega})\to0$. After passing to a subsequence, $x_n$ and $y_n$ would converge to points of $\ol{\Omega}$ with a common image and hence, by injectivity on $\ol{\Omega}$, to a single point $x$. Since $\fii_i'(x)\ne0$, the map $\fii_i$ would be injective on a ball about $x$ containing $x_n$ and $y_n$ for large $n$, a contradiction.

The geometric hypotheses in the theorem of Feng and Rapaport are the following. The system $\Phi$ \emph{preserves a point} $x_0\in\ol{\Omega}$ if every generator fixes $x_0$, and it \emph{preserves an analytic curve} $\Gamma$ if $\Gamma$ is a nonempty relatively closed embedded real-analytic one-dimensional submanifold of $\Omega$, not necessarily connected, with $\fii_i(\Gamma)\subseteq\Gamma$ for every $i\in\II$. It is \emph{holomorphically conjugate to a homothetic system} if there is an injective holomorphic map $h\colon\Omega\to\C$ such that each $h\circ\fii_i\circ h^{-1}\colon h(\Omega)\to\C$ is the restriction of a map $z\mapsto r_iz+b_i$ with $0\ne r_i\in\R$.

\begin{theorem} \label{thm:FR}
  Let $\Phi$ be a planar conformal iterated function system on $\Omega \subset \R^2$ whose generators are injective on $\ol{\Omega}$. If $\Phi$ satisfies the exponential separation condition, preserves neither a point nor an analytic curve, and is not holomorphically conjugate to a homothetic system, then
  \begin{equation*}
    \dimh(X)=\min\{2,\dimconf(\Phi)\}\qquad\text{and}\qquad\dim(\mu_{\mathbf p})=\min\biggl\{2,\frac{H(\mathbf p)}{\chi(\mathbf p)}\biggr\}
  \end{equation*}
  for every positive probability vector $\mathbf p$.
\end{theorem}

Feng and Rapaport prove the measure formula in \cite[Theorem~1.5]{FengRapaport} and derive the set formula in \cite[Corollary~1.6]{FengRapaport}. They also give a criterion guaranteeing the latter two geometric hypotheses. Our genericity argument verifies this criterion.

\begin{proposition} \label{prop:FR-criterion}
  Let $\Phi$ be a planar conformal iterated function system and, for a nonempty word $\iii\in\II^*$, write $x_\iii\in X$ for the fixed point of $\fii_\iii$. If $\fii_\iii'(x_\iii)\notin\R$ for some nonempty $\iii$, then $\Phi$ does not preserve an analytic curve and is not holomorphically conjugate to a homothetic system.
\end{proposition}

The proof of \cite[Corollary~1.7]{FengRapaport} shows the statement. A preserved analytic curve $\Gamma$ contains $X$, since it is relatively closed in $\Omega$ and contains the orbit $\fii_{\jjj|_n}(p)\to\pi(\jjj)$ of any of its points $p$, so it passes through $x_\iii$, where multiplication by $\fii_\iii'(x_\iii)$ would have to preserve its tangent line. A holomorphic conjugacy to a homothetic system preserves the multiplier $\fii_\iii'(x_\iii)$, whereas every homothety has real derivative. Whether the criterion is also necessary is asked in \cite{FengRapaport}.

\begin{remark} \label{rem:FR-genericity}
  Feng and Rapaport draw a genericity consequence from their criterion as well. Combining it with a transversality theorem of Solomyak and Takahashi \cite[Theorem~2.10]{SolomyakTakahashi}, they prove in \cite[Corollary~1.8]{FengRapaport} that for a real-analytic one-parameter family of holomorphic systems that is real or imaginary non-degenerate and has one composition with a non-real multiplier at its fixed point, the two dimension formulas hold for every parameter outside a subset of the interval of Hausdorff dimension zero. \Cref{thm:dim} below is a statement of a different kind, and neither implies the other: ours is topological rather than metric, it is relative to the whole space $\mathcal S'$ rather than to a curve inside it, and it carries no non-degeneracy hypothesis, while theirs applies to families that a $\mathcal C^2$-perturbation of the generators need not respect.
\end{remark}

\subsection{Generic dimensions} \label{sec:dim-gen}

Let $\mathcal S'$ be the subspace of $\mathcal S$ consisting of the systems whose generators are injective on $\ol{\Omega}$, where $\mathcal S$ is defined in \cref{sec:gen-plane}. Let $\mathcal{G}_T'$ be the set of systems $\Phi\in\mathcal G_T\cap\mathcal S'$ which preserve no point and for which $\fii_\iii'(x_\iii)\notin\R$ for some nonempty $\iii\in\II^*$; by \cref{lem:gap}\cref{it:gap-pre}, \cref{thm:main}\cref{it:main-pre}, and \cref{prop:FR-criterion}, every such system satisfies all hypotheses of \cref{thm:FR}. Define $\mathcal{G}_S'$ in the same way with $\mathcal G_S$ in place of $\mathcal G_T$. By \cref{eq:GSsubsetGT},
\begin{equation} \label{eq:GSprime-subset-GTprime}
  \mathcal{G}_S'\subseteq\mathcal{G}_T'.
\end{equation}
By \cref{lem:gap}\cref{it:gap-schw} and \cref{thm:main}\cref{it:main-schw}, every system in $\mathcal{G}_S'$ also satisfies the strong exponential separation condition modulo M\"obius maps.

\begin{theorem} \label{thm:dim}
  Let $\Omega$ be a bounded Jordan domain and let $\Omega'$ be simply connected. Then $\mathcal G_T\cap\mathcal S'$ and $\mathcal G_S\cap\mathcal S'$ are open and dense in $(\mathcal S',d_2)$. Moreover, $\mathcal{G}_T'$ and $\mathcal{G}_S'$ are open and dense in $(\mathcal S',d_2)$. Every $\Phi\in\mathcal{G}_T'$ satisfies the strong exponential separation condition, while every $\Phi\in\mathcal{G}_S'$ satisfies it modulo M\"obius maps. Every system in either class satisfies
  \begin{equation*}
    \dimh(X)=\min\{2,\dimconf(\Phi)\}\qquad\text{and}\qquad\dim(\mu_{\mathbf p})=\min\biggl\{2,\frac{H(\mathbf p)}{\chi(\mathbf p)}\biggr\}
  \end{equation*}
  for every positive probability vector $\mathbf p$. Consequently, the systems in $\mathcal S'$ satisfying these two dimension formulas contain a $d_2$-open and dense subset.
\end{theorem}

\begin{proof}
  The intersections $\mathcal G_T\cap\mathcal S'$ and $\mathcal G_S\cap\mathcal S'$ are open in $(\mathcal S',d_2)$ by \cref{prop:cont}. To prove their density, fix $\Phi\in\mathcal S'$, a number $\lambda\in(\cmax,1)$, and an integer $m>\log N/(2\log(1/\lambda))$, and let $B_1\subseteq\R^{2mN}$ be the ball, $\Phi_t=(g_i^t)_{i\in\II}$, $t\in B_1$, the systems, $C$ the constant, and $U$ the inner neighborhood of \cref{lem:collar} supplied by \cref{prop:transversal}. By \cref{thm:C2dense} there are $t$ arbitrarily close to $0$ with $\Phi_t\in\mathcal G_S$, and the bound $\|g_i^t-\fii_i\|_{\mathcal C^2(\ol{U})}\le C|t|$ of \cref{prop:transversal} shows that $\Phi_t\to\Phi$ in $\mathcal C^2$ on the neighborhood $U$ of $\ol{\Omega}$. We verify that these approximating systems belong to $\mathcal S'$ when $t$ is sufficiently small. By the observation preceding \cref{thm:FR}, we may replace $U$ by a smaller neighborhood of $\ol{\Omega}$ whose closure lies in the inner neighborhood of \cref{lem:collar}, so that every generator $\fii_i$ of $\Phi$ is injective on $U$. Choose constants $c>0$ and $B<\infty$ such that, for all sufficiently small $t$, $|(g_i^t)'|\ge c$ and $|(g_i^t)''|\le B$ on $\ol{U}$ for every $i\in\II$. Choose $\delta>0$ such that the closed $\delta$-neighborhood of $\ol{\Omega}$ lies in $U$ and $B\delta\le c$. If $x,y\in\ol{\Omega}$ and $|x-y|\le\delta$, then the segment from $x$ to $y$ lies in $U$, and Taylor's formula gives $|g_i^t(x)-g_i^t(y)|\ge(c-\tfrac12B|x-y|)|x-y|\ge\tfrac12c|x-y|$. On the compact set of pairs $(x,y)\in\ol{\Omega}^2$ with $|x-y|\ge\delta$, injectivity of $\fii_i$ gives a positive lower bound for $|\fii_i(x)-\fii_i(y)|$, uniformly in $i$, and the uniform convergence $g_i^t\to\fii_i$ preserves half of this bound for all sufficiently small $t$. Thus every $g_i^t$ is injective on $\ol{\Omega}$, so the approximants lie in $\mathcal G_S\cap\mathcal S'$. This proves the density of $\mathcal G_S\cap\mathcal S'$, and \cref{eq:GSsubsetGT} gives the density of $\mathcal G_T\cap\mathcal S'$.

  The sets $\mathcal{G}_T'$ and $\mathcal{G}_S'$ are open in $(\mathcal S',d_2)$. Indeed, $\mathcal G_T\cap\mathcal S'$ and $\mathcal G_S\cap\mathcal S'$ are open there, and the systems preserving a point form a closed set: if $\Phi_k\to\Phi$ in $d_2$ and $x_k$ is fixed by every generator of $\Phi_k$, then, after passing to a subsequence, the compactness of $\ol{\Omega}$ gives $x_k\to x\in\ol{\Omega}$, and uniform convergence gives $\fii_i(x)=x$ for every $i\in\II$. For a fixed nonempty word $\iii$, the compositions and their derivatives depend continuously on the system, uniformly on $\ol{\Omega}$, by induction on $|\iii|$. If $\Phi_k\to\Phi$, compactness and the uniqueness of the fixed point of $\fii_\iii$ therefore give $x_\iii(\Phi_k)\to x_\iii(\Phi)$, and the preceding continuity and the chain rule show that the corresponding multipliers converge. Thus the condition $\fii_\iii'(x_\iii)\notin\R$ is open for each $\iii$, and so is its union over the nonempty finite words.

  To prove density, let $\Phi\in\mathcal S'$ and $\eta>0$. By the first paragraph, there is $\Psi=(\psi_i)_{i\in\II}\in\mathcal G_S\cap\mathcal S'$ with $d_2(\Phi,\Psi)<\eta/3$. The margin $\min_{i\in\II}\dist(\psi_i(\ol{\Omega}),\partial\Omega)$ is positive, each $\psi_i(\ol{\Omega})$ being a compact subset of the open $\Omega$. Since $\mathcal G_S$ is open in $\mathcal S$, choose $r>0$ such that every member of $\mathcal S$ within $d_2$-distance $r$ of $\Psi$ belongs to $\mathcal G_S$, and decrease $r$ so that the image margin is retained by each of the similarity perturbations below whose $d_2$-distance from $\Psi$ is less than $r$. Let $y_2$ be the fixed point of $\psi_2$. Replacing $\psi_1$ by $\psi_1+\tau$, where $\tau$ is small and $\tau\ne y_2-\psi_1(y_2)$, produces a nearby system with no common fixed point. This remains true under sufficiently small further perturbations. Replace $\psi_2$ by $z\mapsto y_2+e^{i\alpha}(\psi_2(z)-y_2)$, whose multiplier at $y_2$ is $e^{i\alpha}\psi_2'(y_2)$; since $\psi_2'(y_2)\ne0$, an arbitrarily small nonzero $\alpha$ can be chosen with $e^{i\alpha}\psi_2'(y_2)\notin\R$. Both perturbations are post-compositions by injective similarities, so they preserve injectivity on $\ol{\Omega}$, as well as conformality and the contraction constants, and they can be chosen within $d_2$-distance $r$ of $\Psi$ and with total $d_2$-distance less than $2\eta/3$. The resulting system belongs to $\mathcal{G}_S'$ and lies within $\eta$ of $\Phi$, proving that $\mathcal{G}_S'$ is dense in $(\mathcal S',d_2)$. The inclusion \cref{eq:GSprime-subset-GTprime} gives the density of $\mathcal{G}_T'$.

  Finally, every $\Phi\in\mathcal{G}_T'$ lies in $\mathcal G_T$, so \cref{lem:gap,thm:main} show that $\Phi$ satisfies the strong, hence the plain, exponential separation condition, and the remaining hypotheses of \cref{thm:FR} hold by the definition of $\mathcal{G}_T'$ and \cref{prop:FR-criterion}. Hence \cref{thm:FR} gives both formulas for every positive probability vector $\mathbf p$. The inclusion \cref{eq:GSprime-subset-GTprime} gives the formulas for members of $\mathcal{G}_S'$, and \cref{lem:gap} gives their strong exponential separation condition modulo M\"obius maps. Thus both classes have all the asserted properties, and either is a required open and dense subset of the systems satisfying the dimension formulas.
\end{proof}

\begin{remark} \label{rem:injectivity-restriction}
  The restriction to $\mathcal S'$ in the density statement cannot be dropped: if a generator takes the same value at two points of $\Omega$, then by Rouch\'e's theorem \cite[p.~153]{Ahlfors} so does every generator sufficiently close to it in $d_2$, so the systems with a generator that is not injective on $\Omega$ form an open set, which the closure of $\mathcal{G}_T'$ misses. Such systems exist: on the convex $\Omega=\{z\in\C : |\re(z)|<1 \text{ and } |\im(z)|<4\}$, with $\Omega'=\{z\in\C : |\re(z)|<\frac{11}{10} \text{ and } |\im(z)|<\frac{41}{10}\}$, the pair $\fii_1(z)=\frac{e^z}{10}$ and $\fii_2(z)=\frac{z}{2}$ is a planar conformal iterated function system, the images $\fii_1(\ol{\Omega})$ and $\fii_2(\ol{\Omega})$ lying in $\Omega$ and the derivatives satisfying $|\fii_1'|\le \frac{e}{10}$ and $|\fii_2'|=\frac{1}{2}$ on $\ol{\Omega}$, whereas $\fii_1(i\pi)=\fii_1(-i\pi)=-\frac{1}{10}$ with $i\pi$ and $-i\pi$ in $\Omega$.
\end{remark}

\begin{example} \label{ex:C2dim}
  Let $\Phi=(\fii_1,\fii_2,\fii_3)$ be the system of \cref{ex:C2}. That example satisfies \cref{eq:mainhyp2}, so \cref{lem:gap} yields $\Phi\in\mathcal G_S$. The generators are injective on $\ol{\Omega}=B(0,1)$: the map $\fii_1$ is linear, and $\fii_2(z)-\fii_2(w)=(z-w)(10+z+w)/100$ and $\fii_3(z)-\fii_3(w)=(z-w)(10+i(z+w))/100$ vanish for $z,w\in B(0,1)$ only when $z=w$, since $|z+w|\le2$. The only fixed point of $\fii_1$ is $0$, whereas $\fii_3(0)=\frac{4}{5}$, so $\Phi$ preserves no point. If $x_3$ is the fixed point of $\fii_3$, then $\fii_3'(x_3)=(5+ix_3)/50$ and $|x_3-\frac{4}{5}|=|\frac{x_3}{10}+i\frac{x_3^2}{100}|\le\frac{11}{100}$, whence $\re(x_3)\ge\frac{69}{100}$ and $\im(\fii_3'(x_3))=\re(x_3)/50>0$. Thus $\Phi\in\mathcal{G}_S'$, the criterion holding with the word $3$ of length one.

  Moreover, $H(\mathbf p)\le\log3$ and $\chi(\mathbf p)\ge-\log\cmax=\log(\frac{50}{6})$, so we have $H(\mathbf p)/\chi(\mathbf p) \le \log3/\log\frac{50}{6} < 1$. The bound $\|\fii_i'\|\le\cmax=\frac{6}{50}$ and the $n=1$ term of \cref{eq:pressure} give $P(t)\le\log3+t\log\cmax$, hence $\dimconf(\Phi)\le\log3/\log(\frac{50}{6})<1$. Therefore \cref{thm:dim} gives 
  \begin{equation*}
    \dimh(X)=\dimconf(\Phi) \qquad \text{ and } \qquad \dim(\mu_{\mathbf p}) = \frac{H(\mathbf p)}{\chi(\mathbf p)}
  \end{equation*}
  for every positive probability vector $\mathbf p$.
\end{example}

\begin{declaration}
  In preparing this paper the author used large language models for language editing throughout, and in some cases to draft English prose from the author's own descriptions of the intended mathematical content. The models were also used to check proofs for correctness and to help locate references. All results and arguments are the author's own and have been verified independently of the models, and every reference has been checked by the author against the original source. The author takes full responsibility for the content of this paper.
\end{declaration}

\printbibliography

\end{document}